\documentclass{TPmod}

\usepackage{tikz-cd}
\usepackage{enumitem}
\usepackage{mathtools}
\usepackage{verbatim}
\usepackage{subcaption}
\usepackage[citestyle=alphabetic,bibstyle=alphabetic,backend=bibtex, maxbibnames=10,minbibnames=5]{biblatex}
\numberwithin{equation}{section}

\newtheorem{conjmain}[mainthm]{Conjecture}

\crefname{thm}{Theorem}{Theorems}
\crefname{prop}{Proposition}{Propositions}

\def\cG{\mathcal{G}}
\def\cL{\mathcal{L}}

\def\hookar{\ar@{^{(}->}}

\def\dbdg{D^b_{\mathsf{dg}}}
\def\val{\mathrm{val}}

\def\LS{\mathcal{S}}

\def\ainf{A_\infty}

\def\fuk{\EuF}
\def\bc{\mathsf{bc}}

\def\dm{r}

\def\Nef{\mathsf{N}}

\def\eps{\epsilon}
\def\sF{\mathbf{F}}
\def\cY{\mathcal{Y}}
\def\cI{\mathcal{I}}

\def\cP{\mathcal{P}}
\def\mafuk{\mathcal{F}_{mon}}
\def\mir0{F_0}
\def\lf{\lambda}
\def\ess{S_{0}}
\def\Sprime{S}
\def\aye{I_0}
\def\Iprime{I}
\def\fo{F_0}
\def\eff{F}

\newcommand{\fm}{\mathfrak{m}}

\renewcommand{\Im}{\textup{Im}}

\DeclareMathOperator{\Log}{Log}
\DeclareMathOperator{\Nbd}{Nbd}

\DeclareMathOperator{\st}{star}

\DeclareMathOperator{\Obs}{Obs}
\newcommand{\calC}{\mathcal{C}}

\def\and{\, \& \,}

\def\cH{\mathcal{H}}
\def\nov{r}

\renewcommand{\vec}[1]{\mathsf{#1}}

\def\cW{\mathcal{W}}
\def\cO{\mathcal{O}}
\def\cL{\mathcal{L}}
\def\cF{\mathcal{F}}
\def\cA{\mathcal{A}}
\def\cB{\mathcal{B}}
\def\cC{\mathcal{C}}
\def\cD{\mathcal{D}}

\def\NE{\mathrm{NE}}

\def\Perf{\operatorname{Perf}}
\def\Hom{\mathrm{Hom}}
\def\HH{\mathrm{HH}}
\def\charac{\mathrm{char}\,}

\def\Ldom{B}
\def\Xizero{P}
\def\Xinonzero{P_0}
\def\kaehl{\kappa}
\def\tee{\mathbf{t}}
\renewcommand{\Re}{\operatorname{Re}}
\renewcommand{\Im}{\operatorname{Im}}
\newcommand{\W}{\mathcal W}
\newcommand{\ZZ}{\mathbb Z}
\newcommand{\vX}{\vec X}
\newcommand{\Lconic}{\tilde L_{cyl}}
\newcommand{\Lmon}{\tilde L_{mon}}
\newcommand{\Lmc}{\tilde L_{mc}}

\newcommand{\tildeoi}{\mathcal{E}_i}
\begin{document}

\title{Homological mirror symmetry for Batyrev mirror pairs}

\author{Sheel Ganatra, Andrew Hanlon, Jeff Hicks, Daniel Pomerleano, and Nick Sheridan}

\begin{abstract} We prove Kontsevich's homological mirror symmetry conjecture for a large class of mirror pairs of Calabi--Yau hypersurfaces in toric varieties. These mirror pairs were constructed by Batyrev from dual reflexive polytopes. The theorem holds in characteristic zero and in all but finitely many positive characteristics.  \end{abstract}
\maketitle

\section{Introduction}
\label{sec:bat}
The earliest constructions of conjectural mirror pairs concern combinatorially defined Calabi--Yau subvarieties of toric varieties. Batyrev conjectured that, if $\Delta$ and $\Delta^*$ are dual reflexive polytopes, then the (crepant resolutions of) anticanonical hypersurfaces in the corresponding toric varieties should be mirror to each other \cite{Batyrev1993}. Batyrev's work generalized Greene--Plesser's earlier construction of conjectural mirror pairs of hypersurfaces in quotients of weighted projective spaces \cite{Greene1990}; and it was subsequently generalized by Batyrev--Borisov's construction of conjectural mirror pairs of complete intersections in toric varieties \cite{Batyrev1994}, which in turn was generalized by Gross--Siebert's construction of conjectural mirror pairs of toric degenerations \cite{GrossSiebertI,GrossSiebertII}. Various other related constructions have also extended mirror symmetry beyond the Calabi--Yau setting.

Kontsevich's homological mirror symmetry conjecture \cite{kontsevich1995homological} posits an equivalence between natural algebro-geometric and symplectic categories associated to mirror pairs. Although homological mirror symmetry has been highly influential in algebraic and symplectic geometry, it has been proved only for a small number of cases in the original setting of mirror pairs of compact Calabi--Yau varieties. Namely, homological mirror symmetry has been proved for elliptic curves in \cite{polishchuk1998categorical,polishchuk2004structures}, higher-dimensional abelian varieties in \cite{kontsevich2001homological,fukaya2002mirror,abouzaid2010homological,abouzaid2021homological}, quartics in $\C\mathbb{P}^3$ in \cite{Seidel2003}, Calabi--Yau hypersurfaces in higher dimensional projective spaces in \cite{sheridan2015homological}, and for Greene--Plesser mirror pairs, together with certain Batyrev--Borisov mirror pairs (`generalized Greene--Plesser mirror pairs', which are complete intersections in quotients of products of weighted projective spaces) in \cite{sheridan2021homological}. In the present work, we prove homological mirror symmetry for a large new class of Batyrev mirror pairs. One noteworthy feature of our result is that it holds not only in characteristic zero, but also in all but finitely many positive characteristics.

The rest of the introduction is organized as follows: in \cref{subsec:tordata}, we introduce the toric data on which our construction of Batyrev mirror pairs depends; and in \cref{subsec:Aintro,subsec:Bintro}, we construct the symplectic and algebraic sides of the mirror pair. We then state the homological mirror symmetry conjecture for Batyrev mirrors in \cref{subsec:hmsnomap} and our HMS statement for Batyrev mirrors (\cref{main:higherdim}) in  \cref{subsec:mainthm}. \Cref{subsec:overview} gives an overview of the proof of \cref{main:higherdim}; we also include a table of notation for the reader's convenience.

\subsection{Toric data}
\label{subsec:tordata}

We start by introducing some notation used in the Batyrev construction, following \cite{Batyrev1993,coxkatz}. 
Let:
\begin{itemize}
\item $\Bbbk$ be a field;
\item $M \cong \Z^n$ be a lattice;
\item $\Delta \subset M_\R$ a reflexive polytope, $\Delta^* \subset M^*_\R$ its polar dual;
\item $\Sigma$ be the fan in $M^*_\R$ dual to $\Delta$ (its rays point along the vertices of $\Delta^*$), and $\Sigma^*$ the fan in $M_\R$ dual to $\Delta^*$;
\item $Y$ be the toric $\C$-variety corresponding to $\Sigma$, and $Y^*$ the toric $\Bbbk$-variety corresponding to $\Sigma^*$. 
\end{itemize}

Note that there is a toric line bundle $\cL_{\Delta^*} \to Y^*$, with a basis of sections indexed by the set $\Xinonzero := \Delta^* \cap M^*$. 
We denote the section corresponding to $\vec{p} \in \Xinonzero$ by $z^{\vec{p}}$; we similarly denote the section of $\cL_{\Delta} \to Y$ corresponding to $\alpha \in A:=\Delta \cap M$ by $z^\alpha$. 

Let $\Xizero \subset \Xinonzero$ be the subset of lattice points lying on a face of $\Delta^*$ which has codimension $\ge 2$. 
Let $\kaehl \in (\R_{>0})^{\Xizero}$ and $K \subset \R$ be a subgroup which contains $\kaehl_{\vec{p}}$ for all $\vec{p} \in \Xizero$. 
Our construction will depend on the data $(\Bbbk,M,\Delta,\kaehl,K)$. 

\subsection{Symplectic construction}
\label{subsec:Aintro}

Consider the following sections of $\cL_{\Delta} \to Y$:
$$W_0 := z^{\vec{0}},\qquad W_1 := \sum_{\alpha \in A \setminus \{0\}}  z^{\alpha}.$$
If $\Sigma^*$ is simplicial, then for sufficiently large $t \in \R_{>0}$, the hypersurface 
$$X_t = \{tW_0 = W_1\} \subset Y$$
intersects each toric stratum smoothly (this is proved using ideas from tropical geometry, as we explain in Section \ref{sec:trop_setup}). 
When $Y$ is not smooth, the hypersurface $X_t$ need not be smooth; we need therefore to take a crepant resolution of $Y$. 
The data for this is determined by $\kaehl$. 

Let $\psi_\kaehl: M^*_\R \to \R$ be the largest convex\footnote{We use the convention that $f(x)=x^2$ is convex, in contrast to some of the literature on toric varieties (and \cite{Ganatra2023integrality}).} function such that $\psi_\kaehl(t \cdot \vec{p}) \le t\cdot \kaehl_{\vec{p}}$ for all $t \in \R_{\ge 0}$ and all $\vec{p} \in \Xizero$. 
The decomposition of $M^*_\R$ into domains of linearity of $\psi_\kaehl$ induces a fan $\Sigma_\kaehl$. 

\begin{defn}
\label{defn:mpcp}
We say that the \emph{MPCP condition} holds if $\Sigma_\kaehl$ is a projective simplicial refinement of $\Sigma$ whose rays are generated by the elements of $\Xizero$.
\end{defn}

\begin{rmk}
MPCP stands for Maximal Projective Crepant Partial desingularization (see \cite[Definition 2.2.13]{Batyrev1993}). 
In the language of \cite[Section 6.2.3]{coxkatz}, $\Sigma_\kaehl$ is called a \emph{simplified projective subdivision} of $\Sigma$, and the MPCP condition holds if and only if $\kaehl$ lies in the interior of a top-dimensional cone $\text{cpl}(\tilde{\Sigma}_\kaehl)$ of the secondary fan (or Gelfand--Kapranov--Zelevinskij decomposition) associated to $\Xizero \subset M_\R$, where $\Sigma_\kaehl$ is a projective simplicial refinement of $\Sigma$ whose rays are generated by $\Xizero$.
\end{rmk}

The morphism of fans $\Sigma_\kaehl \to \Sigma$ induces a toric morphism corresponding to the resolution $Y_\kaehl \to Y$. 
Let $X_{t,\kaehl} \subset Y_\kaehl$ be the proper transform of $X_t$. 
The intersection of $X_{t,\kaehl}$ with each toric stratum is smooth; so if the MPCP condition holds, then $X_{t,\kaehl}$ is a maximal projective crepant partial desingularization of $X_t$ (hence the name of the condition). 

Observe that even if the MPCP condition holds, $Y_\kaehl$ may have finite quotient singularities since $\Sigma_\kaehl$ is only assumed to be simplicial. 
Therefore, $X_{t,\kaehl}$ may also have finite quotient singularities. 
We would like to understand when $X_{t,\kaehl}$ is in fact smooth. 
Let $Y_{\kaehl,0} \subset Y_\kaehl$ denote the pre-image of the toric fixed points of $Y$. 
Observe that $X_{t,\kaehl}$ avoids $Y_{\kaehl,0}$. 

\begin{defn}
We say that the \emph{MPCS condition} holds if the MPCP condition holds, and, furthermore, $Y_\kaehl$ is smooth away from $Y_{\kaehl,0}$. 
MPCS stands for Maximal Projective Crepant Smooth desingularization.
\end{defn}

We remark that the MPCS and MPCP conditions are equivalent when $\dim_\C(Y_\kaehl) \le 4$ (see \cite[\S 2.2]{Batyrev1993}).
If the MPCS condition holds, then $X_\kaehl$ avoids the non-smooth locus of $Y_\kaehl$, and if $\Sigma^*$ is simplicial, then $X_{t,\kaehl}$ is in fact smooth (again using tropical geometry), and Calabi--Yau by \cite[Theorem 4.1.9(i) and Proposition 3.2.2]{Batyrev1993}.

\begin{defn}
The simplicial fan $\Sigma_\kaehl$ determines a subdivision of $\Delta^*$ into simplices. 
We say that the \emph{condition on characteristic} is satisfied if $\charac(\Bbbk)$ does not divide 
\begin{itemize}
    \item the affine volume of any of the simplices in the subdivision of $\Delta^*$ determined by $\Sigma_\kaehl$;\footnote{We normalize the affine volume so that the minimum possible non-zero affine volume of a lattice simplex is $1$.}
    \item the order of the finite abelian group $M^*/L$, where $L$ is the subgroup generated by $P$.
\end{itemize}
\end{defn}

We denote the toric boundary divisor of $Y_\kaehl$ by $D^Y$. 
Note that it has irreducible components $D^Y_{\vec{p}}$ indexed by $\vec{p} \in \Xizero$. 
Let $D^Y_\kaehl := \sum_{\vec {p} \in \Xizero} \kaehl_{\vec{p}} \cdot D^Y_{\vec{p}}$ be the toric $\R$-Cartier divisor with support function $\psi_\kaehl$. 
Because $\psi_\kaehl$ is strictly convex (by our assumption that its domains of linearity are the cones of the fan $\Sigma_\kaehl$), this divisor is ample, so the first Chern class of the corresponding line bundle is represented in de Rham cohomology by a toric orbifold K\"ahler form $\omega^Y_\kaehl$ (see the discussion in \cite[\S 4]{Aspinwall1993b}). 
We denote its restriction to $X_{t,\kaehl}$ by $\omega_\kaehl$. 
Because $X_{t,\kaehl}$ avoids the non-smooth locus of $Y_\kaehl$, $\omega_\kaehl$ is an honest K\"ahler form. 
Its cohomology class is Poincar\'e dual to $\sum_{\vec {p} \in \Xizero} \kaehl_{\vec{p}} \cdot [D_\vec{p}]$, where $D_\vec{p} := X_{t,\kaehl} \cap D^Y_{\vec{p}}$. 

We now introduce the following notation: for any face $S$ of $\Delta$ (or $\Delta^*$), we define $i(S)$ to be the number of interior points on that face.  Furthermore, we define $S^*$ to be the face of $\Delta^*$ (or $\Delta$) which is dual to $S$.

\begin{lem}\label{lem:connectedness}
For $\vec{p} \in \Xinonzero$, let $S_\vec{p}$ denote the face of $\Delta^*$ containing $\vec{p}$ in its interior. Then
$$\#(\text{connected components of 
$D_{\vec{p}}$}) = \left\{ 
    \begin{array}{rl} 0 & \text{if $\dim(S^*_{\vec{p}}) = 0$}\\ 
        i(S^*_{\vec{p}}) & \text{if $\dim(S^*_{\vec{p}}) = 1$}\\ 
        1 & \text{otherwise.} 
\end{array} \right.  $$
\end{lem}
\begin{proof}
The map $Y_\kaehl \to Y$ blows $D^Y_{\vec{p}}$ down to the stratum $Y_{\vec{p}}$ corresponding to $S^*_{\vec{p}}$. Then, $D_{\vec{p}}$ is the preimage of $X \cap Y_{\vec{p}}$. This is an ample hypersurface and therefore connected unless $Y_{\vec{p}}$ is 0-dimensional (in which case it is empty) or $1$-dimensional (in which case it is a union of $i(S^*_{\vec{p}})$ points). 
The preimage of a connected set under a toric morphism is connected, so the result follows.
\end{proof}

\begin{defn}
We say that the \emph{connectedness condition} is satisfied if for every one-dimensional face $S^*$ of $\Delta^*$, either $i(S^*) = 0$ (i.e., $S^*$ has affine length $1$) or $i(S) = 0$. 
\end{defn}

\subsection{Algebraic construction}
\label{subsec:Bintro}

We work over a universal Novikov field:
\begin{equation} \Lambda_{\Bbbk,K} := \left\{ \sum_{j=0}^\infty c_j \cdot q^{\kaehl_j}: c_j \in \Bbbk, \kaehl_j \in K, \lim_{j \to \infty} \kaehl_j = +\infty\right\}.\end{equation}
As $\Bbbk$ and $K$ will be fixed throughout the paper, we usually omit them from the notation. Note that $\Lambda$ is a field extension of $\Bbbk$, and if $\Bbbk$ is algebraically closed and $K$ is divisible, it is algebraically closed. Moreover, it has a valuation
\begin{align} 
\val: \Lambda & \to K \cup \{\infty\} \\
\label{eqn:valdef} \val\left(\sum_{j=0}^\infty c_j \cdot q^{\kaehl_j}\right) &:= \min_j\{\kaehl_j: c_j \neq 0\}.
\end{align}
 
We consider sections of $\cL_{\Delta^*} \to Y^*$:
\begin{equation}
\label{eqn:Wb}
W_\dm(z) := -z^{\vec{0}} + \sum_{\vec{p} \in \Xizero} \dm_{\vec{p}} \cdot z^{\vec{p}},
\end{equation}
for $\dm =(\dm_{\vec{p}}) \in \mathbb{A}^{\Xizero}_\Lambda$. 
We define $X^*_\dm := \{W_\dm = 0\} \subset Y^*_\Lambda$.

\begin{rmk}
Observe that we have a correspondence
\begin{equation} 
\text{monomial $z^{\vec{p}}$ of $W_\dm$} \leftrightarrow \text{divisor $D_\vec{p} \subset X$}.
\end{equation} 
This correspondence is called the `monomial--divisor mirror map' (see \cite{Aspinwall1993b}).
\end{rmk}

\subsection{HMS: conjecture}
\label{subsec:hmsnomap}

On the $B$-side of mirror symmetry, we consider the $\Lambda$-linear $\Z$-graded dg category of coherent complexes $D^b_{dg}Coh(X^*_\dm)$.\footnote{By the results of \cite{orlovlunts2009, stellariuniqueness}, dg-enhancements of the derived category of coherent sheaves are unique up to quasi-equivalence.} On the $A$-side of mirror symmetry, we consider (a version of) the Fukaya category $\fuk(X_{t,\kaehl},\omega_\kaehl).$ We note that most definitions of the Fukaya category work over ground fields of characteristic zero, whereas we consider mirror symmetry over fields of arbitrary characteristic. We give more details about the version of the Fukaya category we are working with below and in  \cref{subsec:relativeFukayaCategory}. For now, we simply recall that the Fukaya category may be curved, i.e., it may have non-vanishing $\mu^0$ and therefore not be an honest $A_\infty$ category. 
Therefore, we consider the version whose objects are bounding cochains on objects of the Fukaya category, which we denote by $ \fuk(X_{t,\kaehl},\omega_\kaehl)^\bc$. 
It is a $\Lambda$-linear $\Z$-graded cohomologically unital $A_\infty$ category, whose idempotent completed pre-triangulated closure we denote $\Perf\fuk(X_{t,\kaehl},\omega_\kaehl)^\bc$.
One part of Kontsevich's homological mirror symmetry conjecture for Batyrev mirrors then reads:

\begin{conjmain}
\label{conj:hgpms}
Suppose $\Sigma^*$ is simplicial. There is a quasi-equivalence of $\Lambda$-linear $\Z$-graded $A_\infty$ categories
\begin{equation} 
\Perf \fuk(X_{t,\kaehl},\omega_\kaehl)^\bc \simeq D^b_{dg}Coh(X^*_{b(\kaehl)}),
\end{equation}
for some $b(\kaehl) \in \mathbb{A}^{\Xizero}_\Lambda$ with $\val(b(\kaehl)_{\vec{p}}) = \kaehl_{\vec{p}}$.
\end{conjmain}

\subsection{HMS: statement}
\label{subsec:mainthm}
The version of the Fukaya category that we work with is the relative Fukaya category, $\fuk(X_{t,\kaehl},D,\kaehl;\Lambda_{\Bbbk,K})$, which is constructed in \cite[Section 1.4]{perutz2022constructing}.
\begin{main}
\label{main:higherdim}
Suppose that $\operatorname{dim}(X_{t,\kaehl}) \geq 3$, the condition on characteristic is satisfied, $\Sigma^*$ is smooth, and the MPCS and connectedness conditions hold. Then: 
\begin{enumerate} 
    \item \label{it:HMS_embed} There is a fully-faithful embedding $D^b_{dg}Coh(X^*_{b(\kaehl)})\hookrightarrow \Perf \fuk(X_{t,\kaehl},D,\kaehl;\Lambda_{\Bbbk,K})^\bc.$  

    \item \label{it:HMS_generate} Suppose the relative Fukaya category $\fuk(X_{t,\kaehl},D,\kaehl;\Lambda_{\Bbbk,K})^{\bc}$ satisfies the hypotheses enumerated in \cite[Section 2.5]{sheridan2021homological}. Then Conjecture \ref{conj:hgpms} holds with $\fuk(X_{t,\kaehl},\omega_\kaehl)=\fuk(X_{t,\kaehl},D,\kaehl;\Lambda_{\Bbbk,K}).$ 
\end{enumerate}
\end{main}
Various hypotheses appear in Theorem \ref{main:higherdim} that are not present in Conjecture \ref{conj:hgpms}. We now provide a short discussion of these hypotheses and where they are used in our proof. 

\begin{rmk}
Theorem \ref{main:higherdim} holds for any version of the Fukaya category which satisfies the hypotheses enumerated in \cite[Section 2.5]{sheridan2021homological}. These hypotheses are expected to hold in whatever version of the Fukaya category one works with (e.g., the general version defined in \cite{FOOO,fukaya2017unobstructed}). 
It is shown in \cite{perutz2022constructing} that the relative Fukaya category satisfies some of the required properties; the remaining ones will be verified in forthcoming work. 
These remaining hypotheses are only needed in the proof of split-generation of the Fukaya category by the image of the embedding from \cref{main:higherdim} \eqref{it:HMS_embed}.
\end{rmk}

\begin{rmk}
One might hope that the MPCS condition could be replaced by the weaker MPCP condition in Theorem \ref{main:higherdim} and that the proof would go through with minimal changes. 
In that case, $(X_{t,\kaehl},\omega_\kaehl)$ would be a symplectic orbifold, so the definition of the Fukaya category would need to be adjusted accordingly (compare to \cite{Cho2014}). 
\end{rmk}

\begin{rmk}\label{rmk:nonamb}
If the connectedness condition does not hold, then we need something more to prove Conjecture \ref{conj:hgpms}. 
That is because there is more than one component of the divisor $D$ corresponding to $\vec{p}$, leading to more than one direction we can deform $\omega_\kaehl$ and more than one class in symplectic cohomology of $X_{t,\kaehl} \setminus D$ corresponding to $\vec{p}$. 
(Deformations in these directions are not ambient: an ambient K\"ahler class would have the same coefficient for each component of $D$ corresponding to $\vec{p}$.) 
On the other hand, there is only one corresponding direction in which we can deform $X^*_\dm$ by deforming $\dm$: we can only deform the coefficient in front of $z^{\vec{p}}$. 
Batyrev shows that the space of deformations of $X^*_\dm$ does have the right dimension, but it must be that some deformations cannot be embedded in $Y^*$. 
\end{rmk}

\begin{rmk}
Smoothness of the fan $\Sigma^*$ is employed most essentially in three places:
\begin{itemize} 
    \item To match various quasi-equivalences of categories, we use the classification of derived autoequivalences of the derived category of smooth Fano varieties from \cite{bondal2001reconstruction}.
  \item The Floer-theoretic proofs of homological mirror symmetry for toric varieties due to \cite{abouzaid2009morse,hanlon2022aspects} work under the assumption that $\Sigma^*$ is smooth. We note that the results from \cite{kuwagaki2020nonequivariant,gammage2017mirror} obtained using microlocal sheaves only require $\Sigma^*$ to be simplicial. 
  \item The assumption is used in the proof of B-side versality in \cref{lem:points} (which is an ingredient in the proof of \cref{lem:Bvers}).
  \end{itemize} 
  Smoothness of $\Sigma^*$ is also used in various other places to simplify arguments (e.g., \cref{lem:tailor_comp}, 
 \cref{lem:Btopologicallfree}, and \cref{lem:smooth}). However, in these latter cases, there  exist straightforward ways to get around the smoothness assumption.  
\end{rmk}

\begin{rmk}
We can refine Conjecture \ref{conj:hgpms} by specifying that the mirror map $b(\kaehl)$ should be given by the formulae in \cite[\S 6.3.4]{coxkatz}.
We can prove this refined version using the results announced in \cite{Ganatra2015}, cf. \cite[Appendix C]{Sheridan2017}. 
By taking care to work throughout over $\Z$, rather than over the coefficient field $\Bbbk$ as we have done, we expect that it would be possible to prove the `integrality of mirror maps' conjecture for the mirror pairs considered in this paper, as was done in the case of Greene--Plesser mirror pairs in \cite{Ganatra2023integrality}. 
However, we have not done this.
\end{rmk}

\subsection{Overview of proof}
\label{subsec:overview}

\begin{table}

\begin{center}
\scalebox{.7}{
\begin{tabular}{c|c|c|c|c}
&& $\overset{MS}{\leftrightarrow}$ && \\ \hline
$M$ & lattice && $M^*$ & dual lattice  \\
$\Delta \subset M_\R$ & reflexive polytope && $\Delta^* \subset M^*_\R$ & dual reflexive polytope \\
$\Sigma \subset M^*_\R$ & fan dual to $\Delta$ && $\Sigma^* \subset M_\R$ & fan dual to $\Delta^*$\\
$Y$ & toric variety corr. to $\Sigma$ && $Y^*$ & toric variety corre. to $\Sigma^*$ \\
$\cL_\Delta \to Y$ & line bundle corr. to $\Delta$ &&  $\cL_{\Delta^*} \to Y^*$ & line bundle corr. to $\Delta^*$ \\
$A$& $\Delta\cap M$&& $\Xinonzero$  &$\Delta^* \cap M^*$, a basis for $\Gamma(\cL_\Delta)$.  \\
$z^\alpha$ & section corr. to $\alpha\in A$ && $z^\vec{p}$ & section corr. to $p\in \Xinonzero$\\
&&& $\Xizero \subset \Xinonzero$ & points on codim $\ge 2$ faces \\
$X_t \subset Y$ & (singular) CY hypersurface  && $X^*_\dm \subset Y^*$ & CY hypersurface corr. to $\dm \in \mathbb{A}^{\Xizero}$ \\ 
$\Sigma_\kaehl \to \Sigma$ & refinement det. by $\kaehl \in \R_{>0}^{\Xizero}$ &&& \\
$Y_\kaehl \to Y$ & resolution &&& \\
$X_{t,\kaehl} \subset Y_\kaehl$ & proper transform &&& \\
$D \subset X_{t,\kaehl}$ & $X_{t,\kaehl} \cap \partial Y_\kaehl$  &&& \\
\hline
$\fo$ & Liouville dom. corr. to $X_{t, \kaehl}\setminus D$ && $\partial Y^*$ & toric boundary\\
$(\Ldom,\fo)$ & Liouville pair && $i:\partial Y^* \hookrightarrow Y^*$ & inclusion of toric boundary  \\
$\ess$ & Liouville sector corr. to $(B, \fo)$ && $Y^*$&\\
\hline
$\iota: \eff \times \C_{Re \ge 0} \hookrightarrow \Sprime$ & inclusion of subsector&&\\
$\cup: \mathcal W(\eff)\to \mathcal W(\Sprime) $ & Cup functor && $ i_* : D^bCoh(\partial Y^*)\to D^bCoh(Y^*)$ & pushforward along inclusion\\
$\cW(\Sprime)$ & Wrapped Fukaya cat. of $\Sprime$ && $D^bCoh(Y^*)$& \\
$\cF(\eff)$ & Fukaya cat. of compact Lag. &&\\
$\mathcal P(\Sprime) \subset \cW(\Sprime)$ & Subcat. of arclike Lagrangians && \\ 
$\cW(\Sprime)^{M^*}, \cW(\eff)^{M^*}, ...$ & $M^*$- graded categories && $D^b_{dg}Coh(Y^*)^{M^*} , D^b_{dg}Coh(\partial Y^*)^{M^*}, ...$&\\
$\cap: \mathcal P(\Sprime)\to \mathcal F(\eff)$ & Geometric cap functor &&\\\hline
$\tilde{L}_i \subset \Sprime$ & ob. of $\cW(\Sprime)$ corr. to $\tildeoi$ && $\tildeoi$ & line bun. on $Y^*$ s.t. $i^!\tildeoi = \cO(i)$ \\
$\mafuk$ & cat. of trop. Lag. sections && $\Pic_{dg}(Y^*)$\\
$L_i = \partial \tilde{L}_i$ & ob. of $\cF(F) = \cF(X_{t,\kaehl} \setminus D)$ && $\cO(i)$ & pullback of line bun. corr. to $i\Delta$ to $\partial Y^*$ \\
$\partial \mafuk$ & subcategory of boundaries && $\Pic_{dg}(\partial Y^*)$\\
\hline
$R_A$ & coefficient ring of $\cF(X_{t,\kaehl},D)$ &&$R_B$& $\Bbbk[[\nov_{\vec{p}}]]_{\vec{p} \in P}$ \\
$\cA_0$ & subcat. of $\cF(X_{t,\kaehl} \setminus D)$ & & $\cB_0$ & subcat. of $D^b_{dg}Coh(X^*_0) = D^b_{dg}Coh(\partial Y^*)$ \\
$\cA_R$ & subcat. of $\cF(X_{t,\kaehl},D)$ & & $\cB_R$ & subcat. of $D^b_{dg}Coh(X^*_R)$ 
\end{tabular}}
\end{center}
\caption{Notation, as well as a rough identification of constructions under the mirror correspondence.}
\end{table}

The proof of \cref{main:higherdim} proceeds as follows. 
On the A-side, we construct certain compact Lagrangians $L_i$ in $X_{t,\kaehl} \setminus D$, and let $\cA_R \subset \fuk(X_{t,\kaehl},D)$ denote the full subcategory with these objects. On the B-side, we let $\mathcal{O}(i)$ denote the restriction to $X_r^*$ of the $i$-th power of the anticanonical bundle on $Y^*_{\Lambda}$ and consider the full subcategory $\cB_R \subset D^b_{dg}Coh(X^*_r)$ with objects $\mathcal{O}(i)$. 
These categories are deformations of the $\Bbbk$-linear categories $\cA_0 \subset \fuk(X_{t,\kaehl} \setminus D)$ and $\cB_0 \subset D^b_{dg}Coh(\partial Y^*)$ with the corresponding objects. 

The bulk of the paper is devoted to constructing a quasi-isomorphism $\cA_0 \simeq \cB_0$. 
Assume for the moment that we have constructed this quasi-isomorphism. 
We then apply the versality result from \cite[Proposition A.6]{Ganatra2023integrality} to upgrade it to a curved filtered quasi-isomorphism $\Psi^*\cB_R \simeq \cA_R$, where $\Psi$ is an appropriate mirror map. 
In order to verify the versality hypotheses on the $A$-side, we need to make certain computations in the symplectic cohomology of $X_{t,\kaehl} \setminus D$, which are accomplished using the results of \cite{ganatra2020symplectic, ganatra2021log} (see \cref{sec:Adef}). 
The versality hypotheses on the $B$-side are more routine (see \cref{sec:Bdef}).

Specializing to appropriate $\Lambda$-points allows us to identify the  subcategory of $D^b_{dg}Coh(X^*_r)$ with objects $\mathcal{O}(i)$, and the subcategory of $\Perf \fuk(X_{t,\kaehl},D,\kaehl;\Lambda_{\Bbbk,K})^\bc$ with objects $L_i$. 
Because $X^*_r$ is smooth, the objects $\mathcal{O}(i)$ split-generate; by the `automatic generation criterion' of \cite{Ganatra2016, Sanda2021}, the objects $L_i$ also split-generate. 
This completes the proof of \cref{main:higherdim} (see \cref{sec:main_result}).

Now, let us explain how the quasi-isomorphism $\cA_0 \simeq \cB_0$ is constructed. 
We start from the homological mirror symmetry result proved by Gammage--Shende in \cite{gammage2017mirror}, building on work of Kuwagaki \cite{kuwagaki2020nonequivariant} and Ganatra--Pardon--Shende \cite{ganatra2018microlocal}. 
Their work produces a commutative diagram
\begin{equation}\label{eq:GS_HMS}
\begin{tikzpicture}[xscale=1.5, yscale=.75,baseline=-1.125cm]
   \node (v1) at (0,0) {$\Perf \cW(X_{t,{\kaehl}} \setminus D)$};
   \node (v2) at (3, 0) {$D^b_{dg}Coh(\partial Y^*) $};
  \node (v3) at (0,-3) {$\Perf \cW(M^*_{\C^*}, \mathfrak{f}_{\Sigma^*})$};
  \node (v4) at (3, -3) {$D^b_{dg}Coh(Y^*)$};
  \draw[<->] (v1) -- (v2);
  \draw[<->] (v3) -- (v4);
  \draw[->] (v1) -- node[midway, left, fill=white]{$\cup$}(v3);
  \draw[->] (v2) -- node[midway, right, fill=white]{ $ i_*$}(v4);
\end{tikzpicture}
\end{equation}
in which $\mathfrak{f}_{\Sigma^*}$ is the boundary at infinity of the Fang--Liu--Treumann--Zaslow skeleton, the horizontal functors are quasi-equivalences, the left vertical arrow is the so-called `cup functor' or `Orlov functor', the right vertical arrow is the pushforward along the inclusion $i:\partial Y^* \hookrightarrow Y^*$, and the diagram commutes up to natural quasi-isomorphism. 
In order to place ourselves in a setting where \eqref{eq:GS_HMS} holds, we first need to deform the affine variety $X_{t,\kaehl} \setminus D$ and its symplectic structure by `tailoring' it, as in \cite{abouzaid2006homogeneous}, to a Weinstein manifold whose Lagrangian skeleton is computed in \cite{zhou2020lagrangian}. 
This is accomplished in \cref{sec:symp_setup}, using general results from \cref{sec:liouville_struc}.

One might think that the subcategory $\cA_0 \subset \cW(X_{t,\kaehl} \setminus D)$ could then be defined to be the category mirror to the subcategory $\cB_0 \subset D^bCoh_{dg}(\partial Y^*)$ under the quasi-equivalence from \eqref{eq:GS_HMS}. 
Unfortunately, life is not so easy: \emph{a priori}, the mirror to this subcategory may consist of non-compact Lagrangians, whereas the objects of the relative Fukaya category $\fuk(X_{t,\kaehl},D)$ are compact Lagrangians (at least as it has been defined to date, cf. \cite{perutz2022constructing}). 
Thus, our main task is to show that the objects mirror to $\mathcal{O}(i)$ are representable by compact Lagrangians.
We do this by constructing a functor $\cap$ which is a `partial adjoint' to $\cup$, and which by construction takes values in the subcategory of compact Lagrangians. 
It is `partial' because it is only defined on a certain subcategory $\cP \subset \Perf \cW(M^*_{\C^*},\mathfrak{f}_{\Sigma^*})$ consisting of `arclike' Lagrangians. 

\begin{rmk} The construction of $\cap$ is carried out using a new result in the theory of partially wrapped Fukaya categories, proved in \cref{sec:adjoints}, which may be of independent interest, and which we now summarize. 
Let $j:S_0 \to S_1$ be an inclusion of Liouville sectors, inducing a functor $j_*:\cW(S_0) \to \cW(S_1)$ in accordance with \cite{ganatra2020covariantly}. 
Given a Lagrangian $L \subset S_1$, one can form its Yoneda module $\cY_L$, and pull it back via $j_*$ to a module over $\cW(S_0)$. 
In general, this pullback module need not be representable, i.e., quasi-isomorphic to the Yoneda module of some geometric Lagrangian in $S_0$. 
The general form of the result proved in \cref{sec:adjoints} gives a criterion under which the pullback module is representable: if $L$ is `arclike' with respect to the inclusion $j$, then we can form a Lagrangian called the `negative pushoff in $S_0$ of $L \cap S_0$'; and we prove that its Yoneda module is isomorphic to the pullback of the Yoneda module of (the negative pushoff in $S_1$ of) $L$ under $j_*$.
The functor $\cup$ is defined to be the pushforward functor associated to the inclusion of a standard neighbourhood of the boundary of a Liouville sector. Applying our general result in this case gives the construction of $\cap$ on the subcategory of arclike Lagrangians (see \cref{sec:fukBackground}). 
\end{rmk}

We consider certain (shifts of) line bundles on $Y^*$,  denoted by $\tildeoi$, which have the property that $i^! \tildeoi = \cO(i)$ (see \S \ref{sec:A0_B0} for the precise description of $\tildeoi$). 
If the objects $\tilde L_i$ of $\cW(M^*_{\C^*},\mathfrak{f}_{\Sigma^*})$ mirror to $\tildeoi$ were arclike, then we could define $\cA_0$ to have objects $L_i = \cap \tilde L_i$, which are compact. 
Because the homological mirror equivalence matches $\cap$ with $i^!$ (by the uniqueness of adjoints), it would necessarily match $\cA_0$ with $\cB_0$ as required.

Unfortunately, we're still not done: under Kuwagaki's construction of the homological mirror equivalence which fits in as the top horizontal arrow in \eqref{eq:GS_HMS}, there is no reason for the objects mirror to $\tildeoi$ to be arclike. 
However, there is a different approach to homological mirror symmetry for Fano toric varieties, due to Abouzaid \cite{abouzaid2006homogeneous,abouzaid2009morse}, Hanlon \cite{hanlon2019monodromy}, and Hanlon--Hicks \cite{hanlon2022aspects}. 
They give a different construction of a mirror equivalence $\Perf \cW(T^* M^*_{S^1},\mathfrak{f}_{\Sigma^*}) \simeq D^b_{dg}Coh(Y^*)$ under which the objects mirror to $\tildeoi$ (and all other toric line bundles) are arclike. 
In order to make use of this, we need to show that the homological mirror equivalence constructed by Abouzaid--Hanlon--Hicks agrees with that constructed by Kuwagaki--Ganatra--Pardon--Shende. 
We prove this in \cref{sec:ASideSetup}, using the classification of derived autoequivalences of smooth Fano varieties due to Bondal--Orlov.

This completes the construction of the quasi-isomorphism $\cA_0 \simeq \cB_0$, and hence the proof of \cref{main:higherdim}: we take $L_i$ to be the objects $\cap \tilde L_i$, where $\tilde L_i$ are the arclike Lagrangians mirror to $\tildeoi$ under the Abouzaid--Hanlon--Hicks mirror equivalence. 
We still need to verify that this equivalence respects certain gradings which are needed in order to apply the versality result; this is done in \cref{sec:A0_B0}, using results from \cref{sec:qiso_grad}.

\begin{rmk} It has long been anticipated that the Floer theoretic approach to toric HMS from \cite{abouzaid2006homogeneous,abouzaid2009morse} and the microlocal sheaf approach initiated in \cite{fang2011categorification,fang2012t} coincide (cf. \cite[Appendix C.2]{fang2012t}). As noted above, Proposition \ref{prop:sameHMS} below settles this question in the affirmative in the Fano case.  \end{rmk}

\begin{rmk}
    We also prove that the subcategories $\cA_0 \subset \fuk(X_{t,\kaehl} \setminus D)$ and $\cB_0 \subset \perf_{\mathsf{dg}}(\partial Y^*)$ split-generate, where  $\perf_{\mathsf{dg}}(\partial Y^*) \subset \dbdg Coh(\partial Y^*)$ denotes the full subcategory of perfect complexes. In particular, this implies (see   \cref{prop:perf_hms}) that Gammage--Shende's quasi-equivalence
    $$\perf \cW(X_{t,\kaehl} \setminus D) \simeq \dbdg Coh(\partial Y^*)$$
    restricts to a quasi-equivalence
    $$\perf \fuk(X_{t,\kaehl} \setminus D) \simeq \perf_{\mathsf{dg}}(\partial Y^*).$$
   
\end{rmk}

\subsection{Acknowledgements}
\label{subsec:acknowledgements}
We are very grateful to Danil Ko\v{z}evnikov for pointing out an error in an earlier version; and to Tatsuki Kuwagaki and Alexander Polishchuk for helpful conversations.
S.G. was supported by NSF grant CAREER DMS-2048055 and a Simons Fellowship (award number 1031288), and would like to thank UCLA and Caltech for their hospitality during visits in which some of this work was completed.
A.H. was supported by the Simons Foundation (Grant Number 814268 via the Mathematical Sciences Research Institute, now renamed to SLMath) and NSF grant DMS-2404882.
J.H. is supported by an EPSRC Postdoctoral Fellowship (project reference: EP/V049097/1). D.P. was partially funded by the Simons Collaboration on Homological Mirror Symmetry (award number 652299) and NSF grant
DMS-2306204 while working on this project. 
N.S. is supported by a Royal Society University Research Fellowship, an ERC Starting Grant (award number 850713 -- HMS), the Simons Collaboration on Homological Mirror Symmetry (award number 652236), the Leverhulme Prize, and a Simons Investigator award (award number 929034). 

\section{Symplectic setup}\label{sec:symp_setup}
We would like to study the Fukaya category of $X_{t, \kaehl}\setminus D = \{tW_0 = W_1\}$ by viewing it as the fiber over $t$ of the potential $W_1/W_0$ on $M^*_{\C^*}$. 
Following \cite{gammage2017mirror}, we replace the spaces $M^*_{\C^*}$ and $X_{t, \kaehl}\setminus D$ with $\Ldom$, a Liouville domain, and $\fo \subset \partial \Ldom$, a Liouville hypersurface in its boundary. 
We define a Liouville sector $\ess$ by removing a neighbourhood of $\fo$ from $\Ldom$.

This construction occurs in several steps. \Cref{sec:rel_kahl} describes the Liouville structure on $X_{t, \kaehl}\setminus D$. \Cref{sec:trop_setup} deforms the hypersurface to its tropical limit while preserving its smooth toric compactification. \Cref{subsec:skeleton} shows that this tropical limit has skeleton described by  \cite{zhou2020lagrangian}. Finally, \Cref{sec:Liouvsec} describes the construction of the sector $\ess$ and its relation to $\fo$. 
\subsection{Relative K\"ahler manifold}\label{sec:rel_kahl}
Assume the MPCS condition holds; then we have a complex manifold $X_{t,\kaehl}$ embedded inside the complex orbifold $Y_\kaehl$. 
Furthermore, we have an orbifold simple normal crossings divisor $D^Y \subset Y_\kaehl$,\footnote{Everything in \cite{Sheridan2017} is written for manifolds, rather than orbifolds. Nevertheless, we make use of the corresponding notions for orbifolds, because it makes the exposition clearer. It would be straightforward to remove the orbifold language from any argument in which it appears because we essentially only work with the symplectic manifolds $X_{t,\kaehl}$ and $Y_\kaehl \setminus D^Y$, both of which avoid the orbifold locus.} whose intersection with $X_{t,\kaehl}$ is a simple normal crossings divisor $D \subset X_{t,\kaehl}$. 
In the terminology of \cite[Definition 3.32]{Sheridan2017}, $(X_{t,\kaehl},D)$ is a sub-$\mathsf{snc}$ pair of $(Y_\kaehl,D^Y)$.

We recall the notion of a relative K\"ahler form on a $\mathsf{snc}$ pair $(X_{t,\kaehl},D)$ from \cite[Definition 3.2]{Sheridan2017}: it consists of a K\"ahler form on $X_{t,\kaehl}$, together with a K\"ahler potential on the complement of $D$, having a prescribed form near $D$. 
It determines `linking numbers', which are strictly positive real numbers associated with the components of $D$. 
One may generalize this to the notion of an orbifold relative K\"ahler form on $(Y_\kaehl,D^Y)$.

Note that $Y_\kaehl \setminus D^Y \cong M^*_{\C^*} = M^*_\R \times M^*_{S^1}$. 
A choice of basis for $M^*$ determines coordinates $x_j$ on $M^*_\R$, and $\theta_j$ on $M^*_{S^1}$, so that the functions $z_j = \exp(x_j+i\theta_j)$ are holomorphic. 

We introduce a restricted class of relative K\"ahler forms on $(Y_\kaehl,D^Y)$ in Definition \ref{def:comp_pot}. 
These have the property that the potential $h^Y$ on $M^*_{\C^*}$ is pulled back from some $\varphi^Y:M^*_\R \to \R$, called a compactification K\"ahler potential. 
The MPCS condition implies that the divisor $D^Y_\kaehl$ is ample, and so by \Cref{lem:comp_pot_exist} there exists a compactification K\"ahler potential $\varphi^Y$ with corresponding linking numbers $\kaehl_{\vec{p}}$. 
We obtain a relative K\"ahler form $(\omega^X,h^X)$ on $(X_{t,\kaehl},D)$ by restricting this one. 

We endow $Y_\kaehl$ with a meromorphic volume from whose restriction to $Y_\kaehl \setminus D^Y = M^*_{\C^*}$ is
$$\Omega^Y|_{M^*_{\C^*}} = d\log z_1 \wedge \ldots \wedge d\log z_n.$$

Observe that $X_{t,\kaehl} \setminus D  \subset M^*_{\C^*}$ is cut out by the equation $W_1/W_0 = t$. 
Therefore, we have a holomorphic volume form $\Omega^Y/d(W_1/W_0)$ on $X_{t,\kaehl} \setminus D$, and this extends to a holomorphic volume form $\Omega^X$ on $X_{t,\kaehl}$. 

\subsection{Liouville domain}\label{sec:trop_setup}

Associated to the relative K\"ahler manifold $(X_{t,\kaehl},D,h)$ from Section \ref{sec:rel_kahl}, we have an exact symplectic manifold $\cH_{t,0} := X_{t,\kaehl} \setminus D \subset M^*_{\C^*}$. 
The Liouville one-form is $\lf^X = -d^ch^X$.  
We denote its Liouville completion by $\widehat{\cH}_{t,0}$ (this is well-defined because the Liouville vector field applied to $h^X$ is positive near $\infty$, cf. Lemma \ref{lem:hol_comp_hom}). 

We wish to make computations in the exact Fukaya category of the Liouville manifold $\widehat{\cH}_{t,0}$. 
To do this, we deform $\cH_{t, 0}$ to a Liouville manifold with a combinatorially described skeleton.

The first step is to `tailor' the affine variety $\cH_{t,0}$, following \cite{abouzaid2006homogeneous}. 
Recall that $A = \Delta \cap M$. 
We define functions $(\nu,c): A\to \Z \times \{\pm 1\}$ which are equal to $(1,-1)$ for $\alpha = 0$, and $(0,+1)$ otherwise.  
Then $\cH_{t,0} \subset M^*_{\C^*}$ is the vanishing locus of the function
$$W_{t,0}(z) = \sum_{\alpha \in A} c(\alpha)\cdot t^{\nu(\alpha)} \cdot z^{\alpha}.$$

We define the corresponding \emph{tropical amoeba} $\Pi \subset M^*_\R$ to be the locus of discontinuity of the piecewise-linear function $M^*_\R \to \R$ sending 
$$ x \mapsto \max_{\alpha \in A} \left(-\nu(\alpha) + \langle \alpha,x\rangle\right).$$
Smoothness of $\Sigma^*$ (which is a hypothesis of Theorem \ref{main:higherdim}) is equivalent to the statement that the subdivision of $\Delta$ induced by $\nu$ is a unimodular triangulation, i.e., it is a subdivision into simplices of the minimal possible volume. 

The components of $M_\R^* \setminus \Pi$ are in one-to-one correspondence with $A$. 
We let $C_\alpha$ denote the closure of the component indexed by $\alpha \in A$. The only interior lattice point of $\Delta$ is $0$, so the only bounded component is $C_0$, which is equal to $\Delta^*$.

We introduce the map
\begin{align*}
\Log_t: M^*_{\C^*} & \to M^*_\R \\
\Log_t(x,\theta) &:= \frac{1}{\log t} x.
\end{align*}
The image $\cA_t := \Log_{t}(W_t^{-1}(0))$ is referred to as the amoeba. As shown in \cite{mikhalkin2004decomposition}, the Gromov-Hausdorff limit of $\cA_t$ as $t$ goes to infinity is $\Pi$. 

Following \cite{abouzaid2006homogeneous}, we can further tailor this convergence by considering an interpolating family of potentials. Namely, we set 
$$ W_{t,s} = \sum_{\alpha \in A} (1 - s\psi_\alpha (t, x)) \cdot c(\alpha) \cdot t^{\nu(\alpha)} \cdot z^\alpha $$
for $t > 0$ and $s \in [0,1]$ and where $\psi_\alpha$ are appropriately chosen cutoff functions (we recall that $x_i = \log |z_i|$). In particular, for large $t$, $\psi_\alpha(t,x/\log t)$ is equal to $0$ for $x$ in a neighbourhood of $ C_\alpha$, and $1$ outside of a larger neighbourhood of $C_\alpha$. In fact, we can use explicitly defined cutoff functions as in \cite{zhou2020lagrangian}. 

We now consider the hypersurfaces $\mathcal{H}_{t,s} = W_{t,s}^{-1}(0)$. 

\begin{lem}\label{lem:tailor_approx_hol}
For any $\delta>0$, the family $(\cH_{t,s})_{s \in [0,1]}$ is $\delta$-approximately holomorphic (in the sense of Definition \ref{def:approx_hol}) for $t$ sufficiently large.
\end{lem}
\begin{proof}
Follows from the proof of \cite[Proposition 1.7]{zhou2020lagrangian}, together with \cite[Lemma 8.3(b)]{cieliebak2007symplectic}.
\end{proof}

\begin{defn}[Restatement of Definition \ref{def:smtorcomp}]
A \emph{smooth toric compactification} of a submanifold $\cH \subset M^*_{\C^*}$ is a simplicial fan $\Sigma_{\cH}$ and smooth submanifold $X_{\cH} \subset Y_{\Sigma_\cH}$, such that $\cH = X_{\cH} \cap M^*_{\C^*}$, $X_{\cH}$ avoids the orbifold points, and the intersection of $X_{\cH}$ with every stratum of the toric boundary divisor is transverse. 
\end{defn}

\begin{lem}\label{lem:tailor_comp}
For sufficiently large $t>0$, the family $(\cH_{t,s})_{s \in [0,1]}$ admits a family of smooth toric compactifications. 
\end{lem}
\begin{proof}
Note that it follows from \Cref{lem:tailor_approx_hol} that $\cH_{t,s}$ is a smooth submanifold of $M^*_{\C^*}$ for sufficiently large $t>0$.

We now consider the closure $X_{t,s}$ of $\cH_{t,s}$ in $Y_\kaehl$. 
For each cone $\sigma$ of $\Sigma_\kaehl$, such that $X_{t,s}$ intersects the corresponding torus orbit $Y_\sigma$, the latter admits a Zariski neighbourhood with coordinates $(w_\sigma,z_\sigma) \in \C^{\dim \sigma} \times (\C^*)^{\dim Y - \dim \sigma}$. 
From the construction of the cutoff functions $\psi_\alpha$ in \cite[Section 1.2]{zhou2020lagrangian}, it is straightforward to check that in a neighbourhood of $Y_\sigma$ they are independent of $w_\sigma$; in other words, there exist cutoff functions $\psi^\sigma_\alpha(t,z_\sigma)$ such that $\psi_\alpha(t,(w_\sigma,z_\sigma)) = \psi^\sigma_\alpha(t,z_\sigma)$. 
As a result, in this neighbourhood of $Y_\sigma$, we have $X_{t,s} = \left(W_{t,s}^{\sigma}\right)^{-1}(0)$ where
\[ W_{t,s}^{\sigma}(w_\sigma,z_\sigma) = \sum_{\alpha \in \Delta}  (1 - s\psi^\sigma_\alpha (t, z_\sigma)) \cdot c(\alpha) \cdot t^{\nu(\alpha)} \cdot (w_\sigma,z_\sigma)^\alpha\]
is a smooth function. 

The toric morphism $Y_\kaehl \to Y$ blows $Y_\sigma$ down to $Y_{\hat \sigma}$, where $\hat \sigma$ is the smallest cone of $\Sigma$ containing $\sigma$. 
Let $\Delta_\sigma$ denote the face of $\Delta$ dual to $\hat \sigma$. 
Then $z^\alpha$ vanishes along $Y_\sigma$ for $\alpha \notin \Delta_\sigma$. 
To show that $X_{t,s}$ is smooth in a neighbourhood of $Y_\sigma$, and intersects it transversely, it then suffices to show that 
\[ X_{t,s} \cap Y_\sigma = \sum_{\alpha \in \Delta_\sigma}  (1 - s\psi^\sigma_\alpha (t, z_\sigma)) \cdot z_\sigma^\alpha\]
is smooth. 
However, since $\Delta_\sigma$ is unimodular, this is simply the product of a tailored pair of pants with $(\C^*)^j$ for some $j$, which is smooth. 
\end{proof}

We now choose a potential $\varphi:M^*_\R \to \R$ which is homogeneous of weight $2$ and adapted to $\Delta^*$ in the sense of \cite[Definition 2.8]{zhou2020lagrangian}, which induces a Liouville manifold structure on $M^*_{\C^*}$. 
By restricting, we then have a Liouville manifold structure on $\cH_{t,1}$, which has well-defined Liouville completion $\widehat{\cH}_{t,1}$ by Lemma \ref{lem:hom_Liouv}. 

\begin{lem} \label{lem:tailor_equivalence}
For sufficiently large $t>0$, we have a Liouville isomorphism $\widehat{\cH}_{t,0} \cong \widehat{\cH}_{t,1}$.
\end{lem}
\begin{proof}
The Liouville isomorphism is constructed in two steps. 
First, we fix $\cH_{t,0}$ and linearly interpolate between the potential $\varphi^Y$ coming from the toric compactification of $\cH_{t,0}$, and the potential $\varphi$ which is homogeneous of weight $2$ and adapted to $\Delta^*$. 
By Lemma \ref{lem:hol_comp_hom}, the latter also has a well-defined Liouville completion; and this completion is Liouville isomorphic to that of $\cH_{t,0}$ by Lemma \ref{lem:change_comp_hom}.

The second step is to deform $\cH_{t,0}$ to $\cH_{t,1}$ using the family $\cH_{t,s}$, equipped with the fixed homogeneous potential $\varphi$. 
This deformation induces a Liouville isomorphism by Lemma \ref{lem:change_tailoring} (whose hypotheses are verified by Lemmas \ref{lem:tailor_comp} and \ref{lem:tailor_approx_hol}).
\end{proof}

\subsection{Skeleton}
\label{subsec:skeleton}

The differential of the potential $\varphi:M^*_\R \to \R$ (which we recall is homogeneous of weight $2$ and adapted to $\Delta^*$) defines a diffeomorphism 
\begin{align*}
\Phi:M^*_\R &\xrightarrow{\sim} M_\R,\\
\Phi(x) &= d\varphi_x.
\end{align*} 
We use it to identify
\begin{equation}\label{eq:Leg_trans}
M^*_{\C^*} = M^*_\R \times M^*_{S^1} \xrightarrow{(d\varphi,\id)} M_\R \times M^*_{S^1}  = T^* M^*_{S^1}     ,
\end{equation}
which fits into a commutative diagram
\begin{equation}\label{eq:Leg_trans_cd} \begin{tikzcd}  M^*_{\C^*} \arrow[d, "\Log"] \arrow[r,"\eqref{eq:Leg_trans}"] & T^*M^*_{S^1}  \arrow[d, "\pi"]& \\
		M^*_\R \arrow[r, "\Phi"] & M_\R 
	\end{tikzcd}
	\end{equation}
	where $\pi$ is the natural projection. 
This identification identifies the Liouville one-form $-d^c \varphi$ on $M^*_{\C^*}$ with the standard Liouville one-form on the cotangent bundle. 

To be more concrete, let us choose a basis for $M^*$, inducing coordinates $(x_1,\ldots,x_n)$ on $M^*_\R$, $(\theta_1,\ldots,\theta_n)$ on $M^*_{S^1}$, and $(u_1,\ldots,u_n)$ on $M_\R$. 
With respect to these coordinates, we have
$$\Phi(x_1,\ldots,x_n) = \left(\frac{\partial \varphi}{\partial x_1},\ldots,\frac{\partial \varphi}{\partial x_n}\right)$$
and
$$-d^c \varphi = \sum_{j=1}^n \frac{\partial \varphi}{\partial x_j} d\theta_j = \sum_{j=1}^n u_j d\theta_j.$$

Associated to the fan $\Sigma^* \subset M_\R$ of $Y^*$ is a combinatorially-defined (negative) FLTZ skeleton
\begin{equation} \label{lsigma}
{\mathbb L}_{\Sigma^*} = \bigcup_{\sigma \in \Sigma^*} \sigma \times \sigma^\perp \subset M_\R \times M^*_{S^1}= T^*M^*_{S^1},
\end{equation}
with boundary at infinity
\begin{equation} \label{fsigma}
    \mathfrak{f}_{\Sigma^*} = \left( {\mathbb L}_{\Sigma^*} \setminus \left(\{0\} \times M^*_{S^1}\right) \right)/\R \subset \partial_\infty (T^*M^*_{S^1}).
\end{equation}

For $t \gg 0$, the complement of the tailored amoeba $\Log_t(\cH_{t,1})$ has a unique bounded component, whose closure we denote by $\mathcal{C}_0$. 
By \cite[Proposition 1.9]{zhou2020lagrangian}, $\cC_0$ can alternatively be characterized as the sublevel set of the function $H_{\mathcal{C}_0}: M^*_\R\to \R_{\geq 0}$:
\begin{equation}
	1 \ge H_{\mathcal{C}_0}:=\sum_{\alpha\in A \setminus \{0\}} (1-\psi_\alpha(t,x))t^{\langle\alpha, x \rangle -1}
	\label{eq:cutoffRealFunction}
\end{equation}
The region $\mathcal C_0$ is not necessarily convex. However, it is star-shaped:

\begin{prop}\label{prop:star}
	For $t$ sufficiently large, there exists $\eps>0$ so that $H_{\mathcal C_0}$ is radially increasing, that is,
	\[  dH_{\mathcal{C}_0}(x  \partial_x)  > 0\]
	at all points $x$ such that $1-\eps\leq H_{\mathcal C_0}(x)\leq 1$ where $x\partial_x$ is the radial vector field on $M_\R^*$.
	\label{prop:cutoffRealFunctionProperties}
\end{prop}
\begin{proof}
	By combining Propositions 1.12(3) and 1.13(2) of \cite{zhou2020lagrangian}, we can choose $t$ sufficiently large so that the outward unit conormal bundle on $\mathcal{C}_0$ is arbitrarily close in the unit conormal bundle of $M_\R^*$ to the Legendrian given by the union of the outward unit conormals to the strata of $C_0$. The vector field $x \partial_x$ pairs positively with the latter as $C_0$ is a convex polytope containing the origin. The result then follows from the compactness of the level sets of $H_{\mathcal{C}_0}$.
\end{proof}

\begin{cor}
For $t$ sufficiently large, $\Log_{t}^{-1}(\mathcal{C}_0)$ is a Liouville domain completing to $M^*_{\C^*}$.
\end{cor}
\begin{proof}
By the commutative diagram \eqref{eq:Leg_trans_cd}, $\Log_t^{-1}(\cC_0) = \pi^{-1}(\Phi^{-1}(\log t \cdot \cC_0))$. 
The Liouville vector field for the standard Liouville one-form on $T^*M^*_{S^1}$ is equal to the radial vector field $u \partial_u$. 
Thus it suffices to show that $\Phi^{-1}(\log t \cdot \cC_0)$ is star-shaped. 
As $\varphi$ is homogeneous of weight $2$ outside a small compact set, and in particular near $\log t \cdot \partial \cC_0$, we have that $\Phi_* (x \partial_x)$ is (pointwise) positively proportional to $u \partial_u$ in that region (cf. \cite[Proposition 2.7]{zhou2020lagrangian}). 
Thus it suffices to check that $\log t \cdot \cC_0$ is star-shaped, which is true by Proposition \ref{prop:star}.
\end{proof}

We then have the following result of \cite{zhou2020lagrangian}, which is a generalization of \cite[Theorem 6.2.4 and Lemma 6.2.5]{gammage2017mirror}:

\begin{prop} \label{prop:GSZ} For $t$ sufficiently large, the intersection $\Lambda_{\Sigma^*} \coloneqq {\mathbb L}_{\Sigma^*} \cap \partial \Log_{t}^{-1}(\mathcal{C}_0)$ is the skeleton of $\mathcal{H}_{t,1}$.
\end{prop}
\begin{proof} 
First, \cite[Theorem 2]{zhou2020lagrangian} shows that for $t$ sufficiently large, the skeleton of $\mathcal{H}_{t,1}$ is a subset of $\partial \Log_{t}^{-1}(\mathcal{C}_0)$. Next, \cite[Theorem 3]{zhou2020lagrangian} identifies the projection of the skeleton of $\mathcal{H}_{t,1}$ to $\partial_\infty T^*M^*_{S^1}$ with $\mathfrak{f}_{\Sigma^*}$. As this is also true for $\Lambda_{\Sigma^*}$ and both are contained in $\partial \Log^{-1}_t (\mathcal{C}_0)$, they must coincide.
\end{proof}

\subsection{Liouville sector}\label{sec:Liouvsec}

From this point forward, we assume that $t$ is chosen sufficiently large so that all the statements in the previous section hold. The following result is due to \cite{zhou2020lagrangian}:
 
\begin{prop} \label{prop:notangentZ}
The Liouville vector field $Z = u \partial_u$ is nowhere tangent to $\mathcal{H}_{t,1}$ in a neighbourhood of $\Lambda_{\Sigma^*}$.
\end{prop}
\begin{proof}
The component of $Z$ that is symplectically orthogonal to $\mathcal{H}_{t,1}$ is positively proportional to the Hamiltonian vector field of $\Im(W_{t,1})$ along $\Log_t^{-1}(\cC_0)$ (and hence along the skeleton, by Proposition \ref{prop:GSZ}) by \cite[Proposition 5.3]{zhou2020lagrangian}, and thus nonvanishing in a neighbourhood of the skeleton.
\end{proof}

We now follow \cite[Corollary 6.2.6]{gammage2017mirror} to construct a sutured Liouville domain (in the sense of \cite[Definition 2.14]{ganatra2020covariantly}) from $(\Log_{t}^{-1}(\mathcal{C}_0), \Lambda_{\Sigma^*})$. We repeat the argument for clarity and later reference. 

\begin{lem} \label{lem:constructDF}
There exists a Liouville subdomain $\Ldom \subset M^*_{\C^*}$ completing to $M^*_{\C^*}$, together with a Liouville subdomain $\fo \subset \cH_{t,1}$ completing to $\widehat{\cH}_{t,1}$, such that $\fo \subset \partial \Ldom$.
\end{lem}
\begin{proof}
The Liouville flow gives an identification $M^*_{\C^*} \setminus M^*_{S^1} \simeq \partial \Log_t^{-1}(\mathcal{C}_0) \times \R$, which identifies $\partial \Log_t^{-1}(\cC_0)$ with $\partial \Log_t^{-1}( \cC_0) \times \{0\}$. 
We choose a Liouville subdomain $\fo \subset \mathcal{H}_{t,1}$, completing to $\cH_{t,1}$, which is a neighbourhood of $\Lambda_{\Sigma^*}$ on which the Liouville vector field is not tangent to $\cH_{t,1}$. 
Then, the projection of $\fo$ to $ \partial \Log_t^{-1}( \cC_0)$ is an immersion, and the projection of $\Lambda_{\Sigma^*}$ is an embedding; so by shrinking $\fo$ further, we may arrange that the projection of $\fo$ is also an embedding. 

Thus, $\fo$ is identified with the graph of a function on a hypersurface in $ \partial \Log^{-1}_t( \mathcal{C}_0)$. 
Extending the function to all of $\partial \Log_t^{-1}$ and taking its graph produces the boundary of a new Liouville domain $\Ldom$ which contains $\fo$ in its boundary. \end{proof}

From the sutured Liouville domain $(\Ldom,\fo)$, we construct a Liouville sector $\ess \subset M^*_{\C^*}$, in the sense of \cite[Definition 2.4]{ganatra2020covariantly}, in accordance with \cite[Definition 2.14]{ganatra2020covariantly}. 
We briefly recall the construction. 
We apply the Reeb flow on $\partial \Ldom$ to obtain a neighbourhood of the form $\fo \times \R_{|t| < \eps} \subset \partial \Ldom$ and set $\ess \subset M^*_{\C^*}$ to be the complement of the set $\R_{>0} \times \R_{|t| < \eps} \times \fo$ where the first coordinate is the flow coordinate for the Liouville vector field from $\fo \times \R_{|t| < \eps} \subset \partial \Ldom$. The $1$-defining function $\aye$ is given (on the boundary of the cornered sector) by the linear extension of $-t$.
By construction, we have that there is a small neighbourhood $U$ of $\fo \subset \partial \ess$ on which 
\begin{equation}
  \label{eq:IandWonFiber} \fo = \aye^{-1}(0) \cap U = \left(\Im \, W_{t,1} \right)^{-1}(0) \cap U.
\end{equation}

The boundary can then be smoothed as in \cite[Remark 2.12]{ganatra2020covariantly}.
    By \cite[Lemma 2.32]{ganatra2020covariantly}, the sector embedding $\ess \hookrightarrow (M^*_{\C^*}, \fo)$ presents the pair $(M^*_{\C^*}, \fo)$, where $\fo$ here denotes the projection to $\partial_{\infty} M^*_{\C^*}$ of $\fo$, as the stopped convex completion of the sector $\ess$, that is, as the inclusion of $\ess$ into the pair $(\bar{\ess}, \fo)$ where $\bar{\ess}$ is the convex completion in the sense of \cite[\S 2.7]{ganatra2020covariantly}, and $\fo$ is the copy of $\fo$ lying at $-\infty$ in the completion region $\eff \times \C_{\Re \leq 0}$, where $\eff:= \widehat{\fo}$ is the Liouville completion of $\fo$.

For the construction in \Cref{subsec:productLagconstr}, we will need one further geometric feature of the sector $\ess$.
\begin{lem} \label{lem:compatibleIandW}
We may choose the pair $(\Ldom, \fo)$ such that there exists a neighbourhood $V$ of ${\mathbb L}_{\Sigma^*}$ and an extension of the $1$-defining function $\aye$ to a cylindrical neighbourhood $\Nbd^Z(\partial \ess)$ of $\partial \ess$ in $M^*_{\C^*}$ such that 
$$ \aye^{-1}(0) \cap V \cap \Nbd^Z(\partial \ess) = \left(\Im W_{t,1}\right)^{-1}(0) \cap V \cap \Nbd^Z(\partial \ess).$$
\end{lem}
\begin{proof} 
We have $Z = Z_{\fo} + Z^\perp$, where $Z_{\fo}$ is the Liouville vector field of $F$ and $Z^\perp$ is symplectically orthogonal to $T\fo$. 
Let $g = \Im W_{t,1}$; then we have that $X_g$ is positively proportional to $Z^\perp$ along $\Lambda_{\Sigma^*}$ by \cite[Proposition 5.3]{zhou2020lagrangian}. 
If $R$ is the Reeb vector field of $\partial B$, then we have
\begin{align*}
    dg(R) &= \omega(X_g,R) \\
    &\propto \omega(Z^\perp,R) \\
    & = \omega(Z,R) \quad \text{as $\omega(Z_{\fo},R) = 0$} \\
    & = \lf(R) = 1
\end{align*}
along $\Lambda_{\Sigma^*}$. 
It follows that $dg(R)>0$ in a neighbourhood of $\Lambda_{\Sigma^*}$ in $\partial B$. 
Hence, by shrinking $\fo$ and choosing $\eps$ sufficiently small, we may arrange that $dg(R) > 0$ along $\fo \times \R_{|t|< \eps}$.  
It follows that the map 
$$\fo \times \R_{|t|<\eps} \xrightarrow{(\id, g)} \fo \times \R_g$$
is a diffeomorphism onto its image. 
As $g(f,0) = 0$, we may write $t = t(f,g)$ as a function of $f \in \fo$, $g \in \R$, with $t(f,0) = 0$.

We now choose $\Nbd^Z(\partial \ess)$, so that there is a smooth retraction $\pi:\Nbd^Z(\partial \ess) \to \partial \ess$. 
We choose $V$ so that $V \cap \Nbd^Z(\partial \ess) = \pi^{-1}(F \times \R_{|t|< \eps})$. 
We define $\pi_F: V \cap \Nbd(\partial \ess) \to \fo$ to be the composition of $\pi$ with the projection to $\fo$. 
We now extend the defining function $\aye$ to $V \cap \Nbd^Z(\partial \ess)$ by setting it equal to $\aye(x) = -t(\pi_F(x),g(x))$. 
Because $\aye = -t$ along $\fo \times \R_{|t|<\eps}$ by construction, this function extends $\aye$; it also has the property that $\aye=0$ if and only if $g=0$. 
We now extend $\aye$ arbitrarily over the rest of $\Nbd^Z(\partial \ess)$.
\end{proof} 

\section{(Partially) wrapped Fukaya categories, cup and cap functors}

\label{sec:fukBackground}
\subsection{(Partially) wrapped Fukaya categories} \label{subsec:fukBackground1}
Recall from \cite{ganatra2018sectorial, ganatra2020covariantly} the construction of the partially wrapped Fukaya category $\W(X,\mathfrak{f})$ of a Liouville sector $X$ stopped along a closed subset $\mathfrak{f} \subset (\partial_{\infty} X)^{\circ}$; when $\mathfrak{f} = \emptyset$ we write $\W(X):=\W(X,\emptyset)$. Objects are cylindrical exact Lagrangians $L$ in $X$ disjoint from $\partial X$ and $\mathfrak{f}$. As described in \cite[Section 5.3]{ganatra2018microlocal}, by equipping $X$ and then its Lagrangians with grading/orientation data, one can ensure the resulting category is $\Z$-graded and defined over a ring of arbitrary characteristic.

In particular, to the Liouville sector $\ess$ introduced in Section \ref{sec:Liouvsec}, together with the holomorphic volume form $\Omega_{\ess}:= \Omega^Y|_{\ess}$, we have a $\Z$-graded $\Z$-linear partially wrapped Fukaya category $\cW(\ess)$ whose objects are cylindrical exact Lagrangians $L$ disjoint from $\partial \ess$, equipped with spin structures and gradings relative to $\Omega_{\ess}$. 
To be precise about the latter choice, we have a squared phase function $L \to S^1$ defined by the argument of $\Omega_{\ess}(v_1 \wedge \ldots \wedge v_{n+1})^2$, where $(v_1,\ldots,v_n)$ is a basis of $TL$; and the grading of $L$ consists of a lift of this squared phase function to $\R$. Note that there are natural maps
\begin{equation} \label{eq:wrappedstoskel}
\mathcal{W}(\ess) \to \mathcal{W}(M^*_{\C^*}, \fo) \to \mathcal{W}(M^*_{\C^*}, \mathfrak{f}_{\Sigma^*} ) 
\end{equation}
which are quasi-equivalences as shown in \cite[Corollary 3.9]{ganatra2018sectorial} (using the fact that $(M^*_{\C^*}, \fo)$ is the stopped convex completion of $\ess$). These quasi-equivalences make it clear that $\mathcal{W}(\ess)$ does not depend on any of the choices involved in constructing $\ess$. 

\begin{rmk} \label{rmk:Sgrading}
    The approach taken to gradings and orientations in \cite[Section 5.3]{ganatra2018microlocal} is equivalent, but formulated in a slightly different language. 
    The grading/orientation data we consider for $\ess$ consist of the constant map $\ess \to K(\Z/2,2)$ (we consider spin structures, rather than relative spin structures), together with the map $\mathrm{LGr}(\ess) \to S^1$ given by the squared phase map. 
    It is easy to see from the definitions how our choice of spin structure and grading on a Lagrangian $L$ induces a choice of grading/orientation data on $L$ in the sense of \cite[Section 5.3]{ganatra2018microlocal}. Moreover, our choice of grading/orientation conventions is compatible with that made in \cite[Section 5.3]{ganatra2018microlocal}, as the tautological polarization of $\ess \subset T^*M^*_{S^1}$ is a trivial bundle (agreeing with our choice of trivial map to $K(\Z/2,2)$) and has phase $0$ with respect to our choice of holomorphic volume form.
\end{rmk}

We also have a $\Z \oplus M^*$-graded version $\cW(\ess)^{M^*}$, whose objects are objects of $\cW(\ess)$ equipped with an additional lift of $L$ to the universal cover of $\ess$ (an `anchoring', in the terminology of \cite{fukaya2010anchored}). 
The additional grading then arises from $\pi_1(\ess) = H_1(\ess) = M^*$ (cf. \cite[Section 3.1]{sheridan2015homological}). 
There is a fully faithful functor $\cW(\ess)^{M^*} \to \cW(\ess)$ which forgets the anchoring on objects, and the $M^*$-grading on morphisms. We can, of course, apply the same construction to obtain $Z\oplus M^*$-graded categories $\cW(M^*_{\C^*})^{M^*}$ and $\cW(M^*_{\C^*}, \mathfrak{f}_{\Sigma^*})^{M^*}$ and a $\Z \oplus M^*$-graded analog of \eqref{eq:wrappedstoskel}. 

Associated to the (finite-type, convex) exact symplectic manifold $\cH_{t,0}$, together with the holomorphic volume form $\Omega_{\cH_{t,0}}:=\Omega^X|_{\cH_{t,0}}$, we have a $\Z$-graded $\Z$-linear wrapped Fukaya category $\cW(\widehat{\cH}_{t,0})$, whose objects are cylindrical exact Lagrangians equipped with spin structures and gradings relative to $\Omega_{\cH_{t,0}}$.
As before, we have a $\Z \oplus M^*$-graded version $\cW(\widehat{\cH}_{t,0})^{M^*}$, whose objects are objects of $\cW(\widehat{\cH}_{t,0})$ equipped with a lift of the composition $L \to \widehat{\cH}_{t,0} \hookrightarrow M^*_{\C^*}$ to the universal cover of $M^*_{\C^*}$. 

\begin{rmk} To simplify notation, we will typically write $\cW(\cH_{t,0})$ for the wrapped category $\cW(\widehat{\cH}_{t,0})$. More generally, we will apply the same convention when discussing wrapped categories of other incomplete exact symplectic manifolds that arise in our arguments (which we may do because in all such cases, the completion is well-defined).  \end{rmk}

\begin{rmk} \label{rmk:Hgrading}
    As in \Cref{rmk:Sgrading}, our choice of grading/orientation data on $\cH_{t,0}$ can be interpreted in the language of \cite[Section 5.3]{ganatra2018microlocal}. From the equation $\Omega_{\cH_{t,0}} \wedge dW_{t,0} = \Omega_{M^*_{\C^*}}$, we have that our grading/orientation data on $\cH_{t,0}$ agrees with the grading/orientation data induced by the stable polarization of $\cH_{t,0}$ coming from the tautological polarization of $T^*M^*_{S^1}$ and the trivialization of the normal bundle of $\cH_{t,0}$ by $dW_{t,0}$. 
\end{rmk}

As it will play a central role going forward, we introduce the abbreviated notation $\eff = \widehat{\fo} = \widehat{\cH}_{t,1}$. We can then similarly associate to $\eff$ a $\Z$-graded $\Z$-linear  wrapped Fukaya category $\cW(\eff)$ and a $\Z \oplus M^*$-graded version $\cW(\eff)^{M^*}$. 
However, $\fo$ is not a holomorphic submanifold, and thus it is more natural to induce grading/orientation data on $\eff$ directly from a stable polarization. 
Indeed, as a Liouville hypersurface of $\Ldom \subset M^*_{\C^*}$, $\fo$ has a natural stable polarization induced by the tautological polarization of $T^*M^*_{S^1} \simeq M^*_{\C^*}$ and the trivialization of the normal bundle of $\fo$ by the pair $(Z^\perp, R_\varphi)$ where $Z^\perp$ is the projection of the Liouville vector field on $M^*_{\C^*}$ to the symplectic orthogonal of $\fo$ and $R_\varphi$ is the Reeb vector field on $\partial \Ldom$. Equivalently, the trivialization of the normal bundle can be seen as coming from a standard neighbourhood of $\partial \ess$ produced in \cite[Proposition 2.25]{ganatra2020covariantly}.

In fact, the wrapped Fukaya categories of $\cH_{t,0}$ and $\eff$ are quasi-equivalent:
\begin{lem}\label{lem:compare_vol_form}
There is a $\Z \oplus M^*$-graded quasi-equivalence 
$$ \cW(\cH_{t,0})^{M^*} \simeq \cW(\eff)^{M^*} $$
for any sufficiently large $t$. 
\end{lem}
\begin{proof}
    From \Cref{lem:tailor_equivalence}, we have a Liouville isomorphism $\widehat{\cH}_{t,0} \simeq \widehat{\cH}_{t,1} = \eff$, resulting in an $M^*$-graded quasi-equivalence. It remains only to check that the Liouville isomorphism respects our chosen grading/orientation data. To show this, we work in terms of stable polarizations (see \Cref{rmk:Sgrading} and  \Cref{rmk:Hgrading}).  

    Through the deformation from $\cH_{t,0}$ to $\cH_{t,1}$ via the family $\cH_{t,s}$, we have a homotopy of grading/orientation data via the stable polarization of $\cH_{t,s}$ obtained from the tautological polarization of $T^*M^*_{S^1}$ and the trivialization of the normal bundle to $\cH_{t,s}$ by $dW_{t,s}$. We now must compare this grading/orientation data on $\cH_{t,1}$ with that on $\eff$. However, \cite[Proposition 5.3]{zhou2020lagrangian} shows that on a subset of $\cH_{t,1}$ homotopy equivalent to both $\eff$ and $\cH_{t,1}$ the Hamiltonian vector field $\vX_{\Im \, W_{t,1}}$ of $\Im \, W_{t,1}$ is positively proportional to $Z^\perp$. Thus, the oriented frames $\left(Z^\perp, R_\varphi\right)$ and $\left(\vX_{\Im \, W_{t,1}}, \vX_{\Re \, W_{t,1}}\right)$ cannot rotate relative to each other around any cycle, i.e., they induce homotopic trivializations of the normal bundle.
\end{proof}

\begin{rmk}
    The equivalence of the grading/orientation data in the proof of \cref{lem:compare_vol_form} and part of the isomorphism of Liouville structures have recently also been proved in \cite[Theorem 9.2]{spenko2024hms}. 
\end{rmk}

\subsection{Cup functor}
\label{subsec:cupFunctor}

By virtue of the fact that $\ess \subset (M^*_{\C^*}, \fo)$ is the inclusion into the stopped convex completion of $\ess$, the complement of $\ess$ can be identified with the stopped Liouville sector $(\eff \times \C_{\Re \leq 0}, \fo)$, where $\eff \times \C_{\Re \leq 0} \hookrightarrow M^*_{\C^*}$ is the canonical neighborhood of $\fo \subset \partial_{\infty} M^*_{\C^*}$.
The sector inclusion $(\eff \times \C_{\Re \leq 0}, \fo) \hookrightarrow (M^*_{\C^*}, \fo)$ induces (after appeal to \cite[Corollary 3.9]{ganatra2018sectorial}) a functor $\W(\eff \times \C_{\Re \leq 0}, \mathfrak{f}_{\Sigma^*}) \to \W(M^*_{\C^*},\mathfrak{f}_{\Sigma^*})$. By composing with stabilization \cite[\S 8.4]{ganatra2018sectorial}, we obtain a functor $\W(\eff) \to \W(M^*_{\C^*},\mathfrak{f}_{\Sigma^*})$, which we call in this paper the cup functor $\cup$ (compare \cite[\S 7.3]{ganatra2018microlocal}, \cite[Example 10.7]{ganatra2018sectorial}, \cite[\S 2.4]{sylvan2019orlov}).

There is an equivalent realization of the $\cup$ functor with target the quasi-equivalent $\cW(\Sprime)$, where $\Sprime$ is a small deformation of $\ess$ (deformed along the lines of \cite[Proposition 2.28]{ganatra2020covariantly}) which has {\em exact boundary} in the sense of \cite[Definition 2.10]{ganatra2020covariantly}. From the fact that $\Sprime$ has exact boundary, one obtains an embedding of Liouville sectors $\eff \times T^*[0,1] \hookrightarrow \Sprime$ covering $\Nbd^Z\partial \Sprime$ --- called here a {\em boundary sector embedding} --- whose induced pushforward can then be composed with stabilization to define $\cup: \cW(\eff) \to \cW(\Sprime)$. 

In \cref{adjustedsectors} below, we recall a particular method of constructing the deformation $\Sprime$, close to that appearing in the proof of \cite[Theorem 1.28]{ganatra2018sectorial}, as a subsector of $(M^*_{\C^*}, \fo)$ contained in an arbitrarily small cylindrical neighborhood of $\ess$ with boundary outside $\ess$. We also review how the deformation argument provides us with an extended defining function $\Iprime$ on $\Sprime$ which has controlled deviation from, and some compatibilities with, our previously fixed extended 1-defining function $\aye$ for $\partial \ess$. We define $(\Sprime,\Iprime)$ to be the outcome of applying the following lemma to our $(B,\fo) \subset M^*_{\C^*}$ with associated sector and 1-defining function $(\ess,\aye)$ constructed in \cref{sec:Liouvsec} and \cref{lem:compatibleIandW}. We will use the third property of $(\Sprime,\Iprime)$ below in \cref{subsecHHP}.

\begin{lem}\label{adjustedsectors}
    Let $(B,\fo)$ be a sutured Liouville domain with $B$ completing to a Liouville manifold $M$, with $\ess \subset M$ the associated Liouville sector. 
        Let $\aye$ be any extended defining function defined in a neighbourhood of $\partial \ess$ within $M$ (for instance, the one from \S \ref{sec:Liouvsec} and \cref{lem:compatibleIandW}).
    In that setting, there exists a small deformation $\ess \subset \Sprime \subset (M,\fo)$, along with an extended 1-defining function $\Iprime$ defined near $\partial \Sprime$ satisfying the following properties:
    \begin{enumerate}
        \item \label{Sclose} $\Sprime$ is contained in $\ess^+:=$ an arbitrarily small outward $X_{\aye}$ pushoff of $\ess$. 
        \item \label{boundaryexact} 
            There is a boundary sector embedding $\eff \times T^*[0,1] \hookrightarrow \Sprime$ (and in particular $\Sprime$ has exact boundary) living in the complement of $\ess$ and fitting into a commutative diagram:
    \begin{equation}\label{cupsquare}
         \begin{tikzcd}
             \eff \times T^*[0,1]\arrow[hookrightarrow]{d}{}  \arrow[hookrightarrow]{r}{} & (M \backslash \ess^{\circ}, \fo) \arrow[hookrightarrow]{d}{}   \\
         \Sprime  \arrow[hookrightarrow]{r}{} & (M,\fo) \end{tikzcd}
        \end{equation}
        where the horizontal inclusions induce equivalences on wrapped Fukaya categories.

    \item \label{arcliketransport} There is a Hamiltonian vector field $V$ on $\Sprime \backslash \ess^{\circ}$ which (while not necessarily linear near infinity on the whole region) agrees with $X_{\Iprime}$ near $\partial \Sprime$ and $X_{\aye}$ near $\partial \ess$ and carries $\partial \ess \stackrel{\sim}{\to} \partial \Sprime$.

    \end{enumerate}
\end{lem}
\begin{proof}

We begin with $\ess \hookrightarrow (M, \fo)$ and $\aye$ as previously constructed. Let $\ess^+$ be the outward flow of $\ess$ by $X_{\aye}$ for some (unspecified)
small time $a > 0$. 
On the region
$\ess^+ \backslash \ess^{\circ} = \Nbd^{Z}(\partial \ess^+)$, $\aye$ determines
an identification with $(\eff \times T^*[0,a], \lambda_{\eff} + \lambda_{T^*[0,a]} +
df)$ for some function $f$ satisfying properties as in \cite[Proposition
2.25]{ganatra2020covariantly} and where $\eff = \widehat{\fo}$; importantly, $f$ may not be compactly supported
(otherwise we'd be done). After rescaling $\aye$, we identify $T^*[0,a] =
T^*[0,1]$ for simplicity.

Next, as in the proofs of \cite[Theorem 1.28]{ganatra2018sectorial} and
\cite[Proposition 2.28]{ganatra2020covariantly} (and closer in construction to
the former), we deform the Liouville one-form $\lambda
\stackrel{\lambda_t}{\leadsto} \lambda'$ within the $\eff \times T^*[0,1]$ region above to turn off the $df$ term in a
subshell $\eff \times T^*[\frac{2}{5},\frac{3}{5}]$ without changing $\lambda$
outside of $\eff \times T^*[\frac{1}{5},\frac{4}{5}]$. Concretely we take the straight line homotopy between $\lambda_0:=\lambda$ and $\lambda_1:= \lambda'= \lambda_{\eff} + \lambda_{T^*[0,1]} + d(\varphi(s) f)$ where $\varphi(s) = 1$ for $s$ outside $[\frac{1}{5},\frac{4}{5}]$ and $\varphi(s) = 0$ for $s$ inside $[\frac{2}{5},\frac{3}{5}]$, and extend this to a homotopy $\lambda_t$ on $M$ which is constant outside this region. 
The family $\lambda_t$ is a convex-at-infinity deformation of Liouville structures as in {\em loc. cit.} It therefore induces a Liouville isomorphism $\Phi: (M, \lambda_1) \to (M,\lambda_0)$, which is identity inside $\ess$ and outside $\ess^+$ by \cite[Proposition 11.8]{cieliebakeliashberg2012}.
Now with respect to the new Liouville form $\lambda_1$ consider  $S_1:= \ess \cup (\eff \times T^*[0,\frac{3}{5}])$ which sits between $\ess$ and $\ess^+$ and is a Liouville subsector of $(M,\lambda_1)$. We observe the following:
\begin{itemize}
    \item $S_1$ is contained in $\ess^+$;
    \item $S_1$ has a boundary sector embedding $(\eff \times T^*[\frac{2}{5},\frac{3}{5}],\lambda_{\eff} + \lambda_{T^*[\frac{2}{5},\frac{3}{5}]}) \hookrightarrow S_1$ which lives in $S_1 \backslash \ess^{\circ}$.  Now identify $T^*[\frac{2}{5},\frac{3}{5}]$ (by scaling up in the base and shrinking in the fiber) with $T^*[0,1]$.

    \item The inclusion of $\eff \times T^*[\frac{2}{5},\frac{3}{5}]$ into $\ess^+ \backslash \ess^{\circ}$ is deformation equivalent to a trivial inclusion, and hence by \cite[Lemma 3.4]{ganatra2018sectorial} induces an equivalence on wrapped Fukaya categories;

    \item Similarly, the inclusion of $S_1$ into $\ess^+$ is deformation equivalent to a trivial inclusion, and hence induces an equivalence on wrapped Fukaya categories.

    \item The vector field $X_{\aye}$ is a Hamiltonian vector field (indeed we haven't changed $\omega$!) taking $\partial \ess$ to $\partial S_1$. The associated Hamiltonian function $\aye$ is an extended 1-defining function near $\partial \ess$ and $\partial S_1$, but may fail to be linear on all of $S_1 \backslash \ess^{\circ}$ (given our modifications to $\lambda$).

\end{itemize}

Now push forward by $\Phi$: let $\Sprime:= \Phi(S_1)$ with extended 1-defining function $\Iprime := \Phi_*(\aye)$. Since $\Phi|_{\ess} = id$ and $\Phi|_{M \backslash (\ess^+)^\circ} = id$, we see that $\Sprime$ satisfies \cref{Sclose}.
The remaining conditions other than the top horizontal arrow of \cref{cupsquare} inducing an equivalence follow immediately from taking $\Phi$ of the above bullet points. 
To study the effect of the top horizontal arrow of \cref{cupsquare} on wrapped Fukaya categories, first use $\Phi$ of the third bullet point to obtain an equivalence $\W(\eff \times T^*[0,1]) \stackrel{\sim}{\to} \W(\ess^+ \backslash \ess^{\circ})$. 

The desired equivalence therefore follows from (by composing with) the claim that the sector embedding $(\ess^+ \backslash \ess^{\circ}) \hookrightarrow (M \backslash \ess^{\circ}, \fo)$ induces an isomorphism on wrapped Fukaya categories. 
To establish the claim, it suffices to replace $\ess^+$ by the sufficiently small time outward flow $(\ess^+)'$ of any other extended 1-defining function $I'$ by the following sandwiching argument, compare \cite[Lemma 3.4]{ganatra2018sectorial}: Denote by $\Nbd^Z(\partial \ess)_{I'}^{+a}$ the complement of $\ess^{\circ}$ in the outward time-$a$ flow of $\ess$ by $X_{I'}$. Then observe that for any $I_1$, $I_2$ there exist sufficiently small $a <c$, $b<d$ and a sequence of inclusions $\Nbd^Z(\partial \ess)_{I_1}^{+a} \hookrightarrow \Nbd^Z(\partial \ess)_{I_2}^{+b} \hookrightarrow \Nbd^Z(\partial \ess)_{I_1}^{+c} \hookrightarrow \Nbd^Z(\partial \ess)_{I_2}^{+d}$. Since any composition of adjacent arrows $\Nbd^Z(\partial \ess)_{I_j}^{+i_1} \hookrightarrow \Nbd^Z(\partial \ess)_{I_j}^{+i_2}$ is a trivial inclusion, inducing an equivalence on wrapped Fukaya categories \cite[Lemma 3.41]{ganatra2020covariantly}, it follows that $\cW(\Nbd^Z(\partial \ess)_{I_1}^{+a}) \stackrel{\sim}{\to} \cW(\Nbd^Z(\partial \ess)_{I_2}^{+b})$. 

It is now enough to show the claim holds after replacing $\aye$ with the extended 1-defining function  $\aye'= -e^s t$, with associated
vector field $X_{\aye'} = -\frac{\partial}{\partial s} + t
\frac{\partial}{\partial t} + Z_{\lambda_{F_0}}$. 
Here recall that in local coordinates $\ess$ is the complement
of the set $(\R_{>0})_s \times \R_{|t| < \eps} \times \fo$ in $M$ (with smoothed boundary) where $s$ is the symplectization and $t$ is a Reeb coordinate as in
the previous section, and the Liouville one-form is $\lambda = e^s (dt - \lambda_{\fo})$.
The outward flow $\ess^+$ of $\ess$ by $X_{\aye}$ for a small
 time $a > 0$
has $\partial \ess^+$ equal to (a smoothing of) the boundary of
the set $(\R_{>a})_s \times \R_{|t| < e^{-a}\eps} \times {\fo}^{-a}$, where
$\fo^{-a}$ denotes the result of inward flow by $Z_{\lambda_{\fo}}$. 
In particular, $S^+$ is the sector constructed from the sutured pair $(B^{+a}, \fo^{-a}) \subset M$ using the neighborhood $\R_{|t| < e^{-a}\eps} \times {\fo}^{-a}$. Hence $(M,\fo)$ is up to deformation the stopped convex completion of $S^+$ by \cite[Lemma 2.32]{ganatra2020covariantly}, and therefore also the convex completion along the boundary component $\partial S^+$ of $S^+ \backslash \ess^{\circ}$. In particular, we may apply \cite[Corollary 3.9]{ganatra2018sectorial} to the inclusion $(\ess^+ \backslash \ess^{\circ}) \hookrightarrow (M \backslash \ess^{\circ},
\fo)$ (note while {\em loc. cit.} is stated for the ``full'' stopped completion of a sector, it also applies immediately to the stopped convex completion along a single boundary component of a sector with disconnected boundary) to obtain
an equivalence $\W(\ess^+ \backslash \ess^{\circ}) \stackrel{\sim}{\to} \W(M \backslash \ess^{\circ}, \fo^{-a})$ which factors through the equivalence (by {\em loc. cit.} again) $\W(M \backslash \ess^{\circ}, \fo) \to \W(M \backslash \ess^{\circ}, \fo^{-a})$. Hence $\W(\ess^+ \backslash \ess^{\circ}) \stackrel{\sim}{\to} \W(M \backslash \ess^{\circ}, \fo)$ is an equivalence.
\end{proof}

Using $\Sprime$ and $\eff \times T^*[0,1] \hookrightarrow \Sprime$ given to us in \cref{adjustedsectors}, we immediately obtain a boundary inclusion functor 
\begin{equation}\label{boundaryinclusionfunctor}
    \W(\eff \times T^*[0,1]) \to \W(\Sprime)
\end{equation}
fitting into a commutative diagram
\begin{equation}\label{boundaryinclusionsquare}
 \begin{tikzcd}
     \W(\eff \times T^*[0,1])\arrow{d}{}  \arrow{r}{} & \W(\eff \times \C_{\Re \leq 0}, \mathfrak{f}_{\Sigma^*}) \arrow{d}{}   \\
 \W(\Sprime)  \arrow{r}{} & \W(M^*_{\C^*}, \mathfrak{f}_{\Sigma^*}) \end{tikzcd}
\end{equation}
where the horizontal arrows are equivalences. Here we have implicitly applied \cite[Corollary 3.9]{ganatra2018sectorial} and also used the fact that $(M \backslash \ess^{\circ}, \fo)$
up to deformation is a model for the canonical $\eff \times \C_{\Re \geq 0}$ neighborhood of $\fo$ in $M$ at infinity \cite[Lemma 2.32]{ganatra2020covariantly}. 

Next, there is a K\"{u}nneth/stabilization embedding $stab: \cW(\eff) \stackrel{\sim}{\to}\cW(\eff \times T^*[0,1])$  \cite[\S 8.4]{ganatra2018sectorial}, defined by multiplying with a cotangent fiber using the K\"{u}nneth embedding
$$
stab: \cW(\eff) \xrightarrow{- \otimes \R} \cW(\eff) \otimes \cW(T^*[0,1]) \xrightarrow{\text{K\"{u}nneth}} \cW(\eff \times T^*[0,1]).
$$
Here we are considering the object $\R$ of $\cW(T^*[0,1])$ given by the cotangent fibre $\{x=\frac{1}{2}\}$, equipped with the standard spin structure (which in particular induces the standard orientation on $\R$). 
The squared phase function is then constant equal to $2\mathrm{arg}(\Omega_{T^*[0,1]}(\partial_y)) = \pi \in S^1$. 
We choose the grading to be equal to $\pi \in \R$. 
Note that these choices agree with the grading/orientation data on $T^*[0,1]$ and $\R$ induced by the tautological polarization of $T^*[0,1]$. 

Composing with the inclusion/equivalence $\W(\eff \times T^*[0,1]) \stackrel{\sim}{\to} \W(\eff \times \C_{\Re \leq 0}, \mathfrak{f}_{\Sigma^*})$ gives a stabilization functor $stab: \cW(\eff) \stackrel{\sim}{\to} \W(\eff \times \C_{\Re \leq 0}, \mathfrak{f}_{\Sigma^*})$.\footnote{It is tempting to pick the stabilization functors $stab: \cW(\eff) \stackrel{\sim}{\to}\cW(\eff \times T^*[0,1])$ and $stab: \cW(\eff) \stackrel{\sim}{\to} \W(\eff \times \C_{\Re \leq 0}, \mathfrak{f}_{\Sigma^*})$ completely independently and then argue that they are the ``same'' after including the first codomain into the second. While this is expected to hold, it is not shown in  \cite{ganatra2018sectorial} that K\"{u}nneth is functorial under arbitrary sector inclusions or independent of all choices made. However, in order to apply the results of \cite{ganatra2018microlocal} later, we are free to pick one such stabilization functor once and for all.
The needed structural results from {\em loc. cit.} --- articulated in \cref{thm:GPSmicro} below --- hold for any such choice.} 
As in \cite[Section 2.4]{sylvan2019orlov}, we then define the cup functor 
\begin{equation}\label{cupeqn}
\cup: \cW(\eff) \to \cW(\Sprime)
\end{equation} 
by composing $stab$ with \eqref{boundaryinclusionfunctor}. At the start of this subsection, we defined $\cup$ as the composition $\cW(\eff) \stackrel{stab}{\to} \W(\eff \times \C_{\Re \leq 0}, \mathfrak{f}_{\Sigma^*}) \to \W(M^*_{\C^*}, \mathfrak{f}_{\Sigma^*})$. The commutative diagram \eqref{boundaryinclusionfunctor} immediately implies:
\begin{cor}\label{cupfunctorcommutativity}
For $\Sprime$ as in the previous Lemma, there is a commutative diagram:
\begin{equation}\label{cupfunctorsquare}
 \begin{tikzcd}
     \W(\eff)\arrow{d}{\cup}  \arrow{r}{=} & \W(\eff) \arrow{d}{\cup}   \\
 \W(\Sprime)  \arrow{r}{\sim} & \W(M^*_{\C^*}, \mathfrak{f}_{\Sigma^*}) \end{tikzcd}
\end{equation}
\end{cor}

Note that $\cup$ is a $\Z$-graded functor with respect to the natural
grading/orientation data on $F$ described in \Cref{subsec:fukBackground1}. As
described in \cref{subsec:fukBackground1} for $\cW(\ess)$, there is also a $\Z
\oplus M^*$ graded version $\cW(\Sprime)^{M^*}$ of the wrapped Fukaya category of
$\Sprime$ (with $\cW(\ess)^{M^*} \to \cW(\Sprime)^{M^*}$ a quasi-equivalence), and it is straightforward
to enhance \eqref{cupeqn} to a $\Z \oplus M^*$-graded functor
\begin{equation}
   \cup^{M^*}: \cW(\eff)^{M^*} \to \cW(\Sprime)^{M^*}
   \label{eq:gradedcup}
\end{equation}

\subsection{Cap functor}

We can associate to the cup functor \eqref{cupeqn} a formal adjoint functor $\cup^*:\mathcal W(\Sprime)\to \operatorname{Mod} \mathcal W(\eff)$ by composing the pullback of modules with Yoneda. In good situations, this functor is representable (i.e., there is an adjoint to the $\cup$-functor). 
For example, in the setting of symplectic Landau-Ginzburg models, it is known that the cup functor is spherical \cite{abouzaidganatra, sylvan2019orlov}, and therefore has both left and right adjoints. The right adjoint is usually called the \emph{cap} functor as for a class of Lagrangian submanifolds $L\subset \ess$ it can be geometrically described as ``intersection of $L$ with the fiber $\eff$''.
Several incarnations of this functor have been studied by various authors such as \cite{abouzaidseidel,abouzaidganatra,abouzaidsmith,SeidelLefschetz6} with different geometric hypotheses.
In the remainder of this section, we define the version of the cap functor that we will use on a particular class of Lagrangians that we now define.

\begin{defn}[Restatement of \cref{def:Iarclike}]
    Let $L\subset \Sprime$ be a compact Lagrangian submanifold with $\partial L\subset \partial \Sprime$. We say that $L$ is \emph{$\Iprime$-arclike} if $L$ is transverse to $\partial \Sprime$, and $I$ is constant along some neighbourhood of $\partial L$ in $L$. Here $I$ is the same extended $1$-defining function which was used in the definition of the cup functor.
\end{defn}

In \Cref{sec:adjoints}, we introduce a subcategory  $\cP(\Sprime) \subset \cW(\Sprime)$ of \emph{negative pushoffs of ($I$-)arclike Lagrangians}, and show that intersection with the fiber $\eff$ is adjoint to the cup functor. More precisely, in \Cref{lem:def_capl}, we show there exists an $A_\infty$ functor
$$\cap: \cP(\Sprime) \to \cF(\eff)$$
which sends a negative-pushoff $L^-$ of an arclike Lagrangian $L$ to the boundary $\partial L$ (which is compact, and in particular defines an object of the compact Fukaya category $\cF(\eff) \subset \cW(\eff)$). 
Furthermore, there exist functorial isomorphisms 

\[\Hom_{H\cW(\Sprime)}(\cup K,L) \cong \Hom_{H\cW(\eff)}(K,\cap L)\]
for all objects $K$ of $\cW(\eff)$, $L$ of $\cP(\Sprime)$ which lift to an equivalence of suitable $\cW(\eff)\!-\!\cP(\Sprime)$ bimodules on the chain level.
Thus, $\cap$ behaves like a right adjoint to $\cup$ restricted to $\cP(\Sprime)$. We say that $\cap$ is a {\em partially-defined right adjoint} to $\cup$ defined along $\cP(\Sprime)$, in the terminology of \cref{sec:adjoints}. The construction of cap is agnostic as to the grading structures imposed upon our categories; in particular, it also gives us a $\Z \oplus M^*$ graded version of $\cap$:
\begin{equation}\label{gradedcap}
   \cap^{M^*}: \cP(\Sprime)^{M^*} \to \cF(\eff)^{M^*}.
\end{equation}

\section{Homological mirror symmetry at large radius}
\label{sec:ASideSetup}
In this section, we recall and analyze homological mirror symmetry relating the sector $\ess$ and its boundary fiber $\eff$ on the A-side to $Y^*$ and its toric boundary $\partial Y^*$ on the B-side.

The main result of this section is \cref{prop:GS_Li}, which (in part) constructs a quasi-equivalence $\Pic_{dg}^{M^*}(\partial Y^*) \simeq \partial \mafuk^{M^*}$ fitting into a diagram of $A_\infty$ categories and functors:
	\[\begin{tikzcd}
            D^b_{dg} Coh(\partial Y^*) \arrow[leftrightarrow]{r}{GS} & \Perf \cW(\eff) \\
			\Pic_{dg}^{M^*}(\partial Y^*)  \arrow[leftrightarrow]{r} \arrow[hookrightarrow]{u} & \partial \mafuk^{M^*}  \arrow[hookrightarrow]{u} \\
			\Pic_{dg}^{M^*}(Y^*) \arrow{u}{ i^!}  \arrow[leftrightarrow]{r}{A} & \mafuk^{M^*} \arrow{u}{\cap_{mon}},\end{tikzcd}.
   \]
The top horizontal arrow is a quasi-equivalence constructed by Gammage--Shende \cite{gammage2017mirror} and the bottom horizontal arrow is a quasi-equivalence constructed by Abouzaid \cite{abouzaid2006homogeneous,abouzaid2009morse}. 
On the lower right-hand side, we have the ``monomially admissible Fukaya categories'':  $\mafuk$ is the category of tropical Lagrangian sections of $S$; $\partial \mafuk$ is the subcategory of the compact Fukaya category $\cF(\eff)$ consisting of boundaries of tropical Lagrangian sections. On the left-hand side, $\Pic_{dg}(Y^*)$ is the category of line bundles on $Y^*$, and $\Pic_{dg}(\partial Y^*)$ is the category of line bundles on $\partial Y^*$ that are restrictions of line bundles on $Y^*$. On both sides, the $M^*$ superscript indicates that these categories are $M^*$-graded. The functors $ i^! $, $A$, and $\cap_{mon}$ are $M^*$-graded; we show that the middle horizontal arrow can also be made $M^*$-graded, on the subcategory we are interested in, in \cref{sec:A0_B0}. 

\subsection{Homological mirror symmetry for Fano toric varieties}\label{sec:hmstoric}

We begin with homological mirror symmetry where the toric variety $Y^*$ is on the $B$-side. The relevant dg category associated with $Y^*$ is $D^b_{dg}Coh(Y^*)$. 
On the $A$-side, the relevant $A_\infty$ category is the wrapped Fukaya category $\mathcal{W}(\Sprime)$ of $\Sprime$ from \Cref{subsec:fukBackground1}. We equip cotangent fibers (and Lagrangians Hamiltonian isotopic to them) and the zero section $M^*_{S^1}$ with canonical grading/orientation data coming from the tautological polarization of $T^*M^*_{S^1} \simeq M^*_{\C^*}$ (see \Cref{rmk:Sgrading}).

We will recall two independently established statements of homological mirror symmetry for toric varieties from the literature relating $\mathcal{W}(\Sprime)$ and $D^b_{dg}Coh(Y^*)$. We will then prove a few properties of these quasi-equivalences and use a theorem of Bondal and Orlov \cite{bondal2001reconstruction} to show that they coincide for smooth Fano toric varieties.

We will use a slight enlargement of the category $\mathcal{W}(\Sprime)$.
Namely, the objects will be pairs $(L,\zeta)$ where $L$ is an exact and cylindrical Lagrangian brane and $\zeta$ is a $\Bbbk$-local system on $L$; counts of holomorphic discs are weighted by the holonomy of the local systems around the boundary. In fact, we will only need to consider local systems on the exact Lagrangian torus $M^*_{S^1}$. Note that the $\Bbbk$-local systems on $\Sprime$
are in natural correspondence with $M_{\Bbbk^*}$ from the isomorphisms $\pi_1(\Sprime) \simeq \pi_1(M^*_{S^1}) \simeq M^*$. 

\begin{rmk} Allowing Lagrangian objects to be equipped with local systems does not change any of the properties that we will use about $\mathcal{W}(\Sprime)$ essentially due to the fact that $\mathcal{W}(\Sprime)$ is generated by simply connected Lagrangians. In particular, the main tools of \cite{ganatra2018sectorial} such as the K\"{u}nneth embedding, wrapping exact triangle, and stop removal are still valid.
\end{rmk}

One powerful approach to proving homological mirror symmetry for toric varieties is to pass through a microlocal sheaf category.
This method has given the most general proofs, and for us, has the important benefit of having been proven to be compatible with restriction to the toric boundary (see \Cref{subsec:boundaryhms}).
There is a long history of computing microlocal sheaf categories corresponding to derived categories of toric varieties beginning with the announcement \cite{bondal2006derived} and results of \cite{fang2011categorification,fang2012t}, and culminating in the general proof of \cite{kuwagaki2020nonequivariant}.
In these works, the category of interest is the dg category $\mathrm{Sh}_{-\mathbb{L}_{\Sigma^*}}(M^*_{S^1})^c$ of compact objects in the category $\mathrm{Sh}_{-\mathbb{L}_{\Sigma^*}}(M^*_{S^1})$ of sheaves of dg modules on $M^*_{S^1}$ with microsupport in $-\mathbb{L}_{\Sigma^*}$, where the negation denotes applying multiplication by $-1$ only on the $M_\R$ factor of $M^*_{S^1} \times M_\R = T^*M^*_{S^1} \simeq M^*_{\C^*}$. When combined with the general correspondence \cite{ganatra2018microlocal} between microlocal sheaf categories and partially wrapped Fukaya categories, we collectively obtain the following (after restricting to our setting): 

\begin{thm}[\cite{kuwagaki2020nonequivariant,ganatra2018microlocal}]
	There are quasi-equivalences of $A_\infty$ categories 
	\[
		\begin{tikzcd} D^b_{dg}Coh(Y^*) \arrow{d}{K }\arrow{dr}{KGPS}\\
			\mathrm{Sh}_{-\mathbb{L}_{\Sigma^*}}(M^*_{S^1})^c \arrow{r}{GPS} &  \Perf \mathcal{W}(M^*_{\C^*}, \mathfrak{f}_{\Sigma^*}) \simeq \Perf \mathcal{W}(\Sprime).
		\end{tikzcd}
	\]
  \label{thm:kgps}
\end{thm}

\begin{rmk} \label{rmk:k=fltz}
    By \cite[Remark 9.3]{kuwagaki2020nonequivariant}, the functor $K$ agrees with the non-equivariant coherent-constructible correspondence in \cite{treumann2010remarks, fang2012t}.
\end{rmk}

The equivalence $GPS$ holds more generally for cotangent bundles of analytic manifolds. We will use the following property of this equivalence.

\begin{lem} \label{lem:gpslocsys}
    Let $Q$ be an oriented real analytic manifold and let $\Lambda \subset S^*Q$ be a closed subanalytic isotropic subset. The image of a local system $\zeta$ on $Q$ under the quasi-equivalence 
        $$ \mathrm{Sh}_{-\Lambda}(Q)^c \xrightarrow{GPS} \Perf \mathcal{W} (T^*Q, \Lambda) $$
    from \cite[Theorem 1.1]{ganatra2018microlocal} is quasi-isomorphic to $(Q, \zeta)$. 
\end{lem}
\begin{proof} 
    The quasi-equivalence $GPS$ is constructed by refining $\Lambda$ to be the conormal to a Whitney triangulation $\mathcal{S}$. We can assume we are in that setting as all linking disks are orthogonal to $(Q, \xi)$ and the stop removal functor is the quotient by linking disks by \cite[Theorem 1.20]{ganatra2018sectorial}.

    In that case, $GPS$ passes through an intermediate quasi-equivalence with perfect complexes over $\mathcal{S}$. The proof of \cite[Theorem 5.36]{ganatra2018microlocal} shows that $(Q, \Bbbk_Q)$ is identified with the indicator module $1_Q$ in $\text{Perf }\mathcal{S}$ where $\Bbbk_Q$ is the trivial local system. The sheaf $\Bbbk_Q$ is also identified with $1_Q$ by \cite[Lemmas 4.6 and 4.7]{ganatra2018microlocal}. For general local systems, the same calculations apply to show that the local system $HF^*((Q, \zeta), L_q)$ for $q \in Q$ defined using continuation elements as in \cite[Lemma 5.10]{ganatra2018microlocal} coincides with $\zeta$.
\end{proof}

Now, we prove two properties of $KGPS$. We let $[0] \in M_{S^1}^*$ be the additive identity and set $L_{[0]}$ to be the inward conormal to a small ball around $[0]$.\footnote{Similar Lagrangians are considered in \cite[Section 5.7]{ganatra2018microlocal}. However, we take the opposite convention as $L_s$ in \cite{ganatra2018microlocal} denotes the outward conormal to a small ball around $s$.}

\begin{prop} \label{prop:kgpsproperties}
	The image of the structure sheaf $\mathcal O_{Y^*}$ under $KGPS$ is quasi-isomorphic to $L_{[0]}$, and for all $\xi \in M_{\Bbbk^*}$, the image of the shifted skyscraper sheaf $\mathcal{O}_{\xi}[-n]$ is quasi-isomorphic to $(M^*_{S^1}, \xi)$. 
\end{prop}
\begin{proof}
    For the structure sheaf, \cite[Lemma 10.1]{kuwagaki2020nonequivariant} shows that $K$ sends $\mathcal{O}_{Y^*}$ to the sheaf $\Bbbk_{[0]}$ where $[0] \in M_{S^1}^*$ is the unit. In particular, this object co-represents the stalk functor at $[0]$. On the other hand, the calculations in \cite[Section 5.7]{ganatra2018microlocal} imply that the image of $L_{[0]}$ under $GPS^{-1}$ co-represents the stalk functor at $[0]$ when $\mathbb{L}_{\Sigma^*}$ is refined to be the conormal to a Whitney triangulation. However, this property persists under stop removal as $L_{[0]}$ is left-orthogonal to the linking disks of the removed strata. 
    
    For $\xi \in M_{\Bbbk^*}$, \cite[Section 2]{treumann2019kasteleyn} shows that $K$ takes $\mathcal{O}_\xi[-n]$ to $\xi$. We can then apply \Cref{lem:gpslocsys} to deduce the result.
  \end{proof}

A second approach to homological mirror symmetry for toric varieties is to stay within the realm of Floer theory. This route will be beneficial to us for its direct access to Lagrangian submanifolds as objects. Following \cite{abouzaid2006homogeneous,abouzaid2009morse,hanlon2019monodromy}, we consider $A$- and $B$-side categories whose objects are integral piecewise-linear functions $\sF:M_\R \to \R$ which are linear on each cone of $\Sigma^*$. 

On the $A$-side, one constructs a Lagrangian submanifold $\Lmon(\sF) \subset T^*M^*_{S^1} \simeq M^*_{\C^*}$ corresponding to each such $\sF$, which comes equipped with a canonical grading as well as an `anchoring' in the sense of \cite{fukaya2010anchored}, i.e., a lift to the universal cover $T^*M^*_\R \to T^*M^*_{S^1}$. 
The anchoring of $\Lmon(\sF)$ is the graph of the differential of a smoothing of $\sF$, under the identification $T^*M^*_\R = M^*_\R \times M_\R = T^*M_\R$ (see \Cref{def:tropLagSection} for more details). 
As a result, the Lagrangians $\Lmon(\sF)$ are sections of $\pi$ (and hence also $\Log$) and are referred to as tropical Lagrangian sections. 
The category of such tropical Lagrangians sections is denoted $\mafuk^{M^*}$, and comes equipped with a $\Z \oplus M^*$-grading (the $\Z$-grading is the cohomological grading, and the $M^*$-grading comes from the anchorings). 

On the $B$-side, one considers the equivariant line bundles $\cO(\sF)$ corresponding to support functions. 
The category of such equivariant line bundles is denoted $\Pic_{dg}^{M^*}(Y^*)$, and comes equipped with a $\Z \oplus M^*$-grading (the $\Z$-grading is the cohomological grading, and the $M^*$-grading is given by characters of the algebraic torus whose action on the morphism spaces is determined by the equivariant structures). 
In \cite{abouzaid2009morse}, Abouzaid constructs a $\Z \oplus M^*$-graded $A_\infty$ quasi-isomorphism
\begin{align*}
\mafuk^{M^*} &\xrightarrow{A} \Pic_{dg}^{M^*}(Y^*),\qquad \text{which identifies} \\
\Lmon(\sF) & \leftrightarrow \cO_{Y^*}(\sF).
\end{align*}
In \cite[Theorem A]{hanlon2022aspects}, a $\Z$-graded fully faithful embedding
\begin{align*}
\mafuk^{M^*} & \xhookrightarrow{HH^{pre}} \Perf \cW(M^*_{\C^*}, \mathfrak{f}_{\Sigma^*}) \qquad \text{sending}\\
\Lmon(\sF) & \mapsto \Lconic(\sF)
\end{align*} 
(which forgets the $M^*$-grading) is constructed, and its image is shown to generate in \cite[Corollary C]{hanlon2022aspects}. 
Denote the composition with the quasi-inverse of the quasi-equivalence $\Perf \cW(\Sprime) \to \Perf \cW(M^*_{\C^*}, \mathfrak{f}_{\Sigma^*})$ by  $HH:\mafuk^{M^*}\to \Perf\cW(\Sprime)$.
Combining that result with the fact that $D^b_{dg}Coh(Y^*)$ is generated by line bundles \cite[Proposition 1.3]{abouzaid2009morse}, one obtains the following when restricting to our setting:

\begin{thm}[\cite{abouzaid2009morse,hanlon2022aspects}] \label{thm:ahh}
	There is a diagram
	\[
		\begin{tikzcd}
			\Pic_{dg}^{M^*}(Y^*) \arrow{r}{A} \arrow[d,hook] & \mafuk^{M^*} \arrow[hook]{d}{HH} \\
   D^b_{dg} Coh(Y^*)  \arrow{r}{AHH} &  \Perf\mathcal{W}(\Sprime)
		\end{tikzcd}
	\]
 which commutes up to quasi-isomorphism of $A_\infty$ functors. 
The top functor is a $\Z \oplus M^*$-graded quasi-isomorphism, the vertical functors are $\Z$-graded and cohomologically fully faithful, and the bottom functor is a $\Z$-graded quasi-equivalence.
\end{thm}

We now check that $AHH$ satisfies the analogue of \Cref{prop:kgpsproperties}.

\begin{prop} \label{prop:ahhproperties}
	The image of the structure sheaf $\mathcal{O}_{Y^*}$ under $AHH$ is quasi-isomorphic to $L_{[0]}$, and for all $\xi \in M_{\Bbbk^*}$, the image of the shifted skyscraper sheaf $\mathcal{O}_\xi[-n]$ under $AHH$ is quasi-isomorphic to $(M^*_{S^1}, \xi)$.
\end{prop}
\begin{proof}
    By definition, $AHH$ takes $\mathcal{O}_{Y^*}$ to $\Lconic(0)$. However, $\Lconic(0)$ is Hamiltonian isotopic to $L_{[0]}$ in the complement of $\mathfrak{f}_{\Sigma^*}$ via a linear isotopy of their primitives using the argument in \cite[Lemma 3.14]{hanlon2022aspects}. 

    We now consider the stop removal functor $\mathcal{W}(M^*_{\C^*}, \mathfrak{f}_{\Sigma^*}) \to \mathcal{W}(M^*_{\C^*})$ which is intertwined with the restriction $D^b_{dg}Coh(Y^*) \to D^b_{dg}Coh(M_{\mathbb{G}_m})$ under $AHH$ by \cite[Proposition 6.8]{hanlon2022aspects}. Moreover, the stop removal functor is the quotient by linking disks by \cite[Theorem 1.20]{ganatra2018sectorial}, and the linking disks are orthogonal to $(M^*_{S^1}, \xi)$. In addition, $AHH$ is compatible with taking products. Thus, we are reduced to checking the claim for the quasi-equivalence between $D^b_{dg} Coh(\mathbb{G}_m)$ and $\Perf \mathcal{W}(T^*S^1)$ which identifies $\mathcal{O}_{\mathbb{G}_m}$ with the cotangent fiber $L_{[0]}$. 
    
    In that case, we have $HW(L_{[0]}, L_{[0]}) \simeq \Bbbk[z^{\pm}]$ respecting product structures and concentrated in degree zero. In fact, we can choose a wrapping of $L_{[0]}$ so that the powers of $z$ are in bijection with intersection points and this isomorphism holds on the chain level. We therefore need to calculate $HW(L_{[0]}, (T, \xi)) \simeq \Bbbk \langle x \rangle$ as a left $\Bbbk[z^{\pm}]$-module where $x$ is a formal basis element of degree $n$. That is, we need to compute $\mu_2(z,x) $ as a multiple of $x$. There is exactly one disc on the cylinder contributing to this product and the portion of its boundary on $S^1$ sweeps out the circle once positively. Moreover, with the chosen orientation data on $S^1$, this disk contributes positively, and we obtain $\mu_2(z,x) = \xi x$ as desired.
\end{proof}

The proof of Proposition \ref{prop:sameHMS} will make use of some facts concerning autoequivalences of $D^b_{dg}Coh(Y^*).$ We first record the following preliminary technical result concerning Fourier-Mukai functors.

\begin{lem} \cite[Theorem 1.2]{CannocoStellari} Let $Y_1$ and $Y_2$ be two smooth projective varieties and let \begin{align*} G: H^0(D^b_{dg}Coh(Y_1)) \to H^0(D^b_{dg}Coh(Y_2))\end{align*} be an exact functor between triangulated categories.  If $G$ is a Fourier-Mukai functor,  meaning $G \cong \phi_\mathcal{E}$ for some $\mathcal{E} \in D^b_{dg}Coh(Y_1 \times Y_2)$,  then the cohomology sheaves of $\mathcal{E}$ are uniquely determined (up to isomorphism) by $G$.  \end{lem} 

We apply this result to obtain the following dg-lift of a celebrated result of Bondal and Orlov: 

\begin{thm} \label{thm:BO} Let $Y^*$ be any smooth Fano variety.  The autoequivalences of $D^b_{dg} Coh(Y^*)$ are generated by automorphisms of $Y^*$, twists by line bundles, and the shift functor.  In other words,  we have 
    \begin{align*} \label{eq:BO} 
    Aut(D^b_{dg} Coh(Y^*)) \cong Aut(Y^*) \ltimes (Pic(Y^*) \oplus \Z) 
\end{align*} 
    where $\Z$ is generated by the shift functor.  \end{thm} 
\begin{proof} 
In \cite[Theorem 3.1]{bondal2001reconstruction},  Bondal and Orlov prove this result at the level of triangulated categories.  To enhance this to the dg-level,  we note that the Fourier-Mukai kernels $\mathcal{E}$ corresponding to each of these autoequivalences are cohomologically concentrated in a single degree.  They thus lift uniquely (up to quasi-isomorphism) to $D^b_{dg} Coh(Y^* \times Y^*)$.   
\end{proof} 

\begin{prop} \label{prop:sameHMS}
	The quasi-equivalences $KGPS$ and $AHH$ are quasi-isomorphic. In particular, $KGPS$ sends $\mathcal {O}_{Y^*}(\sF)\mapsto \Lconic(\sF)$ for every toric divisor on $Y^*$.
\end{prop}
\begin{proof} 
	By Theorem \ref{thm:BO},  $G = KGPS^{-1} \circ AHH$ determines an element of $Aut(Y^*) \ltimes (Pic(Y^*) \oplus \Z).$  By the first parts of \Cref{prop:kgpsproperties} and \Cref{prop:ahhproperties},  $G$ preserves the isomorphism class of the structure sheaf.  Thus, $G$ is induced by an automorphism $f$ of $Y^*$. 

    Begin by assuming that $\Bbbk$ is algebraically closed. Then, we can identify this automorphism by its action on closed points. The second parts of \Cref{prop:kgpsproperties} and \Cref{prop:ahhproperties} show that $f$ preserves every point in the dense torus orbit of $Y^*$. Thus, $f=id$ and hence  $G$ is quasi-isomorphic to the identity.
    
    For general fields, we base change to the algebraic closure $\overline{\Bbbk}$. Then because both $AHH$ and $KGPS$ are compatible with base-change, the autoequivalence $G$ is quasi-isomorphic to the identity and the automorphism $f$ becomes the identity after base change to $\overline{\Bbbk}$. To conclude, note that $f$ is induced by a map on homogeneous coordinate rings which becomes the identity after tensoring with $\overline{\Bbbk}$. This implies that $f$ is already the identity. 
\end{proof}

\begin{rmk} A careful reader will note that \cite{abouzaid2009morse,hanlon2022aspects,kuwagaki2020nonequivariant} are all written with $\Bbbk =\C$ for simplicity. However, their arguments work with no modifications over any other coefficient field. For \cite{kuwagaki2020nonequivariant}, this has been previously claimed in \cite{gammage2017mirror,treumann2019kasteleyn}.
\end{rmk}

\subsection{Homological mirror symmetry for the affine hypersurface} \label{subsec:boundaryhms}
We now discuss homological mirror symmetry with the toric boundary $\partial Y^*$ on the $B$-side and $\eff$ on the $A$-side, where $\eff$ is the completion of the Liouvile domain $\fo \subset \partial B$ from \Cref{sec:Liouvsec}. The relevant categories are $D^b_{dg}Coh(\partial Y^*)$ and $\mathcal{W}(\eff)$ (see \Cref{subsec:fukBackground1}). 

Homological mirror symmetry in this setting was proved by Gammage and Shende \cite{gammage2017mirror} using \cite{ganatra2018microlocal} to convert sheaf-theoretic calculations to Fukaya categories and \cite{zhou2020lagrangian} to calculate skeleta. 
In line with \eqref{fsigma}, we let $-\mathfrak{f}_{\Sigma^*}$ be the boundary at infinity of $-\mathbb{L}_{\Sigma^*}$.
The microlocal sheaf analog of $\mathcal{W}(\eff)$ is the dg category $\mu\mathrm{sh}_{-\mathfrak{f}_{\Sigma^*}}(-\mathfrak{f}_{\Sigma^*})^c$ of compact objects in the dg category of microlocal sheaves on $-\mathfrak{f}_{\Sigma^*}$ (see \cite[Section 7.5]{ganatra2018microlocal} for a complete definition). 
We let $\mu^* : \mu\mathrm{sh}_{-\mathfrak{f}_{\Sigma^*}}(-\mathfrak{f}_{\Sigma^*})^c \to \mathrm{Sh}_{- \mathbb{L}_{\Sigma^*}}(M^*_{S^1})^c$ be the left adjoint to microlocalization. 
Then, \citeauthor{zhou2020lagrangian}'s computation of the skeleton of $F$ suffices to compute the cup functor in terms of these microlocal sheaf categories using the following result of \citeauthor{ganatra2018microlocal}:

\begin{thm}[Cf. Equations (1.4) and (7.31) of \cite{ganatra2018microlocal}] \label{thm:GPSmicro}
There is a diagram of $A_\infty$ functors 
$$\begin{tikzcd}
    \Perf \cW(\eff) \ar{d}{\cup} \ar[leftrightarrow]{r} & \mu\mathrm{sh}_{-\mathfrak{f}_{\Sigma^*}}(-\mathfrak{f}_{\Sigma^*})^c \ar{d}{\mu^*}\\
    \Perf \W(M^*_{\C^*}, \mathfrak{f}_{\Sigma^*})  \ar[leftrightarrow]{r}{GPS} & \mathrm{Sh}_{- \mathbb{L}_{\Sigma^*}}(M^*_{S^1})^c 
\end{tikzcd}$$
in which the horizontal functors are quasi-equivalences, and the diagram commutes up to natural quasi-isomorphism.  
\end{thm}

Then, \citeauthor{gammage2017mirror} compare these microlocal sheaf categories with coherent sheaves:

\begin{thm}[\cite{gammage2017mirror}, Theorem 7.4.1]\label{thm:GSmicro}
	There is a diagram of $A_\infty$ functors
	\[\begin{tikzcd}
			\mu\mathrm{sh}_{-\mathfrak{f}_{\Sigma^*}}(-\mathfrak{f}_{\Sigma^*})^c \ar{d}{\mu^*} \ar[leftrightarrow]{r} & D^b_{dg}Coh(\partial Y^*) \arrow{d}{ i_*}  \\
			\mathrm{Sh}_{- \mathbb{L}_{\Sigma^*}}(M^*_{S^1})^c \ar[leftrightarrow]{r}{K} & D^b_{dg}Coh(Y^*) \end{tikzcd}\]
	in which the horizontal functors are quasi-equivalences, and the diagram commutes up to natural quasi-isomorphism. 
\end{thm}

Putting \Cref{thm:GPSmicro,thm:GSmicro} together and using \cref{cupfunctorcommutativity} to trade $\cW(\eff) \stackrel{\cup}{\to} \W(M^*_{\C^*}, \mathfrak{f}_{\Sigma^*})$ for $\cW(\eff) \stackrel{\cup}{\to} \cW(\Sprime)$, we obtain homological mirror symmetry for the affine hypersurface intertwining $\cup$ and $i_* $ :

\begin{thm}[\cite{gammage2017mirror}]\label{thm:GS}
	There is a diagram of $A_\infty$ functors
	\[\begin{tikzcd}
			D^b_{dg}Coh(\partial Y^*) \arrow{d}{ i_*}  \arrow{r}{GS} & \Perf\mathcal W(\eff) \arrow{d}{\cup} \\
			D^b_{dg}Coh(Y^*) \arrow{r}{KGPS} & \Perf \mathcal W(\Sprime)\end{tikzcd}\]
	in which the horizontal functors are quasi-equivalences, and the diagram commutes up to natural quasi-isomorphism.
\end{thm}

Before proceeding further, we need to discuss the shriek pull-back functor $ i^! $. Note that there is a formal dg-adjoint (see Appendix \ref{section:adjoints}) to $ i_* $: 
\begin{align*} 
( i_* )^*: \Perf D^b_{dg}Coh(Y^*)  \to \operatorname{Mod} D^b_{dg}Coh(\partial Y^*). 
\end{align*} 
    Because $D^b_{dg} Coh(Y^*)$ is already triangulated and split-closed, we may remove `$\operatorname{Perf}$' from the source. 
    Observe that by the adjunction between $H( i_* )$ and (derived) shriek pull-back, we have that this adjoint functor agrees with the shriek  pull-back functor on the level of cohomology, and in particular lands in $\Perf D^b_{dg} Coh(\partial Y^*) = D^b_{dg} Coh(\partial Y^*)$. We take $ i^! $ to be the resulting functor: 
    \begin{align*} 
         i^! : D^b_{dg}Coh(Y^*)  \to D^b_{dg}Coh(\partial Y^*). 
    \end{align*}  
    \begin{rmk} Because dg-adjoints are unique up to quasi-isomorphism of functors (see Appendix \ref{section:adjoints}), the above definition of $ i^! $ will agree (up to quasi-isomorphism) with any other definition of shriek  pull-back that satisfies a dg-adjunction with $ i_* $.  \end{rmk}
\begin{cor}
There is a diagram of $A_\infty$ functors
	\[\begin{tikzcd}
			D^b_{dg}Coh(\partial Y^*)  \arrow{r}{GS} & \Perf \mathcal W(\eff)   \\
			D^b_{dg}Coh(Y^*) \arrow{u}{ i^! } \arrow{r}{KGPS} & \Perf \mathcal W(\Sprime) \arrow{u}{\cup^*}\end{tikzcd}\]
	which commutes up to natural quasi-isomorphism, and where $\cup^*$ is a right-adjoint to $\cup$. 
 \label{cor:adjoint_HMS}
\end{cor}
\begin{proof}
	By Theorem \ref{thm:GS} and the naturality of taking the pullback of modules, we have a commutative diagram
    \begin{equation}\label{adjointsquare}
         \begin{tikzcd}
			D^b_{dg}Coh(\partial Y^*)  \arrow{r}{GS} & \operatorname{Mod} \Perf \mathcal W(\eff)   \\
			D^b_{dg}Coh(Y^*) \arrow{u}{ i^! } \arrow{r}{KGPS} & \Perf \Perf \mathcal W(\Sprime). \arrow{u}{\cup^*}\end{tikzcd}
        \end{equation}
    Because $\Perf \mathcal W(\Sprime)$ is triangulated and split-closed, we may remove the outer `$\Perf$' from the bottom row. It follows by commutativity of the diagram that $\cup^*$ also lands in $\Perf \Perf \mathcal W(\Sprime) = \Perf \mathcal W(\Sprime)$, which yields the claimed commutative diagram.

\end{proof}

\subsection{The cap functor}

We need a geometric understanding of the equivalence $GS$. 
In this subsection, we will show that the restriction of a line bundle to $\partial Y^*$ is mirror to a compact Lagrangian. 
We do this by calculating the adjoint $\cap$ to $\cup$ on the subcategory of Lagrangians mirror to line bundles. This $\cap$ functor will be $\ZZ \oplus M^*$ graded.  
The results of this subsection invoke one technical result, \cref{cor:constructionOfP}, which we temporarily defer to \cref{subsecHHP}.

\begin{prop} \label{lem:innerSquare}
  There exists a diagram of $A_\infty$ categories and functors
  \begin{equation}
   \begin{tikzcd}
    \partial \mafuk^{M^*} \arrow{r}{j_A}& \Perf \mathcal W(\eff)\\
    \mafuk^{M^*} \arrow{r}{HH} \arrow{u}{\cap_{mon}} & \Perf \mathcal W(\Sprime) \arrow{u}{\cup^*}
  \end{tikzcd}
\end{equation}
  which commutes up to natural quasi-isomorphism. Furthermore:
  \begin{enumerate}
      \item The top and bottom arrows are injective on objects and homotopy equivalences on morphisms (in particular, they are cohomologically fully faithful);
      \item $\cap_{mon}$ is $\ZZ\oplus M^*$ graded.
  \end{enumerate}
\end{prop}
\begin{proof}

    The desired diagram follows from commutativity of the outer face of the following  diagram: 
    \[	
\begin{tikzpicture}
	     \node (v4) at (-7,1) {$\partial \mathcal F_{mon}^{M^*}$};
	     \node (v3) at (-7,-5) { $\mathcal F_{mon}^{M^*}$};
	     \node (v5) at (1,1) {$\Perf \mathcal W(\eff)$};
	     \node (v10) at (1,-5) {$\Perf \mathcal W(\Sprime)$};
	     
	     \node (v7) at (-3,-2) {$\mathcal P(\Sprime)^{M^*}$};
	    	     \draw  (v3) edge[->,dashed] node[midway, fill=white] {$\cap_{mon}$}(v4);
	   \draw  (v3) edge[->]  node[midway, fill=white] {$HH$} (v10);
	   \node (v1) at (-3,1) {$\mathcal W(\eff)^{M^*}$};
	   \draw  (v1) edge[->] (v5);
	   \draw  (v7) edge[->] node[midway, fill=white]{$\cap^{M^*}$} (v1);

	   \node at (-1,-1) {\Cref{lem:def_capl}};
	
	     \draw  (v4) edge[->,  dashed] (v1);
	 \draw  (v7) edge[ ->] (v10);
	  \draw (v10) edge[->] node[midway, fill=white]{$\cup^*$} (v5);
\node at (-3,-4) {\Cref{cor:constructionOfP}};
\draw  (v3) edge[->] node[fill=white, midway]{$HH_P$} (v7);
\end{tikzpicture}
    \]
 We define $\partial \mafuk^{M^*} \subset \mathcal \cF(F)^{M^*}$ to be the full subcategory whose objects lie in the image of  $ \cap^{M^*}\circ CF^*(\Lmon^+, -)$ where $\cap^{M^*}$ is as in \cref{gradedcap}.   We thus obtain a $\Z \oplus M^*$-graded functor  $\cap_{mon}:\mafuk^{M^*} \to \partial \mafuk^{M^*}$ which makes the upper left square of the following diagram commute.

\end{proof}

Let $\Pic_{dg}^{M^*}(\partial Y^*)$ be the full subcategory on objects in the image of $\Pic_{dg}^{M^*}(Y^*)\to D^b_{dg}Coh(Y^*)\to D^b_{dg}Coh(\partial Y^*)$.
\begin{prop}\label{prop:GS_Li}
There is a diagram of $A_\infty$ functors
	\[\begin{tikzcd}
			\Pic_{dg}^{M^*}(\partial Y^*)  \arrow[leftrightarrow]{r} & \partial \mafuk^{M^*}  \\
			\Pic_{dg}^{M^*}(Y^*) \arrow{u}{ i^! } \arrow[leftrightarrow]{r}{A} & \mafuk^{M^*} \arrow{u}{\cap_{mon}},\end{tikzcd}
   \]
in which 
\begin{enumerate}
    \item the bottom and vertical functors are $\Z \oplus M^*$-graded; \label{item:gradedMZ}
    \item the top and bottom functors are quasi-isomorphisms; \label{item:qi}
    \item the diagram commutes up to natural quasi-isomorphism;\label{item:commutes}
    \item \label{it:top_arr_GS} the top arrow of the diagram fits into a diagram
    \[\begin{tikzcd}
			D^b_{dg}Coh(\partial Y^*) \arrow[leftrightarrow]{r}{GS} & \Perf \mathcal W(\eff) \\
   \Pic_{dg}^{M^*}(\partial Y^*)  \arrow[leftrightarrow]{r}  \arrow[hookrightarrow]{u} & \partial \mafuk^{M^*}  \arrow[hookrightarrow]{u}
			\end{tikzcd}
   \]
   which commutes up to natural quasi-isomorphism.
\end{enumerate}   
\end{prop}
\begin{proof}
The bottom functor is constructed and shown to be a $\Z \oplus M^*$-graded quasi-isomorphism in \cref{thm:ahh}. 
The left vertical functor is manifestly $\Z \oplus M^*$-graded; the right vertical functor is constructed, and shown to be $\Z \oplus M^*$-graded, in \cref{lem:innerSquare}.

In order to construct the top functor, we observe that both desired commutative diagrams fit in as the left and top faces of the following cube, where every map apart from the dashed line has been defined, and $j_A, j_B, \tilde j_B$ are inclusions of subcategories.
\begin{equation}
  \begin{tikzpicture}    \node (v2) at (-7.5,3) {$Pic_{dg}^{M^*}(\partial Y^*)$};
    \node (v1) at (-7.5,-2) {$Pic_{dg}^{M^*}(Y^*)$};
    \node (v4) at (-5,1.5) {$\partial \mathcal F_{mon}^{M^*}$};
    \node (v3) at (-5,-0.5) { $\mathcal F_{mon}^{M^*}$};
    \node (v5) at (-1,-0.5) {$\Perf \mathcal W(S)$};
    \node (v6) at (-1,1.5) {$\Perf \mathcal W(\eff)$};
    \node (v7) at (2.5,-2) {$D^b_{dg}Coh(Y^*)$};
    \node (v8) at (2.5,3) {$D^b_{dg}Coh(\partial Y^*)$};
    \draw  (v1) edge[->] node[midway, fill=white] {$ i^! $} (v2);
    \draw  (v3) edge[->] node[midway, fill=white] {$\cap_{mon}$}(v4);
    \draw  (v5) edge[->] node[midway, fill=white] {$\cup^*$} (v6);
    \draw  (v7) edge[->] node[midway, fill=white] {$ i^! $} (v8);
    \draw  (v3) edge[->] node[midway, fill=white] {$HH$} (v5);
    \draw  (v4) edge[->] node[midway, fill=white]{$j_A$}(v6);
    \draw  (v2) edge[->] node[midway, fill=white]{$j_B$} (v8);
    \draw  (v1) edge[->] node[midway, fill=white]{$\tilde j_B$}(v7);
    \draw  (v1) edge[->] node[midway, fill=white,] {A} (v3);
    \draw  (v2) edge [dashed, ->] (v4);
    \draw  (v6) edge[<-] node[midway, fill=white] {$GS$} (v8);
    \draw  (v5) edge[<-] node[midway, fill=white] {$KGPS$}(v7);
    \node at (-6.5,0.5) {};
    \node at (-3,-1.5) {\Cref{thm:ahh} and \Cref{prop:sameHMS}};
    \node at (0.5,0.5) {\Cref{cor:adjoint_HMS}};
    \node at (-3,0.5) {\Cref{lem:innerSquare}};
  \end{tikzpicture}	
  \label{eq:bigDiagram}
\end{equation}

The outer face commutes by definition of the category $\Pic_{dg}^{M^*}(\partial Y^*)$.
The bottom, inner, and right faces commute up to natural quasi-isomorphism by \Cref{thm:ahh,prop:sameHMS}, \Cref{lem:innerSquare}, and \Cref{cor:adjoint_HMS} respectively. 
We construct the dashed arrow so that the top face of the commutative diagram commutes up to natural quasi-isomorphism, by applying \Cref{lem:const_fun_easy} to the functors $\cF = GS \circ j_B$ and $\cG = j_A$. 
To check the hypotheses of the Lemma, we first observe that both $j_A$ and $GS \circ j_B$ are cohomologically fully faithful, as $j_A$ and $j_B$ are inclusions of subcategories and $GS$ is a quasi-equivalence.  
To check that $GS \circ j_B$ and $j_A$ have the same essential image, we use the commutativity of the bottom, inner, right, and outer faces, to show that $j_A \circ \cap_{mon} \circ A$ is naturally isomorphic to $GS \circ j_B \circ  i^! $; together with the fact that $A$, $\cap_{mon}$, and the left-hand version of $ i^! $ are essentially surjective by definition. 
Thus we may construct the dashed arrow, and conclude that it is a quasi-equivalence, by \Cref{lem:const_fun_easy}. 
The commutativity of the left-hand and top square then follows from the commutativity of the remaining squares.

\end{proof}

\subsection{Arclike Lagrangians mirror to line bundles on toric varieties}\label{subsecHHP}

In \S \ref{subsec:productLagconstr}, we recall the construction of the tropical Lagrangian section $\tilde L_{mon}(\sF)$ from \cite{hanlon2019monodromy}. We prove it is arclike, and define $\tilde {L}(\sF)$ to be its negative pushoff in $\Sprime$ in the sense of Appendix \ref{sec:adjoints}. In \S \ref{subsec:HanlonHicksLag}, we recall the construction of the cylindrical (but not arclike) tropical Lagrangian section $\Lconic(\sF)$ from \cite{hanlon2022aspects}. Then, in \S \ref{subsec:lcon_lmon}, we prove \cref{cor:HHvsArclike}, which says that the Lagrangians $\Lconic(\sF)$ and $\tilde L(\sF)$ (together with their canonical anchorings) define quasi-isomorphic objects of $\mathcal W(M^*_{\C^*} , \mathfrak f_{\Sigma^*})^{M^*}$. Using this, we easily deduce the following result, which is the main result of this section, and a necessary input for the last subsection's \cref{lem:innerSquare}. 

\begin{prop}
\label{cor:constructionOfP}
There exists a $\Z \oplus M^*$-graded, cohomologically fully faithful $A_\infty$ functor $HH_P$ which makes the following diagram commute up to natural quasi-isomorphism:
\[\begin{tikzpicture}
\node (v1) at (-1.5,0.5) {$\mafuk^{M^*}$};
\node (v3) at (5.5,0.5) {$\Perf \mathcal W(\Sprime)^{M^*}$};
\node (v4) at (5.5,3) {$\mathcal P(\Sprime)$};
\draw  (v1) edge[->] node[midway, below]{$HH$} (v3);
\draw  (v4) edge[->] node[midway, right]{$\mathcal Y_{(-)}$} (v3);
\draw  (v1) edge[dotted, ->] node[midway, left]{$HH_P$} (v4);
\end{tikzpicture}.
\]
\end{prop}
\begin{proof}
    We apply \Cref{lem:const_fun_easy} to the $\Z \oplus M^*$-graded versions of the functors $\cF = HH$ and $\cG =  \mathcal Y_{(-)}$. 
    These are both cohomologically fully faithful.
    
    By \Cref{cor:HHvsArclike}, there exists a quasi-isomorphism in $\cW(M^*_{\C^*}, \mathfrak f_{\Sigma^*})^{M^*}$ between the Lagrangian submanifolds $\Lconic(\sF)$ and $\tilde L(\sF)$. By pulling back this isomorphism along the inclusion (i.e., by applying the quasi-inverse of the inclusion quasi-equivalence) $(i')^*:\Perf\cW(M^*_{\C^*}, \mathfrak f _{\Sigma^*})^{M^*}\to \Perf \cW(\Sprime)^{M^*}$, we obtain quasi-isomorphisms $HH(\Lmon(\sF))= (i')^*\Lconic(\sF) \simeq \mathcal Y_{\tilde L(\sF)}$. 
    Thus the essential image of $HH$ is contained in that of $\mathcal Y_{(-)}$, so the result follows by \Cref{lem:const_fun_easy}.

\end{proof}

\subsubsection{Arclike tropical Lagrangian sections\texorpdfstring{: $\tilde L_{mon}(\sF)$}{} and their negative pushoffs} \label{subsec:productLagconstr}

In this section, we give an explicit construction of Lagrangian sections $\Lmon(\sF)$ which are arclike in $\ess$.
We construct these sections following \cite[Proposition 2.13]{hanlon2019monodromy}. Let $\sF: M_\R \to \R$ be an integral piecewise linear function on $\Sigma^*$, that is, $\sF$ is the support function of a divisor on $Y^*$. Let $\eta: M_\R \to \R$ be a smooth symmetric mollifier function supported on the ball of radius $1$ around the origin with respect to an arbitrarily chosen Euclidean metric. Then, for small $\delta >0$, we set $\eta_\delta(u) := \delta^{-n}\eta(u/\delta)$ and
\[ H_{mon}^{\sF,\delta}(u) := \int_{M_\R} \sF(u-y) \eta_\delta(y) \, dy .\]
\begin{defn} \label{def:tropLagSection}
	The tropical Lagrangian section associated with $\sF$ is
	\[  \Lmon^\delta(\sF) := \text{graph}\left(dH_{mon}^{\sF,\delta}\right) = \left\{ \left(u, dH_{mon}^{\sF,\delta}(u)\right) : u \in M_\R \right\} \subset M_\R \times M^*_{S^1} = T^*M^*_{S^1} \simeq M^*_{\C^*}. \]
 Sometimes we will drop $\delta$ from the notation.
\end{defn}

In \cite{hanlon2019monodromy}, these Lagrangian sections are shown to be mirror objects to line bundles in a modified version of $\mafuk^{M^*}$. 
However, it is observed in \cite[Section 4.4]{hanlon2019monodromy} that in the Fano setting, $\Lmon(\sF)$ is also a tropical Lagrangian section in the sense of \cite{abouzaid2009morse}, i.e., is an object of $\mafuk^{M^*}$.

\begin{prop}
	\label{prop:tropLagInNeighOfSkeleton} 
Let $\eps > 0$. We may choose $\delta$ sufficiently small so that:
\begin{enumerate}
    \item \label{it:Lmonisreal} We have
 \begin{equation} \label{eq:Lmonisreal} \left( (1 - \psi_\alpha (t,x)) \cdot z^\alpha\right)|_{\Lmon^\delta(\sF)} \in \R_{\geq 0}
 \end{equation}
 for all $\alpha \in A$ and all $x$. In particular, $ \left( \Im \, W_{t,1} \right)|_{\Lmon^\delta(\sF)} = 0$.
 \item \label{it:Lmon_in_neigh} $\Lmon^\delta(\sF) $ is contained in the $\eps$-neighbourhood of ${\mathbb L}_{\Sigma^*}$.
 \item \label{it:TLmon} The tangent space to $\Lmon^\delta(\sF)$ is contained in the $\eps$-neighbourhood of the set
\begin{align}\label{eq:where_TL_lies} Q:=\bigcup_{\sigma \in \Sigma^*} \sigma \times \sigma^\perp \times M_\R \times \sigma^\perp \subset M_\R \times M^*_{S^1} \times M_\R \times M^*_\R = T(M_\R \times M^*_{S^1}).
\end{align}
\end{enumerate}
\end{prop}
\begin{proof}
    For each nonzero $\alpha\in A$, we have that
	\begin{equation} \label{eq:monomialsupport}
        dH_{mon}^{\sF,\delta}(\alpha) = \sF(\alpha) 
    \end{equation}
	in the star of $\alpha$ away from a $\delta$-neighbourhood of the boundary. Since $\sF(\alpha)\in \Z$, this implies that $ \arg(z^\alpha)|_{\Lmon^\delta(\sF)}=0 $ in this region. On the other hand, by \cite[Corollary 2.41]{hanlon2019monodromy}, we can choose $\varphi$ so that the image under $\Phi$ of the region on which $\psi_\alpha(t,x) \neq 1$ is contained in the interior of the star of $\alpha$.\footnote{Note that, from \cite[Proposition 2.34]{hanlon2019monodromy} and the discussion following Proposition 2.36 in {\em loc. cit.}, the condition we impose here on $\Phi$ implies that $\varphi$ is adapted to $\Delta^*$.}    In particular, \eqref{eq:monomialsupport} holds over all $x$ satisfying $\psi_\alpha(t,x) \neq 1$ if $\delta$ is sufficiently small. Thus, we see that \eqref{eq:Lmonisreal} holds for all nonzero $\alpha$. However, \eqref{eq:Lmonisreal} holds trivially for $\alpha = 0$, completing the proof of \eqref{it:Lmonisreal}.
 
            Suppose $\left(u,dH_{mon}^{\sF,\delta}(u)\right) \in \Lmon^\delta(\sF)$. In the complement of a neighbourhood $U_\delta$ of the origin in $M_\R^*$ whose size is governed by $\delta$, there is a minimal cone $\tau \in \Sigma^*$ such that $d(u, \tau) <\delta$. Let $\alpha_1, \hdots, \alpha_k \in A$ be the primitive generators of $\tau$. We then have that $dH_{mon}^{\sF,\delta}(u)(\alpha_j) = \sF(\alpha_j)$ for $j = 1,\hdots,k$. In particular, if $v \in \tau$ is a point such that $d(u, v) < \delta$, we have
    $$ \left(v, dH_{mon}^{\sF,\delta}(u)\right) \in \tau \times \tau^\perp \subset {\mathbb L}_{\Sigma^*} $$
    and the distance between $\left(v,dH_{mon}^{\sF,\delta}(u)\right)$ and $\left(u, dH_{mon}^{\sF,\delta}(u) \right)$ is less than $\delta$. 
    This completes the proof of \eqref{it:Lmon_in_neigh}. 
    
    Finally, observe that in the interior of the set of all $u$ with minimal cone $\tau$, we have $dH_{mon}^{\sF,\delta}(u) \in \tau^\perp$, which implies that the tangent space to the graph of $dH_{mon}^{\sF,\delta}$ is contained in $M_\R \times \tau^\perp$. 
    Therefore, the point $(v,dH_{mon}^{\sF,\delta}(u),T_{(u,dH_{mon}^{\sF,\delta}(u))} \Lmon^\delta(\sF))$ is contained in $Q$, and is distance $\delta$ from $(u,dH_{mon}^{\sF,\delta}(u),T_{(u,dH_{mon}^{\sF,\delta}(u))}\Lmon^\delta(\sF))$ as required. The union of the interiors of these sets, over all cones $\tau$, cover the complement of $U$. Thus, we complete the argument by choosing $\delta < \eps$ sufficiently small that $U_\delta \subset B_\eps(0)$. 
    This completes the proof of \eqref{it:TLmon}.
\end{proof}
 We obtain as a corollary: 

\begin{cor}\label{cor:LA_arclike}
    	For $\delta$ sufficiently small and the $\aye$ constructed in \cref{lem:compatibleIandW}, the Lagrangians $\Lmon^\delta(\sF)$ are $\aye$-arclike in $\ess$.
\end{cor}
\begin{proof}
    We first check that $\Lmon^\delta(\sF)$ is transverse to the boundary of $\ess$. Recall from \Cref{lem:constructDF} and the discussion following that $\ess$ is constructed from the data of a sutured Liouville domain $(B, F)$, where $B$ is given by modifying $\Log_t^{-1}(\mathcal C_0)$. 

    The Liouville vector field $(u\partial_u,0)$ on $M_\R \times M^*_{S^1}$ is transverse to the boundary of $B$ by construction. 
    Therefore, we have $d(u\partial_u,T\partial B) > \eps$ for some $\eps>0$, where distances are measured with respect to an arbitrary Riemannian metric. 
    Now let $p \in \Lmon^\delta(\sF) \cap \partial B$. 
    As $\Lmon^\delta(\sF)$ is a Lagrangian section, the tangent vector $u\partial_u$ lifts to a vector $(u\partial_u,v) \in T_p\Lmon^\delta(\sF)$. 
    By \Cref{prop:tropLagInNeighOfSkeleton}, there exists a tangent vector $(\tilde p,(\tilde u, \tilde v)) \in Q$ which is close to $(p,(u\partial_u,v))$. 
    We may arrange that this tangent vector lies over $\partial B$, by flowing along the Liouville flow. 
    By choosing $\delta$ sufficiently small, we may guarantee that $(p,(u\partial_u,v))$ is arbitrarily close to $(\tilde p,(\tilde u,\tilde v))$; and as a consequence that $T_p\partial B$ is arbitrarily close to $T_{\tilde p}\partial B$. 
    
    Now assume that $(\tilde p,(\tilde u,\tilde v)) \in (\sigma \times \sigma^\perp) \times (M_\R \times \sigma^\perp)$. 
    Recall that by construction of $B$, $\mathbb{L}_{\Sigma^*} \cap \partial \Log_t^{-1}(\cC_0) \subset \partial B$. 
    In particular, $\Log_t(\tilde p) \times \sigma^\perp \subset \partial B$, and hence $\{0\} \times \sigma^\perp \subset T_{\tilde p} \partial B$. 
    We now have:
    \begin{align*}
        d\left((p,(u\partial_u,v)),T_p\partial B\right) & > d\left((\tilde p,(\tilde u,\tilde v)),T_{\tilde p}\partial B\right) - \eps/2 \quad\text{for $\delta$ sufficiently small}\\
        &= d\left((\tilde p,(\tilde u,0)),T_{\tilde p}\partial B\right) - \eps/2 \quad \text{as $\tilde v \in \sigma^\perp$ so $(0,\tilde{v})\in T_{\tilde p} \partial B$} \\
        & \ge d\left((\tilde p,(u\partial_u,0)),T_{\tilde p} \partial B\right) - d\left((\tilde p,(u\partial_u,0)),(\tilde p,(\tilde u,0))\right) - \eps/2 \\
        & > \eps - \eps/2 - \eps/2 \quad \text{for $\delta$ sufficiently small} \\
        & = 0.
    \end{align*}
    Therefore, $(u\partial_u,v) \in T_p\Lmon^\delta(\sF)$ is transverse to the hypersurface $T\partial B$, so $\Lmon^\delta(\sF)$ is transverse to $\partial B$ at $p$, and we may choose $\delta$ sufficiently small that this holds for all $p \in \mathbb{L}_{\Sigma^*} \cap \partial B$, by compactness. 
    We now observe that by construction of $\ess$, $\partial \ess$ agrees with $\partial B$ along a neighbourhood of $\mathbb{L}_{\Sigma^*}$. 
    By choosing $\delta$ sufficiently small, we may arrange that $\partial \Lmon^\delta(\sF)$ is contained in this neighbourhood by \Cref{prop:tropLagInNeighOfSkeleton}; thus $\Lmon^\delta(\sF)$ intersects $\partial \ess$ transversely as required.

     By \cref{lem:compatibleIandW}, the zero sets of $\aye$ and $\Im(W_{t,1})$ agree when we are close to both $\partial \ess$ and $\mathbb L_{\Sigma^*}$. By \Cref{prop:tropLagInNeighOfSkeleton}: \cref{it:Lmonisreal,it:Lmon_in_neigh}, $\Lmon^\delta(\sF)$ is contained in the intersection of the zero set of $\Im(W_{t,1})$ and a small neighbourhood of $\mathbb L_{\Sigma^*}$. Therefore $\Lmon^\delta(\sF)\cap \Nbd(\partial \ess) \subset  \aye^{-1}(0)\cap \Nbd(\partial \ess)$. 
     Thus, $\Lmon^\delta(\sF)$ is $\aye$-arclike in $\ess$.
      \end{proof}

By \cref{adjustedsectors}(\ref{arcliketransport}), we may extend (by the flow of $V$ applied to $\partial \Lmon^\delta(\sF)$) $\Lmon^\delta(\sF)$ to an $\Iprime$-arclike Lagrangian 
\begin{equation}\label{lmonprime}
(\Lmon^\delta(\sF))' \subset \Sprime.
\end{equation}
whose intersection with $\ess$ is the $\aye$-arclike $\Lmon^\delta(\sF)$ (to see $\Iprime$-arclikeness of the extension, note that tangency of $(\Lmon^\delta(\sF))'$ to $V = X_{\Iprime}$ near $\partial \Sprime$ implies both transversality to $\partial \Sprime$ and containment in an $\Iprime$-level set by \cref{lem:XItangK}).

\begin{rmk}
    Here is an equivalent way to define the extension $(\Lmon^\delta(\sF))'$ of $(\Lmon^\delta(\sF))$. 
    Observe that $(\Lmon^\delta(\sF))$ itself is tangent to $X_{\aye}$ near and in particular just outside $\partial \ess$ by \cref{lem:compatibleIandW} and \cref{lem:XItangK}, and hence determines an $\aye$-arclike Lagrangian in the enlarged sector $S_1 \supset \ess$ constructed in the proof of \cref{adjustedsectors} with deformed Liouville form $\lambda'$ (in fact $\Lmon^\delta(\sF)$ is tangent to $X_{\aye}$ over all of $S_1 \backslash \ess^{\circ}$). Now apply $\Phi$, the symplectomorphism from $\lambda'$ to $\lambda$ constructed in the same proof, to obtain $(\Lmon^\delta(\sF))'$.

    If in the construction of the deformation $\lambda'$, we modified $\lambda'$ by a compactly supported exact 1-form to equal  $\lambda$ on a compact set containing $(\Lmon^\delta(\sF)) \cap (S_1 \backslash \ess^{\circ})$, then we could further arrange that $\Phi = id$ near $\Lmon^\delta(\sF)$ and hence that $(\Lmon^\delta(\sF))'=\Lmon^\delta(\sF)$ at the expense of slightly modifying $\Sprime:= \Phi(S_1)$. 
       \end{rmk}

Since $\Sprime$ has exact boundary, we can now take the negative pushoff of $(\Lmon^\delta(\sF))'$ in the sense of \cref{sec:adjoints}, using the boundary sector $\eff \times T^*[0,1] \hookrightarrow \Sprime$ induced by $\Iprime$. 
The outcome is a cylindrical exact Lagrangian whose boundary avoids $\Sprime$
\begin{equation}\label{lmonnegativepushoff}
    \tilde{L}(\sF) := ((\Lmon^\delta(\sF))')^- \subset \Sprime;
\end{equation}
this Lagrangian, like $(\Lmon^\delta(\sF))'$ and $\Lmon^\delta(\sF)$, has an anchoring determined by $\sF$ and can be thought of as an object of $\cW(\Sprime)^{M^*}$.
The modifications involved in taking cylindrical pushoff are only made in $\eff \times T^*[0,1] \subset (\Sprime \backslash \ess)$ so in particular 
\begin{equation}\label{restrictionofLmonminus}
    \tilde{L}(\sF) \cap \ess = \Lmon^\delta(\sF).
\end{equation}
We will drop $\delta$ from the notation when appropriate. 

\subsubsection{Cylindrized tropical Lagrangian section\texorpdfstring{: $\Lconic(\sF)$}{}}\label{subsec:HanlonHicksLag}
By instead using a ``cylindrical smoothing'' of $\sF$, we can also construct, as in \cite[Section 3.1]{hanlon2022aspects}, a cylindrical Lagrangian section of $M^*_{\C^*}$ whose Legendrian boundary lies in the complement of $\mathfrak{f}_{\Sigma^*}$, that is, an object of $\mathcal{W}(M^*_{\C^*}, \mathfrak{f}_{\Sigma^*})$. 
The function $\sF$ determines an anchoring of this object, i.e., a lift to $\mathcal{W}(M^*_{\C^*},\mathfrak{f}_{\Sigma^*})^{M^*}$.

We define ${G}_{cyl}^{\sF,\delta}:M_\R \setminus \{0\} \to \mathbb{R}$ by
\begin{equation}
	 {G}_{cyl}^{\sF,\delta}(u) = \int_{M_\R} \sF(u-y) \eta_{\delta |u|} (y) \, dy =  \int_{M_\R} \sF(u - |u|y)  \eta_\delta(y) \, dy 
	 \label{eq:conicalizedTropicalPrimitive}
\end{equation}
where $|u|$ is measured with respect to the Euclidean metric.
It is then shown in \cite[Lemma 3.1]{hanlon2022aspects} that for any $\sF$ there is a constant $C_\sF$  such that
\[ \left| d{G}_{cyl}^{\sF,\delta}(\alpha) -\sF (\alpha) \right| \leq \delta C_\sF \] on the open cones
\begin{equation}
	U^\delta_\alpha = \text{cone} \{u \in M_\R \; : \; | u | = 1,  u \in \st(\alpha), \text{ and } d(u, \sigma) > \delta \text{ for all } \sigma \in \Sigma \text{ such that } \alpha \not \in \sigma \}
	\label{eq:conesUAlpha}
\end{equation}
for all nonzero $\alpha \in A$. 
In the language of \cite[Definition 2.5]{hanlon2022aspects}, $U^\delta_\alpha$ is the $\delta$-star of $\alpha$. 

Given a support function $\sF$ for $\Sigma^*$, we can set
\begin{equation} \label{eq:perturbedPL} \sF_a = \sF - a \sF_K \end{equation}
where $\sF_K$ is the support function of the canonical divisor and $0< a \ll 1$. 
We then define a smooth function $H_{cyl}^{\sF,a,\delta}$ by smoothing ${G}_{cyl}^{\sF_a,\delta}$ arbitrarily on a compact neighbourhood of the origin, which may be taken to be arbitrarily small. We denote such a neighbourhood by $\Nbd(0)$ in what follows.  The resulting $H_{cyl}^{\sF,a,\delta}$ will satisfy
\begin{equation}
     \left| dH_{cyl}^{\sF,a,\delta}(\alpha) - \sF(\alpha) + a \right| \leq \delta C_{\sF_a}
     \label{eq:argumentBound}
\end{equation} 
on $U^\delta_\alpha \setminus \Nbd(0)$. Since $\sF(\alpha)$ is an integer and $a$ is not, we may choose $\delta$ small enough that $dH_{cyl}^{\sF,a,\delta}(\alpha) \notin \Z$ on $U^\delta_\alpha \setminus \Nbd(0)$. 
On the other hand, if $u \in U^\delta_\alpha$ and $(u,\theta) \in \mathbb{L}_{\Sigma^*}$, then $\theta(\alpha) \in \Z$. 
It follows that the Lagrangian 
$$\Lconic^{a,\delta}(\sF) := \text{graph}\left(dH_{cyl}^{\sF,a,\delta}\right)$$ 
avoids ${\mathbb L}_{\Sigma^*}$ outside of the compact set $\pi^{-1}(\Nbd(0))$, for $\delta$ sufficiently small. 
Thus, $\Lconic^{a,\delta}(\sF)$ determines an object of $\mathcal{W}(M^*_{\C^*}, \mathfrak{f}_{\Sigma^*})$. Observe that all choices of $a,\delta$ with $0<a\ll 1$ and $\delta$ sufficiently small yield admissibly Hamiltonian isotopic Lagrangians in $M^*_{\C^*}$; we will therefore drop $a$ and $\delta$ from the notation when appropriate.

\subsubsection{Comparing constructions of Lagrangian sections}
\label{subsec:lcon_lmon}

In the previous two subsections we have constructed two a priori different cylindrical Lagrangians in $\W(M^*_{\C^*}, \mathfrak{f}_{\Sigma^*})$ associated to $\sF$:  the cylindrized tropical section $\Lconic^{a,\delta}(\sF)$ 
and the negative pushoff of the tropical arclike Lagrangian $\tilde{L}(\sF)$
(strictly speaking, each of these depend on small parameters, but only up to
cylindrical isotopy hence isomorphism in the Fukaya category).
\cref{cor:HHvsArclike} below asserts that they are isomorphic. The obvious
isotopy between these two objects is non-cylindrical (passing through the
non-pushed off $\Lmon^\delta(\sF)$), so doesn't immediately induce an isomorphism. 
Instead, we show two objects can be made (after a compactly supported isotopy)
to agree with the arclike Lagrangian $\Lmon^\delta(\sF)$ on a subsector $\ess$, and
also that both objects cofinally wrap away from the subsector.  From here,
locality theory for arclike Lagrangians in Liouville sectors, 
developed in \cref{sec:adjoints}, shows that when tested against objects coming
from the Fukaya category of the (quasi-equivalent) subsector $\ess$, all disks
appearing for either Lagrangian also remain in the subsector, and therefore
coincide. 

Conceptually, we may think of $\Lconic^{a,\delta}(\sF)$ and
$\tilde{L}(\sF)$ as ``generalized negative pushoffs'' of the
same arclike $\Lmon^\delta(\sF)$, in the sense that they agree with $\Lmon^\delta(\sF)$ 
within an equivalent subsector which they cofinally wrap away from. 
The technology of \cref{sec:adjoints} can be used to show that any two
generalized negative pushoffs of an arclike Lagrangian are isomorphic objects;
compare \cref{negativepushoffindependence}.

\begin{lem}[\cite{hanlon2022aspects}, Lemma 3.3]
    \label{cylwrappings}
    $\partial_{\infty}\Lconic^{a,\delta}(\sF)$ admits a cofinal wrapping in $(M^*_{\C^*},\mathfrak{f}_{\Sigma^*})$ converging to $\mathfrak{f}_{\Sigma^*}$ which does not cross $\partial_{\infty} \ess$.
\end{lem}
\begin{proof}
The cofinal wrapping is exhibited in {\em loc. cit.}. As also shown there, the boundary at infinity of $\Lconic^{a,\delta}(\sF)$ and its wrappings can be made to take place in an arbitrarily small neighborhood of $\mathfrak{f}_{\Sigma^*}$, which is, in particular, disjoint from $\partial_{\infty} \ess$.
\end{proof}

\begin{prop}\label{interpolatecyltomon}
    There is a Lagrangian $\Lmc^{a,\delta}(\sF)$ in $M^*_{\C^*}$ which differs from $\Lconic^{a,\delta}(\sF)$ by a compactly supported Hamiltonian isotopy and satisfies 
\[ \Lmc^{a,\delta}(\sF) \cap \ess = \Lmon^\delta(\sF) \]
for sufficiently small $\delta$ and $a$.
\end{prop}
\begin{proof}
    Pick $\eps>0$ and $R > 0$ so that  $\pi(\ess\cap B_\eps(\mathbb L_{\Sigma^*})) \subset B_R(0) \subset M_\R$.
Let $\rho_1: M_\R\to \R_{\ge 0}$ be smooth, equal to $|u|$ outside $B_{R + b}(0)$ for some $b > 0$, and equal to $1$ on $B_R(0)$, and let $\rho_2: M_\R\to [0,1]$ be smooth, equal to $1$ outside $B_{R + b}(0)$, and equal to $0$ on $B_R(0)$.  

We now perform a similar construction to \eqref{eq:conicalizedTropicalPrimitive} by defining
\begin{equation}
  {G}_{mc}^{\sF,\delta} (u) = \int_{M_\R} \sF(u-y) \eta_{\delta \rho_1 (u)} (y) \, dy 
\end{equation} 
for $u \in M_\R$. 
We note that ${G}_{mc}^{\sF,\delta}$ agrees with ${G}_{cyl}^{\sF,\delta}$ on the complement of $B_{R+b}(0)$, and with $H_{mon}^{\sF,\delta}$ on $B_R(0)$. We now set
\[ H_{mc}^{\sF,a,\delta}(u) = G^{\sF,\delta}_{mc}(u) - \rho_2(u)aG^{\sF_K,\delta}_{mc}(u) \]
for all $u \in M_\R$ and define $\Lmc^{a,\delta}(\sF)$ to be the graph of $dH_{mc}^{\sF,a,\delta}$. As before, we note that $H^{\sF, a, \delta}_{mc}$ agrees with $H^{\sF, a, \delta}_{cyl}$ in the complement of $B_{R+b}(0)$ and with $H^{\sF,\delta}_{mon}$ on $B_{R}(0)$. 

To show that $\Lmc^{a, \delta}(\sF)$ agrees with $\Lmon^\delta(\sF)$ on $\ess$, it suffices to show that it lies in a small neighbourhood of $\mathbb L_{\Sigma^*}$ as we have chosen $\eps$ and $R$ so that $\ess \cap B_\eps(\mathbb{L}_{\Sigma^*} )\subset \pi^{-1}(B_R(0))$. 
By applying the same argument as in \cref{prop:tropLagInNeighOfSkeleton}, we can deduce that there exists a choice of $\delta$ and $a$ so that $\Lmc^{a,\delta}(\sF)$ is contained in $B_\eps(\mathbb L_{\Sigma^*})$ by obtaining a uniform bound on $|dH_{mc}^{\sF, a, \delta}(\alpha) - \sF(\alpha)|$ on $U^\delta_\alpha \setminus \Nbd(0)$ for all $\alpha \in A$ that goes to zero as $\delta$ and $a$ go to zero.

We will now produce such a bound. Let $\sF_\sigma$ be the linear function given by restricting $\sF$ to any $\sigma \in \Sigma^*(\dim Y^*)$. Then, for all $u \in U^\delta_\alpha \setminus \Nbd(0)$, we have
\begin{align*} 
    \left| (dG_{mc}^{\sF, \delta})_u(\alpha) - \sF(\alpha) \right| &= \left| d\rho_1(\alpha) \right| \left| \int_{M_\R} d\sF_{u - \rho_1(u)y}(y) \eta_\delta(y) \, dy \right| \\
    &\leq \left| d\rho_1(\alpha) \right| \int_{M_\R} \left| d\sF_{u - \rho_1(u)y} \right| \left| y \right| \eta_\delta(y) \, dy \\
        &\leq \delta | d\rho_1(\alpha)| \max_{\sigma \in \Sigma^*(\dim Y^*)} \left| \sF_\sigma \right| 
\end{align*}
which uniformly tends to $0$ as $\delta$ goes to zero as $|d\rho_1|$ is bounded. The same calculation shows that for all $u \in U^\delta_\alpha$ with $|u| \geq R + b$ (so that $\rho_2 \equiv 1$) we have
\[ \left| d\left(\rho_2aG^{\sF_K,\delta}_{mc}\right)(\alpha)- a \right| \leq \delta C \]
for some constant $C$. Finally, inside of $B_{R+b}(0)$, the norm of  $d\left(\rho_2 G^{\sF_K,\delta}_{mc} \right)$ is bounded so 
\[ \left|  d\left(\rho_2 aG^{\sF_K,\delta}_{mc} \right) \right| \leq a C' \]
for some constant $C'$. Putting all these estimates together, we obtain the desired bound on $|dH_{mc}^{\sF, a, \delta}(\alpha) - \sF(\alpha)|$. 

Therefore, we have shown that $\Lmc^{a,\delta}(\sF) \subset B_\eps(\mathbb{L}_{\Sigma^*})$ for any small $\eps > 0$. 
In particular, it follows that
\[\Lmc^{a,\delta}(\sF) \cap \ess = 
\Lmc^{a, \delta}(\sF) \cap \ess\cap \pi^{-1}(B_R(0))
= \Lmon^\delta(\sF) \]
as required. We also have, by construction, that the flow of $H_{mc}^{\sF,a,\delta} - H_{cyl}^{\sF, a, \delta}$ is a Hamiltonian isotopy, supported in the compact set $\pi^{-1}(B_{R+b}(0))$, from $\Lconic^{a,\delta}(\sF)$ to $\Lmc^{a,\delta}(\sF)$.
\end{proof}

\begin{prop}
    \label{cor:HHvsArclike}
    There is an isomorphism of objects of $\W(M^*_{\C^*}, \mathfrak{f}_{\Sigma^*})^{M^*}$ 
    between  $\Lconic^{a,\delta}(\sF)$ and $\tilde{L}(\sF)$.
\end{prop}
\begin{proof}
    We will use the following notation for the tautological sector embeddings: $i: \ess \to \Sprime$, $i': \Sprime \to (M^*_{\C^*}, \mathfrak{f}_{\Sigma^*})$ and $i' \circ i: \ess \to (M^*_{\C^*}, \mathfrak{f}_{\Sigma^*})$. The induced pushforward functors on wrapped Fukaya categories are all equivalences indicated by the left and right vertical arrows in \Cref{fig:HHvsarclike}.

\begin{figure}
    \[
\begin{tikzpicture}

\fill[rounded corners,fill=gray!20]   (-9.5,-4) rectangle (-5.5,-6.5);

\fill[rounded corners,fill=gray!20]  (2.5,-4) rectangle (5.5,-6.5);

\fill[rounded corners,fill=gray!20]  (-9.5,-0.5) rectangle (5.5,-2.5);

\node (v4) at (-9,-5) {$\Lconic$};

\node (v3) at (3.5,-5) {$\tilde{L}$ };

\node (v6) at (-6.5,-5) {$\Lmc^{a,\delta} $};

\node (v2) at (-6.5,-1.5) {$(i' \circ i)^* \hom_{\cW(M^*_{\C^*}, \mathfrak{f}_{\Sigma^*})}( \Lmc^{a,\delta}, -)$};

\node(v1)  at (-1,-5) {$\Lmc^{a,\delta} \cap \ess= \Lmon^\delta= \tilde{L}\cap \ess$};

\node (v7) at (-1,-1.5) {$CF^*(\Lmon^\delta,-)$};
\node (v5) at (3.5,-1.5) {$i^* \hom_{\cW(\Sprime)}( \tilde{L}, -)$};
\draw  (v6) edge[black, ->, dotted] node[midway, above] {$\cap \ess$} (v1);
\draw  (v3) edge[black, ->,dotted] node[midway, above] {$\cap \ess$} (v1);
\draw  (v1) edge[black, ->,dotted] node[midway, fill=white] {\ref{arclikemodule}} (v7);
\draw  (v4) edge node[midway, fill=gray!20] {\ref{interpolatecyltomon}} (v6);
\draw  (v2) edge[black, ->] node[midway,below ]{\ref{arclikesubsector}} (v7);
\draw  (v5) edge[black, ->] node[midway, below]{\ref{arclikesubsector}} (v7);
\node at (-7.5,-6) {$\mathcal W(M^*_{\C}, \mathfrak f_{\Sigma^*})^{M^*}$};
\node at (4,-6) {$\mathcal W(\Sprime)$};
\node at (-1,-1) {$\operatorname{Mod} \mathcal W(\ess)$};
\draw  (v6) edge[->] node[midway, fill=white]{$(i'\circ i)^* \circ \mathcal Y_{(-)}$} (v2);
\draw  (v3) edge[->] node[midway, fill=white]{$i^* \circ \mathcal Y_{(-)}$}  (v5);
\end{tikzpicture}\]
\caption{Schematic diagram of the proof of \cref{cor:HHvsArclike}.}  \label{fig:HHvsarclike}
\end{figure}
    
We prove the isomorphism by exhibiting isomorphisms indicated in \Cref{fig:HHvsarclike} by horizontal arrows. 
    First, $\Lconic^{a,\delta}(\sF)$ is isomorphic to $\Lmc^{a,\delta}(\sF)$ (which satisfies all of the hypotheses of \cref{interpolatecyltomon}) because they differ by a compactly-supported Hamiltonian isotopy. Now, since $\partial_{\infty} \Lmc^{a,\delta}(\sF) = \partial_{\infty}\Lconic^{a,\delta}(\sF)$ possesses a cofinal wrapping which does not cross $\partial \ess$ by \cref{cylwrappings}, and since $\Lmc^{a,\delta}(\sF) \cap \ess$ is the $\aye$-arclike Lagrangian $\Lmon^\delta(\sF)$, \cref{arclikesubsector} applied to $i' \circ i: \ess \to (M^*_{\C^*}, \mathfrak{f}_{\Sigma^*})$ implies that there is an isomorphism of $\cW(\ess)$ modules 
\begin{equation}\label{arclikecutiso1}
    (i' \circ i)^* \hom_{\cW(M^*_{\C^*}, \mathfrak{f}_{\Sigma^*})}( \Lmc^{a,\delta}(\sF), -) \cong CF^*(\Lmon^\delta(\sF),-),
\end{equation}
where $CF^*(\Lmon^\delta(\sF),-)$ is the ``arclike'' $\cW(\ess)$-module constructed by applying \cref{arclikemodule} to $\Lmon^\delta(\sF)$.

Similarly, the negative pushoff $\tilde{L}(\sF)$ also satisfies $\tilde{L}(\sF)\cap \ess = \Lmon^\delta(\sF)$ (see \eqref{restrictionofLmonminus}), and $\partial_{\infty} \tilde{L}(\sF)$ admits a cofinal wrapping contained entirely in $\partial_{\infty} (F \times T^*[0,1]) \subset \partial_{\infty} \Sprime$ which in particular avoids $\partial_{\infty} \ess$ (see \cref{negativepushoffforwardstopped}). Hence, \cref{arclikesubsector}, this time applied to $i: \ess \to \Sprime$ again implies that
\begin{equation}\label{arclikecutiso2}
    i^* \hom_{\cW(\Sprime)}( \tilde{L}(\sF), -) \cong CF^*(\Lmon^\delta(\sF),-).
\end{equation}
Since we can view $\tilde{L}(\sF)$ as living in $\cW(M^*_{\C^*}, \mathfrak{f}_{\Sigma^*})$ via the pushforward $i'$, and $i'$ is a quasi-equivalence, we also know  $\hom_{\cW(\Sprime)}( \tilde{L}(\sF), -) \cong (i')^* \hom_{\cW(M^*_{\C^*}, \mathfrak{f}_{\Sigma^*})}(\tilde{L}(\sF), -)$, from which we collectively deduce that
\begin{equation}\label{arclikecutiso3}
(i' \circ i)^* \hom_{\cW(M^*_{\C^*}, \mathfrak{f}_{\Sigma^*})}( \tilde{L}(\sF), -) \cong CF^*(\Lmon^\delta(\sF),-);
\end{equation}
in other words, \eqref{arclikecutiso1} and \eqref{arclikecutiso3} imply that 
$$(i' \circ i)^* \hom_{\cW(M^*_{\C^*}, \mathfrak{f}_{\Sigma^*})}( \tilde{L}(\sF), -) \cong (i' \circ i)^* \hom_{\cW(M^*_{\C^*}, \mathfrak{f}_{\Sigma^*})}( \Lmc^{a,\delta}(\sF), -).$$ 
Now, $i' \circ i$ is a quasi-equivalence, so the same isomorphism holds before applying $(i' \circ i)^*$. Therefore Yoneda implies $\tilde{L}(\sF) \cong \Lmc^{a,\delta}(\sF)$.
\end{proof}

\section{Powers of an ample line bundle}\label{sec:A0_B0}
From here on, we drop decorations and use $\cap : \mafuk^{M^*}\to \partial \mafuk^{M^*}$ to denote the cap functor which makes the  diagram from  \cref{prop:GS_Li} commute:
\begin{equation}\label{eq:Pic_sec_cd}\begin{tikzcd}
			\Pic_{dg}^{M^*}(\partial Y^*)  \arrow[leftrightarrow]{r} & \partial \mafuk^{M^*}  \\
			\Pic_{dg}^{M^*}(Y^*) \arrow{u}{ i^! } \arrow[leftrightarrow]{r}{A} & \mafuk^{M^*} \arrow{u}{\cap}.\end{tikzcd}\end{equation}

Recall that all categories have objects indexed by the integral piecewise-linear functions $\sF$. 

Let $\sF_\Delta$ denote the function corresponding to the toric boundary divisor; that is, $\sF_\Delta$ is equal to $-1$ on the generator of each ray of $\Sigma^*$. 
We now restrict the diagram \eqref{eq:Pic_sec_cd} to the respective subcategories with objects indexed by $k\sF_\Delta$. 
We denote the corresponding subcategories by
\begin{equation}
    \label{eq:Pic_sec_cd2}\begin{tikzcd}
			\cB_0  \arrow[leftrightarrow]{r}{\mir0} & \cA_0  \\
			\tilde\cB_0 \arrow{u}{ i^! } \arrow[leftrightarrow]{r} & \tilde \cA_0 \arrow{u}{\cap}.\end{tikzcd}
\end{equation}
 The objects of $\cB_0$ are $\cO(k) = i^*\cO_{Y^*}(k \sF_\Delta)$ and those of $\tilde \cB_0$ are $\mathcal{E}_k = \cO_{Y^*}( (k+1) \sF_\Delta) [1]$ so that $i^! \mathcal{E}_k = \cO(k)$ (see, for example, \cite[Section 3.4]{huybrechts2006fourier}). 
The objects of $\tilde \cA_0$ are the mirror tropical Lagrangian sections $\tilde{L}_k$ to $\mathcal{E}_k$, and those of $\cA_0$ are $L_k = \cap \tilde{L}_k$ (their boundaries). 
In an $A_\infty$ category $\cC$, we use the convention that the cochain-level morphism spaces are denoted $hom^*_\cC(-,-)$, and their cohomologies by $Hom^*_\cC(-,-)$.

The main result of this Section is Lemma \ref{lem:F0respectsM}, which shows that the quasi-isomorphism $\mir0$ above can be made to respect the natural $\mathbb{Z}\oplus M^*$-gradings on both sides. Along the way, we also prove (Corollary \ref{cor:SH_HH}) that the closed--open map
$$\cC \cO: SH^*(\eff;\Bbbk) \to \HH^*(\cA_0)$$
from the symplectic cohomology of $\eff$ to the Hochschild cohomology of $\cA_0$ is a $\Z \oplus M^*$-graded Lie algebra isomorphism. 

\subsection{Hochschild invariants and Calabi--Yau structure}

The symplectic cohomology of $F$ admits a grading by $\Z \oplus M^*$, where the $\Z$ component comes from the Conley--Zehnder index and the $M^* = H_1(M^*_{\C^*})$ component comes from the homology class of the orbit included into $M^*_{\C^*}$. Note that because $\cW(\eff)^{M^*}$ is $\Z\oplus M^*$ graded, we can define a $\Z\oplus M^*$-graded version of  $\HH^*(\cW(\eff)^{M^*})$ (cf. Section \ref{subsec:Ainf_term}).

\begin{prop}\label{prop:COsurj}
The closed--open map
$$\cC \cO: SH^*(\eff;\Bbbk) \to \HH^*(\cW(\eff)^{M^*})$$
is a $\Z \oplus M^*$-graded isomorphism.
\end{prop}
\begin{proof}	
	Given simply connected Lagrangians $L_0, \ldots, L_k$ and intersections  $x_i\in L_{i-1}\cap L_{i}$  with $i\in \{1, \ldots, k\}$ and  $x_0\in L_0\cap L_k$, we compute the $M^*$-grading of the Hochschild cochain $(x_1^\vee\otimes \cdots \otimes x_k^\vee)\otimes x_0$. First, pick paths $\gamma_i: I\to L_i$ from   $x_i$ to $x_{i+1}$. Then the grading of the Hochschild cochain is the homology class of the loop $\gamma_0\cdots \gamma_k$. 

	The structure coefficient of the closed-open map evaluated at  $y$ an orbit and $(x_1^\vee\otimes \cdots \otimes x_k^\vee)\otimes x_0$ a Hochschild cochain is given by a count of decorated $J$-holomorphic disks. The decorations impose the following constraints: an interior marked point labels a puncture limiting towards the orbit $y$;  cyclically labelled boundary marked points are required to map to the $x_i$. Additionally, the $(x_i, x_{i+1})$ component of the boundary maps to   $L_i$. It is immediate that  $[y]=[\partial D^2]=[\gamma_0\cdots \gamma_k]$.

	To check that $\cC \cO: SH^*(\eff;\Bbbk) \to \HH^*(\cW(\eff))$ is an isomorphism, it suffices by \cite[Theorem 1.1]{Ganatrathesis} to show that $\cO\cC:HH_{*-n}(\cW(\eff))\to SH^*(\eff)$ hits the identity; this holds because the map is an isomorphism \cite{chantraine2017geometric, ganatra2018sectorial, Ganatrathesis, Gao}. 
    As the wrapped Fukaya category is generated by cocores of critical handles \cite{chantraine2017geometric, ganatra2018sectorial} which are contractible and therefore admit anchorings, one can carry out the same argument for the $\Z \oplus M^*$-graded versions.
 \end{proof}

Recall that (after shifting $\mathbb{Z}$-gradings up by one) symplectic cohomology has the structure of a graded Lie algebra (see e.g. \cite[\S 2.1]{lekilipascaleff16}). This Lie algebra structure is compatible with $M^*$-gradings. Hochschild cohomology also has a graded Lie algebra structure (again after shifting  $\mathbb{Z}$-gradings up by one; see \cref{eq:Hochdif}) which is also compatible with $M^*$ gradings. 

 \begin{lem}\label{COLie}
In the situation of \cref{prop:COsurj}, $\cC \cO$ intertwines Lie algebra structures.
 \end{lem}
 \begin{proof}
This is a now-standard argument involving using closed-open moduli spaces to construct a homotopy between $[\cC \cO (-),\cC\cO(-)]$ and $\cC\cO ([-,-])$, analogous to the argument that $\cC\cO$ is a ring homomorphism \cite[Proposition 5.3]{Ganatrathesis}. The relevant geometric moduli spaces appear as a special case ($m=2$, $d$ arbitrary) of the moduli spaces used in \cite[Proposition 6.2.1]{pomerleanoseidel} (the relevant statements of results in {\em loc. cit.} are about the existence of bulk-deformed closed-open maps, but the geometric moduli spaces in particular show $\cC \cO$ is an $L_{\infty}$-algebra homomorphism, which is more than we need).
 \end{proof}

 \begin{rmk} \label{rmk:quadraticham1} The definition of wrapped Fukaya categories in \cite[\S 3]{ganatra2020covariantly} uses a `(categorical) localization' model for wrapped Floer theory. On the other hand, the construction of closed-open (and open-closed) maps in \cite{Ganatrathesis} uses quadratic Hamiltonians for the construction of wrapped Floer groups (and symplectic cohomology). These two models for the wrapped Fukaya category are shown to be canonically quasi-equivalent in \cite[Proposition 2.6]{sylvan2019orlov}. \end{rmk}

\begin{prop}\label{prop:HHA_0}
The restriction map 
\begin{align} \label{eq:gradedHHrestriction} \HH^*(\cW(\eff)^{M^*}) \to \HH^*(\cA_0) \end{align}
is a $\Z \oplus M^*$-graded isomorphism. 
\end{prop}
\begin{proof}
Let $\HH^*(\cA_0^{\sharp})$ denote the Hochschild cohomology of $\cA_0$ viewed only as a $\mathbb{Z}$-graded category (in other words, forgetting $M^*$-gradings). We first argue that the restriction map 
\begin{align} \label{eq:HHrestriction2} \HH^*(\cW(\eff)) \to \HH^*(\cA_0^{\sharp}) \end{align} is an isomorphism. To see this, by  \Cref{prop:GS_Li} \eqref{it:top_arr_GS}, it suffices to check that the restriction map from the Hochschild cohomology of $D^bCoh(\partial Y^*)$ to that of the full subcategory with objects $\cO(i)$ is an isomorphism. This follows from \cite[Corollary B.5.1(i)]{preygel2011thom}.  The same corollary also implies that $\HH^*(\cA_0^{\sharp})$ is finite-dimensional in each $\mathbb{Z}$-grading (because it's isomorphic to endomorphisms of the diagonal $\Delta \subset \partial Y^* \times \partial Y^*$). Because each of the hom-spaces in $\cA_0$ is finite dimensional over $\Bbbk$, we have by \eqref{eq:gradednongradedHH1} that $$\HH^*(\cA_0^{\sharp}) \cong \prod_{\vec{q} \in M^*} \HH^{*\oplus \vec{q}}(\cA_0) \cong \bigoplus_{\vec{q} \in M^*}\HH^{*\oplus \vec{q}}(\cA_0),$$
where the second isomorphism uses the aforementioned finite-dimensionality of $\HH^*(\cA_0^{\sharp}).$ From the proof of Proposition \ref{prop:COsurj}, we also have that $\HH^*(\cW(F))$ is a direct sum of its $M^*$-graded pieces. As \eqref{eq:HHrestriction2} is an isomorphism, it follows that \eqref{eq:gradedHHrestriction} is as well. 
\end{proof}
Combining \cref{prop:COsurj} and \cref{prop:HHA_0} with the fact that the restriction map on Hochschild cohomology is a Lie algebra homomorphism, we obtain:

\begin{cor}\label{cor:SH_HH}
The closed--open map
$$\cC \cO: SH^*(\eff;\Bbbk) \to \HH^*(\cA_0)$$
is a $\Z \oplus M^*$-graded isomorphism of Lie algebras.\qed
\end{cor}

We recall that a linear map $\phi:HH_{n}(\cC) \to \Bbbk$ is called a \emph{weak proper $n$-Calabi--Yau structure} on $\cC$ if the pairing
\begin{align*}
(-,-)_\phi: Hom^i_\cC(X,Y) \otimes Hom^{n-i}_\cC(Y,X)    & \to \Bbbk \\
(a,b)_\phi &:= \phi([a \bullet b])
\end{align*}
is a perfect pairing, for all pairs of objects $X$ and $Y$ and $i \in \Z$. 

Compact Fukaya categories of $2n$-dimensional Liouville domains always admit weak proper $n$-Calabi--Yau structures \cite[(12j)]{seidel2008fukaya}, and in particular, $\cA_0$ does.

\begin{lem}\label{lem:HH_minus_n}
    $HH_n(\cA_0)$ has rank 1 and is concentrated in degree $0 \in M^*$. Furthermore, the map $Hom^n_{\cA_0}(L_i,L_i) \to HH_n(\cA_0)$ is an isomorphism for any $i$.
\end{lem}
\begin{proof}
The weak proper $n$-Calabi--Yau structure on $\cA_0$ induces an isomorphism 
$$HH_i(\cA_0) \cong HH^{n-i}(\cA_0)^\vee$$
(see \cite[Lemma A.2]{Sheridan2013}). 
Applying Corollary \ref{cor:SH_HH}, we find that $HH_n(\cA_0)$ is dual to $SH^0(F)$, which has rank $1$ spanned by the unit in this case, by a straightforward application of \cite[Theorem 1.1]{ganatra2020symplectic} (see \S \ref{sec:logPSS} for further discussion of these results). 

We now note that as the composition 
$$Hom^0(L_i,L_i) \otimes Hom^n(L_i,L_i) \to Hom^n(L_i,L_i) \to HH_n(\cA_0) \xrightarrow{\phi} \Bbbk$$
is a perfect pairing, it must be non-zero. 
As both $Hom^n(L_i,L_i)$ and $HH_n(\cA_0)$ have rank $1$, the map $Hom^n(L_i,L_i) \to HH_n(\cA_0)$ must be an isomorphism. 

This map clearly respects the $M^*$-grading, and $Hom^n(L_i,L_i)$ is concentrated in degree $0 \in M^*$ for any $i$, hence $HH_n(\cA_0)$ is too.
\end{proof}

If $\cC$ is a $\Z \oplus M^*$-graded $A_\infty$ category, then a graded weak proper $n$-Calabi--Yau structure is a weak proper $n$-Calabi--Yau structure $\phi$ which vanishes on $HH_{n \oplus \vec{p}}$ for $\vec{p} \neq 0$. 
In this case, the pairing on morphism spaces respects the $M^*$-grading. 

\begin{lem}\label{lem:compat_CY}
    Both $\cA_0$ and $\cB_0$ admit $M^*$-graded weak proper $n$-Calabi--Yau structures, which are respected by the isomorphism $HH_{n}(\cA_0) = HH_{n}(\cB_0)$ induced by the quasi-isomorphism $\cA_0 \simeq \cB_0$.
\end{lem}
\begin{proof}
    The weak proper $n$-Calabi--Yau structure on $\cA_0$ induces one on $\cB_0$ via the isomorphism $HH_n(\cA_0) = HH_n(\cB_0)$. 
    By Lemma \ref{lem:HH_minus_n}, $HH_n(\cA_0)$ is concentrated in degree $0 \in M^*$, hence the Calabi--Yau structure on $\cA_0$ is $M^*$-graded (indeed, the weak proper Calabi--Yau structure on the $M^*$-graded compact Fukaya category is always $M^*$-graded for general reasons).  
    On $\cB_0$, we observe that the map $Hom^n_{\cB_0}(\cO(i),\cO(i)) \to HH_n(\cB_0)$ is $M^*$-graded, and an isomorphism by Lemma \ref{lem:HH_minus_n}. 
    As $Hom^n_{\cB_0}(\cO(i),\cO(i)) = H^n(\cO)$ is graded in degree $0 \in M^*$ for any $i$, so is $HH_n(\cB_0)$, hence the Calabi--Yau structure on $\cB_0$ is also $M^*$-graded.
\end{proof}

\subsection{Upgrading to \texorpdfstring{an $M^*$-graded}{a graded} quasi-isomorphism}

\begin{lem}\label{lem:structure_A0}
    The maps
    \begin{equation}
        \label{eq:Hom_cd}
    \begin{tikzcd}
        Hom^k_{\cB_0}(\cO(i), \cO(j) ) \arrow[leftrightarrow]{r}{\mir0^k} & Hom^k_{\cA_0}(L_i,L_j) \\
        Hom^k_{\tilde \cB_0}(\tildeoi,  \mathcal{E}_j ) \arrow{u}{ i^! } \arrow[leftrightarrow]{r} & Hom^k_{\tilde \cA_0}(\tilde L_i,\tilde L_j) \arrow{u}{\cap}
    \end{tikzcd}
    \end{equation}
    arising in the diagram \eqref{eq:Pic_sec_cd2} have the following properties:
    \begin{enumerate}
    \item \label{it:AO_1} $Hom^k(\cO(i), \cO(j)) = Hom^k(L_i,L_j) = 0$ unless $k=0$ or $n$.
    \item \label{it:AO_2} the vertical maps are surjective for $k=0$;
    \item \label{it:AO_3} the top arrows respect the Serre duality pairings induced by the Calabi--Yau structures chosen in Lemma \ref{lem:compat_CY} in the sense that 
    $$(F^i_0 a,F^{n-i}_0 b) = (a,b)$$
    for any $a \in Hom^i_{\cB_0}(\cO(j),\cO(k))$ and $b \in Hom_{\cB_0}^{n-i}(\cO(k),\cO(j))$. 
                                \end{enumerate} 
\end{lem}
\begin{proof}
    Items \eqref{it:AO_1} and \eqref{it:AO_2} are easily verified on $\cB_0$, and the result follows for $\cA_0$ by commutativity of the diagram. Item \eqref{it:AO_3} follows from the fact that the Calabi--Yau structures are chosen so as to be respected by the mirror equivalence.
\end{proof}

\begin{lem}\label{lem:grading_on_coh_level}
The cohomological isomorphism $H^*(\cA_0) \simeq H^*(\cB_0)$ respects $\Z \oplus M^*$-gradings.
\end{lem}
\begin{proof}
    As $\cap$, $ i^! $, and the bottom arrow in the diagram \eqref{eq:Hom_cd} respect the $M^*$-grading, it follows from Lemma \ref{lem:structure_A0} \eqref{it:AO_2} that $\mir0^0$ respects the $M^*$-grading.
    It follows from the $i=0$ case of Lemma \ref{lem:structure_A0} \eqref{it:AO_3}, together with the fact that the pairings respect the $M^*$-grading because the Calabi--Yau structures were chosen in Lemma \ref{lem:compat_CY} to be $M^*$-graded, that the $F^n_0$ respect the $M^*$-grading too.
\end{proof}

 \begin{lem} \label{lem:F0respectsM} There exists a $\Z \oplus M^*$-graded quasi-isomorphism $\cA_0 \simeq \cB_0$ inducing the isomorphism from Lemma \ref{lem:grading_on_coh_level} at the level of cohomology. \end{lem}
 \begin{proof} This is Lemma \ref{lem:Mgradedqi}. \end{proof}

\subsection{Split-generating the compact Fukaya category}
In this short subsection, we prove a split-generation result --- and therefore an HMS result --- for compact exact Lagrangians in $F$ which is not needed elsewhere in the text, but which may be of independent interest. We let $\fuk(F)$ denote the $\mathbb{Z}$-graded compact (exact) Fukaya category of $F$. Let $\perf_{\mathsf{dg}}(\partial Y^*) \subset \dbdg Coh(\partial Y^*)$ the full subcategory of perfect complexes (see \cref{sec:splitgeneration} for an explicit dg model of $\perf_{\mathsf{dg}}(-)$, and also compare \cref{rmk:perfuniqueness}).

\begin{prop}\label{prop:perf_hms} The subcategories $\cB_0 \subset \perf(\partial Y^*)$ and $\cA_0 \subset \fuk(F)$ split-generate, and in particular, we have a quasi-equivalence
    $$\perf \fuk(F) \simeq \perf_{\mathsf{dg}}(\partial Y^*).$$
\end{prop}
\begin{proof} Because the objects of $\cB_0$ are powers of an ample line bundle, they split-generate $\perf_{\mathsf{dg}}(\partial Y^*)$ by \cite[Theorem 4]{Orlov2009}. 

View $L_i$ as objects of $\mathcal{W}(F)$. By \Cref{prop:GS_Li} \eqref{it:top_arr_GS}, the equivalence $GS$ takes $L_i$ to the line bundles $\cO(i)$. Now consider any other object $E$ of $\mathcal{F}(F) \subset \mathcal{W}(F).$ Let $GS(E)$ denote the image of $E$ under the equivalence $GS$. For any object $B \in \operatorname{Ob}(\dbdg Coh(\partial Y^*))$, we have that:  \begin{align} \operatorname{dim}_{\Bbbk}(H^*(\operatorname{Hom}^*_{\dbdg Coh(\partial Y^*)}(GS(E), B)))<\infty. \end{align} 
It is well-known (\cite[Proof of Prop B.1]{Efimov} or \cite[Remark 1.2.6]{BNP}) that this implies that $GS(E)$ is perfect. Hence $GS(E)$ is split-generated by $\cO(i)$. It follows that $E$ is split-generated by $L_i$.  \end{proof}

% \begin{prop} The Lagrangian branes $L_i$ split-generate $\mathcal{F}(F)$. \end{prop}
% \begin{proof} View $L_i$ as objects of $\mathcal{W}(F)$. By \Cref{prop:GS_Li} \eqref{it:top_arr_GS}, the equivalence $GS$ takes $L_i$ to the line bundles $\cO(i)$. Because these line bundles are powers of an ample line bundle, they split-generate the perfect derived category $\operatorname{Perf}(\partial Y^*)$ by \cite[Theorem 4]{Orlov2009}. Now consider any other object $E$ of $\mathcal{F}(F) \subset \mathcal{W}(F).$ Let $GS(E)$ denote the image of $E$ under the equivalence $GS$. For any object $B \in \operatorname{Ob}(D^bCoh_{dg}(\partial Y^*))$, we have that:  \begin{align} \operatorname{dim}_{\Bbbk}(H^*(\operatorname{Hom}^*_{D^bCoh_{dg}(\partial Y^*)}(GS(E), B)))<\infty. \end{align} It is well-known ( \cite[Proof of Prop B.1]{Efimov} or \cite[Remark 1.2.6]{BNP}) that this implies that $GS(E)$ is perfect. Hence $GS(E)$ is split-generated by $\cO(i)$. It follows that $E$ is split-generated by $L_i$.  \end{proof}

\section{Deforming the \texorpdfstring{$A$}{A}-side}\label{sec:Adef}

In this section, we make the computations needed to show that the subcategory $\cA_R$ of the relative Fukaya category $\fuk(X_{t,\kaehl},D,\Nef)$ having the same objects as $\cA_0$, is a versal deformation of the latter. 
That is, we verify the hypotheses of the versality criterion from \cite[Proposition A.6]{Ganatra2023integrality} (the key results are \cref{lem:Adefgen},
\cref{prop:totalobstruction}, and \cref{lem:Avanishing}). 
In order to make these computations, we identify Hochschild cohomology with symplectic cohomology (via \cref{cor:SH_HH}), which in turn is identified with log symplectic cohomology by \cite{ganatra2021log,ganatra2020symplectic}. 
We also prove a new result, \cref{prop: apCOPSS}, which identifies the classes in log symplectic cohomology which correspond to the first-order deformation classes under this identification. 
The required computations then take place in the log symplectic cohomology, where they are purely topological.
\subsection{Symplectic cohomology} \label{sec:logPSS}

  We wish to understand $SH^*(\mathcal{H}_{t,0})$ using the log PSS map from \cite{ganatra2021log,  ganatra2020symplectic}. We let $H^*_{log}(X_{t,\kaehl},D)$ denote the log cohomology of the pair $(X_{t,\kaehl},D)$, introduced in \cite[\S 3.2]{ganatra2021log}. $H^*_{log}(X_{t,\kaehl},D)$ is a certain $\Z \oplus M^*$-graded $\Bbbk$-module defined topologically from the pair $(X_{t,\kaehl},D).$ We denote the $M^*$-graded pieces of $H^*_{log}(X_{t,\kaehl},D)$ by  $H^{*\oplus \vec{q}}_{log}(X_{t,\kaehl},D)$ for $\vec{q} \in M^*$. 

Recall that the generators of $H^*_{log}(X_{t,\kaehl},D)$ have the form $\alpha \tee^{\vec{v}}$ where $\vec{v} \in \Z_{\ge 0}^{\Xizero}$\footnote{Note that here we use the connectedness condition.} and $\alpha \in H^*(S_{\vec{v}})$ for some manifold $S_{\vec{v}}$, which is empty if and only if the corresponding stratum
\begin{align}\label{eq:stratum_contributing}
     \bigcap_{\vec{p}:v_{\vec{p}} \neq 0} D_{\vec{p}}
     \end{align}
is empty. 
We will call $\vec{v}$ such that $S_{\vec{v}} \neq \emptyset$ \emph{contributing}. 

By \cite[Equation (3.20)]{ganatra2021log}, the $\Z$-degree of such a chain is
\begin{align}\label{eq:zdegree}
    \deg_\Z\left(\alpha \tee^{\vec{v}}\right) = \deg(\alpha) + 2|\vec{v}|
\end{align}
(where $|\vec{v}| := \sum_{\vec{p} \in \Xizero} v_{\vec{p}}$), while the $M^*$-degree is
\begin{align}
\label{eq:Mdegree}   \deg_{M^*}(\vec{v}) = \sum_{\vec{p} \in \Xizero} v_{\vec{p}} \cdot \vec{p}.
\end{align}

\begin{lem}\label{lem:Mdegunique}
    If $\vec{v}_1$ and $\vec{v}_2$ are both contributing, and $\deg_{M^*}(\vec{v}_1) = \deg_{M^*}(\vec{v}_2)$, then $\vec{v}_1 = \vec{v}_2$. 
    In particular, if $\vec{v}$ is contributing, then we have
    \begin{align}\label{eq:Hlog_m_unique}
        H^{* \oplus \deg_{M^*}(\vec{v})}(X_{t,\kaehl},D) \cong H^{*-2|\vec{v}|}(S_{\vec{v}}).
    \end{align}
\end{lem}
\begin{proof}
    Note that if $\vec{v}$ is contributing, then $\deg_{M^*}(\vec{v})$ lies in the relative interior of the cone of $\Sigma_\kaehl$ corresponding to the toric orbit of $Y_\kaehl$ in which \eqref{eq:stratum_contributing} lies. 
    Furthermore, as $\Sigma_\kaehl$ is simplicial, the generators $\vec{p}$ for which $v_{\vec{p}} \neq 0$ form a basis for this cone. 
    In particular, the point $\deg_{M^*}(\vec{v})$ determines the non-zero coordinates $v_{\vec{p}}$; the other coordinates are all $0$. 
    Thus the $M^*$-degree of a contributing $\vec{v}$ determines $\vec{v}$ uniquely, as required. 
\end{proof}

\begin{lem}\label{lem:Hlog_comp}
    We have:
    \begin{enumerate}[label=(\roman*)]
    \item \label{it:hlog0} In $M^*$-grading  $0$: \begin{align}  H^{*\oplus 0} _{log}(X_{t,\kaehl},D) \cong H^*(\mathcal{H}_{t,0}).  \end{align}
    \item \label{it:hlogp} In $M^*$-grading $\vec{p}$, for $\vec{p} \in \Xizero$:
    \begin{align} H_{log}^{*\oplus \vec{p}}(X_{t,\kaehl},D)\cong H^{*-2}(SD_{\vec{p}}),  \end{align} where $SD_{\vec{p}}$ denotes the normal $S^1$ bundle to the interior of $D_{\vec{p}}$. In particular, for $\vec{p} \in \Xizero$, $H^{2\oplus \vec{p}} _{log}(X_{t,\kaehl},D)$ has rank one, generated by the class $\gamma_{\vec{p}}$ corresponding to the identity in $H^0(SD_{\vec{p}})$. 
    \item \label{it:hlogother} In $M^*$-grading $\vec{q} \notin \Xizero \cup \{0\}$, $H_{log}^{i \oplus \vec{q}}(X_{t,\kaehl},D) = 0$ for $i \le 3$.
    \end{enumerate}
\end{lem}
\begin{proof}
    Given Lemma \ref{lem:Mdegunique}, \cref{it:hlog0} follows from the case $\vec{v} = 0$ of \eqref{eq:Hlog_m_unique}, while \cref{it:hlogp} follows from the case where $v_{\vec{q}} = \delta_{\vec{p},\vec{q}}$. 
    \cref{it:hlogother} follows from \eqref{eq:zdegree} and the fact that all other $\vec{v}$ have $|\vec{v}| \ge 2$.
\end{proof}

The log cohomology $H^*_{log}(X_{t,\kaehl},D)$ carries a $\Z \oplus M^*$-graded ring structure (see \cite[Definition 3.4]{ganatra2020symplectic}), which we recall in the special case that we need for our applications.  Let $\mathring{X}_{\vec{p}} \subset X_{t,\kaehl} $ be the complement of all of the divisor components different from $D_{\vec{p}}.$ The real oriented blowup $\hat{X}_{\vec{p}}$ of $\mathring{X}_{\vec{p}}$ along the open stratum of $D_{\vec{p}}$ is then a manifold with boundary $SD_{\vec{p}}$.  The blowup $\hat{X}_{\vec{p}}$ is homotopy equivalent to $\mathcal{H}_{t,0}$ and there is thus a canonically defined restriction map:  \begin{align} r_{\vec{p}}^*: H^*(\mathcal{H}_{t,0}) \to H^*(SD_{\vec{p}}). \end{align}  Given $\alpha_1 \in H^{q_1 \oplus 0} _{log}(X_{t,\kaehl},D)= H^{q_1}(\mathcal{H}_{t,0})$ and $\alpha_2 \in H^{q_2 \oplus \vec{p}} _{log}(X_{t,\kaehl},D) = H^{q_2-2}(SD_{\vec{p}})$,  their product is given by \begin{align} \alpha_1 \ast_{log} \alpha_2 := r_{\vec{p}}^*(\alpha_1) \cup \alpha_2 \in H_{log}^{q_1+q_2 \oplus \vec{p}}(X_{t,\kaehl},D) = H^{q_1+q_2-2}(SD_{\vec{p}}). \end{align}

To define the log PSS map, we rule out sphere bubbling using the following lemma: 

 \begin{lem} \label{lem:modifyprimitive} Fix $\vec{p} \in \Xizero$. We can represent the K{\"a}hler class $[\omega]=[\omega_\kaehl]$ as a sum \begin{align} [\omega] = \sum_{\vec{q} \in \Xizero} \kaehl'_{\vec{q}} \cdot PD(D_{\vec{q}}), \end{align} where $\kaehl'_{\vec{q}}> 0$ for $\vec{q} \neq \vec{p}$, and $\kaehl'_{\vec{p}}=0$. \end{lem}
 \begin{proof} The K{\"a}hler form is restricted from the ambient toric variety $Y_\kaehl$. On $Y_\kaehl$, a class is K{\"a}hler if and only if the piecewise-linear function $\phi_\kaehl$ with value $\kaehl_{\vec{q}}$ on the generator $v_\vec{q}$ of the ray of the fan corresponding to $D_{\vec{q}}$, is strictly convex. Moreover, we have $[\omega_\kaehl] = [\omega_{\kaehl'}]$ iff $\phi_\kaehl - \phi_{\kaehl'}$ is linear. In particular, given a K{\"a}hler class $[\omega]$, we can choose $\kaehl'$ so that $\kaehl'_\vec{p} = 0$ and $\kaehl'_\vec{q} > 0$ for all  $\vec{q} \neq \vec{p}$, by adding an appropriate linear function to $\phi_\kaehl$ (cf. \cite[Lemma 3.31]{Sheridan2017}). 
\end{proof} 

\begin{cor}\label{cor:nospheres}
    Let $\tilde\omega$ be a symplectic form with $[\tilde \omega] = [\omega]$, $J$ an almost-complex structure compatible with $\tilde \omega$, and $\tilde D_{\vec{p}}$ submanifolds homologous to $D_{\vec{p}}$ for $\vec{p} \in P$. 
    Then any stable $J$-holomorphic curve $u$, with at least one non-constant component, satisfies $u \cdot \tilde D_{\vec{p}} > 0$ for at least two of the $\vec{p} \in P$. 
\end{cor}
\begin{proof}
    Suppose, to the contrary, that there exists $\vec{p} \in P$ such that $u \cdot \tilde{D}_{\vec{q}} \le 0$ for $\vec{q} \neq \vec{p}$. 
    Then by \cref{lem:modifyprimitive}, we have
    \begin{align}
        0 \ge  \sum_{\vec{q} \in P} \kaehl'_{\vec{q}}  \tilde{D}_{\vec{q}} \cdot u = \tilde\omega(u) > 0,
    \end{align}
    a contradiction.
\end{proof}

We recall that the construction of the log PSS map involves deforming $\omega$ to a symplectic form $\tilde{\omega}$, within the same cohomology class, so that the divisor $D$ admits a regularization in the sense of \cite[Definition 2.10]{McleanTehraniZinger} (which builds on \cite[\S 5]{Mclean2010}). Corollary \ref{cor:nospheres} implies that for each $\vec{p} \in \Xizero$, any $J$-holomorphic sphere bubble (for any almost-complex structure $J$ compatible with $\tilde{\omega}$) must positively intersect one of the divisor components $D_\vec{q}$ with $\vec{q}\neq \vec{p}$. 
It follows from this that all of the classes $H_{log}^{*\oplus \vec{p}}(X_{t,\kaehl},D)$ are tautologically admissible in the sense of \cite[\S 3.3]{ganatra2021log} and there is therefore a log PSS map:  
\begin{align} \label{eq:PSSlog} PSS^{log}: H_{log}^{*\oplus \vec{p}}(X_{t,\kaehl},D) \to SH^{*\oplus \vec{p}}(\mathcal{H}_{t,0}). \end{align}

Note that the classical PSS construction defines a map: 
\begin{align} \label{eq:PSSclass} PSS: H^{*\oplus 0} _{log}(X_{t,\kaehl},D) \to SH^{*\oplus 0}(\mathcal{H}_{t,0}).  \end{align} 

\begin{prop} \label{prop:productsSH} Given $\alpha_1 \in H^{*\oplus 0} _{log}(X_{t,\kaehl},D)$ and $\alpha_2 \in H^{* \oplus \vec{p}} _{log}(X_{t,\kaehl},D),$ 
\begin{align} PSS(\alpha_1) \ast PSS^{log}(\alpha_2)=PSS^{log}(\alpha_1 \ast_{log} \alpha_2). \end{align} \end{prop}
\begin{proof} This follows from the argument of \cite[Theorem 5.12]{ganatra2020symplectic}. That argument interpolates between the two operations using a certain moduli space of thimbles with two marked points and it suffices to show that no sphere bubbling can arise in the interpolating moduli space. 
Suppose, to the contrary, that such a non-constant sphere bubble arose. 
In the present case, the thimbles intersect $D_{\vec{p}}$ and no other divisor component. 
By Corollary \ref{cor:nospheres}, the union of all the sphere bubbles must positively intersect one of the divisors $D_{\vec{q}}, \vec{q} \neq  \vec{p}.$ 
This would mean the thimble obtained by removing all of the sphere bubbles would have negative intersection with this divisor $D_{\vec{q}}$, but the thimbles intersect each component of $D$ non-negatively by positivity of intersection, a contradiction. 
\end{proof} 

\begin{prop}\label{prop:shgrad}
 We have:
\begin{enumerate}[label=(\roman*)]
\item \label{it:sh20} the map $PSS^{log}:H^2(\mathcal{H}_{t,0};\Bbbk) \to SH^{2 \oplus 0}(\mathcal{H}_{t,0};\Bbbk)$ is surjective. 

\item \label{it:sh2p} if $\vec{p} \in \Xizero$, then $SH^{2 \oplus \vec{p}}(\mathcal{H}_{t,0};\Bbbk)$ has rank one and is generated by the class $PSS^{log}(\gamma_{\vec{p}})$, where $\gamma_{\vec{p}}$ is the class defined in \cref{lem:Hlog_comp}. 

 \item \label{it:sh2other} $SH^{2 \oplus \vec{p}}(\mathcal{H}_{t,0};\Bbbk) = 0$ unless $\vec{p} \in \Xizero \cup \{0\}$. 

\end{enumerate}
\end{prop}
\begin{proof}
By \cite[Theorem 1.1]{ganatra2020symplectic}, there is a convergent spectral sequence \begin{align} E_r^{\ast-q,q} \Rightarrow SH^*(\mathcal{H}_{t,0}) \end{align} and an isomorphism \begin{align} \label{eq:PSSlow} H^*_{log}(X_{t,\kaehl},D) \cong \bigoplus_{q} E_1^{\ast-q,q}. \end{align}
The spectral sequence is also compatible with $M^*$- gradings. The higher differentials in the spectral sequence, when applied to a log cohomology class $\alpha \tee^{\vec{v}}$, only involve terms $\tilde{\alpha}\tee^{\tilde{\vec{v}}}$ with $\tilde{\vec{v}} \neq \vec{v}$. It therefore follows from Lemma \ref{lem:Mdegunique} that the spectral sequence degenerates at the $E_1$ page. With this in place, the proofs of these claims are as follows:

\cref{it:sh20} follows from the fact that on $H^{*\oplus 0} _{log}(X_{t,\kaehl},D)$, which is isomorphic to $H^*(\mathcal{H}_{t,0})$ by Lemma \ref{lem:Hlog_comp}, \eqref{eq:PSSlow} is the associated graded of the PSS map \eqref{eq:PSSclass}.

To prove \cref{it:sh2p}, we first note that by combining Lemma \ref{lem:Hlog_comp} with the collapse of the spectral sequence, we have that $SH^{2 \oplus \vec{p}}(\mathcal{H}_{t,0};\Bbbk) \cong H^{0}(SD_\vec{p})$, which implies that it has rank one. The fact that $PSS^{log}(\gamma_{\vec{p}})$ is a generator for $SH^{2 \oplus \vec{p}}(\mathcal{H}_{t,0};\Bbbk)$ follows from the fact that, on $H^{*\oplus \vec{p}} _{log}(X_{t,\kaehl},D)$, \eqref{eq:PSSlow} is the associated graded of the log PSS map \eqref{eq:PSSlog}. 
\cref{it:sh2other} is immediate from \cref{it:hlogother} of Lemma \ref{lem:Hlog_comp} and the convergence of the spectral sequence.
\end{proof}

\begin{lem}
    \label{lem:affine_lef}
    We have a canonical surjective map
    \begin{align} \label{eq:aff_lef_1} \Lambda^2 M_\Bbbk \to  H^{2 \oplus 0} _{log}(X_{t,\kaehl},D),  \end{align}
    and, for any $\vec{p} \in P$, a canonical isomorphism 
    \begin{align} \label{eq:aff_lef_2} M_\Bbbk \cong H^{3 \oplus \vec{p}} _{log}(X_{t,\kaehl},D).
    \end{align}
\end{lem}
\begin{proof}
    Note that the restriction map $\Lambda^i M_\Bbbk \cong H^i(M^*_{\C^*};\Bbbk) \to H^i(\mathcal{H}_{t,0};\Bbbk)$ is surjective for $i< \dim_\C(X)$ by the affine Lefschetz hyperplane theorem \cite[Theorem 6.5]{dimca1992singularities}, and in particular for $i=2$, by our assumption that $\dim_\C(X) \ge 3$.  This gives rise to \eqref{eq:aff_lef_1}.
    
    Let $\mathring{Y}_{\vec{p}} \subset Y_\kaehl$ denote the complement of all of the divisors except $D_{\vec{p}}^{Y}.$ The real oriented blowup $\hat{Y}_{\vec{p}}$ of  $\mathring{Y}_{\vec{p}}$ along the open stratum of $D_{\vec{p}}^{Y}$ is then a manifold with boundary $S_{\vec{p}}^{Y}$,  the normal circle bundle over the open stratum of $D_{\vec{p}}^{Y}$.  It is not difficult to see that $S_{\vec{p}}^{Y}$ is homeomorphic to $M^*_{S^{1}} \times \Delta_{\vec{p}}$,  where $\Delta_{\vec{p}}$ is the open codimension-one face of the moment polytope of $Y_\kaehl$, and the inclusion map  $S_{\vec{p}}^{Y} \to \hat{Y}_{\vec{p}}$ is a homotopy equivalence. The bundle $SD_{\vec{p}}$ is also a trivial $S^1$-bundle over the interior of $D_\vec{p}$, so again employing the affine Lefschetz hyperplane theorem (together with the K{\"u}nneth formula),  we have that the restriction map $M_\Bbbk \cong H^1(S_{\vec{p}}^{Y};\Bbbk) \to H^1(SD_{\vec{p}};\Bbbk)$ is an isomorphism, yielding \eqref{eq:aff_lef_2}.
	\end{proof}

\begin{prop} \label{prop:combinatorialSH} For any $\vec{p} \in \Xizero$, the following diagram commutes:  \begin{equation} \label{eq:combinatorialbracket} \begin{tikzcd}  SH^{2 \oplus 0}(\mathcal{H}_{t,0})  \arrow[r,  "L_{\vec{p}}"] & SH^{3 \oplus \vec{p}}(\mathcal{H}_{t,0})  & \\
		\Lambda^2 M_\Bbbk  \arrow[u, "PSS"]  \arrow[r,"\iota_{\vec{p}}"]  &  M_\Bbbk \arrow[u,  "PSS_{log}"]. 
	\end{tikzcd} \end{equation}   Here,  the map $L_{\vec{p}}$ denotes the map $[PSS_{log}(\gamma_{\vec{p}}),-]$ given by Lie bracket with $PSS_{log}(\gamma_{\vec{p}})$, and $\iota_{\vec{p}}$ denotes the map given by contracting with $\vec{p} \in M^*.$ \end{prop} 
\begin{proof} Let $\alpha$ be a class in $H^2(\mathcal{H}_{t,0})$ and  let $\Delta$ denote the BV operator on symplectic cohomology.   It is standard that  $\Delta(PSS(\alpha))=0$; furthermore, $\Delta(PSS_{log}(\gamma_{\vec{p}}))=0$ by
\cite[Lemma 4.29]{ganatra2021log}.  Applying the BV-equation,  we therefore have 
\begin{align} [PSS_{log}(\gamma_{\vec{p}}),PSS(\alpha)]=  \Delta(PSS(\alpha) * PSS_{log}(\gamma_{\vec{p}})). \end{align} 
Applying Proposition \ref{prop:productsSH},  this becomes: \begin{align} \label{eq:intermediatebracket} [PSS_{log}(\gamma_{\vec{p}}),PSS(\alpha)]= \Delta(PSS_{log}(r_{\vec{p}}^*(\alpha))), \end{align} where in this equation we view $r_{\vec{p}}^*(\alpha)$ as an element in $H^{4 \oplus \vec{p}} _{log}(X_{t,\kaehl},D)$.

To compute the BV-operator,  note that the $S^1$-action on $SD_{\vec{p}}$ induces a degree $-1$ cohomology operation: \begin{align} B_{\vec{p}}: H^*( SD_{\vec{p}}) \to H^{*-1}(SD_{\vec{p}}).  \end{align} 

 By functoriality of the real oriented blowup \cite{arone-kankaanrinta10}, we have a commutative square:  \begin{equation}\label{eq:blowupfunctorial} \begin{tikzcd}   SD_{\vec{p}}  \arrow[r] \arrow[d] & S_{\vec{p}}^{Y} \arrow[d]& \\
		 \hat{X}_{\vec{p}}  \arrow[r]  &  \hat{Y}_{\vec{p}} 
	\end{tikzcd} \end{equation}  

It follows $r_p^*(\alpha)$ can also be viewed as the restriction of a class on $S_{\vec{p}}^{Y} \cong M^*_{S^{1}} \times \Delta_{\vec{p}}.$ Observe that all of the Borel--Moore homology classes on $M^*_{S^{1}} \times \Delta_{\vec{p}}$ can be canonically represented by $\Bbbk$-linear combinations of submanifolds. In particular, by intersecting these submanifolds with $SD_{\vec{p}}$, we see that $r_p^*(\alpha)$ is Poincar\'{e} dual to a class which can be represented by $\Bbbk$-linear combinations of proper maps from a smooth, oriented manifold. We can therefore combine \eqref{eq:intermediatebracket} with  \cite[Lemma 4.29]{ganatra2021log} again to show that \begin{align} \label{bracket} [PSS_{log}(\gamma_{\vec{p}}),PSS(\alpha)]= PSS_{log}(B_{\vec{p}}(r_{\vec{p}}^*(\alpha))). \end{align}
To conclude,  note that the inclusion of circle bundles $SD_{\vec{p}} \subset S_{\vec{p}}^{Y}$ is $S^1$-equivariant and consequently there is a commutative diagram \begin{equation}\label{eq:squareS1} \begin{tikzcd}   \Lambda^2M_{\Bbbk} \cong H^2(S_{\vec{p}}^{Y})  \arrow[r] \arrow[d, "B_{\vec{p}}^{Y}"] & H^2(SD_{\vec{p}}) \arrow[d, "B_{\vec{p}}  "]& \\
		 M_{\Bbbk} \cong H^1(S_{\vec{p}}^{Y}) \arrow[r, "\cong"]  &   H^1(SD_{\vec{p}}) 
	\end{tikzcd} \end{equation}  

 On $S_{\vec{p}}^{Y}$, the $S^1$-action is just given by the action of the circle $\operatorname{exp}(\vec{p})$ (here, we are again identifying $S_{\vec{p}}^{Y}$ with $M^*_{S^{1}} \times \Delta_{\vec{p}}$). It is then straightforward to see that $B_{\vec{p}}^{Y}=\iota_{\vec{p}}$ from which the result follows immediately.    \end{proof}

\subsection{Relative Fukaya category}
\label{subsec:relativeFukayaCategory}
Since we assume the connectedness condition holds, the intersection of each component $D^Y_{\vec{p}}$ with $X_{t,\kaehl}$ is either connected (when $\vec{p} \in \Xizero$) or empty (otherwise), by Lemma \ref{lem:connectedness}. 
Thus we have $H^2(X_{t,\kaehl},X_{t,\kaehl} \setminus D;\R) \cong \R^{\Xizero}$. 
Recall the notion of a `nice' cone $\Nef \subset \R^{\Xizero}$ from \cite[Definition 3.20]{Sheridan2017}.

\begin{lem}
There exists a nice cone $\Nef \subset \R^{\Xizero}$ containing $\kaehl$ in its interior.
\end{lem}
\begin{proof}
Follows from \cite[Corollary 3.34]{Sheridan2017}, by the connectedness condition and the fact that $\kaehl$ was restricted from a relative K\"ahler form on $(Y_\kaehl,D^Y)$.
\end{proof}

We will consider the relative Fukaya category $\fuk(X_{t,\kaehl},D,\Nef)$, defined as in \cite{perutz2022constructing}. 
There are some choices involved in the construction, which we now specify. 
First, the category will be $\Z \oplus M^*$-graded. 
To define its coefficient ring, let $\mathrm{NE}_A \subset H_2(X_{t,\kaehl},X_{t,\kaehl} \setminus D) \cong \Z^{\Xizero}$ be the cone dual to $\Nef$ and let $ \Bbbk[\mathrm{NE}_A]$ be the monoid ring generated by $\mathrm{NE}_A$. 
For $u \in \mathrm{NE}_A$, we equip the generator $\nov^{u} \in \Bbbk[\mathrm{NE}_A]$ with the grading $0 \oplus [\partial u]$, where $[\partial u]$ is the class of the boundary of $u$ in $H_1(M^*_{\C^*}) = M^*$. The coefficient ring of the relative Fukaya category will be $R_A = \Bbbk[[\mathrm{NE}_A]]$, the $\Z \oplus M^*$-graded completion of $\Bbbk[\mathrm{NE}_A]$. 

The relative Fukaya category $\fuk(X_{t,\kaehl},D,\Nef)$ \cite{perutz2022constructing,Sheridan2017} is a $\Z \oplus M^*$-graded, $R_A$-linear, curved filtered $A_\infty$ category. 
It is a deformation of the $\Bbbk$-linear exact Fukaya category $\fuk(X_{t,\kaehl} \setminus D)$. Let us recall the key elements of the construction. Let $h_1=h^{X}$ denote the K{\"a}hler potential for $\omega$ chosen in \Cref{sec:rel_kahl} and let $c$ be a constant which is larger than the biggest critical value of $h_1$. Then the subspace $$\LS_1:=\lbrace h_1 \leq c \rbrace$$ is a Liouville domain with completion $\widehat{\mathcal{H}}_{t,0}$. An object of the relative Fukaya category is a compact exact Lagrangian $L \subset \LS_1$, equipped with a grading and spin structure. 
The grading consists of a lift of the phase function $L \to S^1 = \R/\Z$ determined by $\Omega^X$ to $\R$, together with a lift of $L$ to the covering space of $X_{t,\kaehl} \setminus D$ classified by the map 
$$\pi_1(X_{t,\kaehl} \setminus D) \to H_1(M^* \otimes_\Z \C^*) = M^*.$$

Morphisms are Hamiltonian chords between Lagrangians. 
The grading of the Lagrangians allows us to specify a degree of each chord (in the grading group $\Z \oplus M^*$), so that our morphism spaces are graded $R_A$-modules. 
The definition of the $\Z$-grading, using the lift of the phase function, is standard; the $M^*$-grading is defined to be the difference between the lifts of the Lagrangians at the ends of the chord. Let $\cA_R$ denote the resulting curved deformation of the category $\cA_0$. \vskip 5 pt

 Let $\vec{p} \in \Xizero$, and $\nov_\vec{p} \cdot a_\vec{p}$ be the corresponding first-order deformation class of $\cA_R$.  In the discussion below, we recall only those technical details concerning the definition of this class which will be relevant to the proof of Proposition \ref{prop: apCOPSS} ---  for full details, see \cite[\S 4.4]{Sheridan2017}.
We set $\mathcal{J}(X_{t,\kaehl})$ to be the space of $\omega$-compatible almost-complex structures on $X_{t,\kaehl}$.   By \cref{lem:modifyprimitive}, we can choose a different K{\"a}hler potential $h_2$ for $\omega$ which is infinite along $\cup_{\vec{q}\neq \vec{p}}D_\vec{q}.$ We let $\LS_2$ be a Liouville domain corresponding to $h_2$ which is chosen so that there is a proper containment: \begin{align} \LS_1 \subset \LS_2. \end{align}
Set $\lf_1=-dh_1 \circ J_0$ and $\lf_2=-dh_2\circ J_0$, where $J_0$ denotes the standard integrable almost-complex structure. \vskip 5 pt

Using the Liouville flow, we can and will assume that all of the Lagrangians $L_i$ lie in $\LS_1.$ We let $\mathbf{L}=(L^{(0)},... ,L^{(s)})$ be a tuple of objects in $\mathcal{A}_0$.  We set $\mathfrak{R}_i(\mathbf{L})$ to be the moduli space of holomorphic discs with boundary punctures $\zeta_0, \zeta_1, . . . , \zeta_s$ in cyclic order and $i$ interior marked points, where $i=0$ or $1$, modulo biholomorphisms of the disc; when $i=1$, we denote the interior marked point by $z_{int}$. We label the boundary components of the discs by $L^{(j)}$. We denote the universal curve by $U_{\mathfrak{R}_i(\mathbf{L})} \to \mathfrak{R}_i(\mathbf{L})$. We make universal choices of strip-like ends $\varepsilon_j$ at each puncture (one negative end at $\zeta_0$, the remaining ends positive). We let: \begin{align} \mathcal{K}_{\mathfrak{R}_i(\mathbf{L})}:= \Omega^1(U_{\mathfrak{R}_i(\mathbf{L})}/\mathfrak{R}_i(\mathbf{L}),  C^\infty(X_{t,\kaehl})), \quad \mathcal{J}_{\mathfrak{R}_i(\mathbf{L})}:= C^{\infty}(\mathcal{U}_{\mathfrak{R}_i(\mathbf{L})}, \mathcal{J}(X_{t,\kaehl})).\end{align}  We choose universal families of perturbation data $(K,J) \in \mathcal{K}_{\mathfrak{R}_i(\mathbf{L})} \times \mathcal{J}_{\mathfrak{R}_i(\mathbf{L})}.$ As usual, these choices (including the choice of strip-like ends) must be made consistently with previously chosen Floer data and the gluing operation for surfaces. 
 
Our arguments will involve working with different classes of perturbation data. 
For the definition of the affine Fukaya category $\fuk(X_{t,\kaehl} \setminus D)$, and the first-order deformation class $a_{\vec{p}}$, we assume that $(K,J) \in \mathcal{K}^{aff}_{\mathfrak{R}_i(\mathbf{L})} \times \mathcal{J}^{aff}_{\mathfrak{R}_i(\mathbf{L})}$, where
\begin{align}    
\mathcal{K}^{aff}_{\mathfrak{R}_i(\mathbf{L})} &:= \lbrace K \in \mathcal{K}_{\mathfrak{R}_i(\mathbf{L})}, K_r=0 \text{ outside } \LS_1 \text{ for all } r \in \mathfrak{R}_i(\mathbf{L})\rbrace\\
\mathcal{J}^{aff}_{\mathfrak{R}_i(\mathbf{L})} &:= \lbrace J \in \mathcal{J}_{\mathfrak{R}_i(\mathbf{L})}, J_r=J_0 \text{ outside } \LS_1 \text{ for all } r \in \mathfrak{R}_i(\mathbf{L}) \rbrace. \end{align}
 
Label the boundary punctures by chords $y_0$ from $L^{(s)}$ to $L^{(0)}$ and $y_j$ from $L^{(j-1)}$ to $L^{(j)}$. The class $a_\vec{p}$ is then defined by counting solutions $(r,u)$, where $r \in \mathfrak{R}_1(\mathbf{L})$ and $u$ is a solution to Floer's equation: 
\begin{equation} \label{infinitesimalmodulispace1}
\left\{
\begin{aligned}
& u: \Sigma_r \to X, \\
& (du-X_{K_r})^{0,1}=0, \\
& u(z) \in L^C \text{ for } z\in C \subset \partial \Sigma_r,\\
& u(\varepsilon_\zeta(s,t))\to y_\zeta(t) \text{ for } \zeta \text{ a boundary puncture}
\end{aligned}
\right.
\end{equation}
which additionally satisfies: 
\begin{equation} \label{infinitesimalmodulispace2}
\left\{
\begin{aligned}
& u \cdot D_{\vec{p}}=1;  u \cdot D_{\vec{q}}=0, \vec{q}\neq \vec{p} \\
& u(z_{int})\in D_{\vec{p}}.
\end{aligned}
\right.
\end{equation}

The following is a minor variation on the argument from \cite[\S 4.4]{Sheridan2017}:

\begin{lem} \label{lem:definingavecp} The count of rigid solutions to \eqref{infinitesimalmodulispace1}, \eqref{infinitesimalmodulispace2} defines a class $a_{\vec{p}} \in HH^{2\oplus \vec{p}}(\cA_0).$ 
This class is independent of the choice of perturbation data $(K,J)$.
\end{lem}
 \begin{proof} Following the standard strategy in Floer theory, we consider the compactifications of moduli spaces of dimension $\leq 1$. We note that all of the Lagrangians $L_i$ are Lagrangian spheres and hence are exact for both the primitives $\lf_1$ and $\lf_2$. The integrated maximum principle (see \cite[Lemma 4.3]{Sheridan2017}) implies that all of our solutions are confined to lie in $\LS_2$. In particular, no sphere bubbling can occur in our moduli spaces.  Another application of the integrated maximum principle implies that for any disc bubble component $u_{bubble}$:
\begin{align} \label{eq:ububbleDp} 
u_{bubble} \cdot D_{\vec{p}} \geq 0; \quad \text{ if }  u_{bubble} \cdot D_{\vec{p}}= 0, u_{bubble}(\Sigma_r) \subset \LS_1. 
\end{align}
With these two observations, the argument of \cite[\S 4.4]{Sheridan2017} implies that this construction still defines a Hochschild cocycle, and that the corresponding Hochschild cohomology class is independent of the choice of perturbation data.
\end{proof}

The class $a_{\vec{p}}$ is equal to the first-order deformation class of the relative Fukaya category (cf. \cite[\S 9.4]{perutz2022constructing} and \cite[Assumption 5.2]{Sheridan2017}).

As noted above, the definition of the log PSS morphism requires deforming the symplectic form in the same cohomology class so that the divisors $D$ admit a regularization. By a Moser argument, we can equivalently recast this as deforming the divisors $D_{\vec{q}}$ to symplectic divisors $\tilde{D}_{\vec{q}}$ which admit a regularization (see \cite[\S 3]{ganatra2021log} or \cite[\S 5]{Mclean2010}). Note that we can assume that this deformation is supported in an arbitrarily small neighbourhood of the codimension two strata of $D$. In particular, we can (and do) assume that: 
\begin{itemize} 
    \item the divisors $\cup_{\vec{q}} \tilde{D}_{\vec{q}}$ remain disjoint from $\LS_1$; 

    \item the divisors $\cup_{\vec{q} \neq \vec{p}} \tilde{D}_\vec{q}$ remain disjoint from $\LS_2$, and $\tilde{D}_{\vec{p}} \cap \mathcal{S}_2 = D_{\vec{p}}\cap \mathcal{S}_2$.
\end{itemize}

\begin{prop}\label{prop: apCOPSS}
We have 
\begin{align} 
    a_{\vec{p}} = \cC \cO(PSS^{log}(\gamma_\vec{p})) \in HH^{2\oplus \vec{p}}(\cA_0).
\end{align} 
\end{prop}
\begin{proof}
We first define a version of the class, $a_{\vec{p}}^{log}$, which is adapted to the log PSS setup. To prepare for this, we introduce the class of perturbation data $\mathcal{K}^{log}_{\mathfrak{R}_i(\mathbf{L})} \times \mathcal{J}^{log}_{\mathfrak{R}_i(\mathbf{L})}$ consisting of all $(K^{log},J^{log})$ such that for each $r \in \mathfrak{R}_i(\mathbf{L})$: 
\begin{enumerate}[label=(\alph*)]
\item For $i=0$, $K_r^{log}$ has non-positive curvature outside of $\LS_1$;
\item $X_{K_{r}^{log}}$ preserves each component $\tilde{D}_\vec{q}$;
\item $J_r^{log}$ satisfies 
    \begin{align} 
        \lf_1= -dh_1\circ J_r \text{ on a contact shell about } \partial \LS_1, 
    \end{align} 
\item $J_r^{log}$ makes each component $\tilde{D}_\vec{q}$ into an almost-complex submanifold.
\end{enumerate}
We choose perturbation data $(K^{log},J^{log}) \in \mathcal{K}^{log}_{\mathfrak{R}_i(\mathbf{L})} \times \mathcal{J}^{log}_{\mathfrak{R}_i(\mathbf{L})}$ such that for $i=0$, the restriction of our data to the union of $\LS_1$ and a small contact shell of $\partial \LS_1$ agree with the perturbation data previously chosen for the affine Fukaya category.  
The integrated maximum principle implies that any Floer solution $u$ satisfying  $u \cdot \tilde{D}_\vec{q}=0$ for all $\vec{q} \in P$ is confined to $\LS_1$. Therefore, while the perturbation data $(K^{log},J^{log})$ for $i=0$ may not agree with the data previously chosen for the affine Fukaya category, the resulting moduli spaces are isomorphic; so the disagreement does not affect the affine Fukaya category. 

We now define the class $a_{\vec{p}}^{log}$ by counting rigid solutions to \eqref{infinitesimalmodulispace1}, for the perturbation data $(K^{log},J^{log})$, which satisfy
\begin{equation} \label{infinitesimalmodulispace3}
\left\{
\begin{aligned}
& u \cdot \tilde{D}_{\vec{p}}=1;  u \cdot \tilde{D}_{\vec{q}}=0, \vec{q}\neq \vec{p} \\
& u(z_{int})\in \tilde{D}_{\vec{p}}.
\end{aligned}
\right.
\end{equation}

To control the compactification of dimension $\leq 1$ moduli spaces, note first that by positivity of intersection, for any disc bubble component $u_{bubble}$, we have that for all $ \vec{q} \in P$,  
\begin{align} \label{eq:randombubblelog} 
u_{bubble} \cdot \tilde{D}_\vec{q} \geq 0. 
\end{align}  
As in the proof of \cref{prop:productsSH}, Corollary \ref{cor:nospheres}  then rules out sphere bubbling in the moduli spaces and it follows that equality holds in \eqref{eq:randombubblelog} for $\vec{q} \neq \vec{p}$. From here, we have that if $u_{bubble} \cdot \tilde{D}_\vec{p}= 0$ as well, the component $u_{bubble}$ does not carry the interior marked point (again using  positivity of intersection) and so $u_{bubble}(\Sigma_r) \subset \LS_1$ by the integrated maximum principle. It now follows, as in the proof of Lemma \ref{lem:definingavecp}, that the argument of \cite[\S 4.4]{Sheridan2017} goes through to define a Hochschild cocycle, whose corresponding Hochschild cohomology class is independent of choices. Finally, a standard degeneration of domain argument (see \cite[Theorem 1.1]{tonkonog2019symplectic} for the analogous result in the context of log Calabi--Yau varieties) shows that 
\begin{align} 
a_\vec{p}^{log}=\cC \cO(PSS^{log}(\gamma_\vec{p})). 
\end{align} 
  \vskip 5 pt

It remains to prove that $a_\vec{p}^{log} = a_\vec{p}$. 
 The key technical issue requiring attention here is that the definition of the log PSS map uses almost-complex structures which preserve $\cup_{\vec{q}} \tilde{D}_{\vec{q}}$, and in particular, do not belong to $\mathcal{J}^{aff}$. Moreover, requiring the almost-complex structure to preserve $\tilde{D}_{\vec{p}}$ is potentially inconsistent with requiring the almost-complex structure to be of contact type along $\partial \LS_2,$ which is needed for the argument of Lemma \ref{lem:definingavecp}. 
 We resolve this by enlarging the families of almost-complex structures used to define $a_\vec{p}^{log}$ and $a_\vec{p}$, to $\mathcal{J}^1$ and $\mathcal{J}^0$ respectively, so that one can still define Hochschild cohomology classes using perturbation data from these enlarged families, the classes are still independent of choices, and the families have non-empty intersection.
 
We first enlarge the family of almost-complex structures used to define $a_\vec{p}$. Let $\mathcal{J}^0_{\mathfrak{R}_i(\mathbf{L})} \subset \mathcal{J}_{\mathfrak{R}_i(\mathbf{L})}$ be the subset of families of almost-complex structures which for all $r \in \mathfrak{R}_1(\mathbf{L})$ satisfy:   
\begin{align} 
\label{eq:contactshellS1} \lf_1= -dh_1\circ J_r \text{ on a contact shell about } \partial \LS_1, 
\end{align} 
\begin{align} 
    \lf_2= -dh_2 \circ J_r \text{ on a contact shell about } \partial \LS_2. 
\end{align}
We choose $(K^0,J^0) \in \mathcal{K}^{aff}_{\mathfrak{R}_i(\mathbf{L})} \times \mathcal{J}^0_{\mathfrak{R}_i(\mathbf{L})}$ such that for $i=0$, the choice agrees with our previous choice of perturbation data in the definition of the affine Fukaya category.

By the same argument as in the definition of $a_\vec{p}$, we can define a Hochschild cohomology class $a^0_{\vec{p}}$ by counting solutions to \eqref{infinitesimalmodulispace1} (for perturbation data $(K^0,J^0)$) which now satisfy \cref{infinitesimalmodulispace3}. This class is independent of choices. 
As in the proof of Lemma \ref{lem:definingavecp}, all solutions satisfying \eqref{infinitesimalmodulispace3} or \eqref{infinitesimalmodulispace2} are contained in $\mathcal{S}_2$. Because  $\tilde{D}_{\vec{p}} \cap \mathcal{S}_2 = D_{\vec{p}}\cap \mathcal{S}_2$, solutions satisfy \eqref{infinitesimalmodulispace3} if and only if they satisfy \eqref{infinitesimalmodulispace2}. 
Now observing that $\mathcal{J}^{aff}_{\mathfrak{R}_1(\mathbf{L})} \subset \mathcal{J}^0_{\mathfrak{R}_1(\mathbf{L})}$, it is immediate that the class $a^0_{\vec{p}}$ agrees with the previously defined $a_{\vec{p}}$. \vskip 5 pt 

Now we enlarge the family of perturbation data used to define $a_{\vec{p}}^{log}$. 
We define  $\mathcal{J}^1_{\mathfrak{R}_1(\mathbf{L})} \subset \mathcal{J}_{\mathfrak{R}_1(\mathbf{L})}$ to be the space of families of almost-complex structures $J^1$ such that $J_r^1$ satisfies \eqref{eq:contactshellS1} and makes each component $\tilde{D}_\vec{q}$, $\vec{q} \neq \vec{p}$ into an almost-complex submanifold.  

We choose $(K^1,J^1) \in \mathcal{K}^{aff}_{\mathfrak{R}_i(\mathbf{L})} \times \mathcal{J}^1_{\mathfrak{R}_i(\mathbf{L})},$ such that for $i=0$ our choice agrees with the previously chosen $(K^{log},J^{log})$ over the union of $\LS_1$ and a small contact shell. We obtain a Hochschild cocycle $a_{\vec{p}}^1$ whose corresponding Hochschild cohomology class is independent of choices. To see this, note that positivity of intersection still implies that for any disc bubble component $u_{bubble}$ (arising in a compactification of dimension $\leq 1$ moduli spaces), 
\begin{align} \label{eq:randombubblelog2} 
u_{bubble} \cdot \tilde{D}_\vec{q} \geq 0, \vec{q} \neq \vec{p}. 
\end{align} 
As before, Corollary \ref{cor:nospheres} then rules out sphere bubbling in the moduli spaces and it follows that equality holds in \eqref{eq:randombubblelog2}. We again have that if $u_{bubble} \cdot \tilde{D}_\vec{p} \leq 0$, $u_{bubble}(\Sigma_r) \subset \LS_1$, this time by the integrated maximum principle. This implies that the argument of \cite[\S 4.4]{Sheridan2017} again goes through. Because $\mathcal{J}^{log}_{\mathfrak{R}_1(\mathbf{L})} \subset \mathcal{J}^1_{\mathfrak{R}_1(\mathbf{L})}$, we have $a_{\vec{p}}^1 = a_{\vec{p}}^{log}.$   

Finally, we observe that $\mathcal{J}^1_{\mathfrak{R}_1(\mathbf{L})} \cap \mathcal{J}^0_{\mathfrak{R}_1(\mathbf{L})}$ is non-empty; thus we have $a_{\vec{p}}^1=a^0_{\vec{p}}.$ 
Thus $a_{\vec{p}} = a_{\vec{p}}^{log}$ as required.
\end{proof}

\subsection{A-side versality}\label{aversality}

\begin{lem}\label{lem:Adefgen}
 Let $\vec{p} \in \Xizero$, and $\nov_\vec{p} \cdot a_\vec{p}$ be the corresponding first-order deformation class of $\cA_R$. 
Then $a_\vec{p}$ spans $\HH^{2 \oplus \vec{p}}(\cA_0) \cong \Bbbk$. 
\end{lem}
\begin{proof} 
    By \cref{prop: apCOPSS}, $a_\vec{p} = \cC \cO(PSS^{log}(\gamma_\vec{p})).$
The result now follows from  \cref{cor:SH_HH}, together with the fact that $SH^{2 \oplus \vec{p}}(\mathcal{H}_{t,0};\Bbbk) \cong \Bbbk$ is spanned by the class $PSS^{log}(\gamma_\vec{p})$ by Proposition \ref{prop:shgrad}.
\end{proof}
 
For any $\vec{p} \in \Xizero$, we define the `$\vec{p}$ obstruction map'
\begin{align*}
    \Obs_\vec{p}: \HH^{2\oplus 0}(\cA_0) & \to \HH^{3\oplus \vec{p}}(\cA_0)\\
    \Obs_\vec{p}(\alpha) &:= [a_\vec{p},\alpha],
\end{align*}
where $[-,-]$ denotes the Gerstenhaber bracket.

\begin{prop}\label{prop:totalobstruction}
Consider the `total obstruction map'
\begin{align*}
    \Obs: \HH^{2\oplus 0 } (\cA_0) &\to \bigoplus_{\vec{p} \in \Xizero} \HH^{3\oplus \vec{p}}(\cA_0),\\
    \Obs(\alpha) &:= \sum_{\vec{p} \in \Xizero}\Obs_\vec{p}(\alpha).
\end{align*}
We have that $ker(\Obs)=0$. 
\end{prop}
\begin{proof} 
    By \cref{cor:SH_HH} and \cref{prop: apCOPSS}, we translate $\Obs_\vec{p}$ into  $[PSS^{log}(\gamma_\vec{p}), -]$ (and hence $\Obs$ into $\sum_{\vec{p} \in \Xizero} [PSS^{log}(\gamma_\vec{p}), -]$).  Now the result follows immediately from \cref{prop:combinatorialSH} together with the fact that, by our assumption on the characteristic of $\Bbbk$, the vectors $\vec{p}$ span $M^*_{\Bbbk}$. 
\end{proof}

\begin{lem}\label{lem:Avanishing} For any $\vec{p} \notin \Xizero \cup \lbrace 0 \rbrace$, $\HH^{2 \oplus \vec{p}}(\cA_0)=0$.  \end{lem}
\begin{proof} By Corollary \ref{cor:SH_HH}, $\HH^{2 \oplus \vec{p}}(\cA_0) \cong SH^{2 \oplus \vec{p}}(\mathcal{H}_{t,0}).$ By Proposition \ref{prop:shgrad}, $SH^{2 \oplus \vec{p}}(\mathcal{H}_{t,0})=0$ for any $\vec{p} \notin \Xizero \cup \lbrace 0 \rbrace.$ \end{proof}

 \begin{rmk} \label{rmk:quadraticham2} As noted in Remark \ref{rmk:quadraticham1}, the arguments of \cite{Ganatrathesis} which enter into the proof of Corollary \ref{cor:SH_HH} (by way of Proposition \ref{prop:COsurj}) use quadratic Hamiltonians for the construction of symplectic cohomology. On the other hand, the constructions of \cite{ganatra2020symplectic, ganatra2021log} use (directed systems of) Hamiltonians of linear growth as the model for symplectic cohomology. These two models for symplectic cohomology are compared in \cite[Lemma 18.1]{rittertop}. \end{rmk}

\section{Deforming the B-side}\label{sec:Bdef}

In \cref{sec:Bsidedgcategory}, we give an explicit chain-level model of a natural $\Z\oplus M^*$-graded lift of the dg category $\cB_R$ deforming $\cB_0$ defined in the introduction, and in \cref{subsec:vers}, we analyze the associated first-order deformation classes of $\cB_0$. The key outputs of this construction and analysis are \cref{lem:Btopologicallfree} and \cref{lem:Bdefgen}, which collectively show that $\cB_R$ and its first-order deformation classes satisfy conditions needed to subsequently appeal to the versality criterion in \cite[Proposition A.6]{Ganatra2023integrality}.

\subsection{The dg category \texorpdfstring{$\cB^{dg}_R$}{BdgR}}
\label{sec:Bsidedgcategory}
We set $R_B=\Bbbk[[\dm_{\vec{p}}]]_{\vec{p} \in P}$ and regard this as an $M^*$-graded ring by letting $\dm_{\vec{p}}$ have degree $-\vec{p} \in M^*$. We let $\fm_B$ denote the unique maximal ideal of $R_B$. To simplify notation, throughout this section, we will drop the subscript and denote this ideal simply by $\fm$. 

Next, let $Y^*_R := Y^* \times_{Spec(\Bbbk)} Spec(R_B)$, and consider the subscheme
$$X^*_R := \{W_\dm = 0\} \subset Y^*_R,$$
where $W_\dm$ is as in \eqref{eqn:Wb} (with $r_{\vec{p}}$ interpreted as above). 

\begin{lem} \label{lem:XRflat} The variety $X_R^*$ is flat over $R_B$.  \end{lem}
\begin{proof} $X_R^*$ is a hypersurface in a smooth variety and hence Cohen-Macaulay of Krull dimension $\operatorname{dim}(R_B)+n$. The result now follow follows from the ``miracle flatness theorem'' (\cite[\href{https://stacks.math.columbia.edu/tag/00R4}{Tag 00R4}]{stacks-project}
    ) if we note that $\operatorname{dim}(\partial Y^*)=n$.  \end{proof}

Let $\mathbb{G}_m := \mathbb{G}_m(\Bbbk)$. 
Note that we have the toric $M_{\mathbb{G}_m}$-action on $Y^*$, and also an $M_{\mathbb{G}_m}$-action on $Spec(R_B)$, via its $M^*$-grading; these combine to give an $M_{\mathbb{G}_m}$-action on $Y^*_R$. 
We observe that this action preserves $X^*_R$, because $W_\dm$ has degree $0 \in M^*$ (using the fact that $z^{\vec{p}}$ has degree $\vec{p}$, and $r_{\vec{p}}$ has degree $-\vec{p}$).
\begin{itemize} \item Let $Y^*=\cup_{\sigma}\mathcal{V}_{\sigma}$ be the standard toric covering of $Y^*$ by affine spaces $\mathcal{V}_{\sigma} \cong \mathbb{A}^{n+1}_{\Bbbk}$ indexed by the maximal cones of the fan and let $\partial Y^*=\cup_{\sigma}\mathcal{W}_{\sigma}$ be the covering of $\partial Y^*$ given by intersecting $\mathcal{V}_{\sigma}$ with $\partial Y^*$.  \item Let $Y_R^*=\cup_{\sigma}\mathcal{V}_{\sigma,R}$ denote the induced cover of $Y_R^*$ and $X_R^*=\cup_{\sigma}\mathcal{W}_{\sigma,R}$ the induced affine cover of $X_R^*$. \end{itemize}

We next define an $R_B$-linear dg category $\cB^{dg}_R$ whose objects are  $\cO(i)$. The morphism spaces in this category are given by \v{C}ech complexes of internal $\underline{Hom}$ sheaves: 
\begin{align} 
\hom_{\cB^{dg}_{R}}(\cO(i),\cO(j)):= \check{C}^*(\lbrace \mathcal{W}_{\sigma,R} \rbrace,\underline{Hom}(\cO(i),\cO(j))). 
\end{align}
Composition in the category is induced by the shuffle product together with the local tensoring operation on internal $\underline{Hom}$ sheaves (cf. \cite[\S 5a]{Seidel2003}). We choose the equivariant lift of $\cO(i)$ corresponding to the support functions $i \sF_\Delta$ (as in \S \ref{sec:A0_B0}), so that this dg category becomes $\mathbb{Z}\oplus M^*$-graded. We note that the $\Bbbk$-linear dg category $\cB^{dg}_R \otimes_{R_{B}} R_B/\fm$ is a model for $\cB_0.$

\begin{lem} \label{lem:Cechflat} The $R_B$-modules $\hom_{\cB^{dg}_{R}}(\cO(i),\cO(j))$ are flat. \end{lem} 
\begin{proof} 
The \v{C}ech complexes $\check{C}^*(\lbrace \mathcal{W}_{\sigma,R} \rbrace,\underline{Hom}(\mathcal{O}(i),\mathcal{O}(j)))$ are finite direct sums of global functions $\Gamma(\cap_{\sigma_i}\mathcal{W}_{\sigma_i,R}, \mathcal{O}_{\cap_{\sigma_i}\mathcal{W}_{\sigma_i,R}})$ on  finite overlaps $\cap_{\sigma_i}\mathcal{W}_{\sigma_i,R}$ of the open cover. By Lemma \ref{lem:XRflat}, the coordinate rings $\Gamma(\cap_{\sigma_i}\mathcal{W}_{\sigma_i,R},\mathcal{O}_{\cap_{\sigma_i}\mathcal{W}_{\sigma_i,R}})$ are flat over $R_B$ and thus $\check{C}^*(\lbrace \mathcal{W}_{\sigma,R} \rbrace,\underline{Hom}(\mathcal{O}(i),\mathcal{O}(j)))$ is flat over $R_B$.  
\end{proof}

\begin{lem}\label{lem:coh_proj}
The cohomology groups $H^*(\hom_{\cB^{dg}_R}(\mathcal{O}(i),\mathcal{O}(j)))$ are free $R_B$-modules of finite rank, for any $i,j \in \mathbb{Z}$. 
\end{lem}
\begin{proof}
It suffices to prove that the groups $H^q(X^*_R,\mathcal{O}(j-i))$ are all free over $R_B.$ This is a straightforward application of the ``cohomology and base change'' theorem \cite[Theorem 3.12.11]{Hartshorne}; note that theorem applies because $X_R^*$ is flat over $R_B$ by Lemma \ref{lem:XRflat}.  Namely, note that at the fiber over $\fm \in Spec(R_B)$, $H^q(\partial Y^*, \mathcal{O}(j-i))=0$ if $q \notin \lbrace 0,n \rbrace.$ Item (a) of \cite[Theorem 3.12.11]{Hartshorne} then implies that 
\begin{align} \label{eq:vanishingHomology} 
H^q(X^*_R,\mathcal{O}(j-i))=0 \text{ if } q \notin \lbrace 0,n \rbrace. 
\end{align}

For $q=0$, note that by item (b) of \cite[Theorem 3.12.11]{Hartshorne}, the vanishing of $H^1(X^*_R, \mathcal{O}(j-i))$ from \eqref{eq:vanishingHomology} implies that the natural map:
\begin{align} 
\phi^0: H^0(X^*_R, \mathcal{O}(j-i))\otimes_R R/\mathfrak{m} \to H^0(\partial Y^*, \mathcal{O}(j-i)) 
\end{align} 
is surjective. Applying item (b) of \cite[Theorem 3.12.11]{Hartshorne} again, we see that 
$H^0(X^*_R, \mathcal{O}(j-i))$ is free. (Note that condition (i) of part (b) of \cite[Theorem 3.12.11]{Hartshorne} is vacuous in this case.) The same argument applies for $q=n$ if we note that $H^{n+1}(X^*_R, \mathcal{O}(j-i))=0$ (automatically for dimension reasons) and $H^{n-1}(X^*_R, \mathcal{O}(j-i))=0$ by \eqref{eq:vanishingHomology}.  
\end{proof}

To put ourselves in a setting appropriate for formal deformation theory arguments, it will be important to arrange that the morphism spaces in our dg category are complete with respect to the $\fm$-adic filtration. For the remainder of this section, given a graded $R_B$-module $N$, we will let $N^{\wedge}$ denote its graded completion with respect to the $\fm$-adic filtration. 

\begin{defn} We let $\cB_R$ denote the dg category whose objects are given by $\mathcal{O}(i)$ and whose morphism spaces are given by the $M^*$-graded $\fm$-adic completions of the corresponding morphism spaces in $\cB_R^{dg}$. In other words: 
\begin{align} 
    \hom_{\cB_R}(\mathcal{O}(i),\mathcal{O}(j)):= \hom_{\cB_R^{dg}}(\mathcal{O}(i),\mathcal{O}(j))^{\wedge}. 
\end{align} 
\end{defn} 

\begin{lem} \label{lem:formalgaga} 
The natural functor $i^{\wedge}: \cB^{dg}_R \to \cB_R$ is a quasi-isomorphism. 
\end{lem}
\begin{proof} Fix $i,j \in \mathbb{Z}.$ We wish to show that 
    \begin{align}\label{eq:iwedgeij} 
    i^{\wedge}: \hom_{\cB_R^{dg}}(\mathcal{O}(i),\mathcal{O}(j)) \to \hom_{\cB_R}(\mathcal{O}(i),\mathcal{O}(j)) 
\end{align} 
is a quasi-isomorphism.  
By Lemma \ref{lem:coh_proj}, the cohomology groups $H^*(\hom_{\cB^{dg}_R}(\mathcal{O}(i),\mathcal{O}(j)))$ are free $R$-modules of finite rank. Thus, we may define a quasi-isomorphism 
\begin{align} \label{eq:freespectralsequenceargument} 
f:H^*(\hom_{\cB^{dg}_R}(\mathcal{O}(i),\mathcal{O}(j))) \to \hom_{\cB_R^{dg}}(\mathcal{O}(i),\mathcal{O}(j))) 
\end{align} 
by sending each basis element to a chain-level representative (where the LHS is regarded as a chain complex with vanishing differential). Moreover, for any $q \geq 0$, the induced maps  
\begin{align*} 
  \fm^q\otimes_{R_{B}} H^*(\hom_{\cB^{dg}_R}(\mathcal{O}(i),\mathcal{O}(j))) \to \fm^q\otimes_{R_{B}} \hom_{\cB_R^{dg}}(\mathcal{O}(i),\mathcal{O}(j))
\end{align*} given by tensoring $f$ with $\fm^q$ are also quasi-isomorphisms. This follows because $\hom_{\cB_R^{dg}}(\mathcal{O}(i),\mathcal{O}(j))$ is a bounded complex of flat modules by Lemma \ref{lem:Cechflat} and $H^*(\hom_{\cB^{dg}_R}(\mathcal{O}(i),\mathcal{O}(j)))$ is free of finite rank. Using flatness again to identify $\fm^q\otimes_{R_{B}} \hom_{\cB_R^{dg}}(\mathcal{O}(i),\mathcal{O}(j)) \cong  \fm^q\hom_{\cB_R^{dg}}(\mathcal{O}(i),\mathcal{O}(j))$, it follows that $f$ is a filtered quasi-isomorphism, namely, for any $q \geq 0$, the maps: 
\begin{align} 
  f_q:  \fm^qH^*(\hom_{\cB^{dg}_R}(\mathcal{O}(i),\mathcal{O}(j))) \to \fm^q\hom_{\cB_R^{dg}}(\mathcal{O}(i),\mathcal{O}(j))
\end{align}
are quasi-isomorphisms.

Now to prove the lemma, it suffices to prove that the composition $i^\wedge \circ f$ is a quasi-isomorphism. The $E_1$-pages of the spectral sequences induced by the $\fm$-adic filtrations on these chain complexes are all given by 
\begin{align} 
\bigoplus_q E_1^{*-q,q} \cong H^*(\hom_{\cB_0}(\mathcal{O}(i),\mathcal{O}(j)))\otimes Gr^q_{\fm}(R_B), 
\end{align}
and the induced maps between them are isomorphisms because $f$ is a filtered quasi-isomorphism.  
As the $\fm$-adic filtrations on the source and target of $i^\wedge \circ f$ are complete and exhaustive (the former because it is free of finite rank, the latter by construction), the map is a quasi-isomorphism, as required. 
\end{proof}

\begin{rmk} [Cf. Section 5b in \cite{Seidel2003}] Lemma \ref{lem:formalgaga} can be viewed as a special case of the formal GAGA theorem \cite[Section 4]{EGA}.  \end{rmk}

\begin{defn}\label{def:top_free_mod}
    Let $N$ be a graded $R_B$-module. 
    We say that $N$ is \emph{topologically free} if there exists an isomorphism $N \simeq ((N/\fm N) \otimes R_B)^{\wedge}$ lifting the natural map $N \to N/\fm N$. 
\end{defn}

\begin{defn}[Cf. Remark 2.2 in \cite{Seidel2003}]\label{def:top_cat}
     A \emph{topological} $R_B$-linear dg category is a dg category whose hom-spaces are topologically free $R_B$-modules.
\end{defn}

\begin{lem} \label{lem:Btopologicallfree} 
    The category $\cB_R$ is a topological $R_B$-linear dg category. 
\end{lem}
\begin{proof} 
As in Lemma \ref{lem:Cechflat}, the \v{C}ech complexes $\check{C}^*(\lbrace \mathcal{W}_{\sigma,R} \rbrace,\underline{Hom}(\mathcal{O}(i),\mathcal{O}(j)))$ are finite direct sums of global functions on finite overlaps of the open cover. Such overlaps are indexed by (not necessarily maximal) cones $\sigma$ of $\Sigma^*$; we denote the overlap indexed by $\sigma$ by $\mathcal{W}_{\sigma,R}$, sitting inside the corresponding open subset $\mathcal{V}_{\sigma,R} \subset Y^*_R$. 
It then suffices to prove that $(\Gamma(\mathcal{W}_{\sigma,R}, \mathcal{O}_{\mathcal{W}_{\sigma,R}}))^{\wedge}$ is topologically free. 

We may choose a basis of $M_\R$ such that the first $j$ basis vectors form a basis for $\sigma$ (because $\Sigma^*$ is smooth). 
This induces an isomorphism $\mathcal{V}_{\sigma,R} \cong (\mathbb{A}_\Bbbk)^j \times (\mathbb{G}_{m}(\Bbbk))^{n-j} \times Spec(R_B)$. We may choose coordinates on $\mathcal{V}_{\sigma,R}$ so that its ring of functions is given by $R_B[\xi_1,\cdots,\xi_j, \xi_{j+1}^{\pm},\cdots,\xi_n^{\pm}]$ and so that the defining function $W_b$ for $\mathcal{W}_{\sigma,R}$ is given by: $W_b =-\xi_1\cdots \xi_j + \sum_{\alpha} b_\alpha \xi^{\alpha}$ where $\alpha \in (\mathbb{Z}^{\geq 0})^j \times (\mathbb{Z})^{n-j}$ and  $b_\alpha \in \fm.$ 

Now take $\mathcal{W}_\sigma$ to be the hypersurface in $\mathcal{V}_\sigma=Spec(\Bbbk[\xi_1,\cdots,\xi_j, \xi_{j+1}^{\pm},\cdots,\xi_n^{\pm}])$ cut out by $-\xi_1\cdots \xi_j=0.$ Let $S \subset (\mathbb{Z}^{\geq 0})^j \times (\mathbb{Z})^{n-j}$ denote the set of vectors such that one of the first $j$ entries is zero. The set $\lbrace \xi^{\alpha} \rbrace_{\alpha \in S}$ defines a $\Bbbk$-basis for the ring of functions on $\mathcal{W}_\sigma$. For the remainder of the argument, we set $\mathcal{Q}=R_B[\xi_1,\cdots,\xi_j, \xi_{j+1}^{\pm},\cdots,\xi_n^{\pm}]$. Consider the natural map of $R_B$-modules: 
\begin{align*}
\bigoplus_{\alpha \in S} R_B \cdot \xi^{\alpha} \to \frac{\mathcal{Q}}{(W_b)} \cong \Gamma(\mathcal{W}_{\sigma,R}, \mathcal{O}_{\mathcal{W}_{\sigma,R}}) 
\end{align*} and let 
\begin{align}\label{eq:completedtopbasis} 
\left(\bigoplus_{\alpha \in S} R_B \cdot \xi^{\alpha}\right)^{\wedge} \to \left(\frac{\mathcal{Q}}{(W_b)}\right)^{\wedge} 
\end{align} 
denote its graded completion. 

 The map \eqref{eq:completedtopbasis} is surjective by Nakayama's lemma for complete modules \cite[\href{https://stacks.math.columbia.edu/tag/00M9}{Tag 00M9 (i)}]{stacks-project}.\footnote{The reference \cite[\href{https://stacks.math.columbia.edu/tag/00M9}{Tag 00M9 (i)}]{stacks-project} is stated for standard completions as opposed to graded completions. However, the proof carries over without change to graded completions.} To prove this map is injective, suppose to the contrary that $f_0$ is a non-zero $\fm$-adically convergent series which is in the kernel of \eqref{eq:completedtopbasis}. Note that the sequence 
 \begin{align} 
 0 \to \mathcal{Q}^{\wedge} \xrightarrow{W_b \cdot} \mathcal{Q}^{\wedge} \to \left(\frac{\mathcal{Q}}{(W_b)}\right)^{\wedge} \to 0 
 \end{align} 
 is exact by \cite[\href{https://stacks.math.columbia.edu/tag/00M9}{Tag 00M9 (iii)}]{stacks-project} and Lemma \ref{lem:XRflat}. This means that   
\begin{align} \label{eq:factoring} 
f_0=W_b f_1 \in \mathcal{Q}^{\wedge} 
\end{align}
for some $f_1 \in \mathcal{Q}^{\wedge}$. As $\mathcal{Q}^\wedge$ is complete and $f_1 \neq 0$, there exists $r \ge 0$ such that $f_1 \in \fm^r\mathcal{Q}^{\wedge} \setminus \fm^{r+1} \mathcal{Q}^{\wedge}.$ This means we can expand 
$$f_1=\sum_\alpha c_\alpha \xi^\alpha+  O(\fm^{r+1}) $$ 
with $c_\alpha \notin \fm^{r+1}$ for some $\alpha$.  Equation \eqref{eq:factoring} then implies that 
\begin{align} \label{eq:factoring2} 
f_0= \sum_\alpha -c_\alpha (\xi_1\cdots\xi_j\cdot \xi^\alpha) +  O(\fm^{r+1}).
\end{align}  
But this is impossible since the right-hand side of \eqref{eq:factoring2} involves monomials $\xi^\alpha$ with $\alpha \notin S.$ 
\end{proof}

\subsection{B-side versality}\label{subsec:vers}

 Throughout this section, we fix $\vec{p} \in P$. 
 Set  $\Bbbk_\eps := \Bbbk[\eps]/\eps^2$ and consider the homomorphism $j_{\vec{p}}:R_B \to\Bbbk_\eps$ defined by letting $j_{\vec{p}}(r_\vec{p}) = \epsilon$,  $j_{\vec{p}}(r_\vec{q})=0$ for $\vec{q} \neq \vec{p}.$ 
The homomorphism $j_{\vec{p}}$ is  $\mathbb{Z}\oplus M^*$-graded if $\eps$ is given degree $(0,-\vec{p}) \in \Z \oplus M^*$. Consider the base change of $\cB_R$ along $j_{\vec{p}}$, 
\begin{align} \label{eq:cBepsilon} 
   \cB_\eps:=\cB_R\otimes_{R_{B}}\Bbbk_\eps. 
\end{align} 
The category $\cB_\eps$ is a $\Z \oplus M^*$-graded infinitesimal deformation of $\cB_0$ and hence gives rise to a first-order deformation class $b_{\vec{p}} \in \HH^{2\oplus \vec{p}}(\cB_0)$. The main result of this section is the following:

\begin{lem}\label{lem:Bvers}
The class $b_{\vec{p}}$ is non-zero, for each $\vec{p} \in \Xizero$.
\end{lem}

\begin{rmk}
    The proof of \cref{lem:Bvers} uses our assumption that $\Sigma^*$ is smooth (see the proof of \cref{lem:points}). However, the proof works just as well if all $2$-dimensional cones of $\Sigma^*$ are smooth (i.e., if $Y^*$ is smooth away from its strata of codimension $\ge 3$). 
    \end{rmk}

Before giving the proof of Lemma \ref{lem:Bvers}, we note the following consequence:  

\begin{lem}\label{lem:Bdefgen}
The class $b_{\vec{p}}$ spans $\HH^{2 \oplus \vec{p}}(\cB_0)$, for any $\vec{p} \in \Xizero$.
\end{lem}
\begin{proof}
We have
$$\HH^{2 \oplus \vec{p}}(\cB_0) \simeq \HH^{2 \oplus \vec{p}}(\cA_0) \simeq SH^{2 \oplus \vec{p}}(F;\Bbbk),$$
and we know that the latter space has rank one by Proposition \ref{prop:shgrad}. 
As $b_{\vec{p}} \neq 0$ by Lemma \ref{lem:Bvers}, it must span $\HH^{2 \oplus \vec{p}}(\cB_0)$. 
\end{proof}

The key step in the proof of Lemma \ref{lem:Bvers} will be to show that the corresponding deformation of homogeneous coordinate rings is non-trivial. We start by recalling the homogeneous coordinate ring of $(Y^*,\cL_{\Delta^*})$. It is the $\Z \oplus M^*$-graded algebra
$$S_{\Delta^*} := \Bbbk[C(\Delta^*) \cap \Z \oplus M^*],$$
where $C(\Delta^*) := \R_{\ge 0} \cdot \{\{1\} \times \Delta^*\} \subset \R \times M^*_\R$. 
The element $z^{(j,n)}$ has degree $(j,n)$. Let $\vec{p} \in \Xinonzero$, and let $X^*_{\vec{p}}$ be the $\Bbbk_\eps$-scheme with defining equation $\{z^0 = \eps z^{\vec{p}}\}$; this is an infinitesimal deformation of the hypersurface $X^*_0 = \{z^0 = 0\}$. It has $\Z \oplus M^*$-graded homogeneous coordinate ring $S^\eps := \left(S_{\Delta^*} \otimes_\Bbbk \Bbbk_\eps \right)/(z^0 - \eps z^{\vec{p}})$, where as above we give $\eps$ degree $(0,-\vec{p}) \in \Z \oplus M^*$. 

We can describe $S^\eps$ explicitly as follows. 
It is generated, as a $\Bbbk_\eps$-module, by monomials $z^{\tilde{n}}$ for $\tilde{n} \in \partial C(\Delta^*) \cap \Z \oplus M^*$. 
(Indeed, it is easy to use the relation $z^0 = \eps z^{\vec{p}}$ to write any other monomial in terms of these.) 
It is straightforward to check that the multiplication is given by
$$ z^{\tilde{n}_1} \cdot z^{\tilde{n}_2} = \delta z^{\tilde{n}_1+\tilde{n}_2} + \eps \cdot  \delta z^{\tilde{n}_1+\tilde{n}_2+(0,\vec{p})},$$
where we use the abbreviation
$$\delta z^{\tilde{n}} := \left\{ 
    \begin{array}{rl} 
        z^{\tilde{n}} & \text{if $\tilde{n} \in \partial C(\Delta^*)$} \\
		0 & \text{otherwise}.
    \end{array}\right. $$

Now, consider the first-order deformation $\cB_\eps$ of $\cB_0$ corresponding to $b_{\vec{p}}$. 
It is $M^*$-graded and linear over $\Bbbk_\eps := \Bbbk[\eps]/\eps^2$, where $\eps$ has degree $-\vec{p} \in M^*$.
Note that
$$\Hom^0_{\cB_\eps}(\cL^i,\cL^j)_n = S^\eps_{(j-i,n)},$$
for $n \in M^*$. 
We will denote the generator of $\Hom^0_{\cB_\eps}(\cL^i,\cL^j)$ corresponding to $z^{(j-i,n)}$ by $z_i^{(j-i,n)}$. 
\begin{lem}\label{lem:Sepsnontriv}
If $\vec{p} \in \Xizero$, then $H^0(\cB_\eps)$ is non-trivial, as a $\Z \oplus M^*$-graded deformation of $H^0(\cB_0)$. 
\end{lem}
\begin{proof}
Suppose $\cB_\eps$ were a trivial graded deformation of $\cB_0$. 
Then there would exist a graded isomorphism $F_\eps: H^0(\cB_\eps) \xrightarrow{\sim} H^0(\cB_0)^{triv}$, with $F_0 = \id$. 
As each $M^*$-graded piece of a morphism space in $H^0(\cB_0)$ has dimension $\le 1$, and $\eps$ has degree $(0,-\vec{p})$, this isomorphism must take the form
$$F_\eps(z_i^{\tilde{n}}) = z_i^{\tilde{n}} + \eps a(i,\tilde{n}) z_i^{\tilde{n} + (0,\vec{p})},$$
where $a(i,\tilde{n}) \in \Bbbk$ vanishes if $\tilde{n} + (0,\vec{p}) \notin \partial C(\Delta^*)$. 

By Lemma \ref{lem:points} below, there exist points $\tilde{n}_1$, $\tilde{n}_2$ in $\partial C(\Delta^*) \cap \Z \oplus M^*$ such that $\tilde n_j + (0,\vec{p}) \notin \partial C(\Delta^*)$, and in particular $F^\eps(z_i^{\tilde n_j}) = z_i^{\tilde n_j}$, for both $j$; $\tilde n_1 + \tilde n_2 \notin \partial C(\Delta^*)$, and in particular $\delta z^{\tilde n_1 + \tilde n_2} = 0$; and $\tilde n_1+\tilde n_2 + (0,\vec{p}) \in \partial C(\Delta^*)$, and in particular $\delta z^{\tilde n_1+\tilde n_2+(0,\vec{p})} = z^{\tilde n_1+\tilde n_2+(0,\vec{p})}$. 
Thus, from
$$F^\eps(z_0^{\tilde{n}_1}) \cdot_{H^0(\cB_0)^{triv}} F^\eps(z_i^{\tilde{n}_2}) = F^\eps(z_0^{\tilde{n}_1} \cdot_{H^0(\cB_\eps)} z_i^{\tilde{n}_2}),$$
we obtain
$$z^{\tilde{n}_1} \cdot_{H^0(\cB_0)^{triv}} z^{\tilde{n}_2} = F^\eps\left(\delta z_0^{\tilde{n}_1+\tilde{n}_2} + \eps \cdot \delta z_0^{\tilde{n}_1+\tilde{n}_2+(0,\vec{p})}\right),$$
and hence (using $F^0 = \id$) that
$$0 = \eps z_0^{\tilde n_1+\tilde n_2+(0,\vec{p})},$$
a contradiction.
\end{proof}

\begin{lem}\label{lem:points}
If $\vec{p} \in \Xizero$, then there exist points $\tilde{n}_1$, $\tilde{n}_2$ in $\partial C(\Delta^*) \cap \Z \oplus M^*$, such that
\begin{itemize}
    \item $\tilde n_i + (0,\vec{p}) \notin C(\Delta^*)$, for both $i$;
    \item $\tilde n_1 + \tilde n_2 \notin \partial C(\Delta^*)$;
    \item $\tilde{n}_1 + \tilde{n}_2+ (0,\vec{p}) \in \partial C(\Delta^*)$.\end{itemize}
\end{lem}
\begin{proof}
Because $\vec{p} \in \Xizero$, it lies on a facet $F$ of $\Delta^*$ having codimension $2$ (it may lie on the boundary of $F$). 
Let us choose a rational point $q$ in the interior of $F$. 
Consider the convex cone $C_{q} := \R_{\ge 0} \cdot (\Delta^* - q)$, which contains $-\vec{p}$ in its interior. 
 
Because $F$ has codimension $2$, $C_q$ is an intersection of half spaces, $\{m_1 \ge 0\} \cap \{m_2 \ge 0\}$ for $m_i \in M$ linearly independent; because $\Sigma^*$ is smooth, we may choose $m_1, m_2 \in M$ to be a subset of a basis. 
Therefore, any element in $C_q \cap M^*$ can be written as a sum of two vectors in $M^*$, one pointing along each boundary facet of $C_q$. In particular, we have $-\vec{p} = v_1 + v_2$. 

There exists $k \in \Z_{>0}$ such that $kq \in M^*$. 
By taking $K$ to be a sufficiently large multiple of $k$ we may then arrange that $q+\frac{1}{K} v_i \in \partial \Delta^*$ for each $i$, and $q + \frac{1}{2K}(v_1 + v_2)$ lies in the interior of $\Delta^*$.  
We now set $\tilde{n}_i = (K,Kq+v_i)$. 

We now check that $\tilde n_1$ and $\tilde n_2$ have the desired properties:
\begin{itemize}
    \item $\tilde{n}_i \in \Z \oplus M^*$: this follows because we chose $K$ so that $Kq \in M^*$.
    \item $\tilde n_i \in \partial C(\Delta^*)$: this follows because $q + \frac{1}{K} v_i \in \partial \Delta^*$.
    \item $\tilde{n}_i + (0,\vec{p}) \notin C(\Delta^*)$: this follows because $-\vec{p}$ points from $v_i$ into the interior of $C_q$, so $v_i + \vec{p} \notin C_q$; and hence, $q + \frac{1}{K}(v_i+\vec{p}) \notin \Delta^*$.
    \item $\tilde n_1 + \tilde n_2 \notin \partial C(\Delta^*)$: this follows because $q + \frac{1}{2K}(v_1+v_2)$ lies in the interior of $\Delta^*$, and in particular, does not lie in $\partial \Delta^*$.
    \item $\tilde{n}_1 + \tilde{n}_2+ (0,\vec{p}) \in \partial C(\Delta^*)$: this follows because $v_1+v_2 = -\vec{p}$, and $q \in \partial \Delta^*$.
\end{itemize}
\end{proof}

\begin{rmk}
Lemma \ref{lem:Sepsnontriv} is false if $\vec{p} \notin \Xizero$: indeed, the points $\vec{p} \in \Xinonzero\setminus ( \{0\} \cup \Xizero)$ are known as `roots' of $Y^*$ and give rise to one-parameter families of automorphisms of the toric variety which trivialize the corresponding deformations of the hypersurface (see \cite[Section 6.1]{coxkatz}).
\end{rmk}

We now observe that the deformation $\cB_\eps$ defined in \eqref{eq:cBepsilon} is a graded uncurved first-order deformation of the $A_\infty$ category $\cB_0$. 
It satisfies $\Hom^1_{\cB_0}(L,L) = \Hom^2_{\cB_0}(L,L) = 0$ for all objects $L$ of $\cB_0$, and furthermore, $\Hom^*_{\cB_0}(K,L)$ is concentrated in two degrees which differ by $n \ge 3$, for all pairs of objects $K,L$; in particular, it is concentrated in degrees, no two of which differ by $1$. 
This ensures that $H^*(\cB_\eps) \simeq H^*(\cB_0) \otimes \Bbbk_\eps$ is a flat deformation of $H^*(\cB_0)$, by considering the spectral sequence associated to the $\eps$-adic filtration on $(\cB_\eps,\mu^1_\eps)$: the $E_1$ page is $H^*(\cB_0) \otimes \Bbbk_\eps$, and the differential on the $E_1$ page is necessarily $0$. 
\begin{lem}\label{lem:Hnontriv}
Suppose that $\cB_\eps$ is any graded uncurved first-order deformation of $\cB_0$, $\Hom^*_{\cB_0}(L,L) = 0$ for $* = 1,2$, and $\Hom^*_{\cB_0}(K,L)$ is concentrated in degrees, no two of which differ by $1$.
If $H^*(\cB_\eps)$ is non-trivial as a graded first-order deformation of $H^*(\cB_0)$, then the corresponding Maurer--Cartan element $KS(\partial/\partial \eps) \in \HH^{2\oplus -\operatorname{deg}(\eps)}(\cB_0)$ is non-zero.
\end{lem}
\begin{proof}
Our first step is to show that we may assume that $\cB_\eps$ is minimal. 
We do this by replacing it with a minimal model $\bar{\cB}_\eps$; we need to prove that this exists, as the base ring $\Bbbk_\eps$ is not a field.

Let $\beta \in CC^2(\cB_0)$ be the Hochschild cocycle corresponding to the deformation $\cB_\eps$; we want to show that $[\beta] \neq 0$. 
Let $\bar{\cB}_0$ be a minimal model for $\cB_0$ (which exists because the base ring $\Bbbk$ is a field), and let $\bar{\beta}$ be a Hochschild cocycle representing the image of $[\beta]$ under the isomorphism $\HH^*(\cB_0) \simeq \HH^*(\bar{\cB}_0)$. 
Let $\bar{\cB}_\eps$ denote the first-order deformation of $\bar{\cB}_0$ corresponding to $\bar{\beta}$. 
Note that $\bar{\beta}^0_L = 0$ for all $L$, because $\hom^2_{\bar{\cB}_0}(L,L) = \Hom^2_{\cB_0}(L,L) = 0$, so this deformation is uncurved. 
Also note that $\bar{\beta}^1 = 0$ by our assumption that $\hom^*_{\bar{\cB}_0}(L,L) = \Hom^*_{\cB_0}(L,L)$ is concentrated in degrees, no two of which differ by $1$. 
Thus $\bar{\beta}$ has length $\ge 2$ (recall from Appendix \ref{subsec:Ainf_term} that the `length' of a Hochschild cochain is the number of inputs). 

It is easy to show from the definitions that the $A_\infty$ quasi-isomorphism $F_0:\cB_0 \xrightarrow{\sim} \bar{\cB}_0$ extends to a (possibly curved) $A_\infty$ homomorphism $F_\eps:\cB_\eps \to \bar{\cB}_\eps$; the curvature $F^0_\eps$ lies in $hom^1_{\bar{\cB}_0}(L,L) = \Hom^1_{\cB_0}(L,L) = 0$ by hypothesis, so in fact $F_\eps$ is uncurved. 
Thus it induces an isomorphism $H^*(\cB_\eps) \simeq H^*(\bar{\cB}_\eps)$. 
Thus we have shown that $\bar{\cB}_\eps$ is a deformation of $\bar{\cB}_0$ satisfying all the hypotheses we imposed on $\cB_\eps$; so it suffices to show that $[\bar{\beta}] \neq 0$.

Suppose, to the contrary, that $\bar{\beta} = \partial \gamma$ for some $\gamma \in CC^1(\bar{\cB}_0)$. 
Recall that we associate to a Hochschild cochain $\varphi^s$ of length $s \ge 0$, and cohomological degree $t$, the bidegree $(s+t,s) \in \Z \oplus \Z$. 
We expand $\gamma = \sum_{s \ge 0} \gamma^s$ where $\gamma^s$ is the part of length $s$; and we expand the Hochschild differential as $\partial = \sum_{s \ge 0} \partial^s$, where $\partial^s$ is the part coming from $\mu^s$, which changes bidegree by adding $(1,s-1)$. 

Note that $\gamma^0 = 0$, because $hom^1_{\bar{\cB}_0}(L,L) = 0$; and $\partial^s = 0$ for $s \le 1$ because $\mu^s = 0$. 
Thus, the length two part of the equation $\bar{\beta} = \partial \gamma$ says that $\bar{\beta}^2 = \partial^2 \gamma^1$. 
It is well-known that this is equivalent to the associative algebra $H^*(\bar{\cB}_\eps)$ being a trivial deformation of $H^*(\bar{\cB}_0)$, which contradicts our hypothesis; thus $[\bar{\beta}] \neq 0$ as required.
\end{proof}

\begin{proof}[Proof of Lemma \ref{lem:Bvers}] 
Follows by combining \cref{lem:Sepsnontriv} and \cref{lem:Hnontriv}.
\end{proof}

\section{Proof of the main result}\label{sec:main_result}

\subsection{Applying versality}

Consider the quasi-isomorphism $\mir0: \cB_0 \to \cA_0$ from \eqref{eq:Pic_sec_cd2}. 
The upshot of Sections \ref{sec:Adef} and \ref{sec:Bdef} is that we have prepared the way to apply Proposition A.6 from \cite{Ganatra2023integrality}, to construct an extension of $\mir0$ to a curved filtered quasi-isomorphism $F: \cB_R \to \cA_R$. 
In the first place, we have that $\cA_R$ is topologically free over $R_A = \Bbbk[[\NE_A]]$ (where $\NE_A$ is `nice' by construction), because its $\operatorname{hom}$-spaces consist of finite rank free modules; and $\cB_R$ is topologically free over $R_B = \Bbbk[[(\Z_{\ge 0})^P]]$ by Lemma \ref{lem:Btopologicallfree}.

\begin{prop}\label{prop:apply_vers}  
There exists an $M^*$-graded $\Bbbk$-algebra homomorphism $\Psi^*: R_B \to R_A$, sending $\nov_p \mapsto \nov_p \cdot \phi_p$ for some units $\phi_p \in R_A$ of degree $0 \in M^*$, together with a curved filtered quasi-isomorphism $F: \Psi^*\cB_R \to \cA_R$. 
\end{prop}
\begin{proof} 
This follows from the versality criterion \cite[Proposition A.6]{Ganatra2023integrality}, whose hypotheses we now verify. Hypothesis (1) of the versality criterion is verified on the A-side by Lemma \ref{lem:Adefgen}, and on the B-side by Lemma \ref{lem:Bdefgen}. Hypothesis (2) is verified by Proposition \ref{prop:totalobstruction}, and Hypothesis (3) by Lemma \ref{lem:Avanishing}.
\end{proof}

\subsection{Specialization to Novikov field}

We consider the $\Bbbk$-algebra homomorphism
\begin{align*}
a(\kaehl)^*: R_A &\to \Lambda \qquad\text{sending}\\
\nov^\beta &\mapsto \nov^{\kaehl(\beta)}
\end{align*}
associated to the class $\kaehl$ in the interior of $\Nef$ (defined in \cite[Lemma 5.20]{Sheridan2017}). 
Let $\cA_\kaehl := \cA_R \otimes_{R_{A}} \Lambda$, where $\Lambda$ is regarded as an $R_A$-algebra via $a(\kaehl)^*$. 
It is a curved filtered $A_\infty$ category, where the filtration is the $\nov$-adic one, so we may define $\cA_\kaehl^\bc$ as in \cite[Section 2.2]{perutz2022constructing}. 

On the other hand, consider the homomorphism $b(\kaehl)^* := a(\kaehl)^* \circ \Psi^*$. 
Let $\cB_\kaehl := \cB_R \otimes_{R_{B}} \Lambda$, where now $\Lambda$ is regarded as an $R_B$-algebra via $b(\kaehl)^*$. 
By Proposition \ref{prop:apply_vers}, we have a curved filtered quasi-isomorphism 
$$F_\kaehl: \cB_\kaehl \to \cA_\kaehl,$$
and therefore a quasi-equivalence
$$F_\kaehl^\bc: \cB_\kaehl^\bc \to \cA_\kaehl^\bc$$
by \cite[Theorem 2.12]{perutz2022constructing}. 

Note that $\cB_R$, and hence $\cB_\kaehl$, are uncurved and cohomologically unital by construction; so we have an embedding $\cB_\kaehl \hookrightarrow \cB_\kaehl^\bc$ by equipping each object with the $0$ bounding cochain. 
Thus we have a quasi-embedding
\begin{equation}
\cB_\kaehl \hookrightarrow \cA_\kaehl^\bc.
\end{equation}

\subsection{Smoothness of the mirror}

\begin{lem}\label{lem:smooth}
The $\Lambda$-variety $X^*_{b(\kaehl)}$ is smooth.
\end{lem}
\begin{proof}
The result follows from \cite[Proposition B.1]{Ganatra2023integrality}, whose hypotheses we now verify. 
Firstly, the coefficient field $\Lambda$ is of the necessary type, and the fan $\Sigma^*$ of the toric variety $Y^*$ is smooth. 
We recall that $X^*_{b(\kaehl)} \subset Y^*$ is the vanishing locus of the function 
$$\sum_{\vec{p} \in \Xinonzero} c_{\vec{p}} \cdot z^{\vec{p}},$$
where 
$$ c_{\vec{p}} = \left\{\begin{array}{ll}
										-1 & \vec{p}=0;\\
										b(\kaehl)^*(\nov_{\vec{p}}) & \vec{p} \in \Xizero;\\
										0 & \text{otherwise.}
										\end{array}\right.$$
As a consequence, we have
$$ \val(c_{\vec{p}}) = \left\{\begin{array}{ll}
										0 & \vec{p}=0;\\
										\kaehl_{\vec{p}} & \vec{p} \in \Xizero\qquad\text{by Lemma \ref{lem:bplambda} below;}\\
										\infty & \text{otherwise.}
										\end{array}\right.$$ 
Therefore, the smallest concave function $\psi:\Delta^* \to \R$ such that $\psi(\vec{p}) \ge -\val(c_{\vec{p}})$ is precisely $-\psi_\kaehl$, where $\psi_\kaehl$ is the function from the introduction. 
The induced subdivision of $\Delta^*$ into domains of linearity of $\psi_\kaehl$ is a decomposition by simplices, by our assumption that $\Sigma_\kaehl$ is simplicial; and the affine volumes of the simplices are not divisible by $\charac(\Lambda) = \charac(\Bbbk)$, by the condition on characteristic. 
Thus the hypotheses of \cite[Proposition B.1]{Ganatra2023integrality} are satisfied, so $X^*_{b(\kaehl)}$ is smooth. 
\end{proof}

\begin{lem}\label{lem:bplambda}
We have $\val\left(b(\kaehl)^*(\nov_{\vec{p}})\right) = \kaehl_{\vec{p}}$.
\end{lem}
\begin{proof} 
Note that $a(\kaehl)^*(\nov_{\vec{p}}) = \nov^{\kaehl_{\vec{p}}}$ by definition. By Proposition \ref{prop:apply_vers}, the homomorphism $\Psi^*$ sends $\nov_{\vec{p}} \mapsto \nov_{\vec{p}} \cdot \phi_{\vec{p}}$ for some units $\phi_{\vec{p}} \in R_A.$ In particular, we have $\phi_{\vec{p}} \neq 0 \in R_A/\fm_A$, so $\val(a(\kaehl)^*(\phi_{\vec{p}})) = 0$. Thus we have
\begin{align*}
    \val\left(b(\kaehl)^*(\nov_{\vec{p}})\right) &= \val\left(a(\kaehl)^*(\Psi^*(\nov_{\vec{p}}))\right) \\
    &= \val\left(a(\kaehl)^*(\nov_{\vec{p}})\right) + \val\left(a(\kaehl)^*(\phi_{\vec{p}})\right) \\
    &= \kaehl_{\vec{p}}
\end{align*}
as required.
\end{proof}
\subsection{Split-generation}\label{sec:splitgeneration}
Let $X^*_{b(\kaehl)}:=\cup_{\sigma}\mathcal{W}_{\sigma,b(\kaehl)}$ be the affine cover of $X^*_{b(\kaehl)}$ given by base-changing the affine cover of $X^*_{R}$ introduced in \S \ref{sec:Bsidedgcategory}.

 We introduce a dg category $\check{\mathcal{S}}(X^*_{b(\kaehl)})$ whose objects are algebraic vector bundles and whose morphism spaces are given by \v{C}ech complexes of internal $\underline{Hom}$ sheaves: 
\begin{align} 
\hom_{\check{\mathcal{S}}(X_{b(\kaehl)}^*)}(E_0,E_1):= \check{C}^*\left(\lbrace \mathcal{W}_{\sigma,{b(\kaehl)}} \rbrace,\underline{Hom}(E_0,E_1)\right). 
\end{align}

\begin{defn} 
\label{defn:cechmodelperf} We define $\operatorname{Perf}_{dg}(X_{b(\kaehl)}^*):=\Perf \check{\mathcal{S}}(X_{b(\kaehl)}^*).$  
\end{defn}

\begin{rmk}\label{rmk:perfuniqueness} 
    It is standard (\cite[Lemma 5.1]{Seidel2003}) that $\operatorname{Perf}_{dg}(X_{b(\kaehl)}^*)$ gives a dg-enhancement of the triangulated category of perfect complexes on $X^*_{b(\kaehl)}$. We also note that by the results of \cite{orlovlunts2009, stellariuniqueness}, all dg-enhancements of this triangulated category are quasi-equivalent. 
\end{rmk}

\begin{proof}[Proof of Theorem \ref{main:higherdim}]
By construction, the category $\cB^{dg}_\kaehl := \cB^{dg}_R \otimes_{R_{B}} \Lambda$ is a strict full subcategory of $\operatorname{Perf}_{dg}(X^*_{b(\kaehl)})$ containing the objects $\cO(i)$. Moreover, these objects split generate $\operatorname{Perf}_{dg}(X^*_{b(\kaehl)})$ by \cite[Theorem 4]{orlov2009remarks}. 
The quasi-isomorphism $\cB^{dg}_R \simeq \cB_R$ induces a quasi-isomorphism $\cB^{dg}_\kaehl \simeq \cB_\kaehl$. 
As $X^*_{b(\kaehl)}$ is smooth, $\operatorname{Perf}_{dg}(X^*_{b(\kaehl)}) \cong D^b_{dg}Coh(X^*_{b(\kaehl)})$. 
Thus we have constructed a quasi-embedding
$$D^b_{dg}Coh(X^*_{b(\kaehl)}) \simeq \perf(\cB_\kaehl) \hookrightarrow \perf(\cA_\kaehl^\bc) \hookrightarrow \perf \fuk(X_{t,\kaehl},D,\kaehl;\Lambda)^\bc;$$
this completes the proof of \cref{it:HMS_embed}.

To prove \cref{it:HMS_generate}, we first observe that because $X^*_{b(\kaehl)}$ is smooth, $\cB_\kaehl$ is homologically smooth by \cite[Corollary 4.4]{Lunts2016}. 
By our assumption that $\fuk(X_{t,\kaehl},D,\kaehl;\Lambda)^\bc$ satisfies the hypotheses enumerated in \cite[Section 2.5]{sheridan2021homological}, we may apply automatic split-generation \cite{Ganatra2016, Sanda2021},
which shows that the image of $\cB_\kaehl$ in $\fuk(X_{t,\kaehl},D,\kaehl;\Lambda)^\bc$ split-generates, because $\cB_\kaehl$ is homologically smooth. 
The result follows.
\end{proof}

\appendix
\section{Adjoints to sectorial inclusions for arclike Lagrangians and the cap functor} \label{sec:adjoints}

Let $i:Y \to X$ be an inclusion of Liouville sectors and $i_*: \cW(Y) \to \cW(X)$ the corresponding functor between partially wrapped Fukaya categories defined in \cite[Section 3.6]{ganatra2020covariantly}. The functor $i_*$ induces a pull-back functor on module categories $i^*:\operatorname{Mod}\mathcal W(X)\to \operatorname{Mod}\mathcal W(Y)$, the ``large'' right adjoint of $i_*$.  In general, the pullback of a representable (i.e., Yoneda) module $i^*\mathcal Y_L$ is not representable, reflecting the possibility that $i_*$ may not have a right adjoint $\cW(X) \to \cW(Y)$ on the level of the smaller categories. The first goal of this appendix is to exhibit a nice sub-class of Lagrangians in $X$, `negative pushoffs of arclike Lagrangians' for which there is a clear geometric interpretation of the functor $i^*$, and to deduce from this the existence of a partial right adjoint on the level of the smaller categories (see \Cref{thm:arclikerepresentable} and \Cref{cor:partialrightadjointgeneral}). 

The second goal of this appendix is to then apply these general statements to the case that  $X = S$ is a Liouville sector with $Y = \Nbd(\partial S) = \eff \times T^*[0,1]$ the boundary sector (strictly speaking, we may need to deform $S$ to obtain boundary coordinates $\Nbd(\partial S) = \eff \times T^*[0,1]$; see \cref{boundarydeformation}. We will assume this is already done for the discussion here). 
The composition of the fully faithful stabilization functor $stab: \W(\eff) \hookrightarrow \W(\eff \times T^*[0,1])$ with the sector inclusion functor $i_*$ is often denoted $\cup: \W(\eff) \to \W(S)$ and called the {\em cup functor}. \Cref{lem:def_capl} below asserts the existence of a partially defined right adjoint $\cap$, the {\em cap functor}, to $\cup$ on the level of the smaller categories.

We say that a Lagrangian submanifold $K$ of a Liouville sector $X$ is {\em arclike} (\Cref{def:arclike}) if:
\begin{itemize} 
    \item $K$ intersects $\partial X$ transversely in $\partial K$ and 
    \item there exists an $\alpha$-linear defining function $I:\partial X \to \R$ (in the sense of \cite[Definition 2.4]{ganatra2020covariantly}) which is constant over $\partial K$, for some $\alpha > 0$. For this $I$, we say $K$ is `$I$-boundary-arclike'. 
\end{itemize} 
We note that by \cref{lem:alpha_alpha'} below, if $K$ is arclike and $\alpha>0$ is \emph{any} positive number, there exists an $\alpha$-linear defining function $I:\partial X \to \mathbb{R}$ for which $K$ is $I$-boundary-arclike. By restricting to a level of such an $I$, $\partial K$ can be viewed as an exact Lagrangian submanifold in the boundary sector $\eff = \partial X / \R$, inheriting all of its brane data from $K$. By \cref{lem:IFunctionAdaptedToK}, such an $I$ function can always be extended to an $\alpha$-linear function defined on some neighbourhood $\Nbd^Z(\partial X)$ of $\partial X$, which we call an `extended $\alpha$-defining function', while still containing $K \cap \Nbd^Z(\partial X)$ in a single level-set --- if this condition holds for a given $K$ and extended $I$, we say $K$ is `$I$-arclike' (see \cref{def:Iarclike}).

Associated to an arclike Lagrangian $K$, there is a {\em (cylindrized) negative pushoff} $K^-$ (see \S\ref{subsec:arclike}), which is an object of $\cW(X)$ that (a priori depends on a choice of extended 1-defining function for which $K$ is $I$-arclike but) is well-defined up to isomorphism (see \cref{negativepushoffindependence}). 
The construction of this negative pushoff requires $X$ to have {\em exact} boundary in the sense of \cite[Definition 2.10]{ganatra2020covariantly} so that $I$ induces a Liouville identification $\Nbd^Z(\partial X) \cong \eff \times T^*[0,1]$, but this condition can always be arranged by a convex-at-infinity deformation of Liouville sectors which preserves the wrapped Fukaya category and the notion of $I$-arclike, see \cref{boundarydeformation}; hence, there is an induced isomorphism class of $K^-$ even if $X$ does not have exact boundary.
\begin{thm}\label{thm:arclikerepresentable}
  Let $i:Y\to X$ be an inclusion of Liouville sectors.
  Let $K\subset X$ be a (compact exact)  Lagrangian $K$ which is arclike for $X$ such that $K \cap Y$ is arclike for $Y$. Then there is an isomorphism of Yoneda modules
  \[i^* \mathcal{Y}_{K^-} \simeq \mathcal{Y}_{(K \cap X)^-}.\]
  In particular, there are functorial isomorphisms $Hom_{H\cW(X)}(i(L),K^-)= Hom_{H\cW(Y)}(L,(K \cap X)^-)$ for any object $L$ of $\W(Y)$ (recall $Hom_{H\cC}(B,A):= H^*(\hom_{\cC}(A,B))$ for an $\ainf$ category $\cC$).
  \label{thm:arclikeRestriction}
\end{thm}

\begin{cor}\label{cor:partialrightadjointgeneral}
    Given $i: Y \to X$ as above, let $\cP({Y \subset X}) \subset \W(X)$ denote the full subcategory consisting of negative pushoffs (in $X$) of Lagrangians $K$ such that $K \subset X$ and $(K \cap Y) \subset Y$ are both arclike.
        Then, $i_*: \W(Y) \to \W(X)$ has a partially defined right adjoint along $\cP({Y \subset X})$, \[i^*: \cP({Y \subset X}) \to \W(Y),\] which on the level of objects sends $K^- \mapsto (K \cap Y)^-$ (up to isomorphism).
\end{cor}
\begin{proof}
    Apply \Cref{cor:partial_adjoint} to \Cref{thm:arclikerepresentable}.
\end{proof}

In the body of this article, we make use of the following application of the results above.
Let $(\Ldom,\fo)$ be a sutured Liouville domain, and $S$ the associated Liouville sector as in \cite[Definition 2.14]{ganatra2020covariantly}.  Fix an extended $\alpha=1$-defining function $I: \Nbd^Z(\partial S) \to \R$. 
After a deformation of Liouville structure near $\partial S$ as in \cite[Proposition 2.28]{ganatra2020covariantly} (see \cref{boundarydeformation}), the function $I$ determines a boundary sector inclusion: $i: \eff \times T^*[0,1] \hookrightarrow S$ covering a neighbourhood of $\partial S$, under which $I$ becomes the projection to $\R$ in $T^*[0,1] = [0,1]\times \R$, and where $\eff$ is the completion of $\fo$.
By composing the induced functor $i_*$ with stabilization $stab: \cW(\eff) \to \cW(\eff \times T^*[0, 1])$ \cite[\S 8.4]{ganatra2018sectorial}, one obtains a cup functor $\cup: \cW(\eff)\to \cW(S)$.  
$I$ also determines a collection of $I$-arclike Lagrangians, which are (implicitly shrinking $\eff \times T^*[0,1]$ if necessary or flowing) also $I$-arclike when intersected with $\eff \times T^*[0,1]$. By applying the above results to $i: \eff \times T^*[0,1] \to S$, we will obtain:

\begin{thm}\label{lem:def_capl}
    Fix an extended $\alpha=1$-defining function $I$, and let $\cP(S) \subset \W(S)$ denote the associated full subcategory of negative pushoffs of $I$-arclike Lagrangians in $S$. The induced cup functor $\cup: \cW(\eff) \to \cW(S)$ has a partially defined right adjoint along $\cP(S)$, called the {\em cap functor} and denoted 
    \begin{equation}
    \cap: \cP(S) \to \cW(\eff),
\end{equation} 
    with image contained in the compact Fukaya category $\fuk(\eff) \subset \cW(\eff)$ and which sends $L^-$ to an object isomorphic to $\partial L$. 
\end{thm}

After a brief subsection on algebraic preliminaries, the remaining subsections are devoted to the study of arclike Lagrangians and their negative pushoffs, and, from there, the proofs of \Cref{thm:arclikerepresentable} and \Cref{lem:def_capl}.

\subsection{Results about \texorpdfstring{$\ainf$}{A-infinity} functors}
\label{section:adjoints}
To begin, we recall a few standard facts about modifying and constructing $\ainf$ functors which are used repeatedly in this appendix and elsewhere in the paper. The first fact concerns our ability to modify an $\ainf$ functor $\cF$ by replacing every object $\cF(X)$ by an isomorphic object. More generally, it allows us to modify $\cF$ to take values in any fully faithfully embedded category of the codomain with the same or larger essential image: 
\begin{lem}\label{lem:const_fun_easy}
Let $\cA \xrightarrow{\cF} \cC \xleftarrow{\cG} \cB$ be $A_\infty$ functors, with $\cG$ cohomologically fully faithful, such that the image of $H(\cF)$ is contained in the essential image of $H(\cG)$ (i.e., for every object $X$ of $\cA$, there exists an object $Y$ of $\cB$ such that $\cF X$ is quasi-isomorphic to $\cG Y$). 
Then, there exists an $A_\infty$ functor $\cI: \cA \to \cB$, unique up to quasi-isomorphism, such that $\cF$ is naturally quasi-isomorphic to $\cG \circ \cI$. 
If $\cF$ is furthermore cohomologically fully faithful and has the same essential image as $\cG$, then $\cI$ is a quasi-equivalence.
\end{lem}
\begin{proof}
    We may replace $\cC$ with the essential image of $\cG$. 
    As $\cG$ is cohomologically fully faithful, it is a quasi-equivalence onto this essential image. 
    Thus, we may invert it, see \cite[Theorem 2.9]{seidel2008fukaya}, and set $\cI = \cG^{-1} \circ \cF$ ($\cG^{-1}$ is unique up to quasi-isomorphism; hence so is $\cI$).
    If $\cF$ is cohomologically fully faithful and has the same essential image as $\cG$, then it is a quasi-equivalence onto the essential image of $\cG$, so $\cI$ is also a quasi-equivalence.
\end{proof}
The second fact, really a corollary of the first, concerns the construction of functors from representable bimodules. 
Recall that given an $\ainf$ functor $f: \cA \to \cB$, one can associate an $\ainf$ $\cA\!-\!\cB$ bimodule, the {\em right graph of $f$} $\Gamma_f^R = \hom_{\cB}(f(-), -)$ and a $\cB\!-\!\cA$ bimodule, the {\em left graph of $f$} $\Gamma_f^L = \hom_{\cB}(-, f(-))$ (these are the two-sided pullbacks of the diagonal bimodule $\cB_{\Delta}$  along $(f, id)$ and $(id,f)$ respectively).\footnote{The confusing convention arises because when we define the cohomology category of an $A_\infty$ category, the order of the inputs gets swapped: $\Hom^*_{H\cA}(X,Y) \coloneqq H^*(hom^*(Y,X))$. In particular, after taking cohomology, our definitions of left/right graph (and subsequent definitions of left/right adjoint functors) agree with the usual ones. We remark that the cohomology category of an $A_\infty$ category was defined in \cite[Section 3.2]{Sheridan2015a} as the \emph{opposite} of the definition just stated (opposite in the usual sense, i.e., no Koszul signs involved); however that definition, when combined with the definition of the $A_\infty$ category associated to a dg category [\emph{op. cit.}, Definition 3.4], does not agree with the standard definition of the cohomology category of a dg category. The last-named author apologizes.}
An $\cA\!-\!\cB$ bimodule $\cP$ is said to be {\em right-representable} (respectively {\em left-representable}) if for all $x \in \cA$, $\cP(x,-)$ (respectively $\cP(-,y)$ for all $y \in \cB$) is representable, i.e., isomorphic to a Yoneda module. 
\begin{cor}\label{cor:representable_bimodules}
    $\cP$ is right representable if and only if there exists a functor $F: \cA\to \cB$ (unique up to isomorphism) such that the right graph of $F$ is isomorphic to $\cP$. Similarly, if $\cP$ is left representable, it must be isomorphic to the left graph of a (unique up to isomorphism) functor $G: \cB \to \cA$.
\end{cor}
\begin{proof}
    For the first assertion, the ``if'' direction is tautological. For the ``only if,'' viewing $\cP$ as a functor from $\cP: \cA \to \operatorname{Mod} \cB$, apply Lemma \ref{lem:const_fun_easy} to $\cA \xrightarrow{\cP} \operatorname{Mod} \cB \xleftarrow{\cY} \cB$ where $\cY$ is the Yoneda embedding.  The Lemma also tells us $F$ is unique up to isomorphism and that the graph of $F$, which as a functor $\cA \to \operatorname{Mod} \cB$ is $\cY \circ F$, is isomorphic to $\cP$. The second assertion is similar (and begins by viewing $\cP$ as $\cP: \cB^{op} \to \operatorname{Mod}\cA^{op}$).
\end{proof}
We will apply \cref{cor:representable_bimodules} to the problem of constructing adjoints or partial adjoints to functors.  
Recall that $f: \cA \to \cB$ is {\em right adjoint to} (or {\em has left adjoint})  $g: \cB \to \cA$ if there is an isomorphism of $\cA\!-\!\cB$ bimodules $\Gamma_f^R = \Gamma_g^L$. 
More generally, given a subcategory $i: \cC \subset \cB$, we say that $f: \cA \to \cB$ has a {\em partially-defined left adjoint} along $\cC$ given by $g: \cC \to \cA$ if there is an isomorphism of $\cA\!-\!\cC$ bimodules between the $\cC$-restricted right graph of $f$, $(\Gamma_f^R)|_{\cA \times \cC} = \hom_{\cB}(f(-),i(-)) $ and the left graph of $g$, $\Gamma_g^L = \hom_{\cA}(-, g(-))$.
An immediate consequence of Corollary \ref{cor:representable_bimodules} is:
\begin{cor}\label{cor:partial_adjoint}
    A functor $f: \cA \to \cB$ has a partially defined left adjoint along $i: \cC \subset \cB$ if and only if the right graph of $f$ restricted to $\cC$, $(\Gamma_f^R)|_{\cA \times \cC} = \hom_{\cB}(f(-),i(-))$ is right representable; that is, if for every $c \in \cC$, if the Yoneda module $i^*\hom_{\cB}(-,c)$ is representable. \qed
\end{cor}
\begin{rmk}\label{rmk:modifyingg}
In the setting of \cref{cor:partial_adjoint}, suppose that there exists a subcategory $\cD \subset \cA$ such that for every $c \in \cC$, $i^*\hom_\cB(-,c)$ is representable by an object of $\cD$.  
Then, by applying \cref{lem:const_fun_easy}, we may assume that the partially-defined left adjoint $g$ has image contained in $\cD$.
\end{rmk}

\subsection{Holomorphic curves in Liouville sectors with Lagrangians touching the boundary}

Let $X$ be a Liouville sector and fix a projection $\pi: \Nbd^Z (\partial X) \to \C_{\Re \geq 0}$ as in \cite[Definition 2.26]{ganatra2020covariantly}. We begin by establishing a slight strengthening of the open mapping theorem argument from \cite[Lemma 2.41]{ganatra2020covariantly}, which will be used to confine disks with boundary on arclike Lagrangians in the following subsections. 

\begin{prop} \label{openmappingboundary} Let $\Sigma$ be a Riemann surface with boundary. Suppose we have fixed a $J$ on $X$ such that $\pi: \Nbd^Z (\partial X) \to \C_{\Re \geq 0}$ is $J$-holomorphic, and let $u: \Sigma \to X$ be a $J$-holomorphic map. If $u^{-1}\pi^{-1}(\C_{\epsilon \ge \Re \ge 0})$ is compact and its intersection with $\partial \Sigma$ maps under $\pi \circ u$ to a horizontal segment $\Gamma$, then $u^{-1}\pi^{-1}(\C_{\epsilon >\Re>0})$ is either empty or consists of closed components of $\Sigma$ along which $u$ is constant.
\end{prop}
\begin{proof}
    The proof follows \cite[Lemma 2.41]{ganatra2020covariantly} closely, and immediately extends to any more general $\Gamma$ which is closed in $S := \C_{\epsilon \ge \Re \ge 0}$, has empty interior and for which $S \backslash \Gamma$ consists only of unbounded components. Let $P = (\pi \circ u)^{-1}(S)$, which is compact by hypothesis. Consider $f = \pi \circ u: P \to S$. Suppose for a contradiction that $f^{-1}(\mathrm{int}(S))$ is neither empty nor a disjoint union of closed components along which $f$ is constant. 
It follows that there exists a connected component $Q$ of $P$, such that $f(Q)$ intersects $\mathrm{int}(S)$ and $Q$ is not a component of $\Sigma$ on which $f$ is constant. We have:
\begin{itemize}
\item $Q$ is compact (as it is a closed subset of $P$, which is compact by hypothesis), hence $K:=f(Q)$ is compact.
\item As $f|_Q^{-1}(\mathrm{int}(S))$ is open and non-empty, $Q$ contains an interior point $q$. As $Q$ is not a component of $\Sigma$ on which $f$ is constant, $f$ is not locally constant at $q$, by unique continuation. Therefore, by the open mapping theorem, $K$ contains a neighbourhood of $f(q)$. Therefore $K$ has non-empty intersection with the set $S' = \mathrm{int}(S) \setminus \Gamma$, seeing as the latter is dense in $S$ (as $\Gamma$ has empty interior).
\item Let $Q' = f^{-1}(S')$. Note that $Q'$ is non-empty by the previous point, is open because $S'$ is open (as $\Gamma$ is closed), and consists entirely of interior points of $\Sigma$ because $u(\partial \Sigma) \subset \Gamma$ is disjoint from $S'$. Therefore, by the open mapping theorem, $K' := u(Q')$ is open and non-empty.
\item In addition to being open and non-empty, $K' = K \cap S'$ is also closed in $S'$, as $K$ is compact; hence it is a non-empty union of connected components of $S'$. However all connected components of $S'$ are unbounded, contradicting the compactness of $K$.
\end{itemize} 
\end{proof}

\subsection{Arclike Lagrangians and their negative pushoffs}\label{subsec:arclike}

Let $X$ be a Liouville sector and $K$ a compact exact Lagrangian submanifold-with-boundary, intersecting $\partial X$ transversely along its boundary $\partial K \subset \partial X$.

\begin{defn}\label{def:arclike}
    For an $\alpha$-defining function $I: \partial X \to \R$ and a $K$ as above, we say $K$ is {\em $I$-boundary-arclike} if $I$ is constant along $\partial K$. Call $K$ {\em arclike} if one of the following equivalent conditions are satisfied (compare \cite[Definition 2.4]{ganatra2020covariantly}):
    \begin{itemize}
        \item For some $\alpha > 0$, $K$ is $I$-boundary-arclike for some $\alpha$-defining function $I$.
        \item For every $\alpha > 0$,  $K$ is $I$-boundary-arclike for some $\alpha$-defining function $I$.
        \item The projection of $\partial K \hookrightarrow \partial X$ to the quotient by the characteristic foliation $F:=\partial X / \R$ is an embedding.
    \end{itemize}
\end{defn}

\begin{lem}\label{lem:alpha_alpha'}
    The three conditions appearing in Definition \ref{def:arclike} are equivalent.
%    If $K$ is arclike meaning that there exists an $\alpha$-defining function $I$ for which $K$ is $I$-boundary-arclike for some $\alpha > 0$, there exists an $\alpha'$-defining function $I'$ for which $K$ is $I'$-boundary-arclike for {\em any} $\alpha' > 0$.
\end{lem}

We make some preparatory remarks prior to the proof of \cref{lem:alpha_alpha'}. 
%First,} it will be helpful to explicitly present our given Liouville sector $X$ as the completion of a compact codimension-0 domain $X^{in}$ with respect to the Liouville vector field along one of the boundary faces of $X^{in}$, analogous to the presentation of a Liouville manifold as the completion of a Liouville domain. As in the proof of \cite[Lemma 2.5]{ganatra2020covariantly}, we produce $X^{in}$ by first picking a contact type hypersurface $Y$ in $X$ (sufficiently close to $\infty$) transversely intersecting $\partial X$ in the region where $Z$ is tangent to $\partial X$, whose projection to $\partial_{\infty} X$ via Liouville flow is a diffeomorphism. The flow by $Z$ from $Y$ determines a symplectization region $\mathbb{R}_{\geq 0} \times Y \subset X$, and by definition $X^{in}$ is the complement of this region $X \backslash (\mathbb{R}_{\geq 0} \times Y)$. $X^{in}$ is a manifold with corners with two boundary faces, $Y$ and $(\partial X)^{in}$ (which is a codimension-0 subdomain of $\partial X$). 
%    We can further arrange, see {\em loc. cit.}, that such an $X^{in}$ is {\em adjusted to $I$}, meaning the corner $Y \cap (\partial X)$ is tangent to the characteristic foliation on $\partial X$ near $(Y \cap (\partial X)) \cap \{I = 0\}$. 
%{ Next, 
    Let $F = \partial X / \R$ denote the leaf space of the characteristic foliation on $\partial X$, and $\pi: \partial X \to F$ the quotient map. Fix some $\alpha$-defining function $I: \partial X \to \R$, inducing an identification $\partial X \stackrel{(\pi,I)}{\to} F \times \R$ under which $\pi$ is projection to the first factor. We observe there is a large space of modifications of $I$ coming from compactly supported functions on $F$ (suitably cut off):
    \begin{lem}\label{lem:modify_I}
    Let $f: F \to \R$ be any function with compact support contained inside $F_0 \subseteq F$ and
    with $\pi^*f: \partial
    X \to \R$ the lift. Then for any $N>0$ there is a compactly supported
    modification $\tilde{I}: \partial X \to \R$ of $I$ to an
    $\alpha$-defining function that is equal to $I - \pi^*f$ on $F_0 \times [-N , N]$.
\end{lem}
\begin{proof}
Pick an arbitrary compactly supported extension $\tilde{f}: F \times \R \to \R$ of $\pi^*f|_{F_0 \times [-N \times N]}$ (e.g., multiply $\pi^*f$ by a function of $I$ which is equal to $1$ for $|I| \le N$ and is compactly supported).  As $\tilde{f}$ is compactly supported, $|\partial_t \tilde{f}|$ is bounded, where $t$ denotes the coordinate on the $\R$ factor.  Therefore, by rescaling $\tilde{f}(y,t) \mapsto \tilde{f}(y,\frac{t}{K})$ for sufficiently large $K$ we obtain a function equal to 
$\pi^*f$ on $F_0 \times [-NK,NK] \supset F_0 \times [-N,N]$ where $|\partial_t \tilde{f}| < \frac{1}{2}$ everywhere. 
Now, set $\tilde{I}:= t - \tilde{f}(y,t) = I - \tilde{f}$. Observe that $\tilde{I}=I - \pi^*f$ on $F_0 \times [-N,N]$ and satisfies $\partial_t \tilde{I} > \frac{1}{2} >0$, implying that $d\tilde{I}$ is positive on the characteristic foliation of $\partial X$. Also at infinity $\tilde{I} = I$ which is $\alpha$-linear.
\end{proof}

%As in \cite{ganatra2020covariantly}, it is sometimes convenient to pass between different values of $\alpha$ appearing in Definition \ref{def:arclike}. We primarily need the cases $\alpha = \frac{1}{2}$ and $\alpha =1$ of the following Lemma:
%\begin{lem}\label{lem:alpha_alpha'}
%    If $K$ is arclike meaning that there exists an $\alpha$-defining function $I$ for which $K$ is $I$-boundary-arclike for some $\alpha > 0$, there exists an $\alpha'$-defining function $I'$ for which $K$ is $I'$-boundary-arclike for {\em any} $\alpha' > 0$.
%        \end{lem}
\begin{proof}[Proof of Lemma \ref{lem:alpha_alpha'}]
     Evidently, the second condition implies the first condition implies the third (if $K$ is $I$-boundary-arclike, then the embedding $\partial K \subset \partial X \stackrel{(\pi,I)}{\to} F \times \R$ is contained in $F \times \{c\}$ for some $c$, and in particular embeds into $F$).
        Therefore it suffices to prove that the third condition implies the second.
%, but for completeness,\marginpar{\tiny{I guess the proof is complete without this, maybe we should just remove it? -NS}} we also show how the first condition implies the second here: }
%    Given an $\alpha$-defining function $I$, \cite[Lemma 2.5]{ganatra2020covariantly} constructs an $\alpha'$-defining function $I'$ by smoothing $\frac{I}{|I|} |I|^{\alpha' / \alpha}$ in a suitable fashion. 
%    The smoothing is constructed with respect to a domain $X^{in}$ whose corner is adjusted to the original $I$ as described above the statement of Lemma currently being proved. One first smooths along $X^{in}$ and then extends the smoothing to $X$ by $\alpha'$-linearity (strictly speaking, {\em loc. it.} does this by first smoothing along the contact boundary $Y$ of $X^{in}$ using the adjusted $I$ condition, and then extending into $X^{in}$ and out into the symplectization region by linearity).
%
%Now, $\partial K$ lies in a compact subset of $\partial X$, so we can assume (by choosing a sufficiently large $X^{in}$ that) $\partial K \subset (\partial X)^{in}$, the boundary face of $X^{in}$ living within $\partial X$. After restricting to $(\partial X)^{in}$, every level set of $I$ is a level set of the smoothing, seeing as in suitable split coordinates one is smoothing the single variable function $f(x) = \frac{x}{|x|} |x|^{\alpha' / \alpha}$ so that the result remains monotonic. Therefore, for such a smoothing, $\partial K$ remains in a level set as desired.  
% Now we show the third condition implies the second. 
    To do so, fix any $\alpha > 0$ and any $\alpha$-defining function $I : \partial X \to \R$ inducing $\partial X \cong  F \times \R_t$, and denote by $\overline{\partial K}:=\pi(\partial K)  \subset F$. Assuming the third condition, $\overline{\partial K}$ is a submanifold of $F$ and
    there is a function $f: \overline{\partial K} \to \R$ whose graph is $\partial K$, namely $f:= I|_{\partial K} \circ \pi^{-1}|_{\overline{\partial K}}$.
Extend $f$ to a compactly supported function on $F$ whose support lies in a Liouville domain $F_0$ containing $\partial K$. By Lemma \ref{lem:modify_I}, we may modify $I$ to another defining function $\tilde{I}$ which equals  $I - \pi^* f$ on $F_0 \times [-N,N]$ where $N$ is such that $\partial K \subset F_0 \times [-N,N]$. Now, at any point $p \in \partial K$, $\tilde{I}(p) = I(p) - \pi^* f(p) = I(p) - I|_{\partial K}(p) = 0$.
\end{proof}

It will be helpful for the following Lemma to explicitly present our given Liouville sector $X$ as the completion of a compact codimension-0 domain $X^{in}$ with respect to the Liouville vector field along one of the boundary faces of $X^{in}$, analogous to the presentation of a Liouville manifold as the completion of a Liouville domain. As in the proof of \cite[Lemma 2.5]{ganatra2020covariantly}, we produce $X^{in}$ by first picking a contact type hypersurface $Y$ in $X$ (sufficiently close to $\infty$) transversely intersecting $\partial X$ in the region where $Z$ is tangent to $\partial X$, whose projection to $\partial_{\infty} X$ via Liouville flow is a diffeomorphism. The flow by $Z$ from $Y$ determines a symplectization region $\mathbb{R}_{\geq 0} \times Y \subset X$, and by definition $X^{in}$ is the complement of this region $X \backslash (\mathbb{R}_{\geq 0} \times Y)$. $X^{in}$ is a manifold with corners with two boundary faces, $Y$ and $(\partial X)^{in}$ (which is a codimension-0 subdomain of $\partial X$). 
%We can further arrange, see {\em loc. cit.}, that such an $X^{in}$ is {\em adjusted to $I$}, meaning the corner $Y \cap (\partial X)$ is tangent to the characteristic foliation on $\partial X$ near $(Y \cap (\partial X)) \cap \{I = 0\}$. 

As in \cite[\S 2.6]{ganatra2020covariantly}, it is useful to consider extensions of an $\alpha$-defining function $I: \partial X \to \R$ to $\alpha$-linear functions $I:\Nbd^Z(\partial X) \to \R$, which we call here {\em extended $\alpha$-defining functions}. 
\begin{defn}\label{def:Iarclike}
    Let $K$ be a compact exact Lagrangian-with-boundary $K \subset X$ which intersects $\partial X$ transversely along $\partial K \subset \partial X$, and $I:\Nbd^Z(\partial X) \to \R$ an extended $\alpha$-defining function. $K$ is called {\em $I$-arclike} if, after possibly shrinking $\Nbd^Z(\partial X)$, $I$ is constant along $K \cap \Nbd^Z(\partial X)$.
\end{defn}
Evidently, if $K$ is $I$-arclike for a given $I$, then $K$ is also $I|_{\partial X}$-boundary-arclike (and therefore arclike). Conversely we have:
\begin{lem}
    \label{lem:IFunctionAdaptedToK}
    If $K$ is arclike, then for any $\alpha>0$ there exists an extended $\alpha$-defining function $I:\Nbd^Z(\partial X) \to \R$ for which $K$ is $I$-arclike.
\end{lem}
\begin{proof}
    For any $\alpha$, fix an $\alpha$-defining function $I:\partial X \to \R$ for which $K$ is $I$-boundary-arclike, which exists by \cref{lem:alpha_alpha'}.  With respect to a choice of domain $X^{in}$ completing to $X$ as described above \cref{lem:alpha_alpha'}, we will extend $I$ first along the intersection $(\Nbd^Z (\partial X)) \cap X^{in}$, then extend to all of $\Nbd^Z (\partial X)$ by $\alpha$-linearity. 
    Recall that the boundary of $X^{in}$ is a union of two faces: $(\partial X)^{in} \subset \partial X$, and $Y$, along which $Z$ is outward pointing.
    Over $\Nbd^Z (\partial X) \cap X^{in}$ one can construct an extension by pulling back $I$ along the projection to $(\partial X)^{in}$ associated to a choice of collar identification $\Nbd (\partial X)^{in} \cong (\partial X)^{in} \times [0,\epsilon)$ induced by an inward pointing vector field defined near $(\partial X)^{in}$ that is tangent to $Y$. We may further choose this inward pointing vector field to commute with $Z$ near $Y$ (so that the resulting extension of $I$ remains $\alpha$-linear near $Y$ hence smoothly extends to $\Nbd^Z (\partial X)$), and to be tangent to $K$ (here we use the requirement that $K$ is transverse to $\partial X$). This implies that the extension --- also called $I$ by abuse of notation --- is constant along $K$.
\end{proof}

Henceforth we fix such an extension $I$ for some $\alpha$ for which $K$ is $I$-arclike.
Since $I|_{K}$ is constant near $\partial X$, it follows that: 
\begin{lem}\label{lem:XItangK}
    At points of $K$ near $\partial X$, the Hamiltonian vector field $\vX_I$ is tangent to $K$. 
\end{lem}
\begin{proof}
    Constancy of $I$ along $K$ implies $T_p K \subset T_p I^{-1}(c)$
    (where $I|_{K} \equiv c$), or equivalently that $\mathrm{span}(\vX_I) = (T_p I^{-1}(c))^{\omega \perp}$ is contained in $T_p K = (T_p K)^{\omega \perp}$.
\end{proof}

By \cite[Proposition 2.25]{ganatra2020covariantly}, the extended $\alpha$-linear function $I$ determines a product decomposition 
\begin{equation}\label{productdecomp}
    (\Nbd^Z(\partial X),\lf_X|_{\Nbd^Z(\partial X)}) = (\eff \times \C_{Re \ge 0}, \lf_{\eff} + \lf^\alpha_\C + df),
\end{equation}
over neighbourhoods of $\partial X$ and $\eff \times \{Re = 0\}$ respectively, where the coordinates on $\C_{Re \ge 0}$ are $R+i I = x + iy$, 
$\lf_{\eff} = (\lf_X)|_{\eff := I^{-1}(0)}$, $\lf^\alpha_\C$ is the Liouville form for $\omega_{\C} = dx \wedge dy$ associated to Liouville vector field $Z_{\C}^{\alpha} = (1-\alpha)x \frac{\partial}{\partial x} + \alpha y \frac{\partial}{\partial y}$, and the function $f$ is a function on $\eff \times \C_{Re \geq 0}$ satisfying properties detailed in {\em loc. cit.} 
This product decomposition is induced by defining 
\begin{equation}
    \label{eq:Rfunction}
    R: \Nbd^Z (\partial X) \to \R,
\end{equation}
to be the unique function characterized by the properties $R|_{\partial X}= 0$, $dR(\vX_I) = -1$. 

\begin{lem}\label{arcimage}
    If $K$ is $I$-arclike, then $R|_K$ is a submersion from $K$ to $\R$ over $\Nbd (\partial K):= \Nbd^Z (\partial X) \cap K$.
\end{lem}
\begin{proof}
By \Cref{lem:XItangK}, $\vX_I$ must be tangent to $K$, and $dR(\vX_I) = -1$, hence $dR|_{TK} \neq 0$.  
\end{proof}

It follows that in the product coordinates $\eff \times \C_{Re \geq 0}$, $K$ simply takes the form $\partial K \times \gamma$ where $\gamma$ is an arc of the form $\{Im = c\}$. By restricting to $I^{-1}(c)$, we see that $\partial K$ is exact for $\lf_X|_{I^{-1}(c)} = \lf_{\eff} + df|_{I^{-1}(c)}$, hence exact for $\lf_{\eff}$. Let $g_K: K \to \R$ denote a chosen primitive of $\lf_X|_{K}$. The restricted primitive for $\partial K$ with respect to $\lf_{\eff}$ is by definition $g_{\partial K}:= (g_K)|_{\partial K} - f|_{\partial K}$. Pick a primitive $g_{\gamma}$ for $\gamma \subset \C_{Re \geq 0}$ with respect to $\lf_{\C}^{\alpha}$ which is equal to $0$ at $(R,I) = (0,c)$. Now observe that in $F \times \C_{Re \geq 0}$, with respect to the Liouville form $\lf_X = \lf_{\eff} + \lf_{\C}^{\alpha} + df$, our (globally defined) primitive $g_K$ for $K$ must be equal to $g_{\partial K} + g_{\gamma} + f$ (as both are primitives of $K = \partial K \times \gamma$ in this region which agree along $\partial K \times (0,c)$).

Next, we explain how, if $X$ has {\em exact boundary} in the sense of \cite[Definition 2.10]{ganatra2020covariantly}, to modify an arclike Lagrangian $K$ in $X$ into a cylindrical Lagrangian 
\begin{equation}
  K^-\subset X
  \label{def:negativePushoff}
\end{equation}
avoiding $\partial X$ (and hence defining an object of $\W(X)$), called its {\em (cylindrized) negative pushoff}. By construction, this modification is supported in an arbitrarily small cylindrical-at-infinity neighbourhood of $\partial X$. The exactness condition on the boundary is used to obtain an embedding of Liouville sectors $\eff \times T^*[0,1] \hookrightarrow \Nbd^Z(X)$ covering a neighborhood of the boundary, or equivalently a Liouville identification $\Nbd^Z(X) \cong \eff \times T^*[0,1]$ up to a compactly supported exact 1-form. As we recall, the exact boundary condition can always be arranged up to a deformation:

\begin{rmk}[Boundary sectors]
\label{boundarydeformation}

    We recall here how an extended $\alpha=1$-defining function $I$ induces a boundary sector inclusion $(\eff \times T^*[0,1], \lf_{\eff} + \lf_{T^*[0,1]}) \hookrightarrow X$ after possibly deforming the Liouville form near $\partial X$ (cf. the proof of \cite[Theorem 2.28]{ganatra2020covariantly}). First, $I$ induces an associated product decomposition \eqref{productdecomp}, where $\lf_{\C}^1 = \lf_{T^*\R_{\geq 0}}$. Now, after applying the convex-at-infinity Liouville deformation $\lf_X \leadsto \lf_X'$ described in the proof of \cite[Proposition 2.28]{ganatra2020covariantly}, we can arrange for $\lf_X'$ in coordinates to simply equal $\lf_{\eff} + \lf_{T^*\R \geq 0}$ for $R$ near $0$ and equal $\lf_X$ for $R$ away from $0$ (concretely in coordinates \eqref{productdecomp}, $\lf_X'= \lf_{\eff} + \lf_{T^* \R_{\geq 0}} + d(\varphi(R) f)$ for a cutoff function $\varphi$ which equals $0$ for $R \in [0,\epsilon]$ and $1$ for $R \notin [0,2\epsilon]$).  The outcome is a (cylindrical-at-infinity) Liouville sector embedding $\eff \times T^*[0,\epsilon] \to (X,\lf_X')$. Now identify $T^*[0,\epsilon] = T^*[0,1]$ via $(R,I) \mapsto (\frac{1}{\epsilon} R, \epsilon I)$. 

    If $X$ already has exact boundary in the sense of \cite[Definition 2.10]{ganatra2020covariantly}, then after possibly modifying $\lf_{\eff}$ by a compactly supported exact 1-form, one can make the $df$ term in \eqref{productdecomp} compactly supported. This gives a boundary sector embedding $\eff \times T^*[0,1] \hookrightarrow X$ without further deforming $\lf_X$; see the discussion below \cite[Definition 2.33]{ganatra2020covariantly} for more details. Indeed, the deformation from \cite[Proposition 2.28]{ganatra2020covariantly} was used to show that any $X$ can be deformed (in a homotopically unique way) to have exact boundary.
\end{rmk}

Returning to the construction of $K^{-}$, use \Cref{lem:alpha_alpha'} and  \Cref{lem:IFunctionAdaptedToK} to
fix an extended $\alpha=1$-defining function $I$ for which $K$ is $I$-arclike. 
As we are assuming $X$ has exact boundary (or implicitly deforming $\lf_X$ to do so as in \cref{boundarydeformation}),
we may consider the resulting boundary sector embedding $\eff \times T^*[0,1] \hookrightarrow X$ (where $\partial X$ is the image of $\eff \times T^*[0,1]|_{0}$); in this smaller region, \Cref{arcimage} implies that $K$ must have the form $\partial K \times \gamma$, where $\gamma: [0,1] \hookrightarrow T^*[0,1]$ is of the form $\{I =c \}$ (this constant may differ from the original $c$ as we may scale $I$ per \cref{boundarydeformation}). The Liouville deformation performed in \cref{boundarydeformation} also changes the primitive in this region to $g_{\partial K} + g_{\gamma}$.
Let $\gamma^-$ be a smoothing of $\{R = \epsilon, I\geq c\} \cup \{R \geq \epsilon, I = c\}$; concretely, one could take
$\gamma^-:=(\gamma^-_R,\gamma^-_I): (0, 1] \to T^*[0,1] = [0,1]_R \times \R_I$ any properly embedded curve satisfying:
\begin{itemize}
    \item $\gamma^-$ agrees with $\gamma$ (i.e., the segment $I = c$) near $t = 1$  
    \item $\gamma^-$ lies on the cotangent fiber at $R = \epsilon$ near $t = 0$
    \item $\gamma^-_R$ is weakly monotonically increasing in $t$ and $\gamma^-_I$ is weakly monotonically decreasing in $t$. In particular, since $d\gamma^- \neq 0$, $\gamma^-_R$ is strictly monotonically increasing near $t=1$ and $\gamma^-_I$ is strictly monotonically decreasing near $t=0$.
\end{itemize}
Let $K^-_{split}:= (K \backslash (\partial K \times \gamma)) \cup (\partial K \times \gamma^-)$. $K^-_{split}$ is evidently Lagrangian and, since the codimension-0 inclusions of $ K \backslash (\partial K \times \gamma)$ into both $K$ and $K^-_{split}$ are homotopy equivalences, $K^-_{split}$ is exact and inherits all brane data from choices of such data on $K$. Moreover, outside a compact set, $K^-_{split} = \partial K \times (T^*_{\epsilon})^+$ where $(T^*_{\epsilon})^+ = \{R = \epsilon, I> 0\}$ is the positive part of the cotangent fiber to $\epsilon$. $K^-_{split}$ inherits a primitive of the form $g_{\partial K} + g_{\gamma^-}$ (which extends to $g_K$ up to a constant shift in the region where $K^-_{split}$ is unmodified), where without loss of generality $g_{\gamma^-}$ vanishes outside a compact subset of $\gamma^-$.
Applying the cylindrization procedure of \cite[\S 7.2]{ganatra2018sectorial} in the product chart $\eff \times T^*[0,1]$, we therefore obtain a cylindrical exact Lagrangian $K^-$ which is well-defined and independent of choice of $\epsilon$ or smoothing up to isotopy. A priori $K^-$ depends on the choice of 1-defining $I$ function for which $K$ is $I$-arclike; however we will see that any two choices give isomorphic objects (see \cref{negativepushoffindependence}).
\begin{rmk}
    The definition of cylindrization of a Lagrangian equal (outside a compact set) to $\partial K \times \gamma^-$ in $\eff \times T^*[0,1]$ given in \cite[\S 7.2]{ganatra2018sectorial} involves flowing near infinity (and outside the aforementioned compact set) by an isotopy taking place in the larger stopped convex completion $\eff \times (\C, \{\pm \infty\})\supset \eff \times T^*[0,1]$, which is well-defined up to contractible choice. However, as {\em loc. cit.} describes, the cylindrization isotopy can be made to be supported in an arbitrarily small neighborhood of $\mathfrak{f} \times \partial_{\infty} \gamma^-$ (where $\mathfrak{f}$ is the core of $F$ and using the fact that $\partial_{\infty} \partial K = \emptyset$) --- and in particular within $\eff \times T^*[0,1]$.
\end{rmk}
\begin{lem} \label{negativepushoffforwardstopped}
   The negative pushoff $K^-$ is forward stopped, i.e., it admits a cofinal positive wrapping converging into $\partial X$.
\end{lem}
\begin{proof} 
    Choose a product-like cofinal wrapping of $K^-_{split}$ converging to $\partial X$. Now, \cite[Proposition 7.5]{ganatra2018sectorial} says that (suitable cylindrizations of) products of cofinal wrappings remain cofinal. While the original result was stated for stopped Liouville manifolds, the result remains true in this case because wrappings of a given Lagrangian $L$ in a Liouville sector $Y$ are cofinal if and only if they are cofinal when thought of as being wrappings in the stopped convex completion $(\bar{Y}, \partial_{\infty} \bar{Y} \backslash (\partial_{\infty} Y)^{\circ})$ (cf. \cite[Lemma 3.8]{ganatra2018sectorial}).
\end{proof} 

\begin{rmk}
    Dually, there is a corresponding {\em (cylindrized) positive pushoff} $K^+$, obtained by applying the same construction to $K$ using the curve $\gamma^+$ defined by negating the $I$ coordinate of $\gamma^-$. Identical arguments imply $K^+$ is `backward stopped', i.e., cofinally negatively wraps into the stop.
\end{rmk}

For the next lemma, recall that for any Liouville manifold $\eff$ there is a stabilization functor $stab: \W(\eff) \to \W(\eff \times T^*[0,1])$, which sends $L$ to the cylindrization of $L \times [\mathrm{fiber}]$ \cite[\S 8.4]{ganatra2018sectorial}.
\begin{lem}\label{lem:boundarystabilization}
    Let $\gamma = \{I = c\} \subset T^*[0,1] = [0,1]_R \times \R_I$, and let $V \subset F$ be a compact exact Lagrangian. Then $V \times \gamma \subset \eff \times T^*[0,1]$ is $I$-arclike near $0$ and $-I$-arclike near $1$, and with respect to these functions, its negative pushoff $(V \times \gamma)^-$ is isotopic to the stabilization $stab(V)$.
\end{lem}
\begin{proof}
    The first assertion (that $I$ near $0$ and $-I$ near $1$ are $(\alpha=1)$ defining functions, and that they make $V \times \gamma$ arclike near $0$ and $1$) is tautological. We can then run the negative pushoff construction near $R=0$ by using the embedding 
    $$\eff \times T^*[0,1] \xrightarrow{(f,R,I) \mapsto (f,\delta R, I/\delta)} \eff \times T^*[0,\delta] \hookrightarrow \eff \times T^*[0,1],$$  and near $R=1$ by using the embedding 
    $$\eff \times T^*[0,1] \xrightarrow{(f,R,I) \mapsto (f,1-\delta R,- I/\delta)} \eff \times T^*[1-\delta,1] \hookrightarrow \eff \times T^*[0,1].$$
        The resulting negative pushoff is a cylindrization of $\partial K \times \tilde{f}$ where $\tilde{f}$, a smoothing of $\{R = \delta\epsilon, I \geq c\}\cup \{\delta\epsilon \leq R \leq 1-\delta\epsilon, I = c\} \cup \{R = 1-\delta\epsilon, I \leq c\}$, is isotopic to a cotangent fiber. By cylindrizing the induced isotopy between $\partial K \times [\mathrm{fiber}]$ and $\partial K \times \tilde{f}$ as in \cite[Lemma 7.3]{ganatra2018sectorial} we obtain a cylindrical isotopy between $(V \times \gamma)^-$ and $stab(V)$.
\end{proof}

\subsection{Wrapped Floer theory with arclike Lagrangians}
Let $X$ be a Liouville sector and $K$ an arclike Lagrangian in $X$. 
Using \Cref{lem:alpha_alpha'} and  \Cref{lem:IFunctionAdaptedToK} to fix an
extended $\alpha= \frac{1}{2}$-defining function $I$ for which $K$ is
$I$-arclike, we obtain by \Cref{arcimage}  a map $\pi: \Nbd^Z (\partial X)
\to \C_{\Re \geq 0}$ in the sense of \cite[Def
2.26]{ganatra2020covariantly} such that $\pi|_{K \cap \Nbd^Z (\partial X)}$
submerses $K \cap \Nbd^Z (\partial X)$ onto a horizontal segment $\Gamma$. 
\begin{lem}\label{arclikemodule}
$K$ induces a well-defined module $CF^*(K,-)$ over the wrapped Fukaya category $\W(X)$ (defined implicitly using a fixed choice of $\pi$ as above), whose cohomology is given by ordinary (rather than wrapped) Floer homology. The resulting module is independent up to isomorphism of choices of Floer data made in its construction.
\end{lem}
\begin{proof}
For any cylindrical Lagrangian $L$ which is transverse to $K$, \Cref{openmappingboundary} along with monotonicity as in \cite[Proposition 3.19]{ganatra2020covariantly} implies that the Floer chain complex $CF^*(K,L)$ is well-defined for a generic cylindrical (near infinity) $J$ which makes this $\pi : \Nbd^Z (\partial X) \to \C_{\Re \geq 0}$ holomorphic (recall from \cite[\S 2.10.1]{ganatra2020covariantly} that, as $\pi$ was built from an $\alpha= \frac{1}{2}$-defining function, there is an abundance of such $J$ satisfying both conditions simultaneously). The same arguments imply that we can define $\ainf$ operations of the form $CF^*(K,L_0) \otimes CF^*(L_0, L_1) \otimes \cdots \otimes CF^*(L_{k-1}, L_k) \to CF^*(K, L_k)$ for any tuple of $L_0, \ldots, L_k$ of cylindrical exact Lagrangians avoiding $\partial X$ mutually transverse to each other and $K$ and any surface-dependent $J$ which is uniformly cylindrical in the sense of \cite[\S 3.2]{ganatra2020covariantly}, makes $\pi : \Nbd^Z (\partial X) \to \C_{\Re \geq 0}$ holomorphic (shrinking this neighbourhood when applying \Cref{openmappingboundary} if needed so it doesn't intersect any $L_i$), and achieves transversality.

To construct the desired $\ainf$ module, modify/augment the construction of the wrapped Fukaya category $\W(X) := \cO[C^{-1}]$ from \cite[\S 3.5]{ganatra2020covariantly} as follows. First, in the construction of the directed category $\cO$, only use Lagrangians which are transversal to $K$ (imposing such a constraint on objects is unproblematic, see, e.g., \cite[Remark 3.28]{ganatra2020covariantly}). Then, add a final object $K$ to the directed category, so $\cO':= \langle \Z, \cO\rangle$. The construction of Floer data for $\cO'$ and $\ainf$ operations from that data carries through given the above discussion, in particular using \Cref{openmappingboundary}. By localizing $\cO'$ at the continuation elements $C$ in $\cO$, we obtain a semi-orthogonal decomposition $\W':= \langle \Z, \W \rangle$, and hence an induced $\W$-module $\hom_{\W'}(K,-)|_{\W}$. Note that this latter module $\hom_{\W'}(K,-)|_{\W}$, which we might call ``$CW^*(K,-)$''\footnote{For instance, its cohomology can be calculated by taking the direct limit of $HF^*(K,L^{-})$ over a cofinal collection of negative wrappings of $L$.} is by definition the localization with respect to $C$, $CF^*(K,-)_{C^{-1}}$, of the module $CF^*(K,-) = \hom_{\cO'}(K,-)|_{\cO}$ in the sense of \cite[Definition 3.17]{ganatra2020covariantly}.

Now, because $K$ is compact, the cylindrical Lagrangians $L \subset X$ and positive isotopies thereof $L \leadsto L^+$ considered in the wrapping category of $X$ do not intersect $K$ at infinity and by definition do not touch $\partial X$; hence, these isotopies stay disjoint from $\Nbd^Z (\partial X)$ (by shrinking this neighbourhood if necessary). As a result, the usual invariance arguments (cf. \cite[\S 3.3]{ganatra2020covariantly}) imply that the continuation maps $HF^*(K,L^+) \to HF^*(K,L)$ are defined and induce isomorphisms. 
This can be reformulated as saying that for any cone $cone(c)$ of any continuation element $c \in C$, $CF^*(K,cone(c))$ is acyclic. It follows from 
\cite[Lemma 3.13]{ganatra2020covariantly} that the natural map of $\cO'$-modules $CF^*(K,-) \to \hom_{\W'}(K,-)$ is a quasi-isomorphism; hence, the same is true of the restriction of these modules to $\cO$. In particular, $CF^*(K,-) \cong CW^*(K,-)|_{\cO}$. Going forward, we therefore refer to the latter module $CW^*(K,-)$ over $\W(Y)$ simply as $CF^*(K,-)$.

Finally, to check independence of choices, suppose we have made two constructions of directed categories $\cO'_1 = \langle \Z, \cO_1 \rangle$ and $\cO'_2 = \langle \Z, \cO_2 \rangle$, for same or possibly different $\cO_1$ and $\cO_2$ computing the same wrapped Fukaya category, and for different choices of Floer data for tuples containing $K$ (to extend to $\cO_i'$); which give different modules $CF^*(K,-)$ over $\cO_i[C_i^{-1}]$ for $i=1,2$. We can include $\cO_1 \coprod \cO_2$ into a sufficiently wrapped (in the sense of \cite[Definition 2.15]{ganatra2018sectorial}, see also \cite[Proposition 3.39]{ganatra2020covariantly}) $\cO_3$ consisting of Lagrangians transverse to $K$, and then build $\cO_3':= \langle \Z, \cO_3\rangle$ by inductively extending the previously chosen Floer data for tuples containing $K$ involving $\cO_1$ and $\cO_2$ separately. The construction of the previous paragraph applied to $\cO_3'$ then produces a module $CF^*(K,-)$ which restricts along the quasi-isomorphisms $\cO_i[C_i^{-1}] \stackrel{\simeq}{\to} \cO_3[C_3^{-1}]$ ($i=1,2$) to the two previously constructed modules, therefore implying they are isomorphic. 
\end{proof}

The modules constructed in \cref{arclikemodule} satisfy a nice locality property with respect to Liouville sector inclusions in the following sense:

\begin{lem}\label{arclikerestriction}
If $K$ is an arclike Lagrangian in $X$, and $Y \subset X$ is a Liouville subsector such that $K \cap Y$ is also an arclike Lagrangian in $Y$, then there is an isomorphism of $\W(Y)$-modules between $i^*CF^*(K,-)$ (where $i: \W(Y) \to \W(X)$ is the inclusion functor and $CF^*(K,-)$ is constructed by applying \Cref{arclikemodule} to $K \subset X$) and $CF^*(K \cap Y,-)$ (constructed by applying \Cref{arclikemodule} to $(K \cap Y) \subset Y$).
\end{lem}
\begin{proof}
    We work with a chain-level model for $\W(X):=\cO_X[C_X^{-1}]$ for which the chain-level inclusion functor $i: \W(Y) \to \W(X)$ comes from a directed subcategory $\cO_Y \subset \cO_X$ as in \cite[\S 3.6]{ganatra2020covariantly}. 
    In the construction of $CF^*(K,-)$ in \Cref{arclikemodule}, we can arrange (given the independence of the outcome from choices of Floer data up to isomorphism) that the Floer datum used for any tuple $(K, L_0, \ldots, L_k)$ with $L_0, \ldots, L_k \in \cO_Y$  coincides, after restricting to $Y$, with the Floer datum used for the tuple $(K \cap Y, L_0, \ldots, L_k)$ and, in particular, makes $\pi = R + i I: \Nbd^Z \partial Y \to \C_{\Re \geq 0}$ $J$-holomorphic for an $I$-function for which $K \cap Y$ is $I$-arclike. \Cref{openmappingboundary} therefore implies that the associated $\ainf$ operations for such tuples agree. Hence, we obtain the isomorphism from a chain level equality of $\cO_Y$ modules $CF^*(K \cap Y,-) = CF^*(K,-)|_{\cO_Y}$. By \Cref{arclikemodule}, $CF^*(K,-)$ is, as an $\cO_X$-module, isomorphic to the restriction of the corresponding (localized) $\W(X)$ module. Hence, we learn that the corresponding localized $\W(X)$ module restricts along $\cO_Y \to \cO_X \to \W(X)$ to $CF^*(K \cap Y,-)$. Since the restriction from $\W(X)$ to $\cO_Y$ factors through the localization $\cO_Y \to \W(Y)$, it follows that the $\W(X)$ module $CF^*(K,-)$ restricts to the corresponding (localized) $\W(Y)$ module $CF^*(K \cap Y,-)$.
\end{proof}

\begin{rmk}\label{independenceofpi}
    As constructed in \cref{arclikemodule}, the module $CF^*(K,-)$ over $\W(X)$ appears to depend implicitly on a fixed choice of $\pi$  for which $K$ is $\Im \pi$-arclike. However, thanks to \cref{arclikerestriction}, one can see that $CF^*(K,-)$ is independent of this choice of $\pi$ along the same lines as the discussion above \cite[Convention 3.1]{ganatra2020covariantly} (which addressed the independence of $\W(X)$ from $\pi$). The argument is as follows: given constructions $CF_1^*(K,-)$ over $\W_1(X)$ using $\pi_1 = R_1 + i I_1$ and $CF_2^*(K,-)$ over $\W_2(X)$ using $\pi_2 = R_2+iI_2$, one observes that there exists a trivial inclusion $X^{-\delta} \subset X$ for which $K \cap X^{-\delta}$ is still arclike (e.g., image of a small inward flow by either of the $X_{I_i}$). Now \cref{arclikerestriction} tells us that $CF_1^*(K,-)$ and $CF_2^*(K,-)$ both restrict to the same $CF^*(K \cap X^{-\delta},-)$ (constructed using any $\pi$ for which $K \cap X^{-\delta}$ is $\Im\pi$-arclike). Since trivial inclusions are quasi-equivalences \cite[Lemma 3.41]{ganatra2020covariantly}, it follows that the induced quasi-equivalence $\W_1(X) \stackrel{\sim}{\leftarrow} \W(X^{-\delta}) \stackrel{\sim}{\rightarrow} \W_2(X)$ identifies $CF_1^*(K,-)$ up to isomorphism with $CF^*_2(K,-)$. 
\end{rmk}

\begin{rmk}
    Let us look at an example where the above locality could possibly fail if $K$ is allowed to be non-compact (even if it has empty boundary). Consider $K = $ a cotangent fiber inside $Y = T^*[0,1]$, included into $X = T^*S^1$. Evidently, making the same choices as before $CF^*_X(K,-)$ as an $\cO_X$ module restricts along $\cO_Y \subset \cO_X$ to $CF^*_Y(K \cap Y,-) = CF^*_Y(K,-)$ as an $\cO_Y$ module. However, if $CW^*_X(K,-)$ (the localization of the former module) were to restrict to $CW^*_Y(K,-)$, we would learn (by testing against a fiber) that the self-wrapped Floer homology of a cotangent fiber in $T^*[0,1]$ and $T^*S^1$ agree which is false.  What fails in the above argument is that the $\cO_X$ module $CF_X^*(K,-)$ is not equivalent to its localization $CW^*_X(K,-)|_{\cO_X}$, hence arguments involving restricting $\cO_X$ modules don't necessarily descend to localized modules.
\end{rmk}

Now, let $Y \subset X$ be a subsector and $K \subset X$ a cylindrical Lagrangian such that $K^{in}:=K \cap Y$ is an arclike (and, in particular, compact) Lagrangian. Let us further assume that the Legendrian $\partial_\infty K$, which is therefore contained in $\partial_{\infty} X \backslash \partial_{\infty} Y$, admits a cofinal wrapping which does not cross $\partial\partial_\infty Y$.
Note that $K^{in}$, an arclike Lagrangian in $Y$, induces a module $CF^*(K^{in}, -)$ over $\W(Y)$ by \Cref{arclikemodule}. There is another module over $\W(Y)$ induced by $K$, namely the Yoneda module over $K$ (a $\W(X)$ module) restricted to $\W(Y)$ along the inclusion functor, $i^*\hom_{\W(X)}(K,-)$. The next result says that these modules are equal:
\begin{lem}\label{arclikesubsector}
In the situation above, there is an isomorphism of $\W(Y)$-modules between
$i^*\hom_{\W(X)}(K,-)$ and $CF^*(K^{in},-)$ (the module constructed by applying \Cref{arclikemodule} to $K^{in} \subset Y$).
\end{lem}
\begin{proof}

We implicitly work with a chain-level model for $\W(X):=\cO_X[C_X^{-1}]$ for which the chain-level inclusion functor $i: \W(Y) \to \W(X)$ comes from a directed subcategory $\cO_Y \subset \cO_X$. We choose $\cO_Y$ arbitrarily to consist of objects transversal to $K^{in}$.
Next, we choose $\cO_X$ by first appending $K$ to be greater than all elements of $\cO_Y$, and then continuing to add elements so that $\cO_X$ contains every isotopy class of Lagrangians in $X$ and is sufficiently wrapped in the sense of \cite[Definition 2.15]{ganatra2018sectorial} (see also \cite[Proposition 3.39]{ganatra2020covariantly}), meaning every Lagrangian $L$ added after $K$ possesses a cofinal wrapping ($L = L^0 \leadsto L^1 \leadsto \cdots$) in $X$ which is also cofinal in the poset $\cO_X$ (an ``$\cO_X$-wrapping sequence''). See \cite[\S 3.6]{ganatra2020covariantly} for one such construction. For $K$, we can also specifically arrange for its $\cO_X$-wrapping sequence $\{K^i\}$ to be induced by a cofinal collection of draggings (in the sense of \cite[Lemma 2.2]{ganatra2018sectorial})
of $K$ along the hypothesized cofinal-in-$\partial_\infty X$ wrapping $(\partial_\infty K)^t \subset (\partial_\infty X \backslash \partial_\infty Y)$.
 
Finally, in order to construct the module $CF^*(K^{in},-)$, we append a maximal element $K^{in}$ to $\cO_Y$ as in \Cref{arclikemodule} and call the result $\cO_Y'$. When choosing Floer data for $\cO_Y'$ and respectively $\cO_X$, we can arrange for any sequence of elements $L_0 > \cdots > L_k$ in $\cO_Y$ (contained in $Y$) that the Floer data chosen to define the $\ainf$ structure maps for the sequence $(K^{in}, L_0, \ldots, L_k)$ in $\cO_Y'$ (as in \Cref{arclikemodule}) and the sequence $(K, L_0, \ldots, L_k)$ in $\cO_X$
agree within $Y$ and use an almost-complex structure $J$ making $\pi = R + i I: \Nbd^Z \partial Y \to \C_{\Re \geq 0}$ $J$-holomorphic, for a $\pi$ with $I$ making $K^{in}$ arclike (this is possible because these tuples are equal when intersected with $Y$, and because each $L_i$ by definition avoids $\partial Y$ ). For this choice of Floer data, \Cref{openmappingboundary} applied to $Y$ implies that there is a chain level equality of $\cO_Y$ modules $CF^*(K^{in},-) = CF^*(K \cap Y,-) = CF^*(K,-)|_{\cO_Y}$.

By cofinality of the wrapping $\{K^i\}$, there is an isomorphism of $\cO_X$-modules
$\hom_{\W(X)}(K, -) \cong\varinjlim_i \hom_{\W(X)}(K^i,-)  \cong \varinjlim_i
\hom_{\cO_X}(K^i,-)$ (here we define $\lim_i \mathcal{M}_i$ for a sequence of
modules $\mathcal{M}_0 \to \mathcal{M}_1 \to \cdots$ as in
\cite[(3.27)]{ganatra2020covariantly}). On the other hand, when restricted to
the subcategory $\cO_Y$, there is a further isomorphism 
\[ 
    CF^*(K,-)|_{\cO_Y} = \hom_{\cO_X}(K,-)|_{\cO_Y} = \hom_{\cO_X}(K^0,-)|_{\cO_Y} \stackrel{\cong}{\to} \varinjlim_i \hom_{\cO_X}(K^i,-)|_{\cO_Y},
\]
as the isotopies $\partial_{\infty}K^i \leadsto \partial_{\infty}K^{i+1}$ don't cross into $Y$ by hypothesis, hence don't cross $\partial_{\infty} L$ for any $L \subset Y$, implying that $HF^*(K^i,L) \stackrel{\cong}{\to} HF^*(K^{i+1},L)$ for each $i$ (see, e.g., \cite[Lemma 3.26]{ganatra2020covariantly}). From the previous paragraph, we know that as $\cO_Y$-modules, $CF^*(K,-) = CF^*(K^{in},-)$. So all together we obtain an isomorphism of $\cO_Y$-modules $i^*\hom_{\W(X)}(K, -)|_{\cO_Y} \cong CF^*(K^{in},-)$. 
Now, the $\cO_Y$-module $i^*\hom_{\W(Y)}(K,-)|_{\cO_Y}$ is by construction $C_Y$-local, i.e., pulled back from the $\W(Y) = \cO_Y[C_Y^{-1}]$ module $i^*\hom_{\W(Y)}(K,-)$, as is $CF^*(K^{in},-)$ by \Cref{arclikemodule}. Given that the pullback to $\cO_Y$ induces a fully faithful functor from $\W(Y)$ modules to $\cO_Y$ modules, we conclude that the $\W(Y)$-modules $i^*\hom_{\W(Y)}(K,-)$ and $CF^*(K^{in},-)$ (constructed using \Cref{arclikemodule}) are isomorphic as desired.
\end{proof}

For the next result, let us return to the setting of an arclike Lagrangian $K$ in a Liouville sector $X$. 
Assume $X$ has been already deformed to have exact boundary as in \cref{boundarydeformation} so that there is a boundary sector inclusion of $F \times T^*[0,1]$ for any extended $\alpha=1$-defining function. 
Although $K$ itself is not an object of $\W(X)$ (because it doesn't avoid $\partial X$), the module $CF^*(K,-)$ defined in \Cref{arclikemodule} is in fact representable by an object in $\W(X)$:

\begin{prop}\label{negativepushoffrep}
    The module $CF^*(K,-)$ is representable by the negative pushoff  $K^- \in \W(X)$ of $K$ (defined in \S \ref{subsec:arclike}). That is, the $\W(X)$-modules $CF^*(K,-)$ and $\hom_{\W(X)}(K^-, -)$ are isomorphic.
\end{prop}
\begin{proof}
    Let $X^{-\delta}$ denote a sufficiently small inward pushoff of $X$ (by definition this is the image of $X$ under the flow generated by $-X_I$ for $I$ a fixed defining function for $X$, cf., \cite[Proof of Lemma 3.4]{ganatra2018sectorial}). Since the inclusion functor $\W(X^{-\delta}) \to \W(X)$ is a quasi-equivalence \cite[Lemma 3.41]{ganatra2020covariantly}, it suffices to prove the stated isomorphism of modules after restriction to $\W(X^{-\delta})$.
    We can arrange the construction of the negative pushoff $K^-$ (which is forward stopped by \Cref{negativepushoffforwardstopped}) to satisfy the following conditions by performing the relevant modifications in a $\eff \times T^*[0,1]$ subsector contained in a sufficiently small neighbourhood of $\partial X$:
    \begin{itemize}
        \item  $K^- \cap X^{-\delta}$ is arclike and equal to $K \cap X^{-\delta} = K \backslash (\gamma' \times \partial K)$ for a subsegment $\gamma' \subset \gamma$;
        \item there exists a cofinal wrapping $\{(K^-)^t\}_{t \geq 0}$ in $X$ of  $K^-$ which does not cross into $X^{-\delta}$ at infinity and for which $(K^-)^t \cap X^{-\delta} = K \cap X^{-\delta}$.
    \end{itemize}

    \Cref{arclikesubsector}, applied to the inclusion $X^{-\delta} \subset X$ therefore implies that there is an isomorphism of $\W(X^{-\delta})$-modules between the restriction of the Yoneda module $\hom_{\W(X)}(K^-,-)|_{\W(X^{-\delta})}$ and $CF^*(K \cap X^{-\delta},-)$ (the latter module constructed by applying \Cref{arclikemodule} to $(K \cap X^{-\delta})$ in $X^{-\delta}$). On the other hand, since $K$ is arclike in $X$ and $K \cap X^{-\delta}$ is arclike in $X^{-\delta}$, \Cref{arclikerestriction} implies that there is an isomorphism of $\W(X^{-\delta})$-modules $CF^*(K \cap X^{-\delta},-) = CF^*(K,-)|_{\W(X^{-\delta})}$. 
    Putting these together, we obtain an isomorphism of $\W(X^{-\delta})$-modules $\hom_{\W(X)}(K^-,-)|_{\W(X^{-\delta})} = CF^*(K,-)|_{\W(X^{-\delta})}$, which as argued earlier completes the proof. 
\end{proof}
\begin{rmk}\label{negativepushoffindependence}
    As constructed, $K^-$ a priori depends on a choice of $\alpha=1$-extended defining function $I$ for which $K$ is $I$-arclike. However, given \cref{negativepushoffrep} and the independence of the module $CF^*(K,-)$ from choices (see \cref{arclikemodule} and \cref{independenceofpi}), Yoneda implies that $K^-$ is well-defined up to isomorphism in $\W(X)$.
\end{rmk}

\begin{rmk} \label{arclikedeformation}
    Suppose $K$ is $I$-arclike in $X$ but $\partial X$ is not exact. 
    Observe that the deformation $\lambda_X \leadsto \lambda_X'$ recalled in \cref{boundarydeformation} preserves $I$ and (if performed close enough to $\partial X$) the $I$-arclikeness of $K$. Under the resulting deformation invariance quasi-isomorphism $\cW(X) \cong \cW(X')$, we claim that the resulting arclike modules $CF^*(K,-; \lambda_X)$ and $CF^*(K,-; \lambda_X')$ are isomorphic. In particular, the representability of the latter module by $K^- \in \cW(X, \lambda_X')$ implies the representability of the former by an isomorphic object. To see the claim, let $X^{-\delta}$ again denote a sufficiently small inward pushoff of $X$ as in the above Lemma so that $K \cap X^{-\delta}$ is $I$-arclike in $X^{-\delta}$ and so that the deformation $\lambda_X \leadsto \lambda_X'$ is supported in $X \backslash (X^{-\delta})^\circ$. We therefore have $X^{-\delta} \hookrightarrow (X, \lambda_X')$, which by deformation invariance/invariance under trivial inclusions \cite[Lemma 3.41]{ganatra2020covariantly} gives us a zig-zag of quasi-equivalences $ \cW(X) \stackrel{\sim}{\leftarrow} \cW(X^{-\delta}) \stackrel{\sim}{\rightarrow} \cW(X,\lambda_X')$. Under these quasi-equivalences \cref{arclikerestriction} implies that $CF^*(K,-; \lambda_X)$ and $CF^*(K,-; \lambda_X')$ have isomorphic restrictions to $\cW(X^{-\delta})$, namely $CF^*(K \cap X^{-\delta},-)$. Therefore the induced quasi-equivalence $\cW(X) \simeq \cW(X,\lambda_X')$ carries $CF^*(K,-; \lambda_X)$ to $CF^*(K,-; \lambda_X')$ as desired.
\end{rmk}

\begin{proof}[Proof of \cref{thm:arclikeRestriction}]
    If the boundaries of $Y$ and $X$ are exact, then combining \Cref{arclikerestriction} with Proposition \ref{negativepushoffrep} applied to both $K$ in $X$ and $K^{in} = K \cap Y$ in $Y$ , we see that $i^*\hom_{\W(X)}(K^-,-)$ is isomorphic to the Yoneda module over the cylindrized negative pushoff of $K\cap Y$, $\hom_{\W(Y)}((K \cap Y)^-,-)$.

    In general, we may deform $Y$ near $\partial Y$ and $X$ near $\partial X$ as in \cref{boundarydeformation} (see also \cref{arclikedeformation}) without changing the arclike property of $K$ and $K \cap Y$, thereby reducing to the above case. 
\end{proof}

Now we return to the setting that $(\Ldom,\fo)$ is a sutured Liouville domain, $S$ the associated Liouville sector (deformed as in \cref{boundarydeformation}), $I:\Nbd^Z(\partial S) \to \R$ an extended $1$-defining function, and $\cup: \W(\eff) \to \W(S)$ the corresponding cup functor defined above the statement of \Cref{lem:def_capl}.
Let $L \subset S$ be an $I$-arclike Lagrangian. For what follows, we think of the boundary $\partial L \subset \partial S$ as a compact exact Lagrangian in $F$, either via restricting to the level set of $I$ containing $\partial L$ or by projection along the characteristic foliation. 
\begin{prop}
  \label{cor:capIsomorphism}
  In the situation described above, the pullback $\cup^* \mathcal{Y}_{L^-}$ is representable by $\partial L$; that is, there is an isomorphism of $\W(\eff)$-modules $\cup^* \mathcal{Y}_{L^-} \simeq \mathcal{Y}_{\partial L}.$
\end{prop}
\begin{proof}
    The subsector $\eff \times T^*[0,1] \subset S$ determined by $I$ has the property that the restriction of $I$ and its associated $R$-function $R$ are simply the projection to the factors of $T^*[0,1] = [0,1]_R \times \R_I$. It follows (see \cref{arcimage} and the discussion that follows) that
    by shrinking the subsector $\eff \times T^*[0,1] \subset S$ (or by flowing $L$), we may assume that $L \cap (\eff \times T^*[0,1]) = \partial L \times \gamma$ where $\gamma: [0,1] \hookrightarrow T^*[0,1]$ is of the form $\{ I = c\}$. By \Cref{lem:boundarystabilization}, we learn that $(L \cap (\eff \times T^*[0,1]))^- = (\partial L \times \gamma)^-$ is isomorphic to the stabilization $stab(\partial L)$ of $\partial L$.
    Let $i: \eff \times T^*[0,1] \subset S$ denote the embedding defined above the statement of \Cref{lem:def_capl}. Theorem \ref{thm:arclikerepresentable} implies that $i^* \mathcal{Y}_{L^-} \simeq \mathcal{Y}_{(L \cap (\eff \times T^*[0,1])^-}$ which we have just argued is isomorphic to $\mathcal{Y}_{stab(\partial L)}$. Now pull back by $stab$ to conclude, as $stab^*i^* = \cup^*$ and $stab^*\mathcal{Y}_{stab(\partial L)}$ is isomorphic to $\mathcal{Y}_{\partial L}$ by full-faithfulness of $stab$ (which is a special case of \cite[Theorem 1.5]{ganatra2018sectorial}).
\end{proof}

\begin{proof}[Proof of \Cref{lem:def_capl}]
    Apply Corollary \ref{cor:partial_adjoint} to \Cref{cor:capIsomorphism}. To arrange that $\cap$ has image contained in the compact Fukaya category, we apply \cref{rmk:modifyingg}.
\end{proof}

\section{Liouville structures on submanifolds of \texorpdfstring{$(\C^*)^n$}{ the complex torus}}\label{sec:liouville_struc}

In this section, we will consider a (real) submanifold $A$ of the complex manifold 
$$M^*_{\C^*} \cong M^*_\R \times M^*_{S^1}$$
for some lattice $M^*$. 

Let $\varphi:M^*_\R \to \R$ be a strictly convex, proper, bounded-below function. 
We will denote its pullback to $M^*_{\C^*}$ by $\varphi$ also. 
Being strictly convex, $\varphi$ is the potential for a K\"ahler form $\omega = d\lf$, with Liouville one-form $\lf = -d^c\varphi$. 

The main results of this appendix give various conditions under which the Liouville structures for different classes of K\"{a}hler potentials can be compared. A version of the comparison of Liouville structures for homogeneous potentials of weight 2 has also recently been obtained in \cite[Theorem 9.2]{spenko2024hms}.

\begin{lem}[cf. Section 4b of \cite{seidel2006biased}]\label{lem:hol_comp_hom}
Suppose that $A$ is a complex submanifold which admits a smooth toric compactification (Definition \ref{def:smtorcomp}), and $\varphi$ is a convex combination of a compactification K\"ahler potential (Definition \ref{def:comp_pot}) and a potential which is homogeneous of weight $2$ (Definition \ref{def:homogeneous}). 
Then $A \cap \{\varphi \le C\}$, equipped with the Liouville one-form $\lf|_A$, is a Liouville domain for sufficiently large $C$; as a consequence, the Liouville completion is independent of $C$ for $C$ sufficiently large. 
\end{lem}
\begin{proof}
Combines Lemmas \ref{lem:hol_crit} and \ref{lem:comp_con}, and Corollary \ref{cor:comb_ang} below.
\end{proof}

\begin{lem}\label{lem:change_comp_hom}
Suppose that $A$ is a complex submanifold admitting a smooth toric compactification, $\varphi_0$ is a compactification K\"ahler potential, and $\varphi_1$ a potential which is homogeneous of weight $2$. 
Then the corresponding Liouville completions $\widehat{(A,\lf_0)}$ and $\widehat{(A,\lf_1)}$ are Liouville isomorphic. 
\end{lem}
\begin{proof}
Consider the family of potentials $\{\varphi_s\}_{s \in [0,1]}$ linearly interpolating between $\varphi_0$ and $\varphi_1$. 
 
It follows from Lemmas \ref{lem:hol_crit} and \ref{lem:comp_con}, and Corollary \ref{cor:comb_ang}, that $A \cap \{\varphi_s \le C\}$ is a Liouville subdomain for all $s$, for $C$ sufficiently large. 
Thus we have a homotopy of Liouville domains associated to these potentials, which in turn induces an isomorphism of Liouville manifolds $\widehat{(A,\lf_0)} \cong \widehat{(A,\lf_1)}$ by \cite[Proposition 11.8]{cieliebakeliashberg2012}.
\end{proof}

\begin{lem}\label{lem:hom_Liouv}
Suppose that $A$ admits a smooth toric compactification, and is $\delta$-approximately holomorphic (Definition \ref{def:approx_hol}) for some $0<\delta <\pi/6$; and $\varphi$ is homogeneous of weight $2$.
Then $A \cap \{\varphi \le C\}$, equipped with the Liouville one-form $\lf|_A$, is a Liouville domain for sufficiently large $C$; as a consequence, the Liouville completion is independent of $C$.
\end{lem}
\begin{proof}
Follows from Lemmas \ref{lem:hom_crit} and \ref{lem:comp_con}, together with \cite[Proposition 11.8]{cieliebakeliashberg2012}.
\end{proof}

\begin{lem}\label{lem:change_tailoring}
Suppose that $(A_t)_{t \in [0,1]}$ is a family of submanifolds which admits a family of smooth toric compactifications, and is $\delta$-approximately holomorphic for some $0<\delta<\pi/6$; and $\varphi$ is homogeneous of weight $2$. 
Then the Liouville completions $\widehat{(A_0,\lf_0)}$ and $\widehat{(A_1,\lf_1)}$ are Liouville isomorphic. 

\end{lem}
\begin{proof}
The proof follows the pattern of that of \Cref{lem:change_comp_hom}, again using \cite[Proposition 11.8]{cieliebakeliashberg2012}. The key step is to show that the sublevel sets $A_t \cap \{ \varphi \le C\}$ are all Liouville subdomains, for $C$ sufficiently large. This follows from the family versions of Lemmas \ref{lem:hom_crit} and \ref{lem:comp_con}. 
\end{proof}

\subsection{Angles}

We recall some material from \cite{donaldson1996symplectic} and  \cite[Section 8]{cieliebak2007symplectic}. 

Given a finite-dimensional real inner product space $(V,\langle \cdot, \cdot \rangle)$, the angle between two non-zero vectors $x,y \in V$ is defined to be 
$$\angle(x,y) := \cos^{-1}\left(\frac{\langle x,y \rangle}{|x||y|}\right) \in [0,\pi].$$
The angle between a linear subspace $X \subset V$ and a vector $y \in V$ is
$$\angle(X,y) := \inf_{0 \neq x \in X} \angle(x,y) \in [0,\pi/2].$$
We remark that $\angle(X,y) = \angle(\pi^\perp y,y)$, where $\pi^\perp$ denotes orthogonal projection to $X$.\footnote{Proof: if $\angle(x,y) = \theta$, we may rescale $x$ so that $|x| = |y|\cos \theta$, then $|x-y| = |y|\sin \theta$. Thus the minimum of $\angle(x,y)$ is achieved when $|x-y|$ is minimal, which is precisely for $x = \pi^\perp(y)$.}
The angle between linear subspaces $X,Y \subset V$ is
$$\angle(X,Y) := \sup_{0 \neq y \in Y} \inf_{0 \neq x \in X} \angle(x,y) \in [0,\pi/2].$$

\begin{lem}\label{lem:equiv_angles}
Suppose that we have two inner products satisfying
$$c |x|^2_1 \le |x|^2_2 \le C|x|^2_1$$
for all $x \in V$, for some $c,C>0$. 
Then
$$\cos \angle_2(x,y) \ge 1 - \frac{C}{c}\left(1-\cos \angle_1(x,y)\right).$$
In particular, $\angle_1(x_j,y_j) \to 0$ if and only if $\angle_2(x_j,y_j) \to 0$.
\end{lem}
\begin{proof}
We may rescale $x$ and $y$ so that $|x|_1 = |y|_1$, then
$$\cos \angle_1(x,y) = 1 - \frac{|x-y|^2_1}{2|x|_1|y|_1}.$$
Then we have
\begin{align*}
\cos \angle_2(x,y) &= \frac{|x|_2^2 + |y|_2^2 - |x-y|_2^2}{2|x|_2|y|_2} \\
& \ge 1 - \frac{C|x-y|_1^2}{2c|x|_1|y|_1} \\
&= 1 - \frac{C}{c} \left(1-\cos \angle_1(x,y)\right)
\end{align*}
as required.
\end{proof}

Now suppose that $\omega$ is a linear symplectic structure on $V$, $J$ a compatible complex structure, and $\langle \cdot, \cdot \rangle = \omega(\cdot,J(\cdot))$ the corresponding inner product. 
The \emph{angle} of a real subspace $X \subset V$ is $\angle(V) := \angle(JV,V)$. 
Note that $\angle(V) = 0$ if and only if $V$ is $J$-holomorphic. 

This notion is closely related to the `K\"ahler angle' $\theta(X)$ associated to an oriented real linear subspace, which was introduced in \cite{donaldson1996symplectic}: in particular, in the case that $X$ is of codimension $2$, $\angle(X) = \min\{\theta(X), \pi - \theta(X)\}$ by \cite[Lemma 8.3 (d)]{cieliebak2007symplectic}.

\begin{lem}\label{lem:pair_orthog}
If $x \in X$ and $y \in X^\omega$ (the symplectic orthogonal complement), then
$$|\langle x, y \rangle| \le |x||y|\sqrt{2-2\cos\angle(X)}.$$
\end{lem}
\begin{proof}
There exists $\tilde{x} \in X$ with $\angle(Jx,\tilde{x}) \le \angle(X)$, and we may rescale it so that $|\tilde{x}| = |x|$. 
It follows easily that $|Jx - \tilde{x}| \le |x| \sqrt{2-2\cos\angle(X)}$. 
Thus
$$|\langle x, y \rangle| = |\omega(y,Jx)| = |\omega(y,Jx - \tilde{x})| \le |y||Jx-\tilde{x}| \le |y||x|\sqrt{2-2\cos\angle(X)} $$
as claimed.
\end{proof}

\begin{cor}\label{cor:ang_symp}
If $\angle(X) < \pi/2$, then $X$ is symplectic.
\end{cor}
\begin{proof}
By Lemma \ref{lem:pair_orthog}, $X \cap X^\omega = \{0\}$.
\end{proof}

\begin{lem}\label{lem:piomegabound}
Suppose that $\angle(X) < \pi/6$, and let $\pi^\omega:V \to X$ denote the symplectic orthogonal projection. Then for any $v \in V$ we have
$$|\pi^\omega(v)| \le \frac{|v|}{\sqrt{2\cos\angle(X) - 1}}.$$
\end{lem}
\begin{proof}
Let $x \in X$ and $y \in X^\omega$, so that $\pi^\omega(x+y) = x$. 
We have
\begin{align*}
|x+y|^2 &= |x|^2+|y|^2 + 2\langle x,y\rangle \\
&\ge |x|^2 + |y|^2 - |x||y| \sqrt{2-2\cos\angle(X)} \quad\text{by Lemma \ref{lem:pair_orthog}}\\
& = |x|^2 + \left(|y| - |x|\sqrt{2-2\cos\angle(X)}\right)^2 - |x|^2(2-2\cos\angle(X))\\
& \ge (2\cos\angle(X) - 1)|x|^2
\end{align*}
from which the result follows. 
\end{proof}

\subsection{\texorpdfstring{Conditions on $A$ and $\varphi$}{Conditions on A or phi}}

A choice of basis for $M^*$ determines coordinates $x_j$ on $M^*_\R$, $\theta_j$ on $M^*_{S^1}$, and holomorphic coordinates $z_j = \exp(x_j + i \theta_j)$ on $M^*_{\C^*}$.

Let $E \in T(M^*_\R)$ be the Euler vector field (equal to $\sum_j x_j \partial_{x_j}$ in coordinates). 
We shall also use $E$ for the corresponding vector field in $T(M^*_{\C^*})$.
If $\varphi:M^*_\R \to \R$ is strictly convex, let $\langle \cdot,\cdot \rangle_\varphi$ denote the K\"ahler metric associated to the K\"ahler form with potential $\varphi$: it is the product of the Riemannian metrics on $M^*_\R$ and $M^*_{S^1}$ given by the Hessian of $\varphi$. 

\begin{defn}\label{def:homogeneous}
We say that $\varphi:M^*_\R \to \R$ is \emph{$k$-homogeneous (near $\infty$)} if  $E(\varphi) = k \cdot \varphi$ (outside a compact set).
\end{defn}

We will call the Riemannian metric associated to a $2$-homogeneous strictly convex function, a $2$-homogeneous metric. 
For example, a Euclidean metric on $M^*_\R$ induces a $2$-homogeneous metric on $M^*_{\C^*}$ (we take $\varphi(x) = |x|^2$). 
A straightforward compactness argument shows that any two $2$-homogeneous metrics are uniformly equivalent:
$$c |v|_{\varphi_1}^2 \le |v|_{\varphi_2}^2 \le C|v|_{\varphi_1}^2$$
for any tangent vector $v$ to $M^*_{\C^*}$, for some constants $c,C>0$ depending on the $\varphi_i$. 

\begin{defn}\label{def:asymp_con}
We say that the submanifold $A \subset M^*_{\C^*}$ is \emph{asymptotically cylindrical} if 
$$\lim_{a \to \infty} \angle(T_aA,E_a) = 0,$$
where the angle is taken with respect to any $2$-homogeneous metric. 
More precisely, for every $\delta>0$, there exists a subset $U_\delta \subset A$ with compact complement, such that for all $a \in U_\delta$, $\angle(T_aA,E_a) < \delta$. 
(Note that the choice of $2$-homogeneous metric doesn't matter, by Lemma \ref{lem:equiv_angles}, and our observation that any two such metrics are uniformly equivalent.)

More generally, we say that a family of submanifolds $(A_t)_{t \in [0,1]}$ is \emph{uniformly asymptotically cylindrical} if for every $\delta>0$, there exists $U_\delta \subset M^*_{\C^*}$ with compact complement, such that for all $t$ and for all $a \in A_t \cap U_\delta$, $\angle(T_a A_t,E_a) < \delta$.
\end{defn}

\subsection{A criterion for holomorphic submanifolds}

The Liouville vector field $Z$ corresponding to the Liouville one-form $\lf$ is  equal to $\nabla \varphi$, where the gradient is taken with respect to the K\"ahler metric induced by $\varphi$. 
We will always assume that $\omega_A = \omega|_A$ is a symplectic form. 
It has Liouville one-form $\lf_A = \lf|_A$, and Liouville vector field $Z_A$. 
It is immediate from the definitions that
$$Z_A = \pi^\omega(\nabla \varphi),$$ 
where $\pi^\omega:T(M^*_{\C^*}) \to TA$ is the $\omega$-orthogonal projection.

Suppose that $A$ is a holomorphic submanifold, i.e., $\angle(TA) = 0$. 
Then the $\omega$-orthogonal projection $T(M^*_{\C^*}) \to TA$ coincides with the orthogonal projection.

\begin{lem}\label{lem:hol_crit}
Suppose that $A$ is an asymptotically cylindrical holomorphic submanifold, and $\varphi_s: M^*_\R \to \R$ is a smooth family of potentials parametrized by $s \in [0,1]$ such that outside of some compact set, $\angle(E,\nabla \varphi_s) \le c$ for all $s$, for some $c<\pi/2$ independent of $s$, where the angle and the gradient $\nabla$ are taken with respect to some fixed $2$-homogeneous metric.
Then for sufficiently large $C$, the sublevel sets $A \cap \{ \varphi_s \le C\}$ are Liouville domains, for all $s$.
\end{lem}
\begin{proof}
We must show that $Z_{A,s}(\varphi_s) > 0$ along $\{\varphi_s = C\}$. 
Using the fact that $A$ is holomorphic, we have that $Z_{A,s}(\varphi_s) = |Z_{A,s}|^2$, so it suffices to prove that $Z_{A,s} \neq 0$. 
This is equivalent to $d\varphi_s|_A \neq 0$, which is equivalent to $\angle(T_aA,\nabla \varphi_s) \neq \pi/2$. 
We have $\angle(E,\nabla \varphi_s) \le c < \pi/2$ outside a compact set, and $\angle(T_aA,E_a) \to 0$ as $a \to \infty$. 
Hence, for $a$ outside some compact set $K$, we have $\angle(T_a A,\nabla \varphi_s) \le c' < \pi/2$ for all $s$. 
Now $\varphi(K \times [0,1])$ is a compact subset of $\R$, so it suffices to take $C$ larger than its supremum. 
\end{proof}

\begin{rmk}
The proof of \cref{lem:hol_crit} works for any continuous family of potentials parametrized by a compact topological space, but we only need the case of a point and an interval.
\end{rmk}

\subsection{A criterion for \texorpdfstring{$2$-homogeneous}{2-homogeneous} potentials}

Let $\varphi:M^*_\R \to \R$ be strictly convex, proper, bounded below, and $2$-homogeneous near $\infty$. 
We will use the associated $2$-homogeneous metric on $M^*_{\C^*}$ below. 

\begin{defn}\label{def:approx_hol}
Let $0 \le \delta < \pi/2$. We say that $A$ is \emph{$\delta$-approximately holomorphic} if $\angle(TA) < \delta$. (Note that $\angle(TA) = \angle_\varphi(TA)$ depends on $\varphi$.)
\end{defn}

Note that if $A$ is $\delta$-approximately holomorphic, then $\omega_A := \omega|_A$ is symplectic by Corollary \ref{cor:ang_symp}, with Liouville one-form $\lf_A = -d^c\varphi|_A$.

\begin{lem}[Proposition 2.7 of \cite{zhou2020lagrangian}]\label{lem:ZpropE}
We have $\angle(Z, E) = 0$ in the locus where $\varphi$ is $2$-homogeneous (in particular, near $\infty$).
\end{lem}

\begin{lem}\label{lem:nablaphiinf}
We have
\begin{enumerate}
\item \label{it:1} $\lim_{a \to \infty} \varphi(a)/|E_a| = \infty$, and
\item \label{it:2} $\lim_{a \to \infty} |\nabla \varphi| = \infty$
\end{enumerate}
(where `$\lim_{a \to \infty}$' is to be interpreted as in Definition \ref{def:asymp_con}).
\end{lem}
\begin{proof}
\eqref{it:1} follows from the fact that along any ray, $\varphi$ is strictly convex and eventually quadratic, while $|E|$ is eventually linear.

For \eqref{it:2}, note that as $\varphi$ is $2$-homogeneous near $\infty$, we have 
\begin{align*}
E (\varphi) &= 2\varphi \\
\Rightarrow \langle \nabla \varphi, E \rangle &= 2\varphi \\
\Rightarrow |\nabla \varphi| |E| &= 2\varphi \quad \text{by Lemma \ref{lem:ZpropE}} \\
\Rightarrow |\nabla \varphi| &= \frac{2 \varphi}{|E|}.
\end{align*}
The result now follows from \eqref{it:1}.
\end{proof}

\begin{lem}\label{lem:hom_crit}
Suppose that $\varphi$ is homogeneous of weight $2$ near $\infty$, and $(A_t)_{t \in [0,1]}$ is a family of submanifolds which are uniformly asymptotically cylindrical and $\delta$-approximately holomorphic for some $0<\delta<\pi/6$. 
Then $Z_{A_t}(\varphi)>0$ outside a compact subset of $M^*_{\C^*}$. 
\end{lem}
\begin{proof} 
Note that as $A_t$ is uniformly asymptotically cylindrical, Lemma \ref{lem:ZpropE} implies that
$$\lim_{a \to \infty} \angle(\pi^\perp_t \nabla \varphi,\nabla \varphi) = \lim_{a \to \infty} \angle(TA_t,\nabla \varphi) = 0$$
uniformly in $t$, from which it follows that 
\begin{align}
\label{eq:perplim}\lim_{a \to \infty} \frac{|\nabla \varphi - \pi^\perp_t \nabla \varphi|}{|\nabla \varphi|} &= \lim_{a \to \infty} \sin\angle(\nabla \varphi,\pi^\perp_t\nabla \varphi) = 0, \qquad \text{and}\\
\label{eq:perplim2}\lim_{a \to \infty} \frac{|\pi^\perp_t \nabla \varphi|}{|\nabla \varphi|} &= \lim_{a \to \infty} \cos\angle(\nabla \varphi,\pi^\perp_t\nabla \varphi) = 1
\end{align}
 uniformly in $t$.

By Lemma \ref{lem:piomegabound}, we have
$$|\pi^\omega_t(\nabla \varphi - \pi^\perp_t \nabla \varphi)| \le \frac{|\nabla \varphi - \pi^\perp_t \nabla \varphi|}{\sqrt{2\cos \delta - 1}}.$$
Combining with \eqref{eq:perplim}, it follows that
\begin{equation}
\label{eq:perplim3}
\lim_{a \to \infty}\frac{|\pi^\omega_t(\nabla \varphi - \pi^\perp_t \nabla \varphi)|}{|\nabla \varphi|} = 0
\end{equation}
uniformly in $t$.

We now have
\begin{align*}
\frac{Z_A(\varphi)}{|\nabla \varphi|^2} &= \frac{\langle \pi^\omega_t(\nabla \varphi),\nabla \varphi \rangle}{|\nabla \varphi|^2}\\
&= \frac{\langle \pi^\omega_t(\nabla \varphi - \pi^\perp_t \nabla \varphi) + \pi^\perp_t \nabla \varphi,\nabla \varphi\rangle}{|\nabla \varphi|^2}\\
& \ge \frac{|\pi^\perp_t \nabla \varphi|^2}{|\nabla \varphi|^2} - \frac{|\pi^\omega_t(\nabla \varphi - \pi^\perp_t \nabla \varphi)|}{|\nabla \varphi|}.
\end{align*}
Applying \eqref{eq:perplim2} and \eqref{eq:perplim3}, we see that the RHS $\to 1$ uniformly in $t$, as $a \to \infty$. 
The result now follows from Lemma \ref{lem:nablaphiinf} \eqref{it:2}.
\end{proof}

\subsection{Submanifolds admitting a toric compactification are asymptotically cylindrical}

Let $\Sigma$ be a simplicial fan in $M^*_\R$, giving rise to a toric orbifold $Y_\Sigma$. 
The complement of the toric boundary divisor in $Y_\Sigma$ is the complex manifold $M^*_{\C^*}$. 

\begin{defn}\label{def:smtorcomp}
A \emph{smooth toric compactification} of a submanifold $A \subset M^*_{\C^*}$ is a smooth submanifold $X \subset Y_\Sigma$, for some simplicial fan $\Sigma$, such that $A = X \cap M^*_{\C^*}$, $X$ avoids the orbifold points, and the intersection of $X$ with every stratum of the toric boundary divisor is transverse. 
\end{defn}

\begin{lem}\label{lem:comp_con}
If $(A_t)_{t \in [0,1]}$ is a family of submanifolds which admits a family $(X_t)_{t \in [0,1]}$ of smooth toric compactifications, then it is uniformly asymptotically cylindrical.
\end{lem}
\begin{proof}
Suppose to the contrary that there exists a sequence of points $a_j \in A_{t_j}$ converging to $\infty$, with $\angle(T_{a_j}A_{t_j},E) > \delta>0$.
Taking a subsequence, we may assume that the $(a_j,t_j)$ converge to a point $(p,s)$ with $p$ in the toric boundary and $s \in [0,1]$. 

Suppose that $p$ lies in the torus orbit $Y_\sigma$ corresponding to the cone $\sigma$ of $\Sigma$. 
As $Y$ is smooth along $Y_\sigma$ by hypothesis, we may choose a basis $(m_j)_{j=1}^n$ of $M^*_{\C^*}$ such that $m_1,\ldots,m_k$ span $\sigma$. 
The choice of basis determines coordinates $x_j$ on $M^*_\R$, $\theta_j$ on $M^*_{S^1}$, and holomorphic coordinates $z_j = \exp(x_j+i\theta_j)$ on $M^*_{\C^*}$. 
These extend to coordinates $(z_j)_{j=1}^n \in \C^k \times (\C^*)^{n-k}$ on a Zariski neighbourhood of $Y_\sigma$, with $Y_\sigma = \{z_1 = \ldots = z_k = 0\}$. 

In a neighbourhood of $p$ in $Y$ we may choose a real coordinate system $(u_1,v_1,\ldots,u_k,v_k,w_{2k+1},\ldots,w_{2n})$ such that:
\begin{itemize}
    \item $p = (0,\ldots,0)$ in these coordinates;
    \item $z_j = u_j+iv_j$ for $j=1,\ldots,k$;
    \item $(w_{2k+1},\ldots,w_{2n})$ are related to $(x_{k+1},\theta_{k+1},\ldots,x_n,\theta_n)$ by an affine-linear change of variables;
    \item $(u_1,v_1,\ldots,u_k,v_k,w_{2k+1},\ldots,w_{\dim X})$ define a coordinate system on a neighbourhood of $p \in X_s$. 
\end{itemize}
In these coordinates, $X_t$ can then be written as the graph of a smooth function
$$F_t:U \to \R^{2n-\dim X}$$
for some $U \subset \R^{\dim X}$ an open neighbourhood of the origin, and all $t$ sufficiently close to $s$.

Now consider the Euler vector field $v = \sum_{j=1}^k u_j \partial/\partial u_j + v_j \partial/\partial v_j$ on $U$. 
It admits the following lift to $A_t$:
$$\tilde v_t = v + \sum_{j=\dim X + 1}^{2n} v(F_t^j) \frac{\partial}{\partial w_j}.$$
Thus we have
$$E - \tilde{v}_t = \sum_{j>k} x_j \frac{\partial}{\partial x_j} - \sum_{j=\dim X + 1}^{2n} v(F_t^j) \frac{\partial}{\partial w_j}.$$
We claim that $|E-\tilde{v}_{t_j}|$ is bounded, as $(a_j,t_j) \to (p,s)$. 
Indeed, this follows from the fact that all of the following quantities are bounded: $|\partial/\partial x_j|$, $|\partial/\partial \theta_j|$ (these are bounded for any $2$-homogeneous metric, by compactness); $|\partial/\partial w_j|$ (as these are a linear combination of $\partial/\partial x_j$ and $\partial/\partial \theta_j$); $x_j$ for $j>k$ (as $z_j(p) \neq 0$); $v(F_t^j)$ (this follows as $v$ is smooth and vanishes at $0$, so $v(F_t^j)$ is smooth and vanishes at $0$). 

On the other hand, $|E|$ is unbounded (as it is eventually linear with positive slope along any ray). 
It follows that $\angle(E,\tilde{v}_t) \to 0$ as $a \to \infty$; as $\tilde{v}_t \in TA_t$ by construction, the claim follows.
\end{proof}

\subsection{Compactification K\"ahler potentials}

\begin{defn}\label{def:comp_pot}
Let $\Sigma$ be a simplicial fan in $M^*_\R$, determining a toric orbifold $Y_\Sigma$ with toric boundary divisor $D$. 
We say that a strictly convex, proper, bounded below function $\varphi: M^*_\R \to \R$ is a \emph{compactification K\"ahler potential (for $\Sigma$)} if the corresponding K\"ahler form extends to an orbifold K\"ahler form on $Y_\Sigma$, so that $(\omega,\varphi)$ define a relative K\"ahler form on $(Y_\Sigma,D)$ in the sense of \cite[Definition 3.2]{Sheridan2017}, and $d\varphi$ extends to a moment map $d\varphi: Y_\Sigma \to M_\R$ for the action of $M^*_{S^1}$ on $Y_\Sigma$. 
\end{defn}

\begin{lem}
\label{lem:comp_pot_exist}
Let $\sum_{i \in \Sigma(1)} \kaehl_i D_i$ be an ample $\R$-divisor on $Y_\Sigma$, with all $\kaehl_i>0$. 
Then there exists a compactification K\"ahler potential $\varphi$ for $\Sigma$ with linking numbers $\kaehl_i$. 
The image of the moment map $d\varphi$ is the moment polytope 
$$\Delta_\kaehl = \{m \in M_\R: v_i(m) \ge -\kaehl_i \text{ for all $i \in \Sigma(1)$}\} \subset M_\R.$$
\end{lem}
\begin{proof}
The claim follows from Guillemin's formula for toric K\"{a}hler potentials \cite{guillemin1994kaehler} extended to orbifolds in \cite{calderbank2001guillemin}.
\end{proof}

\begin{lem}\label{lem:comp_ang}
Suppose that $\varphi$ is a compactification K\"ahler potential. Then $\angle(E,\nabla \varphi) \le c$ near $\infty$ for some $c<\pi/2$, where the angle and gradient are taken with respect to some $2$-homogeneous metric.
\end{lem}
\begin{proof}
It suffices to show that $E(\varphi)/(|d\varphi||E|)$ is bounded below near $\infty$. 
Note that $|d\varphi|$ is bounded above (as it lies in $\Delta$ which is compact), so it suffices to bound $E(\varphi)/|E|$ below.

Suppose, to the contrary, that there exists a sequence $p_i \in M^*_\R$ with $E(\varphi)/|E| \to d$ for some $d \le 0$. 
We consider the $p_i$ as lying inside $M^*_\R \times \{1\} \subset M^*_\R \times M^*_{S^1} \subset Y_\Sigma$. 
Taking a subsequence, we may assume that the $p_i$ converge to a point $y$ lying in some orbit $Y_\sigma$. 
Then $d\varphi_{p_i}$ converges to $d\varphi_y$, which lies on a facet of $\Delta_\kaehl$ contained in the affine space where $\langle m,v_j \rangle = -\kaehl_j$ for all rays $v_j$ of $\sigma$. 

We choose a basis $(m_j)_{j=1}^n$ for $M^*_\R$ such that $m_j = v_j$ for $j=1,\ldots,k$, where $v_j$ are the primitive vectors spanning $\sigma$ (this need not come from a basis for $M^*$). 
The basis induces coordinates $x_j$ on $M^*_\R$, with $x_j(p_i) \to -\infty$ for $j \le k$ and $x_j(p_i)$ bounded for $j > k$. 
Thus we have, as $i \to \infty$:
\begin{align}
\nonumber\frac{E(\varphi)}{|E|} & = \frac{ \sum_{j=1}^n x_j \frac{\partial \varphi}{\partial x_j}}{|E|} \\
\nonumber &\to \frac{\sum_{j=1}^n x_j d\varphi_y(m_j)}{|E|}\\
\label{eq:des_bd} & \to \frac{-\sum_{j=1}^k x_j \kaehl_j}{|\sum_{j=1}^k x_j \partial_{x_j}|}.
 \end{align}
In the last line, we have used the fact that $|E| \to \infty$, while $x_j$ is bounded for $j > k$, as is $d\varphi_y(m_j)$ for any $j$. 

Now observe that as $x_j \to -\infty$ for $j \le k$, we may truncate our sequence so that $x_j<0$ for all $j \le k$. 
Then we have
\begin{align*}
-\sum_{j=1}^k x_j \kaehl_j &\ge \max_{j=1,\ldots,k}(-x_j) \cdot \min_{j=1,\ldots,k} \kaehl_j \\
&\ge \min_{j=1,\ldots,k} \kaehl_j \cdot \frac{1}{\sqrt{k}}\sqrt{\sum_{j=1}^k x_j^2} \\
&\ge \min_{j=1,\ldots,k} \kaehl_j \cdot \frac{1}{\sqrt{k}} \cdot c \left|\sum_{j=1}^k x_j \partial_{x_j}\right|.
\end{align*}
The last line uses the fact that $\sqrt{\sum_j x_j^2} = |\sum_j x_j \partial_{x_j}|'$ with respect to the Euclidean metric $|\cdot|'$, which is uniformly equivalent to $|\cdot|$.
This gives the desired lower bound for \eqref{eq:des_bd}.
\end{proof}

\begin{cor}\label{cor:comb_ang}
Suppose that $\varphi_0$ is a compactification K\"ahler potential, $\varphi_1$ a homogeneous potential of weight $2$, and $\varphi_s = (1-s) \varphi_0 + s \varphi_1$. 
Then $\angle (E, \nabla \varphi_s ) \le c <\pi/2$, for some $c$ independent of $s$, where the angle and gradient are taken with respect to some fixed $2$-homogeneous metric.
\end{cor}
\begin{proof}
Follows immediately from Lemmas \ref{lem:comp_ang} and \ref{lem:ZpropE}.
\end{proof}

\section{Quasi-isomorphisms which respect gradings}\label{sec:qiso_grad}

The goal of this appendix is to show that the quasi-isomorphism $\cA_0 \simeq \cB_0$ can be made to respect the $\Z \oplus M^*$-gradings on both categories. 

\subsection{\texorpdfstring{$A_\infty$}{A infinity} terminology} \label{subsec:Ainf_term}
We follow the conventions concerning $A_\infty$ categories from \cite{Sheridan2015a}.  Let $\Bbbk$ be a field.  Recall that a pre-$A_\infty$ category $\calC$ consists of a collection of objects,  and a graded $\Bbbk$-vector space $hom^{*}(X_0,X_1)$ for each pair of objects $X_0, X_1.$ Set 
\begin{align*} \mathcal{C}(X_0,X_1,\cdots, X_s):= hom^{*}(X_0,X_1)[1]\otimes \cdots \otimes hom^{*}(X_{s-1},X_s)[1].  \end{align*} Define the Hochschild cochains (of length $s$) to be given by:
 \begin{align*} CC^*(\calC)^s:= \prod_{X_{0},  X_{1}, \cdots, X_{s}} Hom^*(\mathcal{C}(X_0,X_1,\cdots, X_s),  \cC(X_0,X_s))[-1] \end{align*} We then set \begin{align*} CC^*(\calC):= \prod_{s\geq 0}  CC^*(\calC)^s  \end{align*}  Given $a_1\otimes \cdots \otimes a_s \in \mathcal{C}(X_0,X_1,\cdots, X_s)$,  it is convenient to introduce the notation \begin{align} \epsilon_j(a_1\otimes \cdots \otimes a_s)= |a_1|+\cdots +|a_j|-j \in \mathbb{Z}/2\mathbb{Z} \end{align}  

Given two Hochschild cochains $\phi, \psi$,  their Gerstenhaber product is defined by
\begin{align} \label{eq:Hochdif} \phi \circ \psi(a_1, \cdots, a_s):= \sum_{j,k}(-1)^{(|\psi|-1) \cdot \epsilon_j} \phi^*(a_1,\cdots,  \psi^*(a_{j+1},  \cdots),a_{k+1}, \cdots), \end{align} and their Gerstenhaber bracket is given by
\begin{align*} \label{eq:Gerstenhaberb} [\phi, \psi]:= \phi \circ \psi-(-1)^{(|\phi|-1)(|\psi|-1)}\psi \circ \phi. \end{align*}
This bracket operation turns $CC^*(\calC)$ into a graded Lie algebra.   
\begin{defn} An $A_\infty$ category is a pre-$A_\infty$ category $\calC$ equipped with an $A_\infty$-structure,  that is an element $\mu \in CC^2(\calC)$ with $\mu \circ \mu=0$ and $\mu^0=0$.  \end{defn}  

Given an $A_\infty$ category $(\calC,\mu)$,  we define the Hochschild cohomology differential  to be given by \begin{align} d_\mu: CC^*(\calC) \to CC^{*+1}(\calC) \\ d_{\mu}(-):=[\mu, -].  \nonumber \end{align} The Hochschild cohomology is by definition $HH^*(\calC,\mu) := H^*((CC^*(\calC),[\mu,-])).$

\begin{rmk} \label{rem:Gdegree} Suppose that $\calC$ carries a grading by some free abelian group $G$ and that the $A_\infty$ structure $\mu$ respects this grading. We say that a Hochschild cochain $\phi$ has $G$-grading $g \in G$ if whenever $a_1,\cdots, a_s$ are homogeneous elements,  \begin{align} \operatorname{deg}_G(\phi(a_1,\cdots, a_s))= \operatorname{deg}_G(a_1)+\cdots + \operatorname{deg}_G(a_s)+g.  \end{align} It is elementary to see that Hochschild cochains of cohomological $G$-degree $g$ define a subcomplex \begin{align}CC^{*}_{g}(\calC,\mu) \subset CC^{*}(\calC,\mu).\end{align}  We let $HH^{*}_{g}(\calC,\mu):=H^*(CC^{*}_{g}(\calC,\mu))$. Furthermore, if $\mathcal{C}$ has finite dimensional hom-spaces, we have a decomposition: \begin{align} \label{eq:gradednongradedHH1} CC^*(\mathcal{C}, \mu) \cong \prod_{g \in G} CC^*_{g}(\mathcal{C},\mu). \end{align} \end{rmk}

\begin{defn} Let $\calC, \tilde{\calC}$ be two pre-$A_\infty$ categories. A pre-functor $F: \calC \to \tilde{\calC}$ is an assignment on the level of objects together with a collection of linear maps \begin{align}  \label{eq:multilinearmaps} F^s: \calC(X_0, \cdots,  X_s) \to \tilde{\calC}(F(X_0),  F(X_s)) \end{align} for every sequence of objects $X_0,\cdots, X_s \in \operatorname{Ob}(\mathcal{C})$ (with $s \geq 1$).  \end{defn}

 A strict formal diffeomorphism from $\calC$ to itself is a pre-functor $F:\calC \to \calC$ such that the assignment on objects is the identity and $F^1=\operatorname{id}$.\footnote{This is a slightly more restrictive notion of formal diffeomorphism than the one introduced in \cite[\S (1c)]{seidel2008fukaya}.}

\begin{defn} Let $(\calC, \mu)$,  $(\tilde{\calC}, \tilde{\mu})$ be two $A_\infty$-categories. An $A_\infty$ functor $F: (\calC,\mu) \to (\tilde{\calC},\tilde{\mu})$ is a pre-functor between the two pre-$A_\infty$ categories which satisfies: 
\begin{align} \label{eq:Ainffunctor}  \sum {\tilde{\mu}}^*(F^*(a_1, \cdots),  F^*(a_{j_{1}+1}\cdots),  F^*(a_{j_{k}+1}, \cdots, a_s))= \sum (-1)^{\epsilon_{j}} F^*(a_1, \cdots, \mu^*(a_{j+1},  \cdots),\cdots, a_s).  \end{align}  
\end{defn} 

 Let $(\calC, \mu)$,  $(\calC, \tilde{\mu})$ be two $A_\infty$-structures on a pre-$A_\infty$ category $\calC$.  An $A_\infty$ functor $F:(\calC, \mu) \to (\calC, \tilde{\mu})$ will be called a strict isomorphism if it arises from a strict formal diffeomorphism. Conversely, given a strict formal diffeomorphism $F$ and an $A_\infty$ structure $(\calC,\mu)$, there is a unique $A_\infty$ structure \begin{align} \label{eq: action} (\calC, F_*(\mu)) \end{align} such that $F$ defines a strict isomorphism $F: (\calC, \mu) \to (\calC,F_*(\mu))$ (\cite[\S (1c)]{seidel2008fukaya}).

\subsection{Obstructions in Hochschild cohomology} \label{subsec:Obstr_hh}
Let $\mathcal{D}$ be a $\Bbbk$-linear graded category.  To match the conventions from \cite{Sheridan2015a},  we view the compositions in $\mathcal{D}$ as being given by operations \begin{align} Hom_{\mathcal{D}}^{*}(X_1,X_2)\otimes Hom_\mathcal{D}^{*}(X_0,X_1) \to Hom_\mathcal{D}^{*}(X_0,X_2). \end{align}  

 This determines an $A_\infty$ category $(\mathcal{C},\mu_{\mathcal{C}_\mathcal{D}})$ with the same objects by setting: \begin{align*} hom_{\mathcal{C}}^*(X,Y):=Hom_{\mathcal{D}}^*(X,Y) \\ \mu_{\mathcal{C}_\mathcal{D}}^1(f)=0, \quad \mu_{\mathcal{C}_\mathcal{D}}^2(f,g):= (-1)^{|f|} g \cdot f , \quad \mu_{\mathcal{C}_\mathcal{D}}^k=0,  k \geq 3.   \end{align*} 

In this context, the $A_\infty$ operations respect the cohomological degree grading. As a consequence by \Cref{rem:Gdegree}, we can consider the Hochschild cochains of cohomological degree $t$, $CC^{*}_{t}(\calC,\mu_{\mathcal{C}_\mathcal{D}}).$ 
Going forward,  we assume that \begin{align} \label{eq:concentrd} Hom_{\mathcal{D}}^*(X,Y) \textrm{ is concentrated in degrees 0 and } d \textrm{ for any } X,Y \in Ob(\mathcal{D}) \end{align} and some $d>0.$  We then have that \begin{align} CC^{*}_{t}(\mathcal{C}, \mu_{\mathcal{C}_\mathcal{D}}) = 0 \text{ unless } t=kd.  \end{align}  

We say that $(\mathcal{C}, \mu)$ is an admissible $A_\infty$ structure on $\mathcal{C}$ if \begin{align} \mu^k= \mu^k_{\mathcal{C}_\mathcal{D}},  \quad k \leq 2.  \end{align} The following is \cite[Lemma 2.2]{Polishchuk} (stated for categories as opposed to algebras) and is easily verified by comparing \eqref{eq:Hochdif} with \eqref{eq:Ainffunctor}.  

\begin{lem} \label{lem:Hoch} Let $\mu$ and $\tilde{\mu}$ be two admissible $A_\infty$-structures on $\mathcal{C}$.  Assume that $\mu^k = \tilde{\mu}^{k}$ for $k< s$, where $s \geq 3$.  \begin{enumerate} \item Set $c(a_1,\cdots,a_s):=(\tilde{\mu}^s-\mu^s)(a_1,\cdots,a_s).$ Then $c \in CC^{2}_{2-s}(\mathcal{C},  \mu_{\mathcal{C}_\mathcal{D}})$ is a Hochschild cocycle.  \item Suppose $\tilde{\mu}= F_*(\mu)$ (recall \eqref{eq: action}) for some formal diffeomorphism $F$.  Then $c$ is a Hochschild coboundary, i.e.,  we have \begin{align} [c]=0 \in HH^2_{2-s}(\mathcal{C}, \mu_{\mathcal{C}_\mathcal{D}}).  \end{align}  \end{enumerate}  \end{lem} 

\begin{prop} \label{prop:obstructions} 
Suppose $\mathcal{D}$ is a proper, graded $\Bbbk$-linear category satisfying \eqref{eq:concentrd}  and  \begin{align} HH^{2}_{-kd}(\mathcal{C}, \mu_{\mathcal{C}_\mathcal{D}})=0,  \quad k>1. \end{align} Suppose further that  $\mathcal{C}$ carries a $G$-grading and that $\mu, \tilde{\mu}$ are two admissible $A_\infty$ structures which preserve this grading. If there is a strict isomorphism \begin{align} F:(\mathcal{C},\mu) \xrightarrow{\sim} (\mathcal{C},\tilde{\mu}),  \end{align} then there exists a strict isomorphism \begin{align} \label{eq:strictisomorphism} \tilde{F} :(\mathcal{C},\mu) \xrightarrow{\sim} (\mathcal{C},\tilde{\mu}).\end{align} which preserves the $G$-grading.   \end{prop} 

\begin{proof} To begin, notice that because $\mathcal{C}$ has finite dimensional hom-spaces, by \eqref{eq:gradednongradedHH1} we have a decomposition: \begin{align} \label{eq:gradednongradedHH2} CC^*_{-kd}(\mathcal{C},\tilde{\mu}) \cong \prod_{g \in G} CC^*_{g \oplus -kd}(\mathcal{C},\tilde{\mu})\end{align}
where $CC^*_{g \oplus -kd}(\mathcal{C},\mu_{\mathcal{C}_\mathcal{D}})$ refers to Hochschild cochains with $G$-grading $g \in G$ and cohomological grading $-kd.$ It follows in particular that $HH^{2}_{0\oplus-kd}(\mathcal{C},\mu_{\mathcal{C}_\mathcal{D}})=0$  for $k>1$.

We follow the inductive argument on the length $s$ in the proof of  \cite[Theorem 1.1 (i)]{Polishchuk}  (page 11 of \emph{loc.  cit.}).  The only difference is that we carry out the argument above in $CC^*_{0 \oplus -kd}(\mathcal{C},\mu_{\mathcal{C}_\mathcal{D}})$ to construct $\tilde{F}$ to preserve gradings. Note that we have $\tilde{F}^s = 0$ unless $s=1+kd$ for some $k$.  Let us begin with $k=1$.  By Lemma \ref{lem:Hoch} (2),  we have $\partial_{\mu_{\mathcal{C}_\mathcal{D}}} F^{1+d} = \tilde{\mu}^{d+2} - \mu^{d+2}$.  Because $\mu$, $\tilde{\mu}$, and $\partial_{\mu_{\mathcal{C}_\mathcal{D}}}$ respect the $G$-grading, we have \begin{align} \partial_{\mu_{\mathcal{C}_\mathcal{D}}} \tilde{F}^{1+d} = \tilde{\mu}^{d+2} - \mu^{d+2} \end{align} for some $\tilde{F}^{1+d} \in CC^*_{(0 \oplus -kd)}(\mathcal{C},\mu_{\mathcal{C}_\mathcal{D}}).$ Putting it together with $\tilde{F}^1 = F^1=\id$, $\tilde{F}^{1+d}$ defines a graded formal diffeomorphism $ \tilde F$ such that $(\tilde F _*\mu)^s = \tilde \mu^s$ for $s \le d+2$.

Now we take $k>1$ and proceed by induction.  For a given $k$,  we now are in the situation of Lemma \ref{lem:Hoch} where there is an equivariant formal diffeomorphism which makes the $A_\infty$ structures agree up to operations of length $2+(k-1)d$.  As $HH^{2}_{(0 \oplus -kd)}(\mathcal{C},\mu_{\mathcal{C}_\mathcal{D}}) = 0$,  we can apply Lemma \ref{lem:Hoch} to find $\tilde{F}^{1+kd} \in CC^*_{(0 \oplus -kd)}(\mathcal{C},\mu_{\mathcal{C}_\mathcal{D}})$ which makes the $A_\infty$ structures agree to order $s=2+kd$.  Composing this sequence of formal diffeomorphisms produces the desired $\tilde{F}.$ Note that the infinite product of formal diffeomorphisms involved in this construction converges because the length increases at each stage. \end{proof} 

\subsection{Obstructions for \texorpdfstring{$H^*(\mathcal{B}_0)$}{H*B0}} 

We now compute the obstructions from Section \ref{subsec:Obstr_hh} in the case where $\mathcal{D}= H^*(\mathcal{B}_0)$. 
Note that this category satisfies \eqref{eq:concentrd} with $d=n$ (cf. \Cref{lem:structure_A0}). 
Our main computation is as follows:

\begin{prop} \label{lem:HHcomp} We have
\begin{align} HH^{j}_{-kn}(\mathcal{C}, \mu_{\mathcal{C}_\mathcal{D}})=0 \end{align}  
for all $j<2k$.\end{prop} 

In order to prove \Cref{lem:HHcomp}, we introduce the `extended homogeneous coordinate ring': \begin{align} A_L:=\bigoplus_{p \oplus q \in \mathbb{Z}\oplus \mathbb{Z}} \operatorname{Hom}^q(\mathcal{O},\mathcal{O}(p)),  \end{align} with its natural ring structure.  Because $A_L$ is graded by $p \oplus q \in \mathbb{Z}\oplus \mathbb{Z}$, we can consider the $\Z \oplus \Z$-graded Hochschild cochains, $CC^{*}_{p \oplus q}(A_L).$ 
Our proof of \Cref{lem:HHcomp} is modelled on the following result of Polishchuk:

\begin{thm}[Theorem 3.3 of \cite{Polishchuk}] \label{thm:Polishchuk}
We have
    \begin{align}
        HH^{j}_{0\oplus -kn}(A_L)=0 \end{align} 
        for all $j<2k$.
\end{thm}

\begin{proof}[Proof of \Cref{lem:HHcomp}] We reduce the computation to that of \Cref{thm:Polishchuk}. 
As both $A_L$ and $(\mathcal{C}, \mu_{\mathcal{C}_\mathcal{D}})$ are strictly unital,  it suffices to work with reduced Hochschild cochains, which we denote by $\bar{CC}$.  We let \begin{align} C^*_A:=\bar{CC}^*_{0 \oplus -kn}(A_L) \\ C^*_{\mathcal{C}}:=\bar{CC}^*_{-kn}(\mathcal{C}, \mu_{\mathcal{C}_\mathcal{D}}).  \nonumber \end{align} 
For a given $i \in \mathbb{Z}$,  the Hochschild cochains $\psi \in C^*_{\mathcal{C}}$ with output in  $hom^*(\mathcal{O}(i),\mathcal{O}(i'))$ are identified with $C^*_A.$ This gives rise to an identification of cochains \begin{align} \label{eq:decomposition} C^*_{\mathcal{C}} \cong \prod_{i \in \mathbb{Z}} C^*_A. \end{align} 
For the moment, this is only a decomposition of cochains and not of complexes because the Hochschild differential does not respect this product decomposition. 

Next,  as in the proof of \cite[Theorem 3.3]{Polishchuk},  we decompose \begin{align} C^*_{A} = C^*_{A}(0) \oplus C^*_{A}(n) \\ C^*_{\mathcal{C}} = C^*_{\mathcal{C}}(0) \oplus C^*_{\mathcal{C}}(n) \nonumber \end{align} where $C^*_{A}(0)$,  $C^*_{\mathcal{C}}(0)$ are the vector subspaces of maps whose output lies in degree $0$ and $C^*_{A}(n),  C^*_{\mathcal{C}}(n)$ those whose output lies in degree $n.$  $C^*_{\mathcal{C}}(n)$ is a subcomplex and we give $C^*_{\mathcal{C}}(0)$ the structure of a complex via the exact sequence \begin{align} 0 \to C^*_{\mathcal{C}}(n) \to C^*_{\mathcal{C}} \to C^*_{\mathcal{C}}(0) \to 0. \end{align} 
The same considerations also apply to $C^*_{A}(0)$ and $C^*_{A}(n).$ It suffices to prove that \begin{align} \label{eq:vanishing} H^j(C^*_{\mathcal{C}}(0))= H^j(C^*_{\mathcal{C}}(n))=0 \end{align} for $j<2k.$ Let us first consider the case of $C^*_{\mathcal{C}}(0).$ For any $r \geq 0$,   let $C^*_{\mathcal{C}}(0)_r \subset C^*_{\mathcal{C}}(0)$ denote the subspace of maps whose output lies in $hom(\mathcal{O}(i),\mathcal{O}(i+s))$ for some $i \in \mathbb{Z}$ and $s \ge r$.   Let $C^*_{A}(0)_r \subset C^*_{A}(0)$ denote the subspace of maps whose output lies in $Hom(\mathcal{O},\mathcal{O}(s))$ for $s\geq r$. These define filtrations of complexes and we let  $C^*_{\mathcal{C}}(0)^{gr}$ and $C^*_{A}(0)^{gr}$ denote the associated graded complexes.  It is easy to see that the piece of the Hochschild differential which fails to respect the product decomposition \eqref{eq:decomposition} strictly increases the $r$ filtration.  As a consequence,  we have a product decomposition of complexes: \begin{align} \label{eq:decomposition2} C^*_{\mathcal{C}}(0)^{gr} \cong \prod_{i \in \mathbb{Z}} C^*_A(0)^{gr}. \end{align}   Polishchuk proves that $C^*_A(0)^{gr}$ is acyclic in the relevant ranges and so the same applies to $C^*_{\mathcal{C}}(0)^{gr}.$  The vanishing  \eqref{eq:vanishing} then follows from a spectral sequence argument, using the fact that the filtration is manifestly bounded above and complete.

For $C^*_{\mathcal{C}}(n)$,  we similarly consider the natural extension of the filtration of $C^p_A(n)$ used in \cite[Proof of Theorem 3.3]{Polishchuk}.  Namely,  given a simple tensor $a_1 \otimes \cdots \otimes a_s \in hom^{q_1}(X_0,X_1)[1]\otimes \cdots \otimes hom^{q_s}(X_{s-1},X_s)[1],$ we let $r_0$ be the number of consecutive degree-zero components $a_1,\hdots, a_{r_{0}}$ at the beginning of the tensor and $r_1$ be the number of consecutive degree-zero components $a_{s-r_1+1},\hdots,a_s$ at the end of the tensor. Let $C^*_{\mathcal{C}}(n)_r$ denote the subcomplex of Hochschild cochains which vanish on inputs with $r_0+r_1 \le r$. As above, it is not difficult to see that the associated graded complex $C^*_{\mathcal{C}}(n)^{gr}$ becomes a direct product of $\mathbb{Z}$-copies of $C^*_A(n)^{gr}$, which Polishchuk proves to be acyclic. Again the filtration is bounded above and complete, so the vanishing \eqref{eq:vanishing} follows from a spectral sequence argument.  \end{proof} 

\begin{cor} \label{cor:vanish} We have \begin{align} HH^{2}_{-kn}(\mathcal{C}, \mu_{\mathcal{C}_\mathcal{D}})=0\end{align} 
for all $k>1$. 
\end{cor} 
\begin{proof} Apply Proposition \ref{lem:HHcomp} with $j=2.$  \end{proof}

We now have everything we need to prove our main cochain-level result:

\begin{lem} \label{lem:Mgradedqi}
    There exists a $\Z \oplus M^*$-graded quasi-isomorphism $\cA_0 \simeq \cB_0$ compatible with the isomorphism $H^*(\cA_0) \cong H^*(\cB_0)$ from \Cref{lem:grading_on_coh_level}.
\end{lem}
\begin{proof}
  We begin by constructing minimal models $\cA_{0,min} \simeq \cA_0$ and $\cB_{0,min}\simeq \cB_0$ using \cite[Proposition 1.12]{seidel2008fukaya}. The minimal model depends on a choice of contracting homotopy for each morphism space. 
We choose these contracting homotopies to respect $\mathbb{Z} \oplus M^*$-gradings. The resulting minimal models $\cA_{0,min}$, $\cB_{0,min}$ will then be $\mathbb{Z}\oplus M^*$-graded $A_\infty$ categories and the quasi-isomorphisms $\cA_{0,min} \simeq \cA_0$, $\cB_{0,min}\simeq \cB_0$ will be $\mathbb{Z}\oplus M^*$-graded. 

The quasi-isomorphism $\cA_0 \simeq \cB_0$ from \S \ref{sec:A0_B0} gives rise to a $\Z$-graded quasi-isomorphism $\cA_{0,min} \simeq \cB_{0,min}$ of these minimal models. We can view this quasi-isomorphism of minimal models as defining two quasi-isomorphic admissible $A_\infty$ structures on $(\mathcal{C}, \mu_{\mathcal{C}_{H^*(\cB_{0})}})$, the $A_\infty$ category corresponding to $H^*(\cB_0).$ By Corollary \ref{cor:vanish}, $$HH^{2}_{-kn}(\mathcal{C}, \mu_{\mathcal{C}_{H^*(\cB_{0})}})=0 $$  for all $k>1$. The result then follows from Proposition \ref{prop:obstructions}.
\end{proof}

\renewbibmacro{in:}{}
\def\bibrangedash{ -- }
\printbibliography 

\end{document}